\documentclass[11pt]{article}
\usepackage[english]{babel}
\usepackage[left=1in,right=1in,top=1in,bottom=1.5in]{geometry} 
\usepackage[skip=5pt, indent=10pt]{parskip}
\usepackage{amssymb,amsmath, amsthm}
\usepackage{hyperref}
\usepackage{xcolor}
\usepackage{nicefrac}
\usepackage[margin=1cm]{caption}
\usepackage{natbib}
\usepackage[makeroom]{cancel}
\usepackage{bbm}
\usepackage{mathrsfs}
\usepackage{comment}
\usepackage{algorithm}
\usepackage{algpseudocode}
\usepackage{enumitem}
\usepackage{times}
\usepackage{graphicx}

\definecolor{darkred}{HTML}{880000}
\definecolor{darkblue}{HTML}{000088}

\hypersetup{
  colorlinks,
  linkcolor={darkred},
  citecolor={darkblue}
}

\renewcommand{\le}{\leqslant}
\renewcommand{\leq}{\leqslant}
\renewcommand{\ge}{\geqslant}
\renewcommand{\geq}{\geqslant}
\renewcommand{\Tilde}{\widetilde}

\newcommand{\eps}{\varepsilon}
\newcommand{\ups}{\upsilon}
\newcommand{\Ups}{\Upsilon}

\newcommand{\rmd}{{\rm d}}
\newcommand{\F}{\rm F}
\newcommand{\diag}{{\rm diag}}
\newcommand{\Tr}{{\rm Tr}}
\newcommand{\Var}{{\rm Var}}
\newcommand{\rmvec}{{\rm\bf vec}}

\newcommand{\sfp}{\mathsf{p}}

\newcommand{\E}{\mathbb E}

\newcommand{\p}{\mathbb P}
\newcommand{\R}{\mathbb R}
\newcommand{\bbX}{\mathbb X}
\newcommand{\1}{\mathbbm 1}

\newcommand{\ttD}{{\mathtt{D}}}
\newcommand{\ttR}{{\mathtt{R}}}
\newcommand{\ttQ}{{\mathtt{Q}}}

\newcommand{\ttL}{{\mathtt{L}}}
\newcommand{\ttr}{{\mathtt{r}}}

\newcommand{\cE}{{\mathcal{E}}}

\newcommand{\KL}{{\mathcal{KL}}}
\newcommand{\cL}{{\mathcal{L}}}
\newcommand{\cN}{{\mathcal{N}}}
\newcommand{\cO}{{\mathcal{O}}}
\newcommand{\cQ}{{\mathcal{Q}}}
\newcommand{\cT}{{\mathcal{T}}}
\newcommand{\cU}{{\mathcal{U}}}
\newcommand{\cV}{{\mathcal{V}}}
\newcommand{\cX}{{\mathcal{X}}}

\newcommand{\m}{\mathfrak m}
\newcommand{\z}{\mathfrak z}

\newcommand{\sfA}{\mathsf{A}}
\newcommand{\sfH}{\mathsf{H}}

\newcommand{\ba}{{\bf a}}
\newcommand{\bb}{{\bf b}}

\newcommand{\be}{{\bf e}}

\newcommand{\bt}{{\bf t}}
\newcommand{\bu}{{\bf u}}
\newcommand{\bv}{{\bf v}}
\newcommand{\bw}{{\bf w}}
\newcommand{\bx}{{\bf x}}

\newcommand{\bs}{{\bf s}}

\newcommand{\bU}{{\bf U}}
\newcommand{\bX}{{\bf X}}
\newcommand{\bY}{{\bf Y}}
\newcommand{\bZ}{{\bf Z}}
\newcommand{\bz}{{\bf \z}}

\newcommand{\bgamma}{\boldsymbol \gamma}
\newcommand{\beps}{\boldsymbol \eps}
\newcommand{\bups}{\boldsymbol \ups}
\newcommand{\btheta}{\boldsymbol \theta}
\newcommand{\bfeta}{\boldsymbol \eta}

\newcommand{\bnu}{\boldsymbol \nu}
\newcommand{\btau}{\boldsymbol \tau}
\newcommand{\bxi}{\boldsymbol \xi}
\newcommand{\bzeta}{\boldsymbol \zeta}

\newcommand{\bnabla}{\boldsymbol \nabla}
\newcommand{\bzero}{{\bf 0}}

\newcommand{\stochasticDelta}{{\Delta}}

\newcommand{\Uniform}{\operatorname{Uniform}}

\newcommand{\avg}[1]{\left \langle #1 \right \rangle}
\newcommand{\biasInner}[1]{\avg{#1}} 

\newcommand{\myendproof}{\hfill$\square$}

\def\argmin{\operatornamewithlimits{argmin}}

\newtheorem{Th}{Theorem}[section]
\newtheorem{Lem}[Th]{Lemma}

\newtheorem{Prop}[Th]{Proposition}
\newtheorem{Co}[Th]{Corollary}
\newtheorem{Rem}[Th]{Remark}

\newtheorem{As}[Th]{Assumption}

\title{Dimension-free bounds in high-dimensional linear regression\\via error-in-operator approach}

\author{
Fedor Noskov\thanks{HSE University, Russian Federation, fnoskov@hse.ru}
\and
Nikita Puchkin\thanks{HSE University, Russian Federation, npuchkin@hse.ru}
\and
Vladimir Spokoiny
\thanks{HSE University, Russian Federation; WIAS Berlin and Humboldt University, Germany, spokoiny@wias-berlin.de}
}

\date{}

\begin{document}

\maketitle

\begin{abstract}
    We consider a problem of high-dimensional linear regression with random design. We suggest a novel approach referred to as error-in-operator which does not estimate the design covariance $\Sigma$ directly but incorporates it into empirical risk minimization. We provide an expansion of the excess prediction risk and derive non-asymptotic dimension-free bounds on the leading term and the remainder. This helps us to show that auxiliary variables do not increase the effective dimension of the problem, provided that parameters of the procedure are tuned properly. We also discuss computational aspects of our method and illustrate its performance with numerical experiments.
\end{abstract}

\tableofcontents
\newpage

\section{Introduction}

Recent advances in supervised machine learning devoted to understanding of deep neural networks revealed surprising effects going beyond the classical statistics theory. In contrast to the standard intuition that a learner should search for a trade-off between approximation and estimation errors, researchers empirically observed that large interpolating rules may still have small test error. Moreover, when the number of parameters exceeds sample size the prediction risk of neural networks passes a U-shaped curve and decreases again \citep{zhang17, nakkiran20}. A bit later it became clear that benign overfitting and double descent are not distinctive features of deep learning. Similar phenomena are ubiquitous for overparametrized models such as random forests and random feature models \citep{belkin19, mei22}, kernel methods \citep{belkin18, liang20}, and linear regression \citep{bartlett2020benign, hastie2022surprises} to name a few.
In \citep{belkin18}, the authors reasonably suggested that we must study more tractable ``shallow'' methods better before diving into deep learning theory.

In the present paper, we consider a classical linear regression problem, where a learner aims to estimate an unknown vector $\btheta^\circ \in \R^d$ from i.i.d. pairs $\{(\bX_i, Y_i) : 1 \leq i \leq n\} \subset \R^d \times \R$ generated from the model
\begin{equation}
    \label{eq:model}
    Y_i = \bX_i^\top \btheta^\circ + \eps_i,
    \quad 1 \leq i \leq n.
\end{equation}
We do not impose any assumptions on the ambient dimension $d$ (in particular, it may be much larger than $n$) or on the structure of $\btheta^\circ$, but we require $\E \|\bX_1\|^2$ to be finite. Given an independent copy $(\bX, Y)$ of $(\bX_1, Y_1)$, the quality of an estimate $\widetilde \btheta$ is measured with its quadratic risk
\begin{align*}
    R(\widetilde\btheta)
    &
    = \E_{(\bX, Y)} (Y - \bX^\top \widetilde\btheta)^2
    \\&
    = \E_{(\bX, Y)} (Y - \bX^\top \btheta^\circ)^2 + \left\|\Sigma^{1/2} (\widetilde\btheta - \btheta^\circ) \right\|^2,
    \quad \text{where} \quad
    \Sigma = \E \bX \bX^\top \in \R^{d \times d}.
\end{align*}
We can rewrite  \eqref{eq:model} in a more compact vector form
\begin{equation}
    \label{eq:model_vector}
    \bY = \bbX^\top \btheta^\circ + \beps,
\end{equation}
where $\bY = (Y_1, \dots, Y_n)^\top$, $\beps = (\eps_1, \dots, \eps_n)^\top$, and $\bbX \in \R^{d \times n}$ is a matrix with columns $\bX_1, \dots, \bX_n$. The problem \eqref{eq:model_vector} provides a convenient framework for studying complex high-dimensional phenomena. For this reason, it has recently got considerable attention.
Researchers put a lot of efforts into careful and subtle analysis of popular minimum norm least squares and ridge regression estimates defined as
\[
    \widehat\btheta{}^{(LS)} = (\bbX \bbX^\top)^\dag \, \bbX \bY
    \quad \text{and} \quad
    \widehat\btheta{}^{(R)} = (\bbX \bbX^\top + \tau I_d)^{-1} \bbX \bY,
\]
respectively. Here and further in the paper, $B^\dag$ stands for the Moore-Penrose pseudoinverse of a matrix $B$.
In \citep{bartlett2020benign, chinot2022robustness, chinot22b}, the authors relied on concentration of measure and statistical learning theory and established dimension-free rates of convergence for $\widehat\btheta{}^{(LS)}$ depending on the eigenvalues of $\Sigma = \E \bX_1 \bX_1^\top$ rather than on the ambient dimension $d$. In the same spirit, \cite{tsigler2023benign} derived non-asymptotic dimension-free bounds on the squared prediction risk of $\widehat\btheta{}^{(R)}$ and supported observations of \cite{kobak2020optimal} that the optimal value of the penalization parameter $\tau$ may be negative due to the implicit regularization. Another line of research (see, for instance, \citep{dobriban2018high, el2018impact, taheri2021fundamental, richards2021asymptotics, mahdaviyeh2019risk, derezinski2020exact, hastie2022surprises, wu20, bach2024high}) examined the least squares and ridge regression estimates through the lens of random matrix theory and studied $\widehat\btheta{}^{(LS)}$ and $\widehat\btheta{}^{(R)}$ in the so-called proportional regime, that is, when the ratio $d / n$ tends to some constant $\gamma \in (0, \infty)$ as $n$ approaches infinity. We would like to emphasize that the aforementioned approaches complement each other. While \cite{bartlett2020benign} argue that $\E \|\bX_1\|^2 = \Tr(\Sigma) = o(n)$, $n \rightarrow \infty$, is a necessary condition for benign overfitting, in the proportional regime one usually assumes the opposite: $\Tr(\Sigma) = \Omega(n)$. However, in a recent paper \citep{cheng2022}, the authors combined resolvent analysis with concentration inequalities and obtained non-asymptotic risk expansion for the ridge regression estimate. 

In the context of overparametrized models, explicit regularization has a very nice property previously unseen in the classical underparametrized setup. In \citep{nakkiran21}, the author noticed that optimally regularized least squares estimate can mitigate double descent. This fact is of great importance, because in the presence of double (or, more generally, multiple descent, see, for instance, \citep{chen21, meng24}) adding new batches to the training data may lead to worse generalization error (see \citep[Section 7]{nakkiran20}).
In the present paper, we study properties of an estimate with double regularization. Let us note that both $\widehat\btheta{}^{(LS)}$ and $\widehat\btheta{}^{(R)}$ have a form $(\breve \Sigma)^\dag \bZ$, where $\bZ = \bbX \bY / n$ and $\breve \Sigma$ is an estimate of $\Sigma = \E \bX_1 \bX_1^\top$. Instead of estimating $\Sigma$ directly, we incorporate learning of $\Sigma$ into an optimization problem. To be more specific, we introduce an objective function
\begin{equation}
    \label{eq:eio_objective}
    \ttL(\bups)
    =
    \ttL(\btheta, \bfeta, A)
    = \frac12 \|\bZ - \bfeta\|^2 + \frac12 \|\bfeta - A \btheta\|^2 + \frac{\mu^2}2 \|\widehat \Sigma - A\|_{\F}^2 + \frac{\lambda}2 \|\btheta\|^2,
\end{equation}
where
\[
    \bZ = \frac1n \bbX \bY,
    \quad
    \widehat\Sigma = \frac1n \bbX \bbX^\top,
\]
and consider a regularized empirical risk minimizer
\begin{equation}
\label{eq: initial optimization problem}
    \widehat \bups
    = (\widehat \btheta, \widehat \bfeta, \widehat A) 
    \in \argmin\limits_{(\btheta, \bfeta, A)} \ttL(\btheta, \bfeta, A).
\end{equation}
Similar ideas were used in linear inverse problems \citep{hoffmann08, trabs18} and Bayesian inference \citep{spokoiny23}.
In the context of kernel ridge regression, \cite{chen23} replaced the standard Euclidean distances $\|x - x'\|$ by $\|\Sigma^{1/2} (x - x')\|$ and considered a family of Gram matrices parametrized by $\Sigma$. However, this is very different from our setup, because, in contrast to \citep{hoffmann08, trabs18, chen23}, we do not assume any low-dimensional parametrization of the operator $A \in \R^{d \times d}$. This means that $\widehat\bups$ estimates a parameter $\bups^\circ = (\btheta^\circ, \Sigma \btheta^\circ, \Sigma)$ of much larger dimension than $d$. At the first glance, we dramatically increased the problem complexity. However, let us recall that we are interested in the target parameter $\btheta^\circ \in \R^d$ rather than in the whole vector $\bups^\circ \in \R^{2d + d^2}$. In the simple case $\mu = +\infty$, it is straightforward to check that the estimate $\widehat\btheta$ corresponds to the plug-in one
\[
    \widehat\btheta_{\infty} = (\widehat\Sigma^2 + 2 \lambda I_d)^{-1} \widehat \Sigma \bZ,
    \quad \text{where} \quad
    \bZ = \frac1n \bbX \bY
    \quad \text{and} \quad
    \widehat \Sigma = \frac1n \bbX \bbX^\top = \frac1n \sum\limits_{i = 1}^n \bX_i \bX_i^\top,
\]
which is nothing but the ridge regression estimate applied to the model
\begin{equation}
    \label{eq:u}
    \bZ = \widehat \Sigma \btheta^\circ + \bU,
    \quad \text{where} \quad
    \bU = \frac1n \bbX \beps.
\end{equation}
Hence, if we take the parameter $\mu$ sufficiently large, the rate of convergence of $\widehat\btheta$ will not suffer (though a rigorous proof of this fact is quite challenging). Moreover, additional regularization with respect to operator helps $\widehat\btheta$ to be less susceptible to double descent. Following \citep{spokoiny23}, we call $\widehat\btheta$ \emph{error-in-operator estimate}.

\medskip

\noindent\textbf{Contribution.}
\quad
The main contributions can be summarized as follows.
\begin{itemize}
    \item We suggest a novel approach for the problem of random design linear regression. We specify leading terms in the bias and variance expansions of the error-in-operator estimate $\widehat\btheta$ and derive an explicit high-probability upper bound on the remainder. In contrast to the standard ridge regression estimate, $\widehat\btheta$ does not have a simple closed-form expression. For this reason, we have to use more sophisticated tools for its analysis.
    \item We establish non-asymptotic dimension-free bounds on the excess quadratic risk $R(\widehat\btheta) - R(\btheta^\circ) = \|\Sigma^{1/2} (\widehat\btheta - \btheta^\circ)\|^2$ and show that the influence of the introduced operator $A$ is negligible once the parameter $\mu$ is tuned properly.
    \item We discuss minimization the non-convex objective \eqref{eq:eio_objective} and prove that a simple alternating optimization procedure converges to the error-in-operator estimate $\widehat\btheta$.
    \item We illustrate the validity of statistical properties of $\widehat\btheta$ with numerical experiments. We also show that the generalization error of the error-in-operator estimate is smaller than the ones of $\widehat\btheta{}^{(R)}$ and $\widehat\btheta_\infty$ and that the risk of $\widehat \btheta$ is less susceptible to double descent.
\end{itemize}

\medskip

\noindent\textbf{Notation.}
\quad
Throughout the paper, $\|B\|_{\F}$ and $\|B\|$ stand for the Frobenius and the operator norm of $B$, respectively. If the matrix $B$ is symmetric positive semidefinite and $B \neq O$, where $O$ is the matrix with zero entries, we denote its effective rank by
\[
    \ttr(B) = \Tr(B) / \|B\|.
\]
For any matrices $A \in \R^{p \times q}$ and $B \in \R^{r \times s}$, their Kronecker product $A \otimes B$ is a matrix of size $pr \times qs$, defined as
\[
    \begin{pmatrix}
        a_{11} B & \dots & a_{1q} B \\
        \vdots & \ddots & \vdots \\
        a_{p1} B & \dots & a_{pq} B. 
    \end{pmatrix}
\]
For any matrix $U \in \R^{p \times q}$ with columns $\bu_1, \dots, \bu_q$, the vectorization operator is given by $\rmvec(U) = (\bu_1^\top, \dots, \bu_q^\top)^\top \in \R^{pq}$. Here and further in the paper, the bold font is reserved for vectors, while matrices and scalars are displayed in regular font. The operator $\nabla_A$ stands for the gradient with respect to $\rmvec(A^\top)$. Thus, if $\ba_1, \dots, \ba_d$ are the rows of $A$ and $f(A)$ is a smooth function, then $\nabla_A f = (\bnabla_{\ba_1} f^
\top, \dots, \bnabla_{\ba_d} f^\top)^\top \in \R^{d^2}$.
For a random variable $\xi$ and a random vector $\bnu$, their Orlicz $\psi_s$-norms, $s \geq 1$, are defined as
\[
    \|\xi\|_{\psi_s} = \inf\left\{ t > 0 : \E e^{|\xi|^s / t^s} \leq 2 \right\}
    \quad \text{and} \quad
    \|\bnu\|_{\psi_s} = \sup\limits_{\|\bu\| = 1} \|\bu^\top \bnu\|_{\psi_s}.
\]
The expressions $(a \vee b)$ and $(a \wedge b)$ denote $\max\{a, b\}$ and $\min\{a, b\}$, respectively.
Sometimes, instead of the standard $\cO$ notation, we use $f \lesssim g$ or $g \gtrsim f$, which mean that there is a universal constant $c > 0$, such that $f \leq c g$.

\medskip

\noindent\textbf{Paper organization.}
\quad
The rest of the paper is organized as follows. In Section~\ref{sec:statistical_properties}, we present our main results concerning statistical properties of the error-in-operator estimate $\widehat\btheta$. In Section \ref{sec:computational_aspects}, we discuss computational aspects of optimizing the non-convex objective \eqref{eq:eio_objective}. We illustrate performance of the suggested procedure with numerical experiments in Section \ref{section: numerical_experiments}. Sections \ref{section: proof of the bias theorem} and \ref{sec:th_stoch_term_proof} are devoted to the proofs of bias and variance expansions of the estimate $\widehat\btheta$. Derivations of numerous auxiliary results are moved to appendix.

\section{Statistical properties of the error-in-operator estimate}
\label{sec:statistical_properties}

In the present section, we discuss main statistical properties of the error-in-operator estimate \eqref{eq: initial optimization problem}. In particular, we specify leading terms in expansion of $\widehat \btheta$ and provide an upper bound on its quadratic risk.
We assume that the covariates $\bX, \bX_1, \dots, \bX_n$ meet the following requirements.
\begin{As}
    \label{as:quadratic}
    There exists a positive constant $C_X$ such that, for any matrix $M$ and any $\lambda$ satisfying the inequality $C_X |\lambda| \leq 1 / \|M\|_{\F}$, it holds that
    \[
        \log \E \exp \left\{ \lambda (\Sigma^{-1/2} \bX)^\top M (\Sigma^{-1/2} \bX) \right\}
        \leq \lambda \Tr(M) + C_X^2 \lambda^2 \|M\|_{\F}^2,
        \quad \text{where} \quad
        \Sigma = \E \bX \bX^\top.
    \]
\end{As}

In \citep{puchkin24}, the authors imposed the same assumption on the distribution of $\bX_1, \dots$, $\bX_n$ in the problem of covariance estimation, and an equivalent condition appeared in \citep{puchkin23}. We refer a reader to these two papers for a comprehensive list of distributions satisfying Assumption \ref{as:quadratic}. In particular, it is fulfilled when $\bX$ has a Gaussian distribution or when $\Sigma^{-1/2} \bX$ has independent sub-Gaussian components. Such conditions often appear in the context of high-dimensional linear regression (see, for instance, \citep{bartlett2020benign, nakkiran21, chinot2022robustness, bach2024high}). More interesting examples include situations when $\bX$ has a convex concentration property, like in \citep{cheng2022}, or when the underlying distribution meets the logarithmic Sobolev inequality (such as the uniform distribution on a sphere \citep{meng24}). Finally, Assumption \ref{as:quadratic} holds in nonlinear feature models of the form $\bX = \varphi(W \bxi)$, where $\varphi$ is a Lipschitz map and $\bxi$ has either a standard Gaussian distribution \citep{hastie2022surprises} or the uniform distribution on a unit sphere \citep{mei22} (the distribution of $\bX$ should be considered conditionally on $W$). This follows from the fact that $\bxi$ with Gaussian  or uniform distribution on a sphere satisfies the logarithmic Sobolev inequality. Then, due to the Herbst argument, it possesses a Lipschitz concentration property. This yields that $\bX$ has the same property as well, and, hence, it satisfies the Hanson-Wright inequality (see \citep{adamczak15}). Let us note that it is enough to check Assumption \ref{as:quadratic} for symmetric matrices only. Indeed, due to the convexity of the squared Frobenius norm, we have
\[
    \left\|0.5 M + 0.5 M^\top \right\|_{\F}^2 \leq 0.5 \|M\|_{\F}^2 + 0.5 \|M^\top\|_{\F}^2 = \|M\|_{\F}^2,
\]
and then it holds that
\begin{align*}
    \log \E \exp \left\{ \lambda (\Sigma^{-1/2} \bX)^\top M (\Sigma^{-1/2} \bX) \right\}
    &
    = \log \E \exp \left\{ \lambda (\Sigma^{-1/2} \bX)^\top (0.5 M + 0.5 M^\top) (\Sigma^{-1/2} \bX) \right\}
    \\&
    \leq \lambda \Tr(0.5 M + 0.5 M^\top) + C_X^2 \lambda^2 \|0.5 M + 0.5 M^\top\|_{\F}^2
    \\&
    \leq \lambda \Tr(M) + C_X^2 \lambda^2 \|M\|_{\F}^2.
\end{align*}
Assumption \ref{as:quadratic} means that the vector $\Sigma^{- 1/2}\bX$ is sub-Gaussian and fulfils the $\psi_2$-$L_2$-equivalence condition (see \citep[Lemma B.2]{puchkin24}) often arising in high-dimensional statistics. However, we would like to mention that in \citep[Proposition 2.3]{puchkin23} the authors gave an example of a random vector satisfying the $\psi_2$-$L_2$-equivalence requirement, which does not meet Assumption \ref{as:quadratic}. We also impose the following condition on the noise distribution.
\begin{As}
    \label{as:noise}
    The i.i.d. pairs $\{(\bX_i, \eps_i) : 1 \leq i \leq n\}$ are such that
    \[
        \E (\eps_1 \,\vert\, \bX_1) = 0,
        \quad 
        \E \bX_1 \bX_1^\top = \Sigma,
        \quad \text{and}
        \quad \|\Sigma^{-1/2} \bX_1 \eps_1\|_{\psi_1}
        = \sigma < +\infty.
    \]
\end{As}
For instance, Assumption \ref{as:noise} holds if $\eps_1$ is Gaussian (or, more generally, sub-Gaussian) and independent of $\bX_1$ \citep{hsu12, nakkiran21}. However, we do not need independence of $\eps_i$ from $\bX_i$ for our purposes.

Two main results of this section are devoted to bias and variance of the estimate $\widehat\btheta$ defined in \eqref{eq: initial optimization problem}. For this reason, let us define a population counterpart of the objective \eqref{eq:eio_objective} given by
\begin{align*}
    \cL(\bups)
    &
    = \frac12 \|\Sigma \btheta^\circ - \bfeta\|^2 + \frac12 \|\bfeta - A \btheta\|^2 + \frac{\mu^2}2 \|\Sigma - A\|_{\F}^2 + \frac{\lambda}2 \|\btheta\|^2
\end{align*}
and introduce the best parametric fit
\begin{equation}
    \label{eq:ups_star}
    \bups^* = (\btheta^*, \bfeta^*, A^*) \in \argmin\limits_{\bups} \cL(\bups).
\end{equation}
One of the key properties of the functional $\cL(\bups)$ we will often exploit in our analysis is that its gradient differs from $\bnabla \ttL(\bups)$ by a parameter-free vector $\z$:
\[
    \bnabla \ttL(\bups) = \bnabla \cL(\bups) - \bz,
    \quad \text{where}
    \quad \bz =
    \begin{pmatrix}
        \bzero\\
        \bZ - \E \bZ\\
        \rmvec(\widehat \Sigma - \Sigma)
    \end{pmatrix}
    \in \R^{2d + d^2}.
\]
In other words, for any $\bups$ the difference $\ttL(\bups) - \cL(\bups)$ is an affine function of $(\bZ - \E \bZ)$ and $(\widehat \Sigma - \Sigma)$. We are ready to present bias and variance expansions of the error-in-operator estimate $\widehat \btheta$.
\begin{Th}
\label{theorem: bias}
    Let the parameters $\mu$ and $\lambda$ be non-negative. Assume that the following inequalities hold:
    \[
        \Vert \btheta^\circ \Vert \le \frac{\mu}{420},
        \quad \text{and} \quad
         \Vert \btheta^\circ \Vert \Vert \Sigma \Vert \le \frac{\mu \sqrt{\lambda}}{168}.
    \]
    Then the vector $\btheta^*$ defined in \eqref{eq:ups_star} satisfies
    \begin{align*}
        \Vert \btheta^* - \btheta^\circ - \bb_\lambda  \Vert \le 210 \left (\frac{\Vert \btheta^\circ \Vert}{\mu} + \frac{\Vert \Sigma \Vert \Vert \bb_\lambda \Vert}{\mu \sqrt{\lambda}} \right ) \left \Vert \bb_\lambda \right \Vert,
        \quad \text{where} \quad
        \bb_\lambda = -\lambda (\Sigma^2/ 2 + \lambda I_d)^{-1} \btheta^\circ.
    \end{align*}
\end{Th}

\begin{Th}
    \label{th:stoch_term}
    Grant Assumptions \ref{as:quadratic} and \ref{as:noise}.
    Let us fix an arbitrary $\delta \in (0, 1)$. Suppose that $4 \ttr(\Sigma)^2 + \log(4 / \delta)$ and $2^{12} (1 + C_X)^2 (\ttr(\Sigma) + \log(2/\delta))$ do not exceed the sample size $n$. Assume that the parameters 
    $\mu \geq 112 \|\btheta^\circ\|$ and $\lambda > 0$ satisfy the inequalities
    \begin{equation}
        \label{eq:stoch_term_mu_lambda_conditions}
        \|\Sigma\| \|\btheta^\circ\| \leq \frac{\mu \sqrt{\lambda}}{8 \cdot 18 \cdot 56},
        \quad
        5 \sigma \|\Sigma\|^{1/2} \sqrt{\frac{\ttr(\Sigma) + \log(4 / \delta)}n} \leq \|\btheta^\circ\| (\sqrt{\lambda} \land \|\Sigma\|),
    \end{equation}
    \begin{equation}
        \label{eq:stoch_term_mu_lambda_conditions_2}
        17 (1 + C_X) \|\Sigma\| \sqrt{\frac{4 \ttr(\Sigma) + \log(2/\delta)}{n}} \le \sqrt{\lambda},
    \end{equation}
    and
    \begin{equation}
        \label{eq:stoch_term_mu_lambda_n_conditions}
        \mu \sqrt{\lambda} \geq 395 \cdot 8 \sigma \|\Sigma\| \sqrt{\frac{\ttr(\Sigma) + \log(4 / \delta)}n},
        \quad
        \sqrt{\lambda} \geq 395 \cdot 8 C_X \|\Sigma\| \sqrt{\frac{\ttr(\Sigma) + \log(4 / \delta)}n}.
    \end{equation}
    Then, with probability at least $(1 - 13\delta / 2)$, the vector
    \[
        \bzeta
        = \Sigma \bU - \Sigma (\widehat \Sigma - \Sigma) \bb_\lambda - (\widehat \Sigma - \Sigma) \Sigma \bb_{\lambda}
        \quad \text{with} \quad
        \bU = \frac1n \bbX \beps = \frac1n \sum\limits_{i = 1}^n \bX_i \eps_i
    \]
    satisfies the inequality
    \begin{align*}
        \left\| \Sigma^{1/2} \left(\widehat \btheta - \btheta^* - \left( \Sigma^2 + 2 \lambda I_d\right)^{-1} \bzeta \right) \right\|
        &
        \leq \frac{19 \|\btheta^\circ\|}{\sqrt{3} \lambda^{3/4}} \left(1 + \sqrt{\frac{\Psi(n, \delta)}{27 \lambda}} \right) \Psi(n, \delta)
        \\&\quad
        + \frac{\|\btheta^\circ\|}{\lambda^{1/4}} \left( \left( \frac{17 \|\btheta^\circ\|}{\mu}\right)^2 + \frac{\|\bb_\lambda\|}{160 \mu} \right) \sqrt{ \Psi(n, \delta)},
    \end{align*}
    where
    \begin{align*}
        \Psi(n, \delta) = 196 \left( \frac{\sigma \|\Sigma\|^{1/2}}{\|\btheta^\circ\|} + 2 C_X \|\Sigma\| \right)^2 \cdot \frac{\ttr(\Sigma)^2 + \log(4 / \delta)}n.
    \end{align*}
\end{Th}
The proofs of Theorem \ref{theorem: bias} and Theorem \ref{th:stoch_term} are postponed to Sections~\ref{section: proof of the bias theorem} and \ref{sec:th_stoch_term_proof}, respectively. These results ensure that
\begin{equation}
    \label{eq:eio_estimate_expansion}
    \widehat\btheta - \btheta^\circ
    \approx
    \bb_\lambda + (\Sigma^2 + 2 \lambda I_d)^{-1} \left(\Sigma \bU - \Sigma (\widehat \Sigma - \Sigma) \bb_\lambda - (\widehat \Sigma - \Sigma) \Sigma \bb_{\lambda} \right)
\end{equation}
with high probability provided that the sample size $n$ and the penalization parameter $\mu$ are large enough. It is not surprising that the leading terms in the expansion \eqref{eq:eio_estimate_expansion} correspond to the ones in decomposition of the estimate $\widehat \btheta_\infty$ defined as
\[
    (\widehat \btheta_\infty, \widehat \bfeta_\infty) \in \argmin\limits_{(\btheta, \bfeta)} \left\{ \frac12 \|\bZ - \bfeta\|^2 + \frac12 \|\bfeta - \widehat \Sigma \btheta\|^2 + \frac{\lambda}2 \|\btheta\|^2 \right\}.
\]
Hence, the insertion of the operator $A$ has no negative influence on the quadratic risk $R(\widehat\btheta)$ provided that the parameter $\mu$ is sufficiently large. Theorems \ref{theorem: bias} and \ref{th:stoch_term} essentially reduce the analysis of the excess risk $R(\widehat\btheta) - R(\btheta^\circ) = \|\Sigma^{1/2}(\widehat \btheta - \btheta^\circ)\|^2$ to large deviation bounds on the right-hand side of \eqref{eq:eio_estimate_expansion}, which easily follow from the next theorems.

\begin{Th}
    \label{th:covariance_concentration}
    Grant Assumption \ref{as:quadratic} and let $A \in \R^{p \times d}$ and $B \in \R^{q \times d}$ be arbitrary matrices. Let us fix any $\delta \in (0, 1)$ satisfying the inequality
    \[
        \ttr(\Sigma^{1/2} A^\top A \Sigma^{1/2}) \, \ttr(\Sigma^{1/2} B^\top B \Sigma^{1/2}) + \log(2 / \delta) \leq 4n.
    \]
    Then, with probability at least $(1 - \delta)$, it holds that
    \begin{align*}
        \|B (\widehat \Sigma - \Sigma) A^\top \|_{\F}
        &
        \leq 4 C_X \|A \Sigma^{1/2}\| \, \|B \Sigma^{1/2}\|
        \sqrt{ \frac{\ttr(\Sigma^{1/2} A^\top A \Sigma^{1/2}) \, \ttr(\Sigma^{1/2} B^\top B \Sigma^{1/2}) + \log(2 / \delta)}n}.
    \end{align*}
\end{Th}

\begin{Th}
    \label{th:noise_concentration}
    Suppose that a sample $\{(\bX_i, \eps_i) : 1 \leq i \leq n\} \subset \R^d \times \R$ consists of i.i.d. pairs of random elements satisfying Assumption \ref{as:noise}.
    Let us fix an arbitrary matrix $B \in \R^{q \times d}$ and any $\delta \in (0, 1)$ such that
    \[
        \ttr(\Sigma^{1/2} B^\top B \Sigma^{1/2}) + \log(2 / \delta) \leq n.
    \]
    Then, with probability at least $(1 - \delta)$, it holds that
    \[
        \left\| \frac1n \sum\limits_{i = 1}^n B \bX_i \eps_i \right\|
        \leq 8 \sigma \left\|B \Sigma^{1/2} \right\| \sqrt{\frac{\ttr(\Sigma^{1/2} B^\top B \Sigma^{1/2}) + \log(2 / \delta)}n}.
    \]
\end{Th}
The proofs of Theorems \ref{th:covariance_concentration} and \ref{th:noise_concentration} are moved to Appendix \ref{sec:concentration_inequalities}. They rely on the PAC-Bayesian variational inequality (see, e.g., \citep[Proposition 2.1]{catoni17}). In \citep{zhivotovskiy21, puchkin24, abdalla22}, the authors use similar arguments in the context of covariance estimation. Let us note that the upper bound in Theorem \ref{th:covariance_concentration} agrees with the results of \cite{bunea15} and \cite{puchkin23} in the case $A = B = I_d$ and with Theorem 16 of \cite{hsu12} when $A = B = (\Sigma + \lambda I_d)^{-1/2}$.

Theorems \ref{th:covariance_concentration} and \ref{th:noise_concentration} applied to the vector $\bzeta$ from Theorem \ref{th:stoch_term} help us to quantify the excess risk of the error-in-operator estimate in terms of
\begin{equation}
    \label{eq:rqk}
    r_q(k) = \sum\limits_{j > k} \left( \frac{\sigma_j}{\sigma_{k + 1}} \right)^q,
    \quad k \in \{1, \dots, d\},
\end{equation}
where $\sigma_1 \geq \sigma_2 \geq \dots \geq \sigma_d$ denote the eigenvalues of $\Sigma$. In Appendix \ref{sec:co_risk_bound_proof}, we prove the following upper bound on the norm of $\Sigma^{1/2}(\widehat\btheta - \btheta^\circ)$ with explicit constants.

\begin{Co}
    \label{co:risk_bound}
    Assume the conditions of Theorem \ref{theorem: bias} and Theorem \ref{th:stoch_term}. Let $\sigma_1, \dots, \sigma_d$ stand for the eigenvalues of $\Sigma$ and define
    \begin{equation}
        \label{eq:k_star}
        k^* = k^*(\lambda) = \max\left\{ k \in \mathbb N: \sigma_{k}^2 \geq 2 \lambda \right\}.
    \end{equation}
    Fix an arbitrary $\delta \in (0, 1)$ such that
    \begin{equation}
        \label{eq:k_star_condition}
        \left(1 + \frac{2\lambda}{\|\Sigma\|^2} \right)^2 \big( k^* + r_4(k^*) \big) + \log(4/\delta) \leq n
        \quad \text{and} \quad
        k^* + r_2(k^*) + 0.25 \log(4/\delta) \leq n
    \end{equation}
    with $r_2(k)$ and $r_4(k)$ given by \eqref{eq:rqk}. Then, with probability at least $(1 - 8\delta)$, it holds that
    \begin{align}
        \label{eq:risk_bound}
        \left\| \Sigma^{1/2} (\widehat \btheta - \btheta^\circ) \right\|
        &\notag
        \leq \left\|\Sigma^{1/2} \bb_\lambda\right\| + 4 \left(2\sigma + C_X \|\Sigma^{1/2} \bb_\lambda\| \right) \sqrt{\frac{k^* + r_4(k^*) + \log(4 / \delta)}n}
        \\&\quad
        + \frac{2 C_X \|\Sigma^{3/2} \bb_{\lambda}\|}{\sqrt{2\lambda}} \sqrt{\frac{4 k^* + 4 r_2(k^*) + \log(4 / \delta)}n}
        + \diamondsuit,
    \end{align}
    where $\bb_\lambda = -\lambda (\Sigma^2/2 + \lambda I_d)^{-1} \btheta^\circ$ and
    \begin{align}
        \label{eq:risk_expansion_remainder}
        \diamondsuit
        &\notag
        = 210 \left (\frac{\Vert \btheta^\circ \Vert}{\mu} + \frac{\Vert \Sigma \Vert \Vert \bb_\lambda \Vert}{\mu \sqrt{\lambda}} \right ) \|\Sigma\|^{1/2} \left \Vert \bb_\lambda \right \Vert
        + \frac{\|\btheta^\circ\|}{\lambda^{1/4}} \left( \left( \frac{17 \|\btheta^\circ\|}{\mu}\right)^2 + \frac{\|\bb_\lambda\|}{160 \mu} \right) \sqrt{ \Psi(n, \delta)}
        \\&\quad
        + \frac{19 \|\btheta^\circ\|}{\sqrt{3} \lambda^{3/4}} \left(1 + \sqrt{\frac{\Psi(n, \delta)}{\lambda}} \right) \Psi(n, \delta).
    \end{align}
\end{Co}
Corollary \ref{co:risk_bound} gives an intuition for the choice of $\mu$. It should be chosen in such a way that $\diamondsuit = \cO(1/\mu)$, $\mu \rightarrow \infty$, does not exceed the leading terms in the risk bound \eqref{eq:risk_bound}.
We would like to note that the expression in the right-hand side of \eqref{eq:risk_bound} can be simplified in view Lemma \ref{lem:bias_upper_bound} from Appendix \ref{sec:bias_properties} below. Let
\[
    \Sigma = \sum\limits_{j = 1}^d \sigma_j \bv_j \bv_j^\top
\]
be the eigenvalue decomposition of $\Sigma$. For any $k \in \{1, \dots, d\}$, we define $\beta_k = \bv_k^\top \btheta^\circ$ and the projections of $\btheta$ onto the linear span of $\bv_1, \dots, \bv_k$ and its orthogonal complement, respectively:
\[
    \btheta_{\leq k}^\circ = \sum\limits_{j = 1}^k \beta_j \bv_j
    \quad \text{and} \quad
    \btheta_{>k}^\circ = \btheta^\circ - \btheta_{\leq k}^\circ.
\]
Then, under the conditions of Corollary \ref{co:risk_bound}, with probability at least $(1 - 8\delta)$ it holds that
\begin{align}
    \label{eq:risk_bound_simplified}
    \left\| \Sigma^{1/2} (\widehat \btheta - \btheta^\circ) \right\|
    &\notag
    \leq \left( \sigma_{k^*}^2 \left\|\Sigma^{-1/2} \btheta^\circ_{\leq k^*}\right\|^2 + \left\|\Sigma^{1/2} \btheta^\circ_{> k^*}\right\|^2 \right)^{1/2} \left(1 + 8 C_X \sqrt{\frac{k^* + r_2(k^*) + \log(4 / \delta)}n} \right)
    \\&
    + 8\sigma \sqrt{\frac{k^* + r_4(k^*) + \log(4 / \delta)}n}
    + \diamondsuit,
\end{align}
where $\diamondsuit$ is defined in \eqref{eq:risk_expansion_remainder}. The first term in the right-hand side of \eqref{eq:risk_bound_simplified} corresponds to the bias and the random design effect while the second one reflects the noise influence. It is worth mentioning that the variance of the standard ridge regression estimate
\begin{equation}
    \label{eq:ridge}
    \widehat\btheta{}^{(R)} = \argmin\limits_{\btheta} \left\{ \left\|\bY - \bbX^\top \btheta \right\|^2 + \tau \|\btheta\|^2 \right\} 
\end{equation}
is of order (see, for instance, \citep{hsu12, cheng2022, bach2024high})
\begin{align}
\label{eq: variance ridge}
    \Var\left( \widehat\btheta{}^{(R)} \,\big\vert\, \bbX \right) = \cO\left( \frac{\Tr\big( \Sigma^2 (\Sigma + \tau I_d)^2 \big)}n \right)
    = \cO\left( \frac{\widetilde k + r_2(\widetilde k)}n \right), \quad n \rightarrow \infty,
\end{align}
where
\[
    \widetilde k = \widetilde k(\tau) = \max\left\{ k \in \mathbb N : \sigma_k \geq \tau \right\}.
\]
\begin{Rem}
    To be more precise, in \citep{cheng2022, bach2024high}, the variance is expressed through the effective regularization parameter $\varkappa(\tau)$. In the present paper, we focus on the case $\ttr(\Sigma) = \cO(\sqrt{n}) = o(n)$ (see Theorem \ref{th:stoch_term}). In this situation, $\varkappa(\tau)$ and $\tau$ do not differ too much. In contrast, \cite{cheng2022} are interested in the opposite scenario. 
\end{Rem}
Let us note that $\widetilde k(\tau)$ coincides with $k^*(\lambda)$ (see \eqref{eq:k_star}) when $\tau^2 = 2 \lambda$. This suggests us to compare the performance of the standard ridge estimate \eqref{eq:ridge} with the prediction risk of the error-in-operator estimate \eqref{eq: initial optimization problem} with $\lambda = \tau^2 / 2$. First, observe that the bound \eqref{eq:risk_bound_simplified} yields that
\[
    \Var\left( \widehat\btheta \,\big\vert\, \bbX \right)
    = \cO\left( \frac{k^* + r_4(k^*)}n + \diamondsuit^2 \right),
    \quad n \rightarrow \infty.
\]
Since $r_q(k)$ monotonously decreases in $q > 0$, $\Var\left( \widehat\btheta \,\big\vert\, \bbX \right)$ exhibits a better behaviour than $\Var\left( \widehat\btheta{}^{(R)} \,\big\vert\, \bbX \right)$. At the same time, the first term in the right-hand side of \eqref{eq:risk_bound_simplified} is of the same order as the upper bound on the bias of $\widehat\btheta{}^{(R)}$ obtained in \citep[Proposition 2.2]{cheng2022}. Moreover, one can easily show that the leading bias term $\|\Sigma^{1/2} \bb_\lambda\|$, $\lambda = \tau^2 / 2$ of the error-in-operator estimate $\widehat\btheta$ does not exceed $\sqrt{2} \tau \|\Sigma^{1/2} (\Sigma + \tau I_d)^{-1} \btheta^\circ\|$ (see Lemma \ref{lem:ridge_comparison}).

\section{Computational aspects}
\label{sec:computational_aspects}

Since the objective \eqref{eq:eio_objective} of the error-in-operator estimate \eqref{eq: initial optimization problem} is non-convex, we must elaborate on computational aspects of the suggested procedure. We defined the estimator $\widehat \btheta$ as a minimizer of a multivariate polynomial of degree $4$. Optimizing a multivariate fourth-degree polynomial is intractable in general. Therefore, to find $\widehat \btheta$, we should harness the inner structure of our problem. First, we reduce the inference to a convex optimization problem. Let us fix an arbitrary $\rho_0$ from $(0, 1/3]$ and define
\begin{equation}
    \label{eq:upsilon}
    \Ups(\rho_0) = \left\{ \bups = (\btheta, \bfeta, A) : (1 - \rho_0^2) \|\btheta\|^2 \leq \rho_0^2 \mu^2, \|\bfeta - A\btheta\|^2 \leq \rho_0^2 \mu^2 \lambda \right\}.
\end{equation}
The introduced set plays a crucial role in our analysis, because, the function $\ttL(\btheta, \bfeta, A)$ has a positive definite Hessian on $\Ups(\rho_0)$ (see Lemma~\ref{lem:block-diagonal_lower_bound} below).
However, note that $\Ups(\rho_0)$ is not convex. For this reason, we restrict our attention on its convex subset of the form (see Proposition~\ref{proposition: product convex subset of Ups})
\begin{align*}
    \Theta \times \sfH \times \sfA \subset \Ups(\rho_0),
\end{align*}
where
\begin{align*}
   \Theta
   & = \left \{\btheta \in \R^d :
   \Vert \btheta \Vert \le \rho_0 \mu \; \text{and} \; \Vert \Sigma \Vert \Vert \btheta - \btheta^* \Vert \le 5 \rho_0 \mu \sqrt{\lambda} / 96\right \},
   \\
   \sfH & = \left \{ \bfeta \in \R^d : \Vert \bfeta - \bZ \Vert \le \rho_0 \mu \sqrt{\lambda} / 3 \right \},
   \\
   \sfA & = \left \{ A \in \R^{d \times d} : \Vert A \Vert \le 3 \Vert \Sigma \Vert + \sqrt{\lambda} / 3   \right \}.
\end{align*}
Next, we show that under mild assumptions the estimate $\widehat \bups$ belongs to $\Theta \times \sfH \times \sfA$.
\begin{Lem}
\label{lemma: localization_probabilistic_conditions}
    Grant Assumptions \ref{as:quadratic} and \ref{as:noise} and fix an arbitrary $\delta \in (0, 1)$ such that $n \ge 2^{12} (1 + C_X)^2 \left ( \ttr(\Sigma) + \log(2/\delta) \right )$. Let $\rho_0 \le 1/8$ and let the sets $\Theta$, $\sfH$, $\sfA$ be as defined above. Assume that the non-negative parameters $\lambda$ and $\mu$ satisfy the inequalities
    \begin{align*}
        \Vert \btheta^\circ \Vert \le \frac{\rho_0 \mu}7,
        \quad &
        \Vert \Sigma \Vert \Vert \btheta^\circ \Vert \le \frac{\rho_0 \mu \sqrt{\lambda}}{56 \cdot 18},
        \quad 
        \frac{2^{11} \sigma \Vert \Sigma \Vert^{1/2}}{\rho_0 \mu} \sqrt{\frac{\ttr(\Sigma) + \log(4/\delta)}{n}}
        \leq \frac{\lambda}{\|\Sigma\|} \land \sqrt{\lambda}, \\
    \end{align*}
    and
    \[
        \frac{2^{13} (1 + C_X) \Vert \Sigma \Vert^2 \Vert \btheta^\circ \Vert}{\rho_0 \mu} \sqrt{\frac{4 \ttr(\Sigma) + \log(2/\delta)}{n}} \le \lambda. 
    \]
    Then, $(\btheta^*, \bfeta^*, A^*) \in \Theta \times \sfH \times \sfA$ and, with probability at least $(1 - \delta)$, the triplets $(\bzero, \bZ, \widehat \Sigma)$ and $(\widehat \btheta, \widehat \bfeta, \widehat A)$ belong to the set $\Theta \times \sfH \times \sfA$ as well.
\end{Lem}

\begin{Rem}
    Careful examination of the proof of Lemma \ref{lemma: localization_probabilistic_conditions} reveals that one can replace Assumption \ref{as:quadratic} with a bit milder $\psi_2$-$L_2$-equivalence condition.    
\end{Rem}

Lemma~\ref{lemma: localization_probabilistic_conditions} implies that $\ttL(\cdot)$ possesses local convexity property, which is widely used in non-convex optimization (see the survey \citep{jain_non_convex_2017}). Thus, instead of~\eqref{eq: initial optimization problem}, we can study the problem
\begin{align}
\label{eq: convex problem}
    \ttL(\btheta, \bfeta, A) \rightarrow \min_{\btheta, \bfeta, A}
    \quad \text{ subject to } (\btheta, \bfeta, A) \in \Theta \times \sfH \times \sfA.
\end{align}
The problem~\eqref{eq: convex problem} could be solved using constrained convex optimization methods, if the sets $\Theta, \sfH$ and $\sfA$ were observable. Unfortunately, it is not the case, but we prove below that the following procedure converges to $(\widehat \btheta, \widehat A)$:
\begin{align*}
    \btheta_t & \in \argmin_{\btheta \in \R^d} \ttL(\btheta, (A_{t - 1} \btheta + \bZ)/2, A_{t - 1}), \\
    A_{t} & \in \argmin_{A \in \R^{d \times d}} \ttL(\btheta_{t}, (A \btheta_t + \bZ)/2, A)
\end{align*}
The estimates $\btheta_t$ and $A_t$ admit explicit expressions, and $\widehat \bfeta$ can be computed as 
\begin{align*}
    \widehat \bfeta = \frac{1}{2} (\bZ + \widehat A \widehat \btheta).
\end{align*}
We summarize the optimization procedure in Algorithm~\ref{algo: main scheme}.

\begin{algorithm}[H]
\caption{Error-in-operator regression}
\label{algo: main scheme}
\begin{algorithmic}[1]
    \Require observations $\bY$, feature matrix $\bbX$, regularization parameters $\mu, \lambda > 0$, number of iterations $T$.
    \State Initialization: set $\bZ = \bbX^\top \bY / n$, $\widehat \Sigma = \bbX^\top \bbX / n$, and $A_0 = \widehat \Sigma$.
    \For{$t = 1, \dots, T$}
        \State Compute
        \[
            \btheta_t = (A_{t - 1}^\top A_{t - 1} + 2 \lambda I_d)^{-1} A_{t - 1}^\top \bZ
        \]
        \State and
        \[
            A_{t} = \bZ \btheta_t^\top \left (2 \mu^2 I_d + \btheta_t \btheta_t^\top \right )^{-1} + \widehat \Sigma \left (I_d + \frac{\btheta_t \btheta_t^\top}{2 \mu^2}\right )^{-1}.
        \]
    \EndFor
    \State \Return $\btheta_T$.
\end{algorithmic}
\end{algorithm}
The following theorem justifies convergence of Algorithm \ref{algo: main scheme} to the error-in-operator estimate $\widehat \btheta$ defined in \eqref{eq: initial optimization problem}.

\begin{Th}
\label{theorem: rates of convergence}
    Grant Assumptions~\ref{as:quadratic} and \ref{as:noise}. Fix arbitrary positive numbers $\delta < 1$ and $\rho_0 \leq 1/8$. Suppose that $n \ge 2^{12} (1 + C_X)^2 \left ( \ttr(\Sigma) + \log(2/\delta) \right )$ and assume that the parameters $\mu$ and $\lambda$ satisfy the conditions
    \begin{equation}
        \label{eq:rho_0_conditions_1}
        \Vert \btheta^\circ \Vert \le \rho_0 \mu/7,
        \quad
        \Vert \Sigma \Vert \Vert \btheta^\circ \Vert \le \rho_0 \mu \sqrt{\lambda} / (56 \cdot 18),
    \end{equation}
    \begin{equation}
        \label{eq:rho_0_conditions_2}
        \lambda \ge \frac{2^{13} (1 + C_X) \Vert \Sigma \Vert^2 \Vert \btheta^\circ \Vert}{\rho_0 \mu} \sqrt{\frac{4 \ttr(\Sigma) + \log(2/\delta)}{n}},
    \end{equation}
    and
    \begin{equation}
        \label{eq:rho_0_conditions_3}
        \frac{\lambda}{\Vert \Sigma \Vert} \wedge \sqrt{\lambda}  \ge \frac{2^{11} \sigma \Vert \Sigma \Vert^{1/2}}{\rho_0 \mu} \sqrt{\frac{\ttr(\Sigma) + \log(4/\delta)}{n}}.
    \end{equation}
    Then, with probability at least $1 - \delta$, Algorithm~\ref{algo: main scheme} outputs $\btheta_T$ such that
    \begin{align*}
       \Vert \btheta_T - \widehat \btheta \Vert \le \rho_0^{2(T - 1)} \left (\Vert \btheta^\circ \Vert +8 \sigma \Vert \Sigma \Vert^{1/2} \, \sqrt{\frac{ \ttr(\Sigma) + \log(4/\delta)}{\lambda n}}\right ).
    \end{align*}
\end{Th}

Let us note that larger values of $\mu$ yield weaker conditions \eqref{eq:rho_0_conditions_1}--\eqref{eq:rho_0_conditions_3} on admissible $\rho_0$.
This means that the algorithm converges faster as $\mu$ increases. This is the behavior one should expect, because in the case $\mu = \infty$ the objective $\ttL(\cdot)$ can be minimized explicitly in one step.

\section{Numerical experiments}
\label{section: numerical_experiments}

In this section, we illustrate our studies with numerical experiments. First, we show that Theorems~\ref{theorem: bias},\ref{th:stoch_term} present correct leading terms in the expansion of $R(\widehat \btheta) - R(\btheta^\circ)$. We fix the following parameters of the random design linear regression problem:
\begin{align}
    d & = 200, \nonumber \\
    \bX_i & = \left ( k^{-1/8} \cdot \sin(\pi k \cdot \xi_k) \right )_{k = 1}^d, \label{eq:experiment setup} \\
    \btheta_k^\circ & = k^{-3}, \quad k = 1, \ldots, d, \nonumber
\end{align}
where $\xi_k$ are i.i.d. samples from $ \Uniform[-1;1]$. It is clear, that $\Sigma$ is diagonal with eigenvalues $k^{-1/4} / 2$, $k = 1, \ldots, d$. We set $\varepsilon_1, \ldots, \varepsilon_n$ to be i.i.d. Gaussian random variables with the standard deviation $0.09$.

For this set of parameters, we plot the dependence of ratio $\Vert \Sigma^{1/2} \bb_\lambda\Vert / \Vert \Sigma^{1/2} (\btheta^* - \btheta^\circ)\Vert$ on $\mu, \lambda$ in Figure~\ref{fig:leading-term-studies} (left). As implied by Theorem~\ref{theorem: bias}, this ratio is close to $1$ for large values of $\mu$.

To illustrate Theorem~\ref{th:stoch_term}, we fix $\mu = 10^8$. We denote
\begin{align*}
    \tilde \bzeta = (\Sigma^2 + 2 \lambda I_d)^{-1} \bzeta.
\end{align*}
If Theorem~\ref{th:stoch_term} yields the correct leading term of $\Sigma^{1/2} (\widehat \btheta - \btheta^*)$, then the ratio $\Vert \Sigma^{1/2} \tilde \bzeta \Vert / \Vert \Sigma^{1/2}(\widehat \btheta - \btheta^*) \Vert$ should be close to $1$ for large enough $n$. We plot the dependence of $\Vert \Sigma^{1/2} \tilde \bzeta \Vert / \Vert \Sigma^{1/2}(\widehat \btheta - \btheta^*) \Vert$ on $\lambda$ and $n$ in Figure~\ref{fig:leading-term-studies} (middle and right). To obtain smoother curves, we averaged the studied ratio over $40$ generations. On the right picture, $\lambda_{opt}(n)$ is estimated by the grid search for each $n$. As one can see, our experiments agree with suggested theory.

Then, we compare the performance of our estimator $\widehat \btheta$ and the ridge estimator $\widehat \btheta{}^{(R)}$ defined in~\eqref{eq:ridge}. We estimate the optimal hyperparameters $\mu, \lambda$ and $\tau$ by the grid search. Due to~\eqref{eq:risk_bound_simplified}, the variance leading term $(\Sigma^2 + 2 \lambda I_d) \bzeta$ of $\widehat \btheta$ depends on the fourth powers of covariance eigenvalues, while the performance of $\widehat \btheta{}^{(R)}$ depends on the second powers. Thus, we expect that our estimator $\widehat \btheta$ should outperform $\widehat\btheta{}^{(R)}$ for $\Sigma$ with heavy-tailed eigenvalues. We support these observations by Figure~\ref{fig:ridge_comparison} (left plot). A reader can observe that our estimator $\widehat \btheta$ outperforms the ridge estimator $\widehat{\btheta}{}^{(R)}$ for the suggested setup~\eqref{eq:experiment setup}.

Next, we study the impact of finite $\mu$ on the performance of $\widehat \btheta$ in the considered setup~\eqref{eq:experiment setup}. Empirically, we show that finite $\mu$ can smooth the double descent curve.  Figure~\ref{fig:ridge_comparison} (right) displays the dependence of the risk $R(\widehat \btheta) -  R(\btheta^\circ)$ on the sample size $n \in \{50, 100, \ldots, 500\}$. For each $n$ and $\mu = \infty$, we find $\lambda_{opt}(n) \in \{1.3^{-40}, 1.3^{-39}, \ldots, 1.3^{39}\}$ by the grid search. Then, for each $n$ and $\lambda \in \lambda_{opt}(n) \cdot \{10^{-6}, 10^{-4}, 10^{-2}, 1 , 2\}$, we estimate $\mu_{opt}(\lambda, n)$ by the grid search on $\{2^0, 2^1, \ldots, 2^{29}\}$. We plot the dependence of the excess risk $R(\widehat \btheta) - R(\btheta^\circ )$ on $n$ and $\lambda$ for $\mu = \infty$ by the dashed line. The solid line corresponds to $\mu = \mu_{opt}(\lambda, n)$. We did not obtain the impact of finite $\mu$ on the risk curve when $\lambda \ge \lambda_{opt}(n)$, but finite $\mu$ can significantly mitigate the effect of double descent if $\lambda < \lambda_{opt}$.

\begin{figure}
    \centering
    \includegraphics[width=\linewidth]{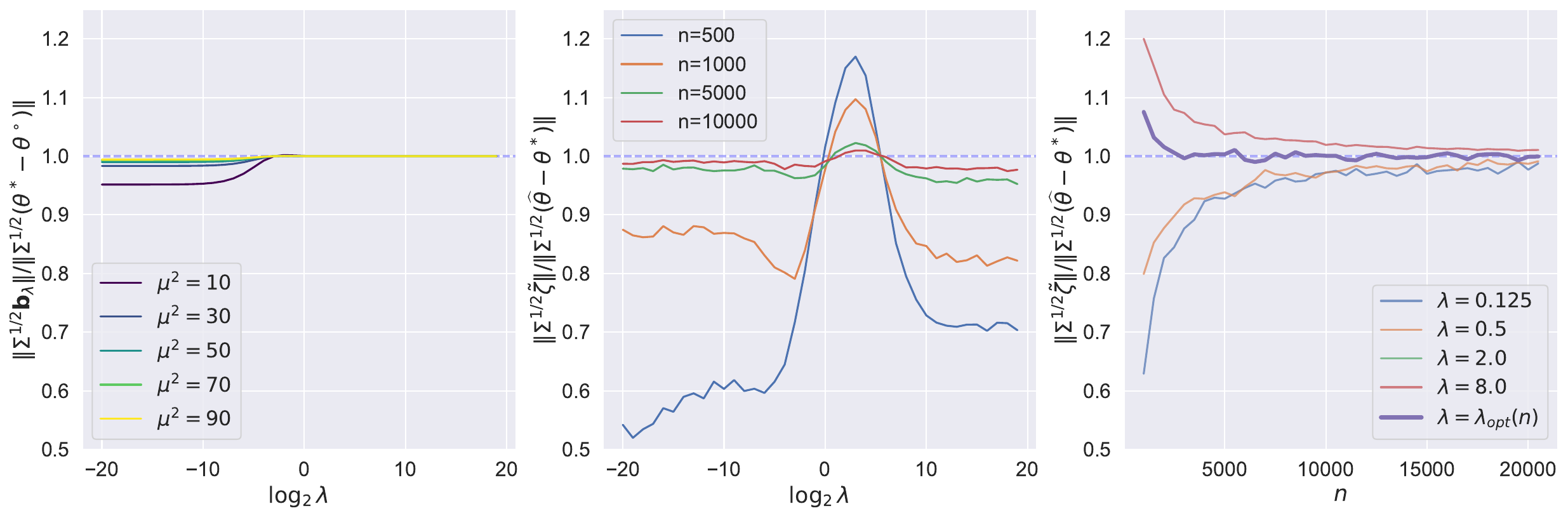}
    \caption{\textit{Left:} The ratio between the bias component $\Vert \Sigma^{1/2}(\btheta^* - \btheta^\circ)\Vert$ of the risk and its leading term established in Theorem~\ref{theorem: bias}. \textit{Middle and Right:} The ratio between the variance component of the risk $\Vert \Sigma^{1/2}(\widehat{\btheta} - \btheta^*) \Vert$ and its leading term established in Theorem~\ref{th:stoch_term}.}
    \label{fig:leading-term-studies}
\end{figure}

\begin{figure}[h!]
    \centering
    \includegraphics[width=\linewidth]{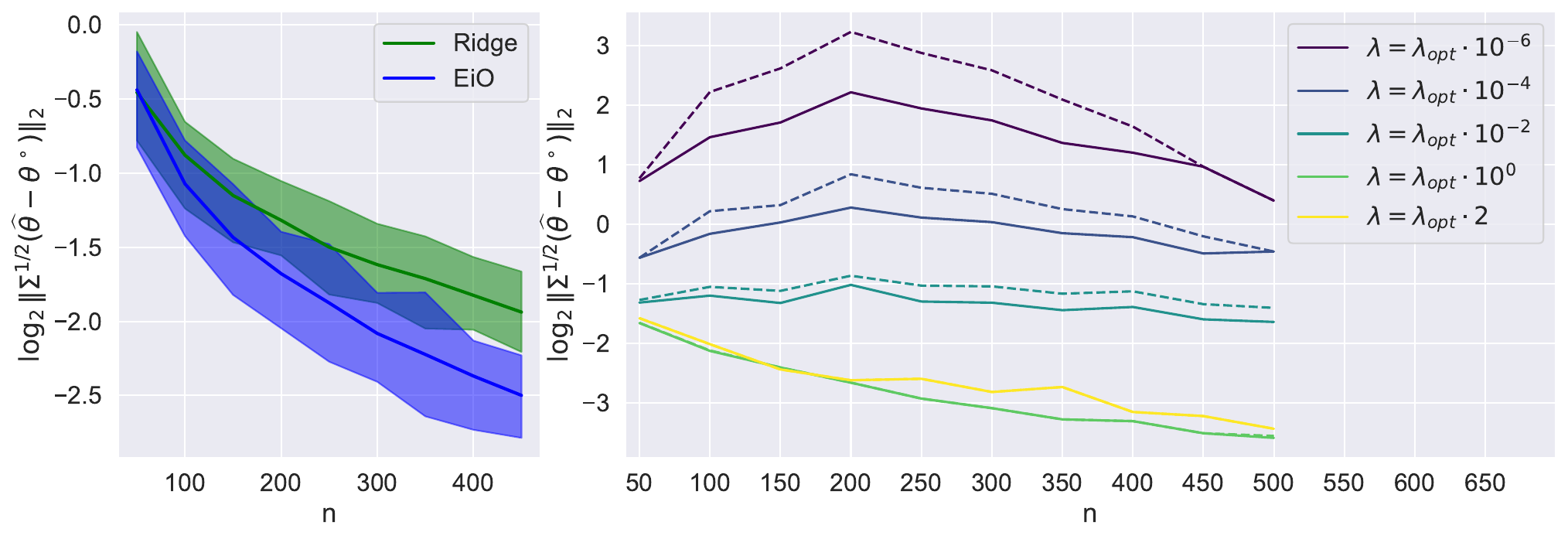}
    \caption{Numerical studies of Error-in-Operator estimator $\widehat \btheta$. {\it Left:} The comparison of the Error-in-Operator estimator and the standard ridge estimator. {\it Right:} The effect of finite $\mu$ on the double descent curve.}
    \label{fig:ridge_comparison}
\end{figure}

\section{Proof of Theorem~\ref{theorem: bias}}
\label{section: proof of the bias theorem}

We start with studying the difference $\bups^* - \bups^\circ$. For any $t \in [0, 1]$, let us define $\bu^{(t)} = (1 - t) \bups^\circ + t \bups^*$. Due to the Newton-Leibniz formula, we have
\begin{align*}
    \bnabla \cL(\bups^*) - \bnabla \cL(\bups^\circ) = \bnabla \cL(\bu^{(1)}) - \bnabla \cL(\bu^{(0)}) = \int_{0}^1 \nabla^2 \cL(\bu^{(t)}) (\bups^* - \bups^\circ) \, \rmd t.
\end{align*}
Let us denote $H(\bu) = \nabla^2 \cL(\bu)$ and $\biasInner{H} = \int_0^1 H(\bu^{(t)}) \rmd t$ for brevity. Generally, for a matrix-valued function $f(\bu) = f(\btheta, \bfeta, A)$, we let
\begin{align*}
    \biasInner{f} = \int_0^1 f(\bu^{(t)}) \rmd t.
\end{align*}
We are going to apply Lemma~\ref{lemma: localization bias corollary} with $\rho_0 = 1/7$. To satisfy its conditions, we bound $\Vert \Sigma \Vert \Vert \btheta^\circ \Vert$ as follows:
\begin{align}
\label{eq: sigma theta circ bound}
    \Vert \Sigma \Vert \Vert \btheta^\circ \Vert \le  \mu \sqrt{\lambda} / (7 \cdot 24) = \frac{\rho_0 \mu \sqrt{\lambda}}{24}.
\end{align}
Then,  Lemma~\ref{lemma: localization bias corollary} and Lemma~\ref{lem:block-diagonal_lower_bound} imply that the matrix $\biasInner{H}$ is invertible. With the introduced notations we have
\begin{align*}
    \bups^* - \bups^\circ = \biasInner{H}^{-1} (\bnabla \cL(\bups^*) - \bnabla \cL(\bups^\circ)).
\end{align*}
Since $\bups^*$ minimizes $\cL(\bups)$, we obtain that
\[
    \bnabla \cL(\bups^*) - \bnabla \cL(\bups^\circ)
    = -\bnabla \cL(\bups^\circ).
\]
It is straightforward to check that the only non-zero components of $\nabla \cL(\bups^\circ)$ correspond to the first $d$ coordinates. For this reason, we can restrict our attention on $\btheta^* - \btheta^\circ$.
Let
\[
    \avg H_{\chi \chi} =
    \begin{pmatrix}
        \avg H_{\bfeta\bfeta} & \avg H_{\bfeta A} \\
        \avg H_{A\bfeta} & \avg H_{AA}
    \end{pmatrix}
\]
stand for the block of $\avg H$ corresponding to the nuisance parameter $\chi = (\bfeta, A)$.
Due to the block-inversion formula \eqref{eq:h_block_inversion}, we have
\begin{equation}
    \label{eq:theta_star_representation}
    \btheta^* - \btheta^\circ = (\biasInner{H} / \biasInner{H}_{\chi \chi})^{-1} (\bnabla_{\btheta} \cL(\bups^*) - \bnabla_{\btheta} \cL(\bups^\circ)) = - \lambda (\biasInner{H} / \biasInner{H}_{\chi \chi})^{-1} \btheta^\circ,
\end{equation}
where $(\biasInner{H} / \biasInner{H}_{\chi \chi})$ is the Schur complement of the submatrix $\biasInner{H}_{\chi \chi}$. Rearranging the terms, we obtain that
\begin{align}
    \label{eq: bias without simplification of Schur complement}
    \big\Vert \btheta^* - \btheta^\circ - \bb_\lambda \big\Vert
    &\notag
    = \big\Vert \btheta^* - \btheta^\circ +  \lambda (\Sigma^2/2 + \lambda I_d)^{-1} \btheta^\circ \big\Vert
    \\&\notag
    = \lambda \left \Vert \left[ (\biasInner{H} / \biasInner{H}_{\chi \chi})^{-1} (\Sigma^2 / 2 + \lambda I_d) - I_d \right] (\Sigma^2/2 + \lambda I_d)^{-1} \btheta^\circ \right \Vert
    \\&
    \le \lambda \left\Vert (\biasInner{H} / \biasInner{H}_{\chi \chi})^{-1} (\Sigma^2 / 2 + \lambda I_d) - I_d \right\Vert \Vert (\Sigma^2/ 2 + \lambda I_d)^{-1} \btheta^\circ \Vert
    \\&\notag
    = \left\Vert (\biasInner{H} / \biasInner{H}_{\chi \chi})^{-1} (\Sigma^2 / 2 + \lambda I_d) - I_d \right\Vert \|\bb_\lambda\|
\end{align}
The next lemma ensures that the rightmost expression is small compared to the norm of $ \lambda (\Sigma^2/ 2 + \lambda I_d)^{-1} \btheta^\circ$. {For the sake of brevity, let us define $\rho = 7 \Vert \btheta^\circ \Vert/\mu$. By the assumptions of Theorem~\ref{theorem: bias}, we have 
\begin{align}
\label{eq: rho_num_bound}
    \rho \le 1/60.
\end{align}}
\begin{Lem}
\label{lemma: Schur complement bound for bias}
Suppose that the assumptions of Theorem~\ref{theorem: bias} hold. Let $\m(\cdot)$ be an arbitrary probability measure supported on $[0, 1]$. Define
\begin{align*}
    \biasInner{H}^{\,\m} = \int_0^1 H((1 - t) \bups^\circ + t \bups^*) \rmd \m(t).
\end{align*}
Then it holds that
    \[
        \left \Vert (\Sigma^2 / 2 + \lambda I_d)^{-1} (\avg{H}^{\,\m} / \avg{H}^{\,\m}_{\chi \chi}) - I_d \right \Vert
        \le 15 \rho + \frac{5 \Vert \Sigma \Vert \Vert \bb_\lambda \Vert}{\mu \sqrt{\lambda}}.
    \]
\end{Lem}

Let us denote $(\Sigma^2 / 2 + \lambda I_d)^{-1} (\avg{H} / \avg{H}_{\chi \chi}) - I_d$ by $E$.
Applying Lemma~\ref{lemma: Schur complement bound for bias} with $\m$ equal to the uniform measure on $[0, 1]$, we obtain that
\[
    \|E\| \le 15 \rho + \frac{5 \Vert \Sigma \Vert \Vert \bb_\lambda \Vert}{\mu \sqrt{\lambda}} \leq 15 \rho +\frac{5 \Vert \Sigma \Vert \Vert \btheta^\circ \Vert}{\mu \sqrt{\lambda}} \le \frac12,
\]
where we used $\rho \le 1/60$ and~\eqref{eq: sigma theta circ bound}.
It yields
\begin{align*}
    &
    \left\Vert (\biasInner{H} / \biasInner{H}_{\chi \chi})^{-1} (\Sigma^2 / 2 + \lambda I_d) - I_d \right\Vert
    \\&
    = \Vert (I_d + E)^{-1} - I_d \Vert
    = \Vert E (I_d + E)^{-1} \Vert
    \le \frac{\Vert E \Vert}{1 - \Vert E \Vert}
    \\&
    \le 30 \rho + \frac{10 \Vert \Sigma \Vert \Vert \bb_\lambda \Vert}{\mu \sqrt{\lambda}}.
\end{align*}
Substituting this bound into~\eqref{eq: bias without simplification of Schur complement}, we get that
\[
    \big\Vert \btheta^* - \btheta^\circ - \bb_\lambda \big\Vert
    \leq 30 \rho \|\bb_\lambda\| + \frac{10 \Vert \Sigma \Vert \Vert \bb_\lambda \Vert^2}{\mu \sqrt{\lambda}}.
\]
\myendproof

\section{Proof of Theorem \ref{th:stoch_term}}
\label{sec:th_stoch_term_proof}

Before we dive into details, let us discuss the main idea of the proof. We split it into four major steps. First, we check that the conditions of Lemma \ref{lemma: localization_probabilistic_conditions} are satisfied with high probability. This yields that both $\widehat\bups$ and $\bups^*$ belong to a convex set around $\bups^\circ$ where, according to Lemma \ref{lem:block-diagonal_lower_bound}, the Hessian $\nabla^2 \cL(\bups)$ is positive definite. Hence, if the penalization parameters $\mu$ and $\lambda$ are chosen in agreement with \eqref{eq:stoch_term_mu_lambda_conditions}, the difference $(\widehat \bups - \bups^*)$ can be approximated by the standardized score $H_*^{-1} \bz$ where $H_*$ stands for $\nabla^2 \cL(\bups^*)$ (see Lemma \ref{lem:stoch_term_expansion}). On the second step, we show that $\widehat \btheta - \btheta^*$ admits an expansion with a leading term $\bs_{\btheta}^*$ equal to the first $d$ components of $H_*^{-1} \bz$. After that, we quantify the remainder $(\widehat \btheta - \btheta^* - \bs_{\btheta}^*)$ handling the subvectors of the standardized score corresponding to the target and the nuisance parameters separately (see Lemma \ref{lem:standardized_score_theta_bound} and Lemma \ref{lem:standardized_score_chi_bound}, respectively). Summing up the bounds obtained on the second and the third steps, we provide an upper bound of the form
\begin{equation}
    \label{eq:stoch_term_remainder_bound_asymp}
    \left\| \Sigma^{1/2} (\widehat \btheta - \btheta^* - \bs_{\btheta}^*) \right\|
    \lesssim \lambda^{-3/4} \left(1 + \sqrt{\frac{\Psi(n, \delta)}{27 \lambda}} \right) \Psi(n, \delta)
\end{equation}
holding with high probability. Finally, on the fourth step we elaborate on $\bs_{\btheta}^*$ and argue that it is close to the vector $\bzeta$ defined in the statement of Theorem \ref{th:stoch_term}.

\noindent\textbf{Step 1: localization.}\quad
To start with, let us define
\[
    \rho_0 = \frac{7 \|\btheta^\circ\|}{\mu} \left(1 \vee \frac{72 \|\Sigma\|}{\sqrt{\lambda}} \right).
\]
Note that $\rho_0 \leq 1/16$, because of the condition \eqref{eq:stoch_term_mu_lambda_conditions} and the inequality $\mu \geq 112 \|\btheta^\circ\|$. Moreover, \eqref{eq:stoch_term_mu_lambda_conditions} implies that
\begin{align*}
    &
    \frac{2^{11} \sigma \|\Sigma\|^{1/2}}{\rho_0 \mu} \sqrt{\frac{\ttr(\Sigma) + \log(4 / \delta)}n}
    \leq \frac{2^{11} \sigma \|\Sigma\|^{1/2} \sqrt{\lambda}}{7 \cdot 72 \|\Sigma\| \|\btheta^\circ\|} \sqrt{\frac{\ttr(\Sigma) + \log(4 / \delta)}n}
    \\&
    \leq \frac{5 \sigma \|\Sigma\|^{1/2} \sqrt{\lambda}}{ \|\Sigma\| \|\btheta^\circ\|} \sqrt{\frac{\ttr(\Sigma) + \log(4 / \delta)}n}
    \leq \frac{\sqrt{\lambda}}{\|\Sigma\| \|\btheta^\circ\|} \cdot \|\btheta^\circ\| (\sqrt{\lambda} \land \|\Sigma\|)
    = \sqrt{\lambda} \land \frac{\lambda}{\|\Sigma\|},
\end{align*}
while \eqref{eq:stoch_term_mu_lambda_conditions_2} yields
\begin{align*}
    &
    \frac{2^{13} (1 + C_X) \Vert \Sigma \Vert^2 \Vert \btheta^\circ \Vert}{\rho_0 \mu} \sqrt{\frac{4 \ttr(\Sigma) + \log(2/\delta)}{n}}
    \\&
    \leq \frac{2^{13} (1 + C_X) \Vert \Sigma \Vert^2 \Vert \btheta^\circ \Vert \sqrt{\lambda}}{7 \cdot 72 \|\Sigma\| \|\btheta^\circ\|} \sqrt{\frac{4 \ttr(\Sigma) + \log(2/\delta)}{n}}
    \\&
    \leq 17 (1 + C_X) \|\Sigma\| \sqrt{\lambda} \cdot \sqrt{\frac{4 \ttr(\Sigma) + \log(2/\delta)}{n}} \le \lambda.
\end{align*}
This means that the assumptions of Lemma \ref{lemma: localization_probabilistic_conditions} are satisfied. Hence, there is an event $\cE_0$ of probability at least $1 - \delta$, such that $\widehat \bups$ and $\bups^*$ belong to $\Theta \times \sfH \times \sfA$ on $\cE_0$, where
\begin{align*}
   \Theta
   & = \left \{\btheta \in \R^d :
   \Vert \btheta \Vert \le \rho_0 \mu \text{ and } \Vert \Sigma \Vert \Vert \btheta - \btheta^* \Vert \le 5 \rho_0 \mu \sqrt{\lambda} / 96\right \},
   \\
   \sfH & = \left \{ \bfeta \in \R^d : \Vert \bfeta - \bZ \Vert \le \rho_0 \mu \sqrt{\lambda} / 3 \right \},
   \\
   \sfA & = \left \{ A \in \R^{d \times d} : \Vert A \Vert \le 3 \Vert \Sigma \Vert + \sqrt{\lambda} / 3   \right \}.
\end{align*}
In the rest of the proof, we restrict our attention on the event $\cE_0$.

\medskip

\noindent\textbf{Step 2: expansion of $\widehat \bups - \bups^*$.}\quad
The result obtained on the previous step allows us to exploit geometric properties of $\cL(\bups)$. To be more precise, Proposition \ref{proposition: product convex subset of Ups} implies that
\[
    \widehat \bups, \bups^* \in \Theta \times \sfH \times \sfA \subseteq \Ups(\rho_0)
    \quad \text{on $\cE_0$,}
\]
where the set $\Ups(\rho_0)$ is defined in \eqref{eq:upsilon}.
Then, due to Lemma \ref{lem:block-diagonal_lower_bound} and Lemma \ref{lem:second_derivative_lipschitzness}, $\nabla^2 \cL(\bups)$ is positive definite and smooth on the segment $\{ t \widehat \bups + (1 - t) \bups^* : 0 \leq t \leq 1\}$. These properties play a crucial role in the proof of the following technical result.

\begin{Lem}
    \label{lem:stoch_term_expansion}
    Let $S$, $\ttD$, and $\ttD_0$ be $(2d + d^2) \times (2d + d^2)$ block-diagonal matrices of the form
    \[
        S = \diag\big( \Sigma^{1/2}, I_{d + 1} \otimes O_d \big),
    \]
    \[
        \ttD_*^2
        = \diag\left( (A^*)^\top A^* + \lambda I_d, 2 I_d, I_d \otimes \left(\mu^2 I_d + \btheta^* (\btheta^*)^\top \right) \right),
    \]
    and
    \begin{equation}
        \label{eq:D0}
        \ttD_0^2
        = \diag\left( (A^*)^\top A^* + \lambda I_d, 2 I_d, I_d \otimes \left(\mu_0^2 I_d + \btheta^* (\btheta^*)^\top \right) \right),
    \end{equation}
    where $\mu_0 \in (0, \mu]$ is an arbitrary constant.
    Fix an arbitrary $\rho_0 \in (0, 1/16]$ and consider an event $\cE$ such that $\bups^*$ and $\widehat \bups$ belong to a convex set $ \cU \subset \Upsilon(\rho_0)$ and $395 \|\ttD_*^{-1} \z\| \leq \mu \sqrt{\lambda}$
    on $\cE$. Then, on this event, it holds that
    \[
        \left\| S \big( \widehat\bups - \bups^* - H_*^{-1} \bz \big) \right\|
        \leq \frac{16}{\mu_0 \sqrt{\lambda}} \, \| \ttD_0 H_*^{-1} \bz \|^2 \left( \lambda^{-1/4} + \sqrt{\frac{\left\| \Sigma - A^*\right\|}{\lambda}} \right) \left(1 + \frac{\| \ttD_0 H_*^{-1} \bz \|}{3 \mu_0 \sqrt{\lambda}} \right).
    \]
\end{Lem}
The proof of Lemma \ref{lem:stoch_term_expansion} is postponed to Section \ref{sec:lem_stoch_term_expansion_proof}. 
Let us note that the first $d$ components of the vector $S (\widehat \bups - \bups^* - H_*^{-1} \z)$ coincide with $\Sigma^{1/2} (\widehat \btheta - \btheta^* - \bs_{\btheta}^*)$ while the other entries of $S (\widehat \bups - \bups^* - H_*^{-1} \z)$ are equal to zero. This yields that
\[
    \left\| \Sigma^{1/2} (\widehat \btheta - \btheta^* - \bs_{\btheta}^*) \right\|
    = \left\| S (\widehat \bups - \bups^* - H_*^{-1} \z) \right\|.
\]
Moreover, according to Theorems \ref{th:covariance_concentration} and \ref{th:noise_concentration}, with probability at least $(1 - 3\delta / 2)$, we simultaneously have
\begin{equation}
    \label{eq:e1_sigma}
    \|\widehat \Sigma - \Sigma\|_{\F} \leq 4 C_X \|\Sigma\| \sqrt{\frac{\ttr(\Sigma)^2 + \log(4 / \delta)}n},
    \quad 
    \|(\widehat \Sigma - \Sigma) \btheta^\circ\| \leq 4 C_X \|\Sigma\| \|\btheta^\circ\| \sqrt{\frac{\ttr(\Sigma) + \log(4 / \delta)}n},
\end{equation}
and
\begin{equation}
    \label{eq:e1_u}
    \|\bU\| \leq 8 \sigma \|\Sigma\|^{1/2} \sqrt{\frac{\ttr(\Sigma) + \log(4 / \delta)}n}.
\end{equation}
We denote the corresponding event by $\cE_1$.
Then the inequalities
\eqref{eq:stoch_term_mu_lambda_n_conditions} ensure that
\[
    395 \|\widehat \Sigma - \Sigma\|_{\F}
    \leq \frac{\sqrt{\lambda}}2,
    \quad
    395 \|\bU\| \leq \mu \sqrt{\lambda},
    \quad \text{and} \quad
    395 \|(\widehat \Sigma - \Sigma) \btheta^\circ\| \leq \frac{\|\btheta^\circ\| \sqrt{\lambda}}{2}.
\]
Taking into account the triangle inequality $\|\bZ - \E \bZ\| \leq \|\bU\| + \|(\widehat \Sigma - \Sigma) \btheta^\circ\|$, we obtain that
\begin{align*}
    395^2 \|\ttD_*^{-1} \z\|^2
    \leq \frac{395^2}2 \|\bZ - \E \bZ\|^2 + 395^2 \mu^2 \|\widehat \Sigma - \Sigma\|_{\F}^2
    \leq \frac{\mu^2 \lambda}2 \left( 1 + \frac{\|\btheta^\circ\|}{2 \mu} \right)^2 + \frac{\mu^2 \lambda}4
    \leq \mu^2 \lambda.
\end{align*}
Thus, the conditions of Lemma \ref{lem:stoch_term_expansion} are satisfied on the intersection of $\cE_0$ and $\cE_1$ with $\mu_0 = \|\btheta^\circ\| \sqrt{3}$. Due to the union bound $\p(\cE_0 \cap \cE_1) \geq 1 - 5 \delta / 2$. Moreover, applying Lemma \ref{lemma: Sigma bias optimal bound} with $14 \|\btheta^\circ\| / \mu \leq 1/8$, we obtain that
\[
    \|A^* - \Sigma\| \leq \frac{\sqrt{\lambda}}{64} + \frac{5 \|\Sigma\| \|\bb_\lambda\|}{16 \mu}.
\]
Since
\[
    \mu \sqrt{\lambda} \geq 8 \cdot 18 \cdot 56 \|\Sigma\| \|\btheta^\circ\|
    \quad \text{and} \quad
    \|\bb_\lambda\| = \lambda \left\| \left( \frac12 \Sigma^2 + \lambda I_d \right)^{-1} \btheta^\circ \right\| \leq \|\btheta^\circ\|,
\]
it holds that
\[
    \|A^* - \Sigma\|
    \leq \frac{\sqrt{\lambda}}{64} + \frac{5 \sqrt{\lambda}}{16 \cdot 8 \cdot 18 \cdot 56}
    = \frac{4 \sqrt{\lambda}}{16^2} + \frac{5 \sqrt{\lambda}}{16^2 \cdot 9 \cdot 56}
    \leq \frac{9 \sqrt{\lambda}}{16^2}.
\]
Hence, we have just proved that the remainder $(\widehat\btheta - \btheta^* - \bs_{\btheta}^*)$ satisfies
\begin{align}
    \label{eq:stoch_term_expansion_simplified}
    \left\| \Sigma^{1/2} (\widehat \btheta - \btheta^* - \bs_{\btheta}^*) \right\|
    &\notag
    = \left\| S \big( \widehat\bups - \bups^* - H_*^{-1} \bz \big) \right\|
    \\&
    \leq \frac{16}{\mu_0 \sqrt{\lambda}} \, \| \ttD_0 H_*^{-1} \bz \|^2 \left( \lambda^{-1/4} + \frac{3 \lambda^{-1/4}}{16} \right) \left(1 + \frac{\| \ttD_0 H_*^{-1} \bz \|}{3 \mu_0 \sqrt{\lambda}} \right)
    \\&\notag
    \leq \frac{19}{\mu_0 \lambda^{3/4}} \, \| \ttD_0 H_*^{-1} \bz \|^2 \left(1 + \frac{\| \ttD_0 H_*^{-1} \bz \|}{3 \mu_0 \sqrt{\lambda}} \right)
\end{align}
on the event $\cE_0 \cap \cE_1$ of probability at least $(1 - 5\delta / 2)$. We proceed with the analysis of $\ttD_0 H_*^{-1} \bz$.

\medskip

\noindent\textbf{Step 3: bound on the norm of $\ttD_0 H_*^{-1} \bz$.}\quad
The proof of the upper bound on $\|\ttD_0 H_*^{-1} \bz\|$ is based on a block representation of $\nabla^2 \cL(\bups)$. Let  $\chi = (\bfeta, A) \in \R^{d + d^2}$ stand for the nuisance parameter. For any $\bups = (\btheta, \chi) = (\btheta, \bfeta, A) \in \R^{2d + d^2}$, we denote
\[
    H_{\btheta \btheta} = \nabla_{\btheta \btheta}^2 \cL(\bups),
    \quad
    H_{\btheta \chi} = H_{\chi \btheta}^\top = \nabla_{\btheta \chi}^2 \cL(\bups),
    \quad \text{and} \quad
    H_{\chi \chi} = \nabla_{\chi \chi}^2 \cL(\bups).
\]
With the introduced notations, the Hessian of $\cL(\bups)$ is equal to 
\[
    H = H(\bups) =
    \begin{pmatrix}
        H_{\btheta \btheta} & H_{\btheta \chi} \\
        H_{\chi \btheta} & H_{\chi \chi}
    \end{pmatrix}.
\]
Let $\breve H_{\btheta \btheta} = H / H_{\chi \chi}$ and $\breve H_{\chi \chi} = H / H_{\btheta \btheta}$ stand for the Schur complements of $H_{\chi \chi}$ and $H_{\btheta \btheta}$, respectively. Using the block-matrix inversion formula (see \eqref{eq:h_block_inversion})
\[
    \begin{pmatrix}
        H_{\btheta \btheta} & H_{\btheta \chi} \\
        H_{\chi \btheta} & H_{\chi \chi}
    \end{pmatrix}^{-1}
    =
    \begin{pmatrix}
        \breve H_{\btheta \btheta}^{-1} & -\breve H_{\btheta \btheta}^{-1} H_{\btheta \chi} H_{\chi \chi}^{-1} \\
        -\breve H_{\chi \chi}^{-1} H_{\chi \btheta} H_{\btheta \btheta}^{-1} & \breve H_{\chi \chi}^{-1}
    \end{pmatrix}
\]
and taking into account that $\bz$ has a form $(\bzero, \bz_\chi)$, where $\bz_\chi = \big( \bZ - \E \bZ, \mu^2 \, \rmvec(\widehat \Sigma - \Sigma) \big)$,
we obtain that
\begin{equation}
    \label{eq:standardized_score}
    H^{-1} \bz
    =
    \begin{pmatrix}
        \breve H_{\btheta \btheta}^{-1} & -\breve H_{\btheta \btheta}^{-1} H_{\btheta \chi} H_{\chi \chi}^{-1} \\
        -\breve H_{\chi \chi}^{-1} H_{\chi \btheta} H_{\btheta \btheta}^{-1} & \breve H_{\chi \chi}^{-1}
    \end{pmatrix}
    \begin{pmatrix}
        \bzero \\ \z_\chi
    \end{pmatrix}
    =
    \begin{pmatrix}
        -\breve H_{\btheta \btheta}^{-1} H_{\btheta \chi} H_{\chi \chi}^{-1} \z_\chi \\
        \breve H_{\chi \chi}^{-1} \z_\chi.
    \end{pmatrix}
\end{equation}
Then, recalling the definition of $\ttD_0$ (see \eqref{eq:D0}), we observe that
\begin{equation}
    \label{eq:standardized_score_pythagoras}
    \| \ttD_0 H_*^{-1} \bz \|^2
    = \left\| \left( \frac12 (A^*)^\top A^* + \lambda I_d \right)^{1/2} \breve H_{\btheta \btheta}^{-1}(\bups^*) H_{\btheta \chi}(\bups^*) H_{\chi \chi}^{-1}(\bups^*) \z_\chi \right\|^2
    + \left\| \ttD_{0 \chi} \breve H_{\chi \chi}^{-1}(\bups^*) \z_\chi \right\|^2,
\end{equation}
where we introduced
\[
    \ttD_{0 \chi}^2
    = \diag\big( 2 I_d, I_d \otimes (\mu_0^2 I_d + \btheta^* (\btheta^*)^\top) \big).
\]
The two terms in the right-hand side of \eqref{eq:standardized_score_pythagoras} should be studied separately. We start with the analysis of the subvector of $\ttD_0 H_*^{-1} \bz$ corresponding to the target parameter.
\begin{Lem}
    \label{lem:standardized_score_theta_bound}
    Grant Assumptions \ref{as:quadratic} and \ref{as:noise} and let $\bups \in \Ups(\rho)$ for some $\rho \leq 1/5$ (see \eqref{eq:upsilon} for the definition of $\Ups(\rho)$). Let $\delta \in (0, 1)$ and suppose that
    \[
        \ttr(\Sigma) + \log(4/\delta) \leq n.
    \]
    Let $\cE_1$ be the event where \eqref{eq:e1_sigma} and \eqref{eq:e1_u} hold. Then there exists an event $\cE_2$ of probability at least $1 - 5 \delta / 2$ such that
    \begin{itemize}
        \item[(i)] the following inequalities hold simultaneously on $\cE_2$:
        \begin{align}
            \label{eq:hat_sigma_linear_functional_1_deviation_bound}
            \| (\widehat \Sigma - \Sigma)(\btheta - \btheta^\circ) \|
            &
            \leq 4 C_X \|\Sigma\|^{1/2} \|\Sigma^{1/2} (\btheta - \btheta^\circ) \| \sqrt{\frac{\ttr(\Sigma) + \log(4 / \delta)}n},
            \\
            \label{eq:hat_sigma_linear_functional_2_deviation_bound}
            \| (\widehat \Sigma - \Sigma) (A \btheta - \bfeta) \|
            &
            \leq 4 C_X \|\Sigma\|^{1/2} \|\Sigma^{1/2} (A \btheta - \bfeta)\| \sqrt{\frac{\ttr(\Sigma) + \log(4 / \delta)}n},
            \\
            \label{eq:hat_sigma_linear_functional_3_deviation_bound}
            | (A \btheta - \bfeta)^\top (\widehat \Sigma - \Sigma) \btheta |
            &
            \leq 4 C_X \|\Sigma^{1/2} \btheta\| \|\Sigma^{1/2} (A \btheta - \bfeta)\| \sqrt{\frac{1 + \log(4 / \delta)}n},
            \\
            \label{eq:hat_sigma_linear_functional_4_deviation_bound}
            | (A \btheta - \bfeta)^\top (\widehat \Sigma - \Sigma) \btheta^\circ |
            &
            \leq 4 C_X \|\Sigma^{1/2} \btheta^\circ\| \|\Sigma^{1/2} (A \btheta - \bfeta)\| \sqrt{\frac{1 + \log(4 / \delta)}n},
            \\
            \label{eq:z_linear_functional_deviation_bound}
            \left| (A \btheta - \bfeta)^\top \bU \right|
            &
            \leq 8 \sigma \|\Sigma^{1/2} (A \btheta - \bfeta)\| \sqrt{\frac{1 + \log(4 / \delta)}n};
        \end{align}
        \item[(ii)] on the intersection of $\cE_1$ and $\cE_2$, one has
        \begin{align*}
            &
            \left\| \left( \frac12 A^\top A + \lambda I_d \right)^{1/2} \breve H_{\btheta \btheta}^{-1} H_{\btheta \chi} H_{\chi \chi}^{-1} \bz_\chi \right\|
            \\&
            \leq \frac1{1 - 10\rho^2} \left\| \left( \frac12 A^\top A + \lambda I_d \right)^{-1/2} H_{\btheta \chi} H_{\chi \chi}^{-1} \bz_\chi \right\|
            \\&
            \leq \frac{2 \sqrt{2} \|\Sigma\|^{1/2}}{1 - 10\rho^2} \left(
            2 \sigma + C_X \|\Sigma^{1/2} (\btheta - \btheta^\circ)\| + \frac{\sqrt{2} C_X \|\Sigma^{1/2} (A \btheta - \bfeta)\|}{\sqrt{\lambda}} \right) \sqrt{\frac{\ttr(\Sigma) + \log(4 / \delta)}n}
            \\&\quad
            + \frac{4 \sqrt{2} \rho^2 \|\Sigma\|^{1/2}}{1 - 10\rho^2} \left(
            2 \sigma + 3 C_X \|\Sigma^{1/2} \btheta^\circ\| \right) \sqrt{\frac{\ttr(\Sigma) + \log(4 / \delta)}n}
            \\&\quad
            + \frac{8 \rho \|\Sigma^{1/2} (A \btheta - \bfeta)\|}{(1 - 10\rho^2) \mu \sqrt{\lambda}} \left( C_X \|\Sigma^{1/2} \btheta\| + \sigma \right) \sqrt{\frac{1 + \log(4 / \delta)}n}.
        \end{align*}
    \end{itemize}
\end{Lem}
We provide the proof of Lemma \ref{lem:standardized_score_theta_bound} in Section \ref{sec:lem_standardized_score_theta_bound_proof}. Let us elaborate on the result of this lemma when $\bups = \bups^*$. According to Lemma \ref{lemma: localization lemma for bias} (see \ref{point: eta weak bound on bias, bias locating convex set} and \ref{point: theta weak bound on bias, bias locating convex set}), we have
\[
    \|A^* \btheta^* - \bfeta^*\| \leq \sqrt{\lambda} \|\btheta^\circ\|,
    \quad
    \|\btheta^*\| \leq \|\btheta^\circ\|,
\]
and
\[
    \|\btheta^* - \btheta^\circ\|
    \leq \|\bb_\lambda\| \left(1 + \frac{3 \|\btheta^\circ\|}{\mu} \right)
    \leq \|\btheta^\circ\| \left(1 + \frac{3}{111} \right)
    = \frac{38 \|\btheta^\circ\|}{37}.
\]
This means that $\btheta^*$ belongs to $\Ups(\rho^*)$ with $\rho^* = \|\btheta^\circ\| / \mu \leq 1/112$. Hence, there exists an event $\cE_2^*$ such that $\p(\cE_2^*) \geq (1 - 5\delta / 2)$ and the inequalities
\begin{align}
    \label{eq:hat_sigma_linear_functional_star_1_deviation_bound}
    \| (\widehat \Sigma - \Sigma) \btheta^*\|
    &
    \leq 4 C_X \|\Sigma\|^{1/2} \|\Sigma^{1/2} \btheta^*\| \sqrt{\frac{\ttr(\Sigma) + \log(4 / \delta)}n},
    \\
    \label{eq:hat_sigma_linear_functional_star_2_deviation_bound}
    \| (\widehat \Sigma - \Sigma) (A^* \btheta^* - \bfeta^*) \|
    &
    \leq 4 C_X \|\Sigma\|^{1/2} \|\Sigma^{1/2} (A^* \btheta^* - \bfeta^*)\| \sqrt{\frac{\ttr(\Sigma) + \log(4 / \delta)}n},
    \\
    \label{eq:hat_sigma_linear_functional_star_3_deviation_bound}
    | (A^* \btheta^* - \bfeta^*)^\top (\widehat \Sigma - \Sigma) \btheta^* |
    &
    \leq 4 C_X \|\Sigma^{1/2} \btheta^*\| \|\Sigma^{1/2} (A^* \btheta^* - \bfeta^*)\| \sqrt{\frac{1 + \log(4 / \delta)}n},
    \\
    \label{eq:hat_sigma_linear_functional_star_4_deviation_bound}
    | (A^* \btheta^* - \bfeta^*)^\top (\widehat \Sigma - \Sigma) \btheta^\circ |
    &
    \leq 4 C_X \|\Sigma^{1/2} \btheta^\circ\| \|\Sigma^{1/2} (A^* \btheta^* - \bfeta^*)\| \sqrt{\frac{1 + \log(4 / \delta)}n},
    \\
    \label{eq:z_linear_functional_star_deviation_bound}
    \left| (A^* \btheta^* - \bfeta^*)^\top \bU \right|
    &
    \leq 8 \sigma \|\Sigma^{1/2} (A^* \btheta^* - \bfeta^*)\| \sqrt{\frac{1 + \log(4 / \delta)}n}
\end{align}
hold simultaneously on $\cE_2^*$.
Moreover, on $\cE_1 \cap \cE_2^*$ it holds that
\begin{align}
    \label{eq:standardized_score_theta_bound_simplified_1}
    &\notag
    \left\| \left( \frac12 (A^*)^\top A^* + \lambda I_d \right)^{1/2} \breve H_{\btheta \btheta}^{-1}(\bups^*) H_{\btheta \chi}(\bups^*) H_{\chi \chi}^{-1}(\bups^*) \bz_\chi \right\|
    \\&
    \leq 1.001 \left\| \left( \frac12 (A^*)^\top A^* + \lambda I_d \right)^{-1/2} H_{\btheta \chi}(\bups^*) H_{\chi \chi}^{-1}(\bups^*) \bz_\chi \right\|
\end{align}
and
\begin{align}
    \label{eq:standardized_score_theta_bound_simplified_2}
    &\notag
    \left\| \left( \frac12 (A^*)^\top A^* + \lambda I_d \right)^{-1/2} H_{\btheta \chi}(\bups^*) H_{\chi \chi}^{-1}(\bups^*) \bz_\chi \right\|
    \\&\notag
    \leq 2 \sqrt{2} \|\Sigma\|^{1/2} \left(
    2 \sigma + \frac{38 C_X \|\Sigma\|^{1/2} \|\btheta^\circ\|}{37} + \sqrt{2} C_X \|\Sigma\|^{1/2} \|\btheta^\circ\| \right) \sqrt{\frac{\ttr(\Sigma) + \log(4 / \delta)}n}
    \\&\quad
    + \frac{4 \sqrt{2}\|\Sigma\|^{1/2}}{112^2} \left(
    2 \sigma + 3 C_X \|\Sigma\|^{1/2} \|\btheta^\circ\| \right) \sqrt{\frac{\ttr(\Sigma) + \log(4 / \delta)}n}
    \\&\quad\notag
    + \frac{8 \|\Sigma\|^{1/2}}{112^2} \left( C_X \|\Sigma\|^{1/2} \|\btheta^\circ\| + \sigma \right) \sqrt{\frac{1 + \log(4 / \delta)}n}
    \\&\notag
    \leq \left(4 \sqrt{2} + \frac{8 \sqrt{2}}{112^2} + \frac{8}{112^2} \right) \|\Sigma\|^{1/2}\left( \sigma + 2 C_X \|\Sigma\|^{1/2} \|\btheta^\circ\|\right) \sqrt{\frac{\ttr(\Sigma) + \log(4 / \delta)}n}.
\end{align}

Finally, the next lemma concerns the norm of $(\ttD_0 H_*^{-1} \bz)_{\chi}$ and finishes the third step of the proof.
\begin{Lem}
    \label{lem:standardized_score_chi_bound}
    Let $\rho \leq 1/16$ and fix an arbitrary $\bups \in \Ups(\rho)$, where $\Ups(\rho)$ is defined in \eqref{eq:upsilon}. Introduce
    \[
        \widetilde \ttD^2 = \diag\left(2 I_d, I_d \otimes (\widetilde\mu^2 I_d + \btheta \btheta^\top) \right) \in \R^{(d + d^2) \times (d + d^2)},
    \]
    where $\widetilde \mu$ is an arbitrary number from $[0, \mu]$. Finally, let $\cE_1$ be the event where \eqref{eq:e1_sigma} and \eqref{eq:e1_u} hold. Then, on this event, we have
    \begin{align*}
        \left\| \widetilde \ttD \breve H_{\chi \chi}^{-1} \z_\chi \right\|
        &
        \leq 8 \left(1 + \frac{9 \rho}2 \right) \sigma \|\Sigma\|^{1/2} \sqrt{\frac{\ttr(\Sigma) + \log(4 / \delta)}n}
        \\&\quad
        + 4 \left(1 + \frac{9 \rho}2 \right) C_X \|\Sigma\| \|\btheta^\circ\| \sqrt{\frac{\ttr(\Sigma) + \log(4 / \delta)}n}
        \\&\quad
        + 4 \left(\sqrt{\widetilde \mu^2 + \|\btheta\|^2} + \frac{9 \rho \mu}2\right) C_X \|\Sigma\| \sqrt{\frac{\ttr(\Sigma)^2 + \log(4 / \delta)}n}.
    \end{align*}
\end{Lem}
The proof of Lemma \ref{lem:standardized_score_chi_bound} is moved to Section \ref{sec:lem_standardized_score_chi_bound_proof}.
Similarly to Lemma \ref{lem:standardized_score_theta_bound}, we are going to apply Lemma \ref{lem:standardized_score_chi_bound} to $\bups = \bups^*$. As we discussed earlier, $\bups^*$ lies in $\Ups(\rho)$, where $\rho^* = \|\btheta^\circ\| / \mu \leq 1 / 112$. Taking $\widetilde \mu = \mu_0 = \|\btheta^\circ\| \sqrt{3}$, $\rho = \rho^* \leq 1 / 112$, and $\bups = \bups^*$, we obtain that
\begin{align}
    \label{eq:standardized_score_chi_bound_simplified}
    \left\| \ttD_{0 \chi} \breve H_{\chi \chi}^{-1}(\bups^*) \z_\chi \right\|
    &\notag
    \leq 8 \left(1 + \frac{9}{224} \right) \sigma \|\Sigma\|^{1/2} \sqrt{\frac{\ttr(\Sigma) + \log(4 / \delta)}n}
    \\&\quad\notag
    + 4 \left(1 + \frac{9}{224} \right) C_X \|\Sigma\| \|\btheta^\circ\| \sqrt{\frac{\ttr(\Sigma) + \log(4 / \delta)}n}
    \\&\quad
    + 4 \left( 2\|\btheta^\circ\| + \frac{9 \|\btheta^\circ\|}{224} \right) C_X \|\Sigma\| \sqrt{\frac{\ttr(\Sigma)^2 + \log(4 / \delta)}n}
    \\&\notag
    \leq \frac{17}2 \left( \sigma \|\Sigma\|^{1/2} + 2 C_X \|\Sigma\| \|\btheta^\circ\| \right) \sqrt{\frac{\ttr(\Sigma)^2 + \log(4 / \delta)}n}. 
\end{align}
on the event $\cE_1$. Here we used the bound $\|\btheta^*\| \leq \|\btheta^\circ\|$ following from the definition of $\btheta^*$ (see Lemma \ref{lemma: localization lemma for bias}\ref{point: theta weak bound on bias, bias locating convex set}). The inequalities \eqref{eq:standardized_score_pythagoras}, \eqref{eq:standardized_score_theta_bound_simplified_1}, \eqref{eq:standardized_score_theta_bound_simplified_2}, and \eqref{eq:standardized_score_chi_bound_simplified} immediately imply that
\begin{align*}
    \| \ttD_0 H_*^{-1} \bz \|
    &
    \leq 1.001 \left(4 \sqrt{2} + \frac{8 \sqrt{2}}{112^2} + \frac{8}{112^2} \right) \|\Sigma\|^{1/2}\left( \sigma + 2 C_X \|\Sigma\|^{1/2} \|\btheta^\circ\|\right) \sqrt{\frac{\ttr(\Sigma) + \log(4 / \delta)}n}
    \\&\quad
    +  8 \left(1 + \frac{9}{224} \right) \left( \sigma \|\Sigma\|^{1/2} + 2 C_X \|\Sigma\| \|\btheta^\circ\| \right) \sqrt{\frac{\ttr(\Sigma)^2 + \log(4 / \delta)}n}
    \\&
    \leq 14 \left( \sigma \|\Sigma\|^{1/2} + 2 C_X \|\Sigma\| \|\btheta^\circ\| \right) \sqrt{\frac{\ttr(\Sigma)^2 + \log(4 / \delta)}n}
    \equiv \|\btheta^\circ\| \sqrt{\Psi(n, \delta)}
\end{align*}
on the intersection of the events $\cE_0$, $\cE_1$, and $\cE_2^*$ (that is, with probability at least $1 - 5\delta$). In view of \eqref{eq:stoch_term_expansion_simplified}, on the same event, we have
\begin{align}
    \label{eq:stoch_term_remainder_bound}
    \left\| \Sigma^{1/2} (\widehat \btheta - \btheta^* - \bs_{\btheta}^*) \right\|
    &\notag
    = \left\| S (\widehat \bups - \bups^* - H_*^{-1} \z) \right\|
    \\&
    \leq \frac{19}{\mu_0 \lambda^{3/4}} \, \| \ttD_0 H_*^{-1} \bz \|^2 \left(1 + \frac{\| \ttD_0 H_*^{-1} \bz \|}{3 \mu_0 \sqrt{\lambda}} \right)
    \\&\notag
    \leq \frac{19 \|\btheta^\circ\|}{\sqrt{3} \lambda^{3/4}} \left(1 + \sqrt{\frac{\Psi(n, \delta)}{27 \lambda}} \right) \Psi(n, \delta).
\end{align}
and \eqref{eq:stoch_term_remainder_bound_asymp} follows.

\medskip

\noindent\textbf{Step 4: elaborating on $\bs_{\btheta}^*$.}
\quad
It remains to quantify the difference between $\bs_{\btheta}^*$ and
\[
    \frac12 \left(\frac12 \Sigma^2 + \lambda I_d \right)^{-1} \left( \Sigma (\bZ - \E\bZ) - \Sigma (\widehat \Sigma - \Sigma) \btheta^* - (\widehat \Sigma - \Sigma) \Sigma \bb_\lambda \right).
\]
We perform this task in three tiny steps. First, we note that $\bs_{\btheta}^*$ can be approximated by
\[
    -\left( \Sigma^2 + \lambda I_d \right)^{-1} H_{\btheta \chi}(\bups^*) H_{\chi \chi}^{-1}(\bups^*) \z_\chi.
\]
We support this claim by the following result.

\begin{Lem}
    \label{lem:standardized_score_simplified}
    Assume that the parameters $\mu$ and $\lambda$ fulfil \eqref{eq:stoch_term_mu_lambda_conditions} and that $\mu \geq 112 \|\btheta^\circ\|$. Then, on the event where \eqref{eq:standardized_score_theta_bound_simplified_2} holds, we also have
    \begin{align*}
        &
        \left\| \Sigma^{1/2} \bs_{\btheta}^* 
        + \Sigma^{1/2} \left( \Sigma^2 + \lambda I_d \right)^{-1} H_{\btheta \chi}(\bups^*) H_{\chi \chi}^{-1}(\bups^*) \z_\chi \right\|
        \\&
        \leq \frac{\|\btheta^\circ\|}{2 \lambda^{1/4}} 
         \left( \left( \frac{17 \|\btheta^\circ\|}{\mu}\right)^2 + \frac{\|\bb_\lambda\|}{160 \mu} \right) \sqrt{\Psi(n, \delta)},
    \end{align*}
    where
    \[
        \Psi(n, \delta) = 196 \left( \frac{\sigma \|\Sigma\|^{1/2}}{\|\btheta^\circ\|} + 2 C_X \|\Sigma\| \right)^2 \cdot \frac{\ttr(\Sigma)^2 + \log(4 / \delta)}n.
    \]
\end{Lem}
We postpone the proof of Lemma \ref{lem:standardized_score_simplified} to Section \ref{sec:lem_standardized_score_simplified_proof} and finish the proof of Theorem \ref{th:stoch_term} first. The next auxiliary lemmata ensure that $H_{\btheta \chi}(\bups^*) H_{\chi \chi}^{-1}(\bups^*) \z_\chi$ is close to
\[
    -\frac12 \left( \Sigma (\bZ - \E \bZ) - \Sigma (\widehat \Sigma - \Sigma) \btheta^* - (\widehat \Sigma - \Sigma) \Sigma \bb_\lambda \right).
\]
\begin{Lem}
    \label{lem:semiparametric_remainder_1}
    Assume that the parameters $\mu$ and $\lambda$ fulfil \eqref{eq:stoch_term_mu_lambda_conditions} and that $\mu \geq 112 \|\btheta^\circ\|$. Then, on the event $\cE_1 \cap \cE_2^*$, the vector
    \[
        \btau =
        (A^*)^\top \bU
        - (A^*)^\top (\widehat \Sigma - \Sigma) (\btheta^* - \btheta^\circ)
        - 2(\widehat \Sigma - \Sigma) (A^* \btheta^* - \bfeta^*)
    \]
    satisfies the inequality
    \[
        \left\| \Sigma^{1/2} \left( \Sigma^2 / 2 + \lambda I_d \right)^{-1} \left( H_{\btheta \chi}(\bups^*) H_{\chi \chi}^{-1}(\bups^*) \z_\chi + 0.5 \btau \right) \right\|
        \leq \frac{\|\btheta^\circ\|^3}{\mu^2 \lambda^{1/4}} \sqrt{\Psi(n, \delta)}.
    \]
\end{Lem}
\begin{Lem}
    \label{lem:semiparametric_remainder_2}
    Assume that the parameters $\mu \geq 112 \|\btheta^\circ\|$ and $\lambda > 0$ fulfil \eqref{eq:stoch_term_mu_lambda_conditions} and that $4n \geq \ttr(\Sigma) + \log(4 / \delta)$. Let $\btau$ be as defined in Lemma \ref{lem:semiparametric_remainder_2} and let
    \[
        \bzeta =
        \Sigma \bU
        - \Sigma (\widehat \Sigma - \Sigma) \bb_\lambda
        - (\widehat \Sigma - \Sigma) \Sigma \bb_\lambda.
    \]
    Then there exists an event $\cE_3$ of probability at least $(1 - 3\delta / 2)$ such that on $\cE_1 \cap \cE_3$, it holds that
    \[
        \|\Sigma^{1/2} \left( \Sigma^2 + 2 \lambda I_d \right)^{-1} (\btau - \bzeta)\| \leq \frac{12 \|\btheta^\circ\|}{25 \lambda^{1/4}} \left( \left(\frac{17 \|\btheta^\circ\|}\mu \right)^2 + \frac{\|\bb_\lambda\|}{160 \mu} \right) \sqrt{\Psi(n, \delta)}.
    \]
\end{Lem}
The proofs of Lemma \ref{lem:semiparametric_remainder_1} and Lemma \ref{lem:semiparametric_remainder_2} are moved to Sections \ref{sec:lem_semiparametric_remainder_1_proof} and \ref{sec:lem_semiparametric_remainder_2_proof}, respectively. Summing up Lemmata \ref{lem:standardized_score_simplified}, \ref{lem:semiparametric_remainder_1}, and \ref{lem:semiparametric_remainder_2}, we obtain that 
\begin{align}
    \label{eq:standardized_score_simplified_bound}
    \left\| \Sigma^{1/2} \bs_{\btheta}^* 
    - \Sigma^{1/2} \left( \Sigma^2 + 2\lambda I_d \right)^{-1} \bzeta \right\|
    &\notag
    \leq \left( \frac12 + \frac{12}{25} + \frac1{17^2} \right) \frac{\|\btheta^\circ\|}{\lambda^{1/4}} \left( \left(\frac{17 \|\btheta^\circ\|}\mu \right)^2 + \frac{\|\bb_\lambda\|}{160 \mu} \right) \sqrt{\Psi(n, \delta)}
    \\&
    \leq \frac{\|\btheta^\circ\|}{\lambda^{1/4}} \left( \left(\frac{17 \|\btheta^\circ\|}\mu \right)^2 + \frac{\|\bb_\lambda\|}{160 \mu} \right) \sqrt{\Psi(n, \delta)}
\end{align}
on the event $\cE_0 \cap \cE_1 \cap \cE_2^* \cap \cE_3$.
Finally, the triangle inequality and the bounds \eqref{eq:stoch_term_remainder_bound}, \eqref{eq:standardized_score_simplified_bound}, imply that, with probability at least $(1 - 13\delta / 2)$,
\begin{align*}
    \left\| \Sigma^{1/2} \left(\widehat \btheta - \btheta^* - \left( \Sigma^2 + 2 \lambda I_d\right)^{-1} \bzeta \right) \right\|
    &
    \leq \frac{19 \|\btheta^\circ\|}{\sqrt{3} \lambda^{3/4}} \left(1 + \sqrt{\frac{\Psi(n, \delta)}{27 \lambda}} \right) \Psi(n, \delta)
    \\&\quad
    + \frac{\|\btheta^\circ\|}{\lambda^{1/4}} \left( \left( \frac{17 \|\btheta^\circ\|}{\mu}\right)^2 + \frac{\|\bb_\lambda\|}{160 \mu} \right) \sqrt{ \Psi(n, \delta)}.
\end{align*}

\myendproof

\bibliographystyle{abbrvnat}
\bibliography{references}

\appendix

\section{Auxiliary results: properties of the Hessian of $\cL(\bups)$}
\label{sec:hessian}

Appendix \ref{sec:hessian} contains some useful properties of the Hessian of $\cL(\bups)$. Throughout this section, denote $\nabla^2 \cL(\bups)$ by $H = H(\bups)$ for brevity. The value of $\bups$ will always be clear from context. We will extensively use the block form of $H$:
\begin{equation}
    \label{eq:h_block_form}
    H =
    \begin{pmatrix}
        H_{\btheta\btheta} & H_{\btheta\bfeta} & H_{\btheta A} \\
        H_{\bfeta\btheta} & H_{\bfeta\bfeta} & H_{\bfeta A} \\
        H_{A\btheta} & H_{A\bfeta} & H_{AA}
    \end{pmatrix}.
\end{equation}
Here the diagonal blocks correspond to
\begin{align*}
    H_{\btheta\btheta} = A^\top A + \lambda I_d,
    \quad
    H_{\bfeta\bfeta} = 2 I_d,
    \quad \text{and} \quad
    H_{AA} = \diag \big(H_{\ba_1 \ba_1}, \dots, H_{\ba_d \ba_d}\big),
\end{align*}
while the non-diagonal ones are given by
\[
    H_{\btheta \bfeta} = H_{\bfeta \btheta}^\top = -A^\top,
    \quad
    H_{\btheta A} = H_{A \btheta}^\top = \begin{pmatrix} H_{\btheta\ba_1} & \dots & H_{\btheta\ba_d} \end{pmatrix},
    \quad
    H_{\bfeta A} = H_{A \bfeta}^\top = \begin{pmatrix} H_{\bfeta\ba_1} & \dots & H_{\bfeta\ba_d} \end{pmatrix},
\]
where, for any $j \in \{1, \dots, d\}$,
\begin{align*}
    H_{\ba_j \ba_j} = \mu^2 I_d + \btheta \btheta^\top,
    \quad
    H_{\btheta \ba_j} = (\ba_j^\top \btheta - \bfeta_j) I_d + \ba_j \btheta^\top
    \quad \text{and} \quad
    H_{\bfeta \ba_j} = - \be_j \btheta^\top.
\end{align*}
Throughout the proofs, we will often use a representation of $H$ in terms of the Kronecker product. Let us recall that, for any matrices $V \in \R^{p \times q}$ and $W \in \R^{r \times s}$, their Kronecker product $U \otimes W$ is a matrix of size $pq \times rs$, defined as
\[
    \begin{pmatrix}
        V_{11} W & \dots & V_{1q} W \\
        \vdots & \ddots & \vdots \\
        V_{p1} W & \dots & V_{pq} W 
    \end{pmatrix}.
\]
Then we can rewrite $H$ in the following form:
\begin{equation}
    \label{eq:h_block_form_kronecker}
    H =
    \begin{pmatrix}
        A^\top A + \lambda I_d & -A^\top & (A \btheta - \bfeta)^\top \otimes I_d + A^\top \otimes \btheta^\top \\
        -A & 2 I_d & -  I_d \otimes \btheta^\top \\
        (A \btheta - \bfeta) \otimes I_d + A \otimes \btheta & - I_d \otimes \btheta & I_d \otimes (\mu^2 I_d + \btheta \btheta^\top) \\
    \end{pmatrix}.
\end{equation}

The first result indicates that $\nabla^2 \cL(\bups)$ is positive definite on a set which is sufficiently large for our purposes. In particular, according to Lemma \ref{lemma: localization_probabilistic_conditions}, the proper choice of the parameters $\mu$, $\lambda$, and $\rho$ yields that both $\widehat \bups$ and $\bups^*$ belong to this set.

\begin{Lem}
    \label{lem:block-diagonal_lower_bound}
    Let us fix any $\rho \in [0, 1/7]$ and let $\bups = (\btheta, \bfeta, \rmvec(A))$ be any triplet from
    \[
        \Ups(\rho) = \left\{ \bups = (\btheta, \bfeta, \rmvec(A)) : (1 - \rho^2) \|\btheta\|^2 \leq \rho^2 \mu^2, \|\bfeta - A\btheta\|^2 \leq \rho^2 \mu^2 \lambda \right\}.
    \]
    Then it holds that
    \[
        \nabla^2 \cL(\bups) = H \succeq \frac{1 - 2 \rho}4 \, \ttD^2,
    \]
    where
    \[
        \ttD^2
        = \ttD^2(\bups)
        = \diag \left(H_{\btheta \btheta}, H_{\bfeta \bfeta}, H_{AA} \right)
        = \diag\left( A^\top A + \lambda I_d, 2 I_d, I_d \otimes (\mu^2 I_d + \btheta \btheta^\top) \right).
    \]
\end{Lem}
We postpone the proof of Lemma \ref{lem:block-diagonal_lower_bound} to Section \ref{sec:lem_block_diagonal_lower_bound_proof} and move to the following auxiliary result. The error-in-operator estimate $\widehat \btheta$ obviously depends on the nuisance parameter $\chi = (\bfeta, A)$. Let
\[
    H_{\btheta \chi} = H_{\chi \btheta}^\top =
    \begin{pmatrix}
        H_{\btheta \bfeta} & H_{\btheta A}    
    \end{pmatrix}
    \quad \text{and} \quad
    H_{\chi \chi} =
    \begin{pmatrix}
        H_{\bfeta\bfeta} & H_{\bfeta A} \\
        H_{A\bfeta} & H_{AA}
    \end{pmatrix}
\]
stand for the corresponding blocks $H$. Our analysis of bias and variance of the error-in-operator estimate $\widehat \btheta$ requires examination of $H^{-1}$. According to the block-matrix inversion formula, it holds that
\begin{equation}
    \label{eq:h_block_inversion}
    H^{-1} =
    \begin{pmatrix}
        \breve H_{\btheta \btheta}^{-1} & -\breve H_{\btheta \btheta}^{-1} H_{\btheta \chi} H_{\chi \chi}^{-1} \\
        -\breve H_{\chi \chi}^{-1} H_{\chi \btheta} H_{\btheta\btheta}^{-1} & \breve H_{\chi \chi}^{-1}
    \end{pmatrix}
    =
    \begin{pmatrix}
        \breve H_{\btheta \btheta}^{-1} & -H_{\btheta \btheta}^{-1} H_{\btheta \chi} \breve H_{\chi \chi}^{-1} \\
        -H_{\chi \chi}^{-1} H_{\chi \btheta} \breve H_{\btheta\btheta}^{-1} & \breve H_{\chi \chi}^{-1}
    \end{pmatrix}.
\end{equation}
We will be particularly interested in
\[
    \breve H_{\btheta \btheta}
    = H_{\btheta \btheta} - H_{\btheta \chi} H_{\chi \chi}^{-1} H_{\chi \btheta} = H / H_{\chi \chi},
\]
which is nothing but the Schur complement of
$H_{\chi \chi}$. Before we proceed, we have to introduce additional notation. Let $\m(\cdot)$ be an arbitrary probability measure on the space of triplets $\bups = (\btheta, \bfeta, A)$. For a function $f(\bups)$ (scalar, vector- or matrix-valued), we denote
\[
    \biasInner f^{\, \m} = \int f(\bups) \, \rmd \m(\bups).
\]
In particular, $\biasInner{H}_{\chi \chi}^{\,\m}$ stands for 
\begin{equation}
    \label{eq:h_chi_chi_m}
    \biasInner{H}_{\chi \chi}^{\,\m} = \int H_{\chi \chi}(\bups) \, \rmd \m(\bups).
\end{equation}
The next proposition shows that the inverse of $\biasInner{H}_{\chi \chi}^{\,\m}$ (and, hence, $H_{\chi \chi}^{-1}$ as well) has a nice structure.

\begin{Prop}
    \label{proposition: inverse of nuisance parameters}
    Let $\rho \leq 1/2$ and let $\m(\cdot)$ be an arbitrary probability measure supported on
    \begin{align}
    \label{eq: norm theta set}
        \{(\btheta, \bfeta, A) \in \R^{d} \times \R^d \times \R^{d \times d} : \Vert \btheta \Vert \le \rho \mu \}.
    \end{align}
    Denote the inverse of $\biasInner{H}_{\chi \chi}^{\,\m}$ defined in \eqref{eq:h_chi_chi_m} by $J$. Then the matrix $J$ can be represented in the form
    \begin{align*}
        J =
        \begin{pmatrix}
            J_{\bfeta \bfeta} & J_{\bfeta A} \\
            J_{A \bfeta} & J_{AA}
        \end{pmatrix}
        =
        \begin{pmatrix}
            (0.5 + r_1) I_d & r_2 I_d \otimes \Tilde{\btheta}^\top \\
            r_2 I_d \otimes \Tilde{\btheta} & I_d \otimes R_3
        \end{pmatrix},
        \quad \text{where $\Tilde{\btheta} = \mu^2 \left( \mu^2 I_d + \avg{\btheta \btheta^\top}^{\m} \right)^{-1}\avg{\btheta}^{\,\m}$,}
    \end{align*}
    and the non-negative numbers $r_1, r_2 $ and the matrix $R_3 \in \R^{d \times d}$ satisfy the inequalities
    \[
        0 \le r_1 \le \rho^2,
        \quad
        0 \le r_2 \le \frac1{\mu^{2}},
        \quad \text{and} \quad
        \Vert R_3 \Vert \le \frac{2}{\mu^{2}}.
    \]
    In addition, there exists $r_2' \in [0, \rho^2 / \mu^2]$, $r_3' \in [0, \mu^{-4}]$, and $R_3' \in \R^{d \times d}$, $\Vert R_3'\Vert \le \mu^{-4}$, such that
    \[
        r_2 = \frac{1}{2\mu^2} + r_2',
        \quad \text{and} \quad
        R_3 = \mu^{-2} I_d + R_3' \biasInner{\btheta \btheta^\top}^{\m} + r_3' \Tilde{\btheta} \Tilde{\btheta}^\top.
    \]
\end{Prop}

We move the proof of Proposition \ref{proposition: inverse of nuisance parameters} to Section \ref{sec: prop inverse of nuisance parameters} and proceed with a lemma allowing us to control the smoothness of the Hessian of $\cL(\bups)$.

\begin{Lem}
    \label{lem:second_derivative_lipschitzness}
    Let us fix arbitrary $\lambda_0 > 0$ and $\mu_0 > 0$ and denote
    \[
        \ttD_0^2
        = \diag\left( (A^*)^\top A^* + \lambda_0 I_d, 2 I_d, I_d \otimes \big(\mu_0^2 I_d + \btheta^* (\btheta^*)^\top \big) \right).
    \]
    Let $\bups^*$ be as defined in \eqref{eq:ups_star}.
    Then, for any vectors $\bu$ and $\bups$ of dimension $2d + d^2$, it holds that
    \[
        \left| \bu^\top \left( \nabla^2 \cL(\bups^* + \bups) - \nabla^2 \cL(\bups^*) \right) \bu \right|
        \leq \frac{4}{\mu_0 \sqrt{\lambda_0}} \, \| \ttD_0 \bups \| \| \ttD_0 \bu\|^2 + \frac{2}{\mu_0^2 \lambda_0} \| \ttD_0 \bups \|^2 \| \ttD_0 \bu\|^2.
    \]
\end{Lem}

The proof of Lemma \ref{lem:second_derivative_lipschitzness} is moved to Section \ref{sec:lem_second_derivative_lipschitzness_proof}.

\subsection{Proof of Lemma \ref{lem:block-diagonal_lower_bound}}
\label{sec:lem_block_diagonal_lower_bound_proof}

Note that, for any $\bu, \bv \in \R^d$ and $\bw \in \R^{d^2}$, the quadratic form
\[
    \begin{pmatrix} \bu^\top & \bv^\top & \bw^\top \end{pmatrix}
    \begin{pmatrix}
    H_{\btheta\btheta} & H_{\btheta\bfeta} & H_{\btheta A} \\
    H_{\bfeta\btheta} & H_{\bfeta\bfeta} & H_{\bfeta A} \\
    H_{A\btheta} & H_{A\bfeta} & H_{AA}
    \end{pmatrix}
    \begin{pmatrix} \bu \\ \bv \\ \bw \end{pmatrix}
\]
is equal to
\begin{align}
    \label{eq:hessian_quadratic_form}
    &\notag
    \bu^\top H_{\btheta\btheta} \bu + \bv^\top H_{\bfeta\bfeta} \bv + \bw^\top H_{AA} \bw
    + 2\bu^\top H_{\btheta\bfeta} \bv + 2\bu^\top H_{\btheta A} \bw + 2\bv^\top H_{\bfeta A} \bw
    \\&
    = \|A \bu\|^2 + \lambda \|\bu\|^2 + 2 \|\bv\|^2 + \|H_{AA}^{1/2} \bw\|^2 + 2 \bu^\top A^\top \bv + 2\bu^\top H_{\btheta A} \bw + 2\bv^\top H_{\bfeta A} \bw. 
\end{align}
Our goal is to show that the right-hand side of \eqref{eq:hessian_quadratic_form} is not smaller than
\[
    \frac{1 - 2\rho}4 \left( \|A \bu\|^2 + \lambda \|\bu\|^2 + 2 \|\bv\|^2 + \|H_{AA}^{1/2} \bw\|^2 \right).
\]
For this purpose, we prove the following lemma in Appendix \ref{sec:lem_non-diagonal_blocks_bound_proof}.
\begin{Lem}
    \label{lem:non-diagonal_blocks_bound}
    Under the conditions of Lemma \ref{lem:block-diagonal_lower_bound}, it holds that
    \[
        \|H_{\bfeta A} H_{AA}^{-1/2}\| \leq \rho
        \quad \text{and} \quad
        \|H_{\btheta \btheta}^{-1/2} H_{\btheta A} H_{AA}^{-1/2}\| \leq \rho \sqrt{2}.
    \]
\end{Lem}
The desired lower bound easily follows from Lemma \ref{lem:non-diagonal_blocks_bound}. Indeed, using to this lemma and the Cauchy-Schwarz inequality $2 \sqrt{2} ab \leq 3a^2 / 4 + 8 b^2 / 3$, we obtain that
\[
    2 \bu^\top H_{\btheta A} \bw
    \geq - 2 \|H_{\btheta\btheta}^{1/2} \bu\| \|H_{\btheta\btheta}^{-1/2} H_{\btheta A} H_{AA}^{1/2}\| \|H_{AA}^{1/2} \bw\|
    \geq -\rho \left( \frac34 \|A \bu\|^2 + \frac{3 \lambda}4 \|\bu\|^2 + \frac83 \|H_{AA}^{1/2} \bw\|^2 \right)
\]
and
\[
    2\bv^\top H_{\bfeta A} \bw
    \geq -2 \|\bv\| \|H_{\bfeta A} H_{AA}\| \|H_{AA}^{1/2} \bw\|
    \geq -\rho \left( \|\bv\|^2 + \|H_{AA}^{1/2} \bw\|^2 \right).
\]
This yields that
\begin{align*}
    &
    \|A \bu\|^2 + \lambda \|\bu\|^2 + 2 \|\bv\|^2 + \bw^\top H_{AA} \bw + 2 \bu^\top A^\top \bv + 2\bu^\top H_{\btheta A} \bw + 2\bv^\top H_{\bfeta A} \bw
    \\&
    \geq \left(1 - \frac{3 \rho}4 \right) \left( \|A \bu\|^2 + \lambda \|\bu\|^2 \right)
    + 2 \left(1 - \frac{\rho}2 \right) \|\bv\|^2
    + \left(1 - \frac{8 \rho}3 - \rho \right) \|H_{AA}^{1/2} \bw\|^2
    + 2 \bu^\top A^\top \bv.
\end{align*}
Taking into account the lower bound
\[
    \|A \bu\|^2 + 2 \|\bv\|^2 + 2 \bu^\top A^\top \bv
    \geq \|A \bu\|^2 + 2 \|\bv\|^2 - \left( \frac23 \|A \bu\|^2 + \frac32 \|\bv\|^2 \right)
    \geq \frac13 \|A \bu\|^2 + \frac12 \|\bv\|^2,
\]
we conclude that
\begin{align*}
    &
    \|A \bu\|^2 + \lambda \|\bu\|^2 + 2 \|\bv\|^2 + \bw^\top H_{AA} \bw + 2 \bu^\top A^\top \bv + 2\bu^\top H_{\btheta A} \bw + 2\bv^\top H_{\bfeta A} \bw
    \\&
    \geq \left(\frac13 - \frac{3 \rho}4 \right) \left( \|A \bu\|^2 + \lambda \|\bu\|^2 \right)
    + 2 \left(\frac14 - \frac{\rho}2 \right) \|\bv\|^2
    + \left(1 - \frac{8 \rho}3 - \rho \right) \|H_{AA}^{1/2} \bw\|^2.
\end{align*}
Finally, since $\rho \in [0, 1/7]$, it holds that
\[
    \frac13 - \frac{3 \rho}4
    \geq \frac14 - \left( \frac34 - \frac14 \right) \rho
    = \frac14 - \frac\rho2,
    \quad \text{and} \quad
    1 - \frac{11 \rho}3
    \geq \frac{10}{21}
    \geq \frac14 - \frac\rho2,
\]
and we obtain the desired bound:
\begin{align*}
    &
    \|A \bu\|^2 + \lambda \|\bu\|^2 + 2 \|\bv\|^2 + \bw^\top H_{AA} \bw + 2 \bu^\top A^\top \bv + 2\bu^\top H_{\btheta A} \bw + 2\bv^\top H_{\bfeta A} \bw
    \\&
    \geq \frac{1 - 2\rho}4 \left( \|A \bu\|^2 + \lambda \|\bu\|^2 + 2 \|\bv\|^2 + \|H_{AA}^{1/2} \bw\|^2 \right),
\end{align*}
which is equivalent to
\begin{align*}
    &
    \begin{pmatrix} \bu^\top & \bv^\top & \bw^\top \end{pmatrix}
    \begin{pmatrix}
    H_{\btheta\btheta} & H_{\btheta\bfeta} & H_{\btheta A} \\
    H_{\bfeta\btheta} & H_{\bfeta\bfeta} & H_{\bfeta A} \\
    H_{A\btheta} & H_{A\bfeta} & H_{AA}
    \end{pmatrix}
    \begin{pmatrix} \bu \\ \bv \\ \bw \end{pmatrix}
    \\&
    \geq \frac{1 - 2\rho}4 \begin{pmatrix} \bu^\top & \bv^\top & \bw^\top \end{pmatrix}
    \begin{pmatrix}
    H_{\btheta\btheta} & O & O \\
    O & H_{\bfeta\bfeta} & O \\
    O & O & H_{AA}
    \end{pmatrix}
    \begin{pmatrix} \bu \\ \bv \\ \bw \end{pmatrix}.
\end{align*}
The proof is finished.

\myendproof

\subsection{Proof of Proposition~\ref{proposition: inverse of nuisance parameters}}
\label{sec: prop inverse of nuisance parameters}

Since the measure $\m$ is clear from the context, we omit it throughout the proof and write $\avg{H}$ and $\avg{\btheta}$, instead of $\avg{H}^\m$ and $\avg{\btheta}^\m$ for brevity.
Let us note that, according to \eqref{eq:h_block_form_kronecker}, the matrix $H_{\chi \chi}(\bups)$ admits the following form:
\[
    H_{\chi \chi} =
    \begin{pmatrix}
        2 I_d & I_d \otimes \btheta^\top \\
        I_d \otimes \btheta & I_d \otimes (\mu^2 I_d + \btheta \btheta^\top)
    \end{pmatrix}.
\]
Then, due to the definition of $\avg{H}$ it holds that
\[
    \avg{H}_{\chi \chi} =
    \begin{pmatrix}
        2 I_d & I_d \otimes \avg{\btheta}^\top \\
        I_d \otimes \avg{\btheta} & I_d \otimes \left(\mu^2 I_d + \avg{\btheta \btheta^\top} \right)
    \end{pmatrix}.
\]
Due to the blockwise inversion formula, we have
\begin{align*}
    J_{\bfeta \bfeta}
    &
    = \left( \avg{H}_{\bfeta \bfeta} - \avg{H}_{\bfeta A} \avg{H}_{AA}^{-1} \avg{H}_{A \bfeta} \right)^{-1}
    \\&
    = \left ( 2 I_d - \left(I_d \otimes \avg{\btheta}^\top \right) \left[ I_d \otimes \left(\mu^2 I_d +  \avg{\btheta \btheta^\top} \right) \right]^{-1} \left(I_d \otimes \avg{\btheta} \right) \right )^{-1}.
\end{align*}
The latter expression can significantly simplified using the Kronecker product properties:
\[
    J_{\bfeta \bfeta}
    = \left ( 2 - \avg{\btheta}^\top \left(\mu^2 I_d +  \avg{\btheta \btheta^\top} \right)^{-1} \avg{\btheta} \right )^{-1} I_d.
\]
We are going to show that the remainder
\begin{align*}
    r_1 = \left ( 2 - \avg{\btheta}^\top \left(\mu^2 I_d +  \avg{\btheta \btheta^\top} \right)^{-1} \avg{\btheta}  \right )^{-1} - \frac{1}{2}
\end{align*}
is small and, hence, $J_{\bfeta \bfeta}$ is close to $0.5 I_d$. Indeed,
taking into account that $\m$ is supported on the set defined in~\eqref{eq: norm theta set}, we observe that
\[
    \left\|\avg{\btheta} \right\|^2 \leq \rho^2 \mu^2.
\]
Then, according to the Sherman-Morrison formula,
\begin{align*}
    \avg{\btheta}^\top \left(\mu^2 I_d + \avg{\btheta \btheta^\top} \right)^{-1} \avg{\btheta}
    &
    \leq \avg{\btheta}^\top \left(\mu^2 I_d + \avg{\btheta} \avg{\btheta}^\top \right)^{-1} \avg{\btheta}
    \\&
    = \avg{\btheta}^\top \left(\frac1{\mu^2} I_d - \frac{\avg{\btheta} \avg{\btheta}^\top}{\mu^2 (\mu^2 + \|\avg{\btheta}\|^2)} \right) \avg{\btheta}
    \\&
    = \frac{\|\avg{\btheta}\|^2}{\mu^2 + \|\avg{\btheta}\|^2}
    \leq \rho^2.
\end{align*} 
This yields that
\begin{align}
    r_1
    &
    = \frac{\avg{\btheta}^\top (\mu^2 I_d +  \avg{\btheta \btheta^\top} )^{-1} \avg{\btheta} }{2 (2 - \biasInner{\btheta}^\top (\mu^2 I_d +  \avg{\btheta \btheta^\top} )^{-1} \biasInner{\btheta}) }
    \le \frac{\rho^2}{4 - 2 \rho^2}. \label{eq: r1 bound}
\end{align}
Since $\rho \le 1/2$, the expression in the right-hand side is at most $2\rho^2 / 7 < \rho^2$.

We proceed with the block-inversion formula and compute:
\begin{align}
    \label{eq:J_eta_A}
    J_{\bfeta A}
    &\notag
    = - J_{\bfeta \bfeta} \avg{H}_{\bfeta A} \avg{H}_{AA}^{-1}
    \\&
    = \left (\frac{1}{2} + r_{1} \right) \left( I_d \otimes \avg{\btheta}^\top \right) \left[ I_d \otimes \left(\mu^2 I_d + \avg{\btheta \btheta^\top} \right)^{-1} \right] = r_2 I_d \otimes \Tilde{\btheta}^\top,
\end{align}
where
\[
    \Tilde{\btheta} = \mu^2 \left(\mu^2 I_d + \avg{\btheta \btheta^\top} \right)^{-1}\avg{\btheta}
    \quad \text{and} \quad
    r_2 = \mu^{-2} \left(\frac12 + r_1 \right) \leq \mu^{-2}.
\]
Note that
\begin{equation}
    \label{eq: tilde theta bound}
    \Vert \Tilde{\btheta} \Vert
    \le \Vert \avg{\btheta} \Vert \le \rho \mu,
\end{equation}
because of $(1 - \rho^2) \|\biasInner{\btheta}\|^2 \leq \rho^2 \mu^2$.
    
It remains to consider $J_{AA}$. Since the product $\avg{H} J$ equals to the identity matrix, it holds that
\[
    \avg{H}_{A \bfeta} J_{\bfeta A} + \avg{H}_{AA} J_{AA} = I_{d^2}. 
\]
This yields that $J_{AA} = \avg{H}_{AA}^{-1} - \avg{H}_{AA}^{-1} \avg{H}_{A \bfeta} J_{\bfeta A}$
and, applying \eqref{eq:J_eta_A}, we obtain that
\begin{align*}
    J_{AA}
    &
    = I_d \otimes \left( \mu^2 I_d + \avg{\btheta \btheta^\top} \right)^{-1} + r_2 \left[ I_d \otimes \left(\mu^2 I_d + \avg{\btheta \btheta^\top} \right) \right]^{-1} (I_d \otimes \avg{\btheta}) (I_d \otimes \Tilde{\btheta}^\top) \\
    &
    = I_d \otimes R_3,
\end{align*}
where
\begin{align*}
    R_3
    &
    = \left( \mu^2 I_d + \avg{\btheta \btheta^\top} \right)^{-1} + r_2 \left(\mu^2 I_d + \avg{\btheta \btheta^\top} \right)^{-1} \avg{\btheta} \Tilde{\btheta}^\top
    \\&
    = \left( \mu^2 I_d + \avg{\btheta \btheta^\top} \right)^{-1} + \mu^{-4} \left( \frac{1}{2} + r_1 \right ) \Tilde{\btheta} \Tilde{\btheta}^\top.
\end{align*}
Since the norm of $\Tilde{\btheta}$ does not exceed $\rho \mu$ (see \eqref{eq: tilde theta bound}), we have
\[
    \Vert R_3 \Vert
    \le  \mu^{-2} + \left ( \frac{1}{2} + r_1 \right ) \mu^{-4} \Vert \Tilde{\btheta} \Vert^2 \le \mu^{-2} + \rho^2 \left ( \frac{1}{2} + r_1 \right ) \mu^{-2} \le 2 \mu^{-2}.
\]

Finally, we study the leading terms of $J_{\bfeta A}$ and $J_{AA}$. Set $r_1' = r_1 / \mu^2$. Clearly, $r_1'$ is not greater than $\rho^2 \mu^{-2}$, and, due to \eqref{eq:J_eta_A}, it holds that
\begin{align*}
    J_{\bfeta A} & 
    = r_2 I_d \otimes \Tilde{\btheta}^\top
    = \mu^{-2} \left(\frac12 + r_1 \right) I_d \otimes \Tilde{\btheta}^\top
    = \frac12 I_d \otimes \Tilde{\btheta}^\top + r_1' I_d \otimes \Tilde{\btheta}^\top.
\end{align*}
Next, we determine the leading term in the expression for $R_3$:
\begin{align*}
    R_3
    &
    = \left( \mu^2 I_d + \avg{\btheta \btheta^\top} \right)^{-1} + \mu^{-4} \left( \frac{1}{2} + r_1 \right ) \Tilde{\btheta} \Tilde{\btheta}^\top
    \\&
    = \frac1{\mu^2} \left( \mu^2 I_d + \avg{\btheta \btheta^\top} \right)^{-1} \left( \mu^2 I_d + \avg{\btheta \btheta^\top} - \avg{\btheta \btheta^\top}\right) + \mu^{-4} \left( \frac{1}{2} + r_1 \right ) \Tilde{\btheta} \Tilde{\btheta}^\top
    \\&
    = \frac1{\mu^2} I_d + R_3' \biasInner{\btheta \btheta^\top} + r_3' \Tilde{\btheta} \Tilde{\btheta}^\top,
\end{align*}
where the introduced matrix
\[
    R_3' = \frac1{\mu^2} \left( \mu^2 I_d + \avg{\btheta \btheta^\top} \right)^{-1}
\]
and the coefficient $r_3' = \mu^{-4} (1/2 + r_1)$ satisfy the bounds
\[
    \Vert R_3' \Vert
    = \left\Vert \mu^{-2} \left(\mu^2 I_d + \biasInner{\btheta \btheta^\top} \right) \right\Vert \le \mu^{-4}
    \quad \text{and} \quad
    r_3' = \mu^{-4} \left( \frac12 + r_1 \right)
    \le \mu^{-4}.
\]
\myendproof

\subsection{Proof of Lemma \ref{lem:second_derivative_lipschitzness}}
\label{sec:lem_second_derivative_lipschitzness_proof}

Throughout the proof, we will represent $\bups^*$, $\bups$, and $\bu$ as triplets:
\[
    \bups = (\btheta, \bfeta, \rmvec(A)),
    \quad
    \bups^* = (\btheta^*, \bfeta^*, \rmvec(A^*)),
    \quad
    \bu = (\bt, \bnu, \rmvec(B)).
\]
Then, calculating the Hessian of $\cL(\bups)$ (see, for instance, the proof of Lemma \ref{lem:block-diagonal_lower_bound}), we can rewrite the quadratic form $\bu^\top \nabla^2 \cL(\bups^*) \bu$ in the following way:
\begin{align*}
    \bu^\top \nabla^2 \cL(\bups^*) \bu
    &
    = \|A^* \bt\|^2 + \lambda \|\bt\|^2 + 2 \|\bnu\|^2
    + \|B \btheta^*\|^2 + \mu \|B\|_{\F}^2
    \\&\quad
    - 2 \bnu^\top A^* \bt
    + 2 \bt^\top (A^*)^\top B \btheta^* + 2 (A^* \btheta^* - \bfeta^*)^\top B \bt - 2 \bnu^\top B \btheta^*.
\end{align*}
Note that the parameters $\lambda$ and $\mu$, appearing in the definition of $\cL(\bups)$, may differ from $\lambda_0$ and $\mu_0$. It does not affect the proof, because $\lambda$ and $\mu$ vanish, once we subtract $\bu^\top \nabla^2 \cL(\bups^*) \bu$ from $\bu^\top \nabla^2 \cL(\bups^* + \bups) \bu$. Indeed, it holds that
\begin{align*}
    \bu^\top \left( \nabla^2 \cL(\bups^* + \bups) - \nabla^2 \cL(\bups^*) \right) \bu
    &
    = 2 \bt^\top A^\top A^* \bt + \|A \bt\|^2 + 2 \btheta^\top B^\top B \btheta^* + \|B \btheta\|^2 - 2 \bnu^\top A \bt
    \\&\quad
    + 2 \btheta^\top B^\top A^* \bt  + 2 \bt^\top A^\top B \btheta^* + 2 \bt^\top A^\top B \btheta
    \\&\quad
    + 2 (A^* \btheta + A \btheta^* + A \btheta - \bfeta)^\top B \bt - 2 \bnu^\top B \btheta.
\end{align*}
It is straightforward to observe that the difference of $\bu^\top \nabla^2 \cL(\bups^* + \bups) \bu$ and $\bu^\top \nabla^2 \cL(\bups^*) \bu$ consists of the third order and the fourth order terms (in $\bu$ and $\bups$):
\begin{equation}
    \label{eq:tq_decomposition}
    \bu^\top \left( \nabla^2 \cL(\bups^* + \bups) - \nabla^2 \cL(\bups^*) \right) \bu = \cT + \cQ,
\end{equation}
where
\begin{align*}
    \cT
    &
    = 2 \bt^\top A^\top A^* \bt + 2 \btheta^\top B^\top B \btheta^* - 2 \bnu^\top A \bt + 2 \btheta^\top B^\top A^* \bt
    \\&\quad
    + 2 \bt^\top A^\top B \btheta^* + 2 (A^* \btheta + A \btheta^* - \bfeta)^\top B \bt - 2 \bnu^\top B \btheta
\end{align*}
and
\[
    \cQ
    = \|A \bt\|^2 + \|B \btheta\|^2 + 2 \bt^\top A^\top B \theta + 2 \btheta^\top A^\top B \bt.
\]
It is more convenient to study $\cT$ and $\cQ$ separately.
For this reason, we split the rest of the proof into two steps and show that
\[
    |\cT| \leq \frac{4}{\mu_0 \sqrt{\lambda_0}} \|\ttD_0 \bups\| \|\ttD_0 \bu\|^2
    \quad \text{and} \quad
    |\cQ| \leq \frac{2}{\mu_0^2 \lambda_0} \|\ttD_0 \bups\|^2 \|\ttD_0 \bu\|^2.
\]

\medskip

\noindent
\textbf{Step 1: upper bound on $|\cQ|$.}
\quad
We start with a simpler inequality
\[
    |\cQ| \leq \frac{2}{\mu_0^2 \lambda_0} \|\ttD_0 \bu\|^2 \|\ttD_0 \bups\|^2
\]
and postpone the proof of the upper bound on the absolute value $\cT$ to the next step. It is enough to show that
\[
    \frac{2}{\mu_0^2 \lambda_0} \|\ttD_0 \bu\|^2 \|\ttD_0 \bups\|^2 \leq \cQ \leq \frac{2}{\mu_0^2 \lambda_0} \|\ttD_0 \bu\|^2 \|\ttD_0 \bups\|^2.
\]
The lower bound trivially follows from the inequalities
\begin{align*}
    \cQ
    &
    = \|A \bt + B \btheta\|^2 + 2 \btheta^\top A^\top B \bt
    \geq -2 \|\btheta\| \|A\|_{\F} \|B\|_{\F} \|\bt\|
    \geq - \frac{2}{\mu^2 \lambda} \|\ttD_0 \bu\|^2 \|\ttD_0 \bups\|^2.
\end{align*}
From now on, we focus on the upper bound on $\cQ$. It holds that
\begin{align}
    \label{eq:q_intermediate_upper_bound}
    \cQ
    &\notag
    = \|A \bt\|^2 + \|B \btheta\|^2 + 2 \bt^\top A^\top B \theta + 2 \btheta^\top A^\top B \bt
    \\&
    \leq 2 \|A\|_{\F}^2 \|\bt\|^2 + 2 \|B\|_{\F}^2 \|\btheta\|^2 + 2 \|A\|_{\F} \|\btheta\| \|B\|_{\F} \|\bt\|.
\end{align}
The right-hand side can be viewed as a quadratic form of $\mu_0 \|A\|_{\F}$ and $\sqrt{\lambda_0} \|\btheta\|$:
\begin{align}
    \label{eq:q_quadratic_form_representation}
    &\notag
    2 \|A\|_{\F}^2 \|\bt\|^2 + 2 \|B\|_{\F}^2 \|\btheta\|^2 + 2 \|A\|_{\F} \|\btheta\| \|B\|_{\F} \|\bt\|
    \\&
    = \frac{1}{\mu_0^2 \lambda_0}
    \begin{pmatrix} \mu_0 \|A\|_{\F} & \sqrt{\lambda_0} \|\btheta\| \end{pmatrix}
    \begin{pmatrix}
        2 \lambda_0 \|\bt\|^2 & \mu_0 \sqrt{\lambda_0} \|B\|_{\F} \|\bt\| \\
        \mu_0 \sqrt{\lambda_0} \|B\|_{\F} \|\bt\| & 2 \mu_0^2 \|B\|_{\F}
    \end{pmatrix}
    \begin{pmatrix} \mu \|A\|_{\F} \\ \sqrt{\lambda} \|\btheta\| \end{pmatrix}.
\end{align}
Let us introduce $\varphi \in [0, \pi / 2]$, such that
\[
    \sin \varphi = \frac{\mu_0 \|B\|_{\F}}{\sqrt{\lambda_0 \|t\|^2 + \mu_0^2 \|B\|_{\F}^2}},
    \quad
    \cos \varphi = \frac{\sqrt{\lambda_0} \|t\|}{\sqrt{\lambda_0 \|t\|^2 + \mu_0^2 \|B\|_{\F}^2}}
\]
and consider the matrix
\[
    \frac1{\mu_0^2 \|B\|_{\F}^2 + \lambda_0 \|t\|^2}
    \begin{pmatrix}
        2 \lambda_0 \|\bt\|^2 & \mu_0 \sqrt{\lambda_0} \|B\|_{\F} \|\bt\| \\
        \mu_0 \sqrt{\lambda_0} \|B\|_{\F} \|\bt\| & 2 \mu_0^2 \|B\|_{\F}
    \end{pmatrix}
    = \begin{pmatrix}
        2 \cos^2 \varphi & \sin \varphi \cos \varphi \\
        \sin \varphi \cos \varphi & 2 \sin^2 \varphi
    \end{pmatrix}.
\]
Direct calculations show that both its eigenvalues belong to $[0, 2]$. Hence, the operator norm of this matrix does not exceed $2$ and we obtain that
\begin{align}
    \label{eq:q_quadratic_form_upper_bound}
    &\notag
    \begin{pmatrix} \mu_0 \|A\|_{\F} & \sqrt{\lambda_0} \|\btheta\| \end{pmatrix}
    \begin{pmatrix}
        2 \lambda_0 \|\bt\|^2 & \mu_0 \sqrt{\lambda_0} \|B\|_{\F} \|\bt\| \\
        \mu_0 \sqrt{\lambda_0} \|B\|_{\F} \|\bt\| & 2 \mu_0^2 \|B\|_{\F}
    \end{pmatrix}
    \begin{pmatrix} \mu_0 \|A\|_{\F} \\ \sqrt{\lambda_0} \|\btheta\| \end{pmatrix}
    \\&\notag
    \leq \left( \mu_0^2 \|A\|_{\F}^2 + \lambda_0 \|\btheta\|^2 \right) \left\|
    \begin{pmatrix}
        2 \lambda_0 \|\bt\|^2 & \mu_0 \sqrt{\lambda_0} \|B\|_{\F} \|\bt\| \\
        \mu_0 \sqrt{\lambda_0} \|B\|_{\F} \|\bt\| & 2 \mu_0^2 \|B\|_{\F}
    \end{pmatrix}
    \right\|
    \\&
    = \left( \mu_0^2 \|A\|_{\F}^2 + \lambda_0 \|\btheta\|^2 \right) 
    \left( \mu_0^2 \|B\|_{\F}^2 + \lambda_0 \|t\|^2 \right)
    \left\|
    \begin{pmatrix}
        2 \cos^2 \varphi & \sin \varphi \cos \varphi \\
        \sin \varphi \cos \varphi & 2 \sin^2 \varphi
    \end{pmatrix}
    \right\|
    \\&\notag
    \leq 2 \left( \mu_0^2 \|A\|_{\F}^2 + \lambda_0 \|\btheta\|^2 \right) 
    \left( \mu_0^2 \|B\|_{\F}^2 + \lambda_0 \|t\|^2 \right).
\end{align}
Taking \eqref{eq:q_intermediate_upper_bound}, \eqref{eq:q_quadratic_form_representation}, and \eqref{eq:q_quadratic_form_upper_bound} into account, we obtain that
\begin{equation}
    \label{eq:q_upper_bound}
    \cQ
    \leq \frac{2}{\mu_0^2 \lambda_0} \left( \mu_0^2 \|A\|_{\F}^2 + \lambda_0 \|\btheta\|^2 \right) 
    \left( \mu_0^2 \|B\|_{\F}^2 + \lambda_0 \|t\|^2 \right)
    \leq \frac{2}{\mu_0^2 \lambda_0} \|\ttD_0 \bups\|^2 \|\ttD_0 \bu\|^2.
\end{equation}

\medskip

\noindent
\textbf{Step 2: upper bound on $|\cT|$.}
It remains to bound the absolute value of $\cT$. First, note that
\begin{align}
    \label{eq:t_intermediate_upper_bound}
    \frac{\mu_0 \sqrt{\lambda_0} \; |\cT|}{2}
    &\notag
    \leq \mu_0 \sqrt{\lambda_0} \|\bt\| \|A\|_{\F} \|A^* \bt\|
    + \mu_0 \sqrt{\lambda_0} \|\btheta\| \|B\|_{\F} \|B \btheta^*\|
    + \mu_0 \sqrt{\lambda_0} \|\bnu\| \|A\|_{\F} \|\bt\|
    \\&\quad
    + \mu_0 \sqrt{\lambda_0} \|\btheta\| \|B\|_{\F} \|A^* \bt\|
    + \mu_0 \sqrt{\lambda_0} \|\bt\| \|A\|_{\F} \|B \btheta^*\|
    + \mu_0 \sqrt{\lambda_0} \|\bt\| \|B\|_{\F} \|A^* \btheta\|
    \\&\quad\notag
    + \mu_0 \sqrt{\lambda_0} \|\bt\| \|B\|_{\F} \|A \btheta^*\|
    + \mu_0 \sqrt{\lambda_0} \|\bt\| \|B\|_{\F} \|\bfeta\|
    + \mu_0 \sqrt{\lambda_0} \|\bnu\| \|B\|_{\F} \|\btheta\|.
\end{align}
The right-hand side of \eqref{eq:t_intermediate_upper_bound} is quite massive. Let us first examine the terms $\sqrt{\lambda_0} \|\btheta\| \|A^* \bt\| + \sqrt{\lambda_0} \|\bt\| \|A^* \btheta\|$ and $\mu_0 \|A\|_{\F} \|B \btheta^*\| + \mu_0 \|B\|_{\F} \|A \btheta^*\|$. Due to the Cauchy-Schwarz inequality, it holds that
\[
    \sqrt{\lambda_0} \|\btheta\| \|A^* \bt\| + \sqrt{\lambda_0} \|\bt\| \|A^* \btheta\|
    \leq \sqrt{\|A^* \bt\|^2 + \lambda_0 \|\bt\|^2} \sqrt{\|A^* \btheta\|^2 + \lambda_0 \|\btheta\|^2}
\]
and
\[
    \mu_0 \|A\|_{\F} \|B \btheta^*\| + \mu_0 \|B\|_{\F} \|A \btheta^*\|
    \leq \sqrt{\mu_0^2 \|B\|_{\F}^2 + \|B \btheta^*\|^2} \sqrt{\mu_0^2 \|A\|_{\F}^2 + \|A \btheta^*\|^2}.
\]
Thus,
\begin{align}
    \label{eq:t_4_terms_cauchy_schwarz_bound}
    &\notag
    \mu_0 \sqrt{\lambda_0} \|\btheta\| \|B\|_{\F} \|A^* \bt\|
    + \mu_0 \sqrt{\lambda_0} \|\bt\| \|A\|_{\F} \|B \btheta^*\|
    + \mu_0 \sqrt{\lambda_0} \|\bt\| \|B\|_{\F} \|A^* \btheta\|
    \\&\quad\notag
    + \mu_0 \sqrt{\lambda_0} \|\bt\| \|B\|_{\F} \|A \btheta^*\|
    + \mu_0 \sqrt{\lambda_0} \|\bt\| \|B\|_{\F} \|\bfeta\|
    \\&
    \leq \mu_0 \|B\|_{\F} \sqrt{\|A^* \bt\|^2 + \lambda_0 \|\bt\|^2} \sqrt{\|A^* \btheta\|^2 + \lambda_0 \|\btheta\|^2}
    \\&\quad\notag
    + \sqrt{\lambda_0} \|\bt\| \sqrt{\mu_0^2 \|B\|_{\F}^2 + \|B \btheta^*\|^2} \sqrt{\mu_0^2 \|A\|_{\F}^2 + \|A \btheta^*\|^2} 
    + \mu_0 \sqrt{\lambda_0} \|\bt\| \|B\|_{\F} \|\bfeta\|.
\end{align}
Applying the Cauchy-Schwarz inequality again, we obtain that
\begin{align}
    \label{eq:t_5_terms_cauchy_schwarz_bound}
    &\notag
    \mu_0 \|B\|_{\F} \sqrt{\|A^* \bt\|^2 + \lambda_0 \|\bt\|^2} \sqrt{\|A^* \btheta\|^2 + \lambda_0 \|\btheta\|^2}
    \\&\quad\notag
    + \sqrt{\lambda_0} \|\bt\| \sqrt{\mu_0^2 \|B\|_{\F}^2 + \|B \btheta^*\|^2} \sqrt{\mu_0^2 \|A\|_{\F}^2 + \|A \btheta^*\|^2} 
    + \mu_0 \sqrt{\lambda_0} \|\bt\| \|B\|_{\F} \|\bfeta\|
    \\&
    \leq \sqrt{\|A^* \bt\|^2 + \lambda_0 \|\bt\|^2} \sqrt{\mu_0^2 \|B\|_{\F}^2 + \|B \btheta^*\|^2}
    \\&\quad \notag
    \cdot \left( \sqrt{\|A^* \btheta\|^2 + \lambda_0 \|\btheta\|^2} + \sqrt{\mu_0^2 \|A\|_{\F}^2 + \|A \btheta^*\|^2}  + \frac1{\sqrt 2} \cdot \sqrt{2}\|\bfeta\| \right)
    \\& \notag
    \leq \sqrt{\|A^* \bt\|^2 + \lambda_0 \|\bt\|^2} \sqrt{\mu_0^2 \|B\|_{\F}^2 + \|B \btheta^*\|^2} \cdot \sqrt{\frac52} \; \|\ttD_0 \bups\|.
\end{align}
The inequalities \eqref{eq:t_4_terms_cauchy_schwarz_bound} and \eqref{eq:t_5_terms_cauchy_schwarz_bound} help us to simplify the upper bound \eqref{eq:t_intermediate_upper_bound} drastically. In particular, they yield that
\begin{align*}
    \frac{\mu_0 \sqrt{\lambda_0} \; |\cT|}{2}
    &
    \leq \mu_0 \sqrt{\lambda_0} \|\bt\| \|A\|_{\F} \|A^* \bt\|
    + \mu_0 \sqrt{\lambda_0} \|\btheta\| \|B\|_{\F} \|B \btheta^*\|
    + \mu_0 \sqrt{\lambda_0} \|\bnu\| \|A\|_{\F} \|\bt\|
    \\&\quad
    + \mu_0 \sqrt{\lambda_0} \|\bnu\| \|B\|_{\F} \|\btheta\|
    + \sqrt{\|A^* \bt\|^2 + \lambda_0 \|\bt\|^2} \sqrt{\mu_0^2 \|B\|_{\F}^2 + \|B \btheta^*\|^2} \cdot \sqrt{\frac52} \; \|\ttD_0 \bups\|.
\end{align*}
The expression in the right-hand side can be simplified even further. Using the inequalities
\[
    \mu_0 \|A\|_{\F} \leq \|\ttD_0 \bups\|,
    \quad \text{and} \quad
    \sqrt{\lambda_0} \|\btheta\| \leq \|\ttD_0 \bups\|,
\]
we obtain that
\begin{align*}
    \frac{\mu_0 \sqrt{\lambda_0} \; |\cT|}{2}
    &
    \leq \mu_0 \sqrt{\lambda_0} \|\bt\| \|A\|_{\F} \|A^* \bt\|
    + \mu_0 \sqrt{\lambda_0} \|\btheta\| \|B\|_{\F} \|B \btheta^*\|
    + \mu_0 \sqrt{\lambda_0} \|\bnu\| \|A\|_{\F} \|\bt\|
    \\&\quad
    + \mu_0 \sqrt{\lambda_0} \|\bnu\| \|B\|_{\F} \|\btheta\|
    + \sqrt{\|A^* \bt\|^2 + \lambda_0 \|\bt\|^2} \sqrt{\mu_0^2 \|B\|_{\F}^2 + \|B \btheta^*\|^2} \cdot \sqrt{\frac52} \; \|\ttD_0 \bups\|
    \\&
    \leq \sqrt{\lambda_0} \|\bt\| \|A^* \bt\| \|\ttD_0 \bups\|
    + \mu_0 \|B\|_{\F} \|B \btheta^*\| \|\ttD_0 \bups\|
    + \sqrt{\lambda_0} \|\bnu\| \|\bt\| \|\ttD_0 \bups\|
    \\&\quad
    + \mu_0 \|\bnu\| \|B\|_{\F} \|\ttD_0 \bups\|
    + \sqrt{\|A^* \bt\|^2 + \lambda \|\bt\|^2} \sqrt{\mu^2 \|B\|_{\F}^2 + \|B \btheta^*\|^2} \cdot \sqrt{\frac52} \; \|\ttD_0 \bups\|.
\end{align*}
Moreover, introducing
\[
    \bw = \left(\sqrt{\|A^* \bt\|^2 + \lambda \|\bt\|^2}, \sqrt{2} \|\bnu\|, \sqrt{\mu^2 \|B\|_{\F}^2 + \|B \btheta^*\|^2} \right)^\top
    \quad \text{and} \quad
    \Phi = 
    \begin{pmatrix}
        1 & \sqrt{2} / 4 & \sqrt{10} / 4 \\
        \sqrt{2} / 4 & 0 & \sqrt{2} / 4 \\
        \sqrt{10} / 4 & \sqrt{2} / 4 & 1
    \end{pmatrix},
\]
we observe that
\begin{align*}
    \frac{\mu_0 \sqrt{\lambda_0} \; |\cT|}{2}
    &
    \leq \sqrt{\lambda_0} \|\bt\| \|A^* \bt\| \|\ttD_0 \bups\|
    + \mu_0 \|B\|_{\F} \|B \btheta^*\| \|\ttD_0 \bups\|
    + \sqrt{\lambda_0} \|\bnu\| \|\bt\| \|\ttD_0 \bups\|
    \\&\quad
    + \mu_0 \|\bnu\| \|B\|_{\F} \|\ttD_0 \bups\|
    + \sqrt{\|A^* \bt\|^2 + \lambda \|\bt\|^2} \sqrt{\mu^2 \|B\|_{\F}^2 + \|B \btheta^*\|^2} \cdot \sqrt{\frac52} \; \|\ttD_0 \bups\|
    \\&
    \leq \|\ttD_0 \bups\| \; \bw^\top \Phi \bw.
\end{align*}
Finally, direct calculations show that $\|\Phi\| < 2$, and then
\begin{equation}
    \label{eq:t_upper_bound}
    \frac{\mu_0 \sqrt{\lambda_0} \; |\cT|}{2}
    \leq 2 \|\ttD_0 \bups\| \left( \|A^* \bt\|^2 + \lambda_0 \|\bt\|^2 +  2 \|\bnu\|^2 + \mu_0^2 \|B\|_{\F}^2 + \|B \btheta^*\|^2 \right)
    = 2 \|\ttD_0 \bups\| \|\ttD_0 \bu\|^2.
\end{equation}
Hence, summing up \eqref{eq:tq_decomposition}, \eqref{eq:q_upper_bound}, and \eqref{eq:t_upper_bound}, we obtain that
\[
    \left| \bu^\top \left( \nabla^2 \cL(\bups) - \nabla^2 \cL(\bups^*) \right) \bu \right|
    \leq |\cT| + |\cQ|
    \leq \frac{4}{\mu_0 \sqrt{\lambda_0}} \, \| \ttD_0 \bups \| \| \ttD_0 \bu\|^2 + \frac{2}{\mu_0^2 \lambda_0} \| \ttD_0 \bups \|^2 \| \ttD_0 \bu\|^2.
\]
\myendproof

\subsection{Proof of Lemma \ref{lem:non-diagonal_blocks_bound}}
\label{sec:lem_non-diagonal_blocks_bound_proof}

The proofs of the upper bounds $\|H_{\bfeta A} H_{AA}^{-1/2}\| \leq \rho$ and $\|H_{\btheta \btheta}^{-1/2} H_{\btheta A} H_{AA}^{-1/2}\| \leq \rho \sqrt{2}$ is based on a simple observation that the operator norm of a matrix $B$ does not exceed $\rho$ if and only if $B^\top B \preceq \rho^2 I_d$. In our case, it is enough to show that
\[
    H_{AA}^{-1/2} H_{\bfeta A} H_{A \bfeta} H_{AA}^{-1/2} \preceq \rho^2 I_d
    \quad \text{and} \quad
    H_{\btheta \btheta}^{-1/2} H_{\btheta A} H_{AA}^{-1} H_{A \btheta} H_{\btheta \btheta}^{-1/2} \preceq 2\rho^2 I_d.
\]
For convenience, we split the derivations into two steps.

\medskip

\noindent
\textbf{Step 1: upper bound on $\|H_{\bfeta A} H_{AA}\|$.}
\quad
Let us recall that (see \eqref{eq:h_block_form} and \eqref{eq:h_block_form_kronecker})
\[
    H_{\bfeta A} = H_{A \bfeta}^\top = -I_d \otimes \btheta^\top
    \quad \text{and} \quad
    H_{AA} = I_d \otimes \left(\mu^2 I_d + \btheta \btheta^\top \right).
\]
According to the conditions of the lemma, the triplet $\bups = (\btheta, \bfeta, \rmvec(A))$ belongs to the set $\Ups(\rho)$.
Then it is straightforward to observe that
\[
    H_{A \bfeta} H_{\bfeta A}
    = I_d \otimes \big( \btheta\btheta^\top \big) 
    \preceq I_d \otimes \left( \rho^2 \btheta \btheta^\top + (1 - \rho^2) \|\btheta\|^2 I_d \right)
    \preceq I_d \otimes \left( \rho^2 \btheta \btheta^\top + \rho^2 \mu^2 I_d \right)
    = \rho^2 H_{AA}.
\]
This yields that
\[
    \left\|H_{\bfeta A} H_{AA}^{-1/2} \right\|^2
    = \left\|H_{AA}^{-1/2} H_{A \bfeta} H_{\bfeta A} H_{AA}^{-1/2} \right\|
    \leq \rho^2.
\]

\medskip

\noindent
\textbf{Step 2: upper bound on $\|H_{\btheta \btheta}^{-1/2} H_{\btheta A} H_{AA}^{-1/2}\|$}.
\quad
The proof of the upper bound on the operator norm of $H_{\btheta \btheta}^{-1/2} H_{\btheta A} H_{AA}^{-1/2}$ relies on the explicit representation of $H_{AA}^{-1}$. According to the Woodbury matrix identity, it holds that
\[
    \left( \mu^2 I_d + \btheta \btheta^\top \right)^{-1}
    = \frac1{\mu^2} I_d - \frac1{\mu^2 (\mu^2 + \|\btheta\|^2)} \btheta \btheta^\top
\]
Then
\[
    H_{AA}^{-1} = \frac1{\mu^{2}} \left( I_d \otimes \left( I_d - \frac1{\mu^2 + \|\btheta\|^2} \btheta \btheta^\top \right) \right),
\]
and we obtain that
\begin{align*}
    H_{\btheta A} H_{AA}^{-1} H_{A \btheta}
    &
    = \frac1{\mu^{2}} \left( (A \btheta - \bfeta)^\top \otimes I_d + A^\top \otimes \btheta^\top \right) \left( I_d \otimes \left( I_d - \frac1{\mu^2 + \|\btheta\|^2} \btheta \btheta^\top \right) \right) \Big( (A \btheta - \bfeta) \otimes I_d + A \otimes \btheta \Big)
    \\&
    = \frac{\|A \btheta - \bfeta\|^2}{\mu^{2}} \left( I_d - \frac1{\mu^2 + \|\btheta\|^2} \btheta \btheta^\top \right) + \frac{A^\top (A \btheta - \bfeta) \btheta^\top}{\mu^2 + \|\btheta\|^2} + \frac{\btheta (A \btheta - \bfeta)^\top A}{\mu^2 + \|\btheta\|^2} + \frac{A^\top A \|\btheta\|^2}{\mu^2 + \|\btheta\|^2}
    \\&
    \preceq \frac{\|A \btheta - \bfeta\|^2}{\mu^{2}} + \frac{A^\top (A \btheta - \bfeta) \btheta^\top}{\mu^2 + \|\btheta\|^2} + \frac{\btheta (A \btheta - \bfeta)^\top A}{\mu^2 + \|\btheta\|^2} + \frac{A^\top A \|\btheta\|^2}{\mu^2 + \|\btheta\|^2}.
\end{align*}
Thus, for any unit vector $\bu \in \R^d$, it holds that
\begin{align*}
    \bu^\top H_{\btheta A} H_{AA}^{-1} H_{A \btheta} \bu
    &
    \leq \frac{\|A \btheta - \bfeta\|^2}{\mu^{2}} + \frac{2 \bu^\top \btheta (A \btheta - \bfeta)^\top A \bu}{\mu^2 + \|\btheta\|^2} + \frac{\|A \bu\|^2 \|\btheta\|^2}{\mu^2 + \|\btheta\|^2}
    \\&
    \leq \frac{\|A \btheta - \bfeta\|^2}{\mu^{2}} + \frac{2 \bu^\top \btheta (A \btheta - \bfeta)^\top A \bu}{\mu^2 + \|\btheta\|^2} + \frac{\|A \bu\|^2 \|\btheta\|^2}{\mu^2 + \|\btheta\|^2}
    \\&
    \leq \frac{2 \|A \btheta - \bfeta\|^2}{\mu^{2}} + \frac{\|A \bu\|^2 (\bu^\top \btheta)^2 + \|A \bu\|^2 \|\btheta\|^2}{\mu^2 + \|\btheta\|^2}.
\end{align*}
By the definition of the set $\Ups(\rho)$, $\btheta$, $\bfeta$, and $A$ satisfy the inequalities $\|A \btheta - \bfeta\|^2 \leq \rho^2 \mu^2 \lambda$ and $\|\btheta\|^2 \leq \rho^2 (\mu^2 + \|\btheta\|^2)$. This yields that
\[
    \bu^\top H_{\btheta A} H_{AA}^{-1} H_{A \btheta} \bu
    \leq 2 \rho^2 \lambda + 2\rho^2 \|A \bu\|^2
    = 2 \rho^2 \; \bu^\top H_{\btheta \btheta} \bu.
\]
Hence, it holds that
\[
    H_{\btheta\btheta}^{-1/2} H_{\btheta A} H_{AA}^{-1} H_{A \btheta} H_{\btheta\btheta}^{-1/2}
    \preceq 2 \rho^2 I_d
\]
and then
\[
    \|H_{\btheta\btheta}^{-1/2} H_{\btheta A} H_{AA}^{-1/2}\| \leq \rho \sqrt{2}.
\]
\myendproof

\section{Properties of the target functional minimizer}

In this section, we list auxiliary properties of the triplet $\bups^* = (\btheta^*, \bfeta^*, A^*)$ minimizing the objective function $\cL(\bups)$. They play a fundamental role in studying the bias of the estimate $\widehat \bups$ defined in \eqref{eq: initial optimization problem}. Our first observation is that $\bups^*$ belongs to a convex subset of $\Ups(\rho_0)$, $\rho_0 \in [0, 1/7]$ (see \eqref{eq:upsilon}), under mild assumptions. In view of Lemma \ref{lem:block-diagonal_lower_bound}, this means that the functional $\cL(\bups)$ is strongly convex in a vicinity of $\bups^*$.

\begin{Lem}
\label{lemma: localization bias corollary}
    Let us fix an arbitrary $\rho_0 \in [0, 1/7]$ and let $\Ups(\rho_0)$ be as defined in \eqref{eq:upsilon}. Let
    \begin{align*}
        \Theta^*
        &
        = \left\{ \btheta \in \R^d : \Vert \btheta \Vert \le \frac{\rho_0 \mu}2 \text{ and } \Vert \Sigma \btheta - \Sigma \btheta^\circ \Vert \le \frac{\rho_0 \mu \sqrt{\lambda}}3 \right\},
        \\
        \sfA^*
        &
        = \left\{ A \in \R^{d \times d} : \Vert A - \Sigma \Vert \le \frac{\sqrt{\lambda}}3 \right\},
        \\
        \sfH^*
        &
        = \left\{\bfeta \in \R^d : \Vert \bfeta - \Sigma \btheta^\circ \Vert \le \frac{\rho_0 \mu \sqrt{\lambda}}3 \right\}
    \end{align*}
    and assume that $\Sigma$ and $\btheta^\circ$ satisfy the inequalities
    \begin{equation}
        \label{condition: Sigma bound on bias components}
        \Vert \btheta^\circ \Vert \le \frac{\rho_0 \mu}7
        \quad \text{and} \quad
        \Vert \Sigma \Vert \Vert \btheta^\circ \Vert \le \frac{\rho_0 \mu \sqrt{\lambda}}{24}.
    \end{equation}
    Then it holds that
    \begin{align*}
        \bups^*
        = (\btheta^*, \bfeta^*, A^*)
        \in \Theta^* \times \sfH^* \times \sfA^* 
        \subseteq \Ups(\rho_0).
    \end{align*}
\end{Lem}
We provide the proof of Lemma \ref{lemma: localization bias corollary} in Appendix \ref{sec:lemma localization bias corollary proof} below. Its main ingredient is the next result following from the definition of $\bups^*$.

\begin{Lem}
\label{lemma: localization lemma for bias}
Assume that $\Vert \Sigma \Vert \Vert \btheta^\circ \Vert \le 4 \mu \sqrt{\lambda}$ and $\Vert \btheta^\circ \Vert \le \mu/7$. Then we have
\begin{enumerate}[label=(\roman*)]
    \item \label{point: A* weak bound on bias, bias locating convex set}$\Vert A^* - \Sigma \Vert_{\F} \le \mu^{-1} \sqrt{\lambda} \Vert \btheta^\circ \Vert$ and $\Vert A^* - \Sigma \Vert \le \mu^{-2} \Vert \Sigma \Vert \Vert \btheta^\circ \Vert^2$;
    \item \label{point: eta weak bound on bias, bias locating convex set} $\Vert \bfeta^* - \Sigma \btheta^\circ \Vert \le \sqrt{\lambda} \Vert \btheta^\circ \Vert$;
    \item \label{point: bb bfeat weak bound, bias locating convex set} $\Vert A^* \btheta^* - \bfeta^* \Vert \le \sqrt{\lambda} \Vert \btheta^\circ \Vert$;
    \item \label{point: theta weak bound on bias, bias locating convex set} $\Vert \btheta^* \Vert \le \Vert \btheta^\circ \Vert$ and $\Vert \btheta^* - \btheta^\circ\Vert \le \|\bb_\lambda\| + \Vert \btheta^\circ \Vert^2 / \mu$.
\end{enumerate}
\end{Lem}

The proof of Lemma \ref{lemma: localization lemma for bias} is moved to Section \ref{sec:lemma localization lemma for bias proof}.
The inequalities \ref{point: A* weak bound on bias, bias locating convex set} will often be enough for our purposes. However, sometimes we will need sharper bounds on the operator norm of $A^* - \Sigma$ and on the Euclidean norm of $A^* \btheta^* - \bfeta^*$.

\begin{Lem}
\label{lemma: Sigma bias optimal bound}
Assume that $\|\btheta^\circ\| \leq \mu / 49$ and $\|\Sigma\| \|\btheta^\circ\| \leq \mu \sqrt{\lambda} / 24$. Then it holds that
\begin{align*}
    \Vert  A^* - \Sigma \Vert \le \left (\frac{14 \Vert \btheta^\circ \Vert}{\mu} \right )^2 \sqrt{\lambda} + \frac{ 35\Vert \Sigma \Vert \Vert \btheta^\circ \Vert \Vert \bb_\lambda \Vert}{\mu^2}.
\end{align*}
\end{Lem}

\begin{Lem}
    \label{lem:a_theta_eta_expansion}
    Assume that $\|\btheta^\circ\| \leq \mu / 49$ and $\|\Sigma\| \|\btheta^\circ\| \leq \mu \sqrt{\lambda} / 24$. Then it holds that
    \[
        \left\|A^* \btheta^* - \bfeta^* - \frac{\Sigma \bb_{\lambda}}2 \right\| \leq 3 \|\btheta^\circ\| \sqrt{\lambda} \left( \left(\frac{14 \|\btheta^\circ\|}\mu \right)^2 + \frac{35 \|\Sigma\| \|\btheta^\circ\| \|\bb_\lambda\|}{\mu^2 \sqrt{\lambda}} \right),
    \]
    where $\bb_\lambda = -\lambda (\Sigma^2 / 2 + \lambda I_d)^{-1} \btheta^\circ$.
\end{Lem}

We provide the proofs of Lemma \ref{lemma: Sigma bias optimal bound} and Lemma \ref{lem:a_theta_eta_expansion} in Sections \ref{sec:lemma Sigma bias optimal bound proof} and \ref{sec:lem_a_theta_eta_expansion}, respectively.

\subsection{Proof of Lemma~\ref{lemma: localization bias corollary}}
\label{sec:lemma localization bias corollary proof}

First, let us show that $\btheta^* \in \Theta^*$. Applying Lemma~\ref{lemma: localization lemma for bias}\ref{point: theta weak bound on bias, bias locating convex set} with $\rho_0 \leq 1/7$ in place of $\rho$, we obtain that
\begin{align}
    \label{eq: six theta circ bound on bias}
    \Vert \btheta^* - \btheta^\circ \Vert \le 2 \Vert \btheta^\circ \Vert.
\end{align}
Next, taking the condition~\ref{condition: Sigma bounded via lambda} of the lemma into account and using Lemma~\ref{lemma: localization lemma for bias}\ref{point: theta weak bound on bias, bias locating convex set}, we observe that
\begin{align}
    \Vert \Sigma \btheta^* - \Sigma \btheta^\circ \Vert & \le \Vert \Sigma \Vert (\Vert \btheta^* \Vert + \Vert \btheta^\circ \Vert) \le 2 \Vert \Sigma \Vert \Vert \btheta^\circ \Vert \le \rho_0 \mu \sqrt{\lambda} / 3 \label{eq: bias of theta wrt Sigma}.
\end{align}
The inequalities \eqref{eq: six theta circ bound on bias}
and \eqref{eq: bias of theta wrt Sigma} yield that the vector $\btheta^*$ belongs to $\Theta^*$. Our next goal is to prove that $A^* \in \sfA^*$. Due to Lemma~\ref{lemma: localization lemma for bias}\ref{point: A* weak bound on bias, bias locating convex set}, we have 
\begin{align}
\label{eq: bias of A wrt set A*}
    \Vert A^* - \Sigma \Vert
    \le \frac{\sqrt{\lambda} \Vert \btheta^\circ \Vert}{\mu}
    \le \frac{\rho_0 \sqrt{\lambda}}7
    \le \frac{\sqrt{\lambda}}3,
\end{align}
and, hence, $A^* \in \sfA^*$. An upper bound on $\Vert \bfeta^* - \Sigma \btheta^\circ \Vert$ easily follows from \eqref{eq: bias of theta wrt Sigma} and \eqref{eq: bias of A wrt set A*}:
\begin{align*}
    \Vert \bfeta^* - \Sigma \btheta^\circ \Vert & = \frac{1}{2} \Vert A^* \btheta^* - \Sigma \theta^\circ \Vert \le \frac{1}{2} (\Vert A^* \btheta^* - \Sigma \btheta^* \Vert + \Vert \Sigma \btheta^* - \Sigma \btheta^\circ \Vert) \\
    & \le \frac{1}{2} (\Vert A^* - \Sigma \Vert \Vert \btheta^* \Vert + \Vert \Sigma \btheta^* - \Sigma \btheta^\circ \Vert)
    \leq \frac{\rho_0 \mu \sqrt{\lambda}}3,
\end{align*}
where the last inequality is due to the fact that $\btheta^* \in \Theta^*$ and $A \in \sfA^*$. Thus, $\bfeta^* \in \sfH^*$. It only remains to prove that the product $\Theta^* \times \sfA^* \times \sfH^*$ is a subset of $\Upsilon(\rho_0)$. For this purpose, let us fix an arbitrary $(\btheta, A, \bfeta) \in \Theta^* \times \sfA^* \times \sfH^*$. Then, due to the definition of $\Theta^*, \sfA^*$ and $\sfH^*$, it holds that
\begin{align*}
    \Vert A \btheta - \bfeta \Vert \le \Vert A - \Sigma \Vert \Vert \btheta \Vert + \Vert \Sigma \btheta - \Sigma \btheta^\circ \Vert + \Vert \Sigma \btheta^\circ - \bfeta \Vert \le \rho_0 \mu \sqrt{\lambda}.
\end{align*}
This implies that the triplet $(\btheta, A, \bfeta)$ belongs to $\Upsilon(\rho_0)$ as well. Hence, $\Theta^* \times \sfA^* \times \sfH^* \subseteq \Upsilon(\rho_0)$, and the proof is finished.

\myendproof

\subsection{Proof of Lemma~\ref{lemma: localization lemma for bias}}
\label{sec:lemma localization lemma for bias proof}

Let us remind the reader that
\[
    \cL(\bups)
    = \cL(\btheta, \bfeta, A)
    = \frac12 \|\Sigma \btheta^\circ - \bfeta\|^2 + \frac12 \|\bfeta - A \btheta\|^2 + \frac{\mu^2}2 \|\Sigma - A\|_{\F}^2 + \frac{\lambda}2 \|\btheta\|^2.
\]
Since $(\btheta^*, \bfeta^*, A^*)$ minimizes $\cL(\btheta, \bfeta, A)$, we have
\[
    \cL(\btheta^*, A^*, \bfeta^*)
    \le \cL(\btheta^\circ, \Sigma, \Sigma \btheta^\circ)
    = \frac{\lambda \Vert \btheta^\circ \Vert^2}2
\]
and then
\begin{align}
    \label{eq: Delta(A) bias bound}
    \notag
    \Vert A^* \btheta^* - \bfeta^* \Vert
    &
    \le \sqrt{\lambda} \Vert \btheta^\circ \Vert,
    \quad
    \Vert A^* - \Sigma \Vert_{\F} \le \mu^{-1} \sqrt{\lambda} \Vert \btheta^\circ \Vert,
    \\
    \Vert \bfeta^* - \Sigma \btheta^\circ \Vert
    &
    \le \sqrt{\lambda} \Vert \btheta^\circ \Vert,
    \quad
    \Vert \btheta^* \Vert \le \Vert \btheta^\circ \Vert.
\end{align}
This yields the statements \ref{point: eta weak bound on bias, bias locating convex set} and \ref{point: bb bfeat weak bound, bias locating convex set} of the lemma and the first part of~\ref{point: A* weak bound on bias, bias locating convex set} and \ref{point: theta weak bound on bias, bias locating convex set}. In contrary, the derivation of the inequalities 
\[
    \Vert A^* - \Sigma \Vert \le \Vert \Sigma \Vert \Vert \btheta^\circ \Vert^2 / \mu^2
    \quad \text{and} \quad
    \Vert \btheta^* - \btheta^\circ\Vert \le \left \Vert \bb_\lambda \right \Vert + 3 \Vert \btheta^\circ \Vert \Vert \bb_\lambda \Vert / \mu
\]
is not straightforward. For the ease of exposure, we split the rest of the proof into several steps.

\medskip

\noindent \textbf{Step 1. An intermediate bound on $\|\btheta^* - \btheta^\circ\|$.}\quad
To bound the norm of $\btheta^* - \btheta^\circ$, we consider the gradient of $\cL$. Since $\bnabla_{\btheta} \cL(\btheta^*, \bfeta^*, A^*) = \bzero$ and
\[
    \bnabla_{\btheta} \cL(\btheta, \bfeta, A) = A^\top (A \btheta - \bfeta) + \lambda \btheta,
\]
we have
\[
    -\lambda \btheta^\circ
    = \bnabla_{\btheta} \cL(\btheta^*, \bfeta^*, A^*) - \bnabla_{\btheta} \cL(\btheta^\circ, \Sigma \btheta^\circ, \Sigma)
    = (A^*)^\top (A^* \btheta^* - \bfeta^*) + \lambda (\btheta^* - \btheta^\circ)
\]
Let us note that $\bfeta^* = (A^* \btheta^* + \Sigma \btheta^\circ) / 2$. This yields that
\[
    \frac12 (A^*)^\top (A^* \btheta^* - \Sigma \btheta^\circ) + \lambda (\btheta^* - \btheta^\circ)
    = -\lambda \btheta^\circ
\]
or, equivalently,
\begin{align}
    \label{eq:first-order_optimality_corollary}
    \left( \frac12 (A^*)^\top A^* + \lambda I_d \right) (\btheta^* - \btheta^\circ)
    &\notag
    = -\lambda \btheta^\circ - \frac12 (A^*)^\top (A^* - \Sigma) \btheta^\circ 
    \\&
    = -\lambda \btheta^\circ - \frac12 \Sigma (A^* - \Sigma) \btheta^\circ - \frac12 (A^* - \Sigma)^\top (A^* - \Sigma) \btheta^\circ.
\end{align}
Thus, due to the triangle inequality, we obtain that
\begin{align*}
    \left\| \btheta^* - \btheta^\circ \right\|
    &
    \leq \left\| \left( \frac12 (A^*)^\top A^* + \lambda I_d \right)^{-1} \left( \lambda \btheta^\circ + \frac12 \Sigma (A^* - \Sigma) \btheta^\circ \right) \right\|
    \\&\quad
    + \frac12 \left\| \left( \frac12 (A^*)^\top A^* + \lambda I_d \right)^{-1} (A^* - \Sigma)^\top (A^* - \Sigma) \btheta^\circ \right\|
    \\&
    \leq \left\| \left( \frac12 (A^*)^\top A^* + \lambda I_d \right)^{-1} \left( \lambda \btheta^\circ + \frac12 \Sigma (A^* - \Sigma) \btheta^\circ \right) \right\| + \frac{\|A^* - \Sigma\|^2 \|\btheta^\circ\|}{2 \lambda}.
\end{align*}
Let us introduce $\tilde \rho = \|\btheta^\circ\| / \mu \le 1/7$. This, together with the bound \eqref{eq: Delta(A) bias bound}, implies that
\[
    \frac{\|A^* - \Sigma\|^2 \|\btheta^\circ\|}{2 \lambda}
    \leq \frac{\|\btheta^\circ\|}{2 \lambda} \cdot \frac{\lambda \|\btheta^\circ\|^2}{\mu^2}
    = \frac{\|\btheta^\circ\|^3}{2 \mu^2}
    \leq \frac{\tilde \rho^2 \|\btheta^\circ\|}2.
\]
Hence, it holds that
\begin{equation}
    \label{eq:bias_norm_intermediate_bound}
    \left\| \btheta^* - \btheta^\circ \right\|
    \leq \left\| \left( \frac12 (A^*)^\top A^* + \lambda I_d \right)^{-1} \left( \lambda \btheta^\circ + \frac12 \Sigma (A^* - \Sigma) \btheta^\circ \right) \right\| + \frac{\tilde \rho^2 \|\btheta^\circ\|}2.
\end{equation}

\medskip

\noindent \textbf{Step 2. Final bound on $\|\btheta^* - \btheta^\circ\|$.}\quad
In what follows, we are going to show that
\begin{equation}
    \label{eq:bias_norm_first_term_bound}
    \left\| \left( \frac12 (A^*)^\top A^* + \lambda I_d \right)^{-1} \left( \lambda \btheta^\circ + \frac12 \Sigma (A^* - \Sigma) \btheta^\circ \right) \right\|
    \leq \|\bb_\lambda\| + \frac{35 \tilde \rho \|\btheta^\circ\|}{12}.
\end{equation}
Then the inequalities \eqref{eq:bias_norm_intermediate_bound} and \eqref{eq:bias_norm_first_term_bound} will yield the desired bound on $\|\btheta^* - \btheta^\circ\|$. Let us introduce
\[
    S = \frac12 (A^*)^\top A^* - \frac12 \Sigma^2
    \quad \text{and} \quad
    B = \left( \frac12 \Sigma^2 + \lambda I_d \right)^{-1/2} S \left( \frac12 \Sigma^2 + \lambda I_d \right)^{-1/2}.
\]
Note that, according to Lemma \ref{lem:diff_squares_op_norm_bound} and \eqref{eq: Delta(A) bias bound}, the operator norm of $B$ does not exceed
\begin{equation}
    \label{eq:standardized_remainder_norm_bound}
    \|B\|
    \leq \|A^* - \Sigma\| \sqrt{\frac{2}{\lambda}} + \frac{\|A^* - \Sigma\|^2}\lambda
    \leq \frac{\|\btheta^\circ\| \sqrt{2}}{\mu} + \frac{\|\btheta^\circ\|^2}{\mu^2}
    = \tilde \rho\sqrt{2} + \tilde \rho^2 \leq \frac14.
\end{equation}
Here we used the inequalities $\|\btheta^\circ\| = \tilde \rho \mu$ and $0 \leq \tilde \rho \leq 1/7$ holding by the definition of $\tilde \rho$ and conditions of the lemma. In particular, \eqref{eq:standardized_remainder_norm_bound} ensures that the matrix $(I_d + B)$ is invertible and
\[
    \left\| (I_d + B)^{-1} \right\|
    \leq \frac1{1 - \|B\|}
    \leq \frac43.
\]
Using the relation
\[
    (I_d + B)^{-1}
    = (I_d + B)^{-1} (I_d + B - B)
    = I_d - (I_d + B)^{-1} B,
\]
we obtain that
\begin{align}
    \label{eq:bias_norm_first_term_expansion}
    &\notag
    \left( \frac12 (A^*)^\top A^* + \lambda I_d \right)^{-1} \left( \lambda \btheta^\circ + \frac12 \Sigma (A^* - \Sigma) \btheta^\circ \right)
    \\&\notag
    = \left( \frac12 \Sigma^2 + \lambda I_d \right)^{-1/2} (I_d + B)^{-1} \left( \frac12 \Sigma^2 + \lambda I_d \right)^{-1/2} \left( \lambda \btheta^\circ + \frac12 \Sigma (A^* - \Sigma) \btheta^\circ \right)
    \\&
    = -\bb_\lambda - \lambda \left( \frac12 \Sigma^2 + \lambda I_d \right)^{-1/2} (I_d + B)^{-1} B \left( \frac12 \Sigma^2 + \lambda I_d \right)^{-1/2} \btheta^\circ
    \\&\quad\notag
    + \frac12 \left( \frac12 \Sigma^2 + \lambda I_d \right)^{-1/2} (I_d + B)^{-1} \left( \frac12 \Sigma^2 + \lambda I_d \right)^{-1/2} \Sigma (A^* - \Sigma) \btheta^\circ.
\end{align}
The norm of the second term in the right-hand side of \eqref{eq:bias_norm_first_term_expansion} in not greater than
\begin{align}
    \label{eq:bias_norm_first_term_expansion_i}
    &\notag
    \left\| \lambda \left( \frac12 \Sigma^2 + \lambda I_d \right)^{-1/2} (I_d + B)^{-1} B \left( \frac12 \Sigma^2 + \lambda I_d \right)^{-1/2} \btheta^\circ \right\|
    \\&
    \leq \lambda \cdot \frac1{\sqrt{\lambda}} \cdot \frac{\|B\|}{1 - \|B\|} \cdot \frac{\|\btheta^\circ\|}{\sqrt{\lambda}}
    = \frac43 (\tilde \rho \sqrt{2} + \tilde \rho^2) \|\btheta^\circ\|
    \leq \frac{9 \tilde \rho \|\btheta^\circ\|}4,
\end{align}
while the third one does not exceed
\begin{align}
    \label{eq:bias_norm_first_term_expansion_ii}
    &\notag
    \left\| \frac12 \left( \frac12 \Sigma^2 + \lambda I_d \right)^{-1/2} (I_d + B)^{-1} \left( \frac12 \Sigma^2 + \lambda I_d \right)^{-1/2} \Sigma (A^* - \Sigma) \btheta^\circ \right\|
    \\&
    \leq \frac1{2 \sqrt{\lambda}} \cdot \frac1{1 - \|B\|} \cdot \left\| \left( \frac12 \Sigma^2 + \lambda I_d \right)^{-1/2} \Sigma \right\| \|A^* - \Sigma\| \|\btheta^\circ\|
    \\&\notag
    \leq \frac1{2 \sqrt{\lambda}} \cdot \frac1{1 - \|B\|} \cdot \frac{\|\btheta^\circ\|^2 \sqrt{\lambda}}{\mu}
    \leq \frac43 \cdot \frac{\tilde \rho \|\btheta^\circ\|}{2}
    = \frac{2 \tilde \rho \|\btheta^\circ\|}{3}.
\end{align}
The inequalities \eqref{eq:bias_norm_first_term_expansion}, \eqref{eq:bias_norm_first_term_expansion_i}, and \eqref{eq:bias_norm_first_term_expansion_ii} imply that
\begin{align*}
    &
    \left\| \left( \frac12 (A^*)^\top A^* + \lambda I_d \right)^{-1} \left( \lambda \btheta^\circ + \frac12 \Sigma (A^* - \Sigma) \btheta^\circ \right) \right\|
    \\&
    \leq \|\bb_\lambda\| + \frac{9 \tilde \rho \|\btheta^\circ\|}4 + \frac{2 \tilde \rho \|\btheta^\circ\|}{3}
    = \|\bb_\lambda\| + \frac{35 \tilde \rho \|\btheta^\circ\|}{12}.
\end{align*}
Hence, \eqref{eq:bias_norm_first_term_bound} holds and then \eqref{eq:bias_norm_intermediate_bound} yields
\[
    \|\btheta^* - \btheta^\circ\|
    \leq \|\bb_\lambda\| + \frac{35 \tilde \rho \|\btheta^\circ\|}{12} + \frac{\tilde \rho^2 \|\btheta^\circ\|}2
    \leq \|\bb_\lambda\| + 3 \tilde \rho \|\btheta^\circ\| = \Vert \bb_\lambda \Vert + \Vert \btheta^\circ \Vert^2 / \mu,
\]
where the last inequality is due to the fact that $\tilde \rho \leq 1/7$.

\medskip

\noindent \textbf{Step 3. The bound on $\Vert A^* - \Sigma \Vert$.} Finally, we are going to prove that $\Vert A^* - \Sigma \Vert \le  \Vert \Sigma \Vert \Vert \btheta^\circ \Vert^2/ \mu^2$ using the derived bound on $\Vert \btheta^* - \btheta^\circ \Vert$. Note that $\|\bb_\lambda\| \leq \|\btheta^\circ\|$ and then
\begin{align}
    \label{eq: 6 theta circ bound on bias}
    \Vert \btheta^* - \btheta^\circ \Vert
    \le \Vert \btheta^\circ \Vert + 3 \tilde \rho \Vert \btheta^\circ \Vert
    \leq \Vert \btheta^\circ \Vert + \frac{3 \Vert \btheta^\circ \Vert}7
    \le \frac{3 \Vert \btheta^\circ \Vert}2.
\end{align}
Similarly to the analysis of $\btheta^* - \btheta^\circ$, our study of $A^* - \Sigma$ relies on the first-order optimality condition:
\[
    O_d = \nabla_A \cL(\btheta^*, \bfeta^*, A^*) = (A^* \btheta^* - \bfeta^*) (\btheta^*)^\top + \mu^2 (A^* - \Sigma).
\]
Since $\bfeta^* = (A^* \btheta^* + \Sigma \btheta^\circ) / 2$, it holds that
\[
    \mu^2 (A^* - \Sigma) = \frac12(\Sigma \btheta^\circ - A^* \btheta^*) (\btheta^*)^\top.
\]
Thus, we have
\[
    \left( \mu^2 I_d + \btheta^* (\btheta^*)^\top \right) (A^* - \Sigma) = \frac12 \Sigma (\btheta^\circ - \btheta^*) (\btheta^*)^\top.    
\]
Since the smallest eigenvalue of $\mu^2 I_d + \btheta^* (\btheta^*)^\top$ is equal to $\mu^2$, we obtain that
\[
    \|A^* - \Sigma\|
    \leq \frac{\|\Sigma\| \|\btheta^*\| \|\btheta^* - \btheta^\circ\|}{2 \mu^2} 
    \leq \frac{3 \|\Sigma\| \|\btheta^\circ\|^2}{4 \mu^2} \le \Vert \Sigma \Vert \Vert \btheta^\circ \Vert^2 / \mu^2.
\]

\myendproof

\subsection{Proof of Lemma~\ref{lemma: Sigma bias optimal bound}}
\label{sec:lemma Sigma bias optimal bound proof}

Similarly to the proof of Theorem \ref{theorem: bias}, the analysis of $A^* - \Sigma$ starts with the Newton-Leibniz formula:
\begin{align*}
    - \bnabla \cL(\bups^\circ)
    = \bnabla \cL(\bups^*) - \bnabla \cL(\bups^\circ)
    = \int_{0}^1 \nabla^2 \cL\big(t \bups^* + (1 - t) \bups^\circ \big) (\bups^* - \bups^\circ) \, \rmd t
    = \avg{H} (\bups^* - \bups^\circ).
\end{align*}
Here we used the fact that $\bnabla \cL(\bups^*) = \bzero$ and introduced the notation
\[
    \avg H = \int_{0}^1 \nabla^2 \cL\big(t \bups^* + (1 - t) \bups^\circ \big) \, \rmd t.
\]
More generally, throughout the proof, for a matrix-valued function $f(\bups) = f(\btheta, \bfeta, A)$, $\biasInner{f}$ stands for
\begin{align*}
    \biasInner{f} = \int_0^1 f\big(t \bups^* + (1 - t) \bups^\circ \big) \, \rmd t.
\end{align*}
Note that, according to Lemma \ref{lemma: localization bias corollary} and Lemma \ref{lem:block-diagonal_lower_bound}, the matrix $\biasInner{H}$ is invertible. Then it holds that $\bups^* - \bups^\circ = - \avg{H}^{-1} \bnabla \cL(\bups^\circ)$, and we have to study the last $d^2$ components of $\avg{H}^{-1} \bnabla \cL(\bups^\circ)$. For this purpose, we use the block form representation of the matrix $\avg H$ and its inverse:
\begin{align*}
    \avg H
    =
    \begin{pmatrix}
        \avg H_{\btheta \btheta} & \avg H_{\btheta \bfeta} & \avg H_{\btheta A}\\
        \avg H_{\bfeta \btheta} & \avg H_{\bfeta \bfeta} & \avg H_{\bfeta A} \\
        \avg H_{A \btheta} & \avg H_{A \bfeta} & \avg H_{AA}
    \end{pmatrix}
    =
    \begin{pmatrix}
        \avg H_{\btheta \btheta} & \avg H_{\btheta \bfeta} & \avg H_{\btheta \ba_1} & \dots & \avg H_{\btheta \ba_d} \\
        \avg H_{\bfeta \btheta} & \avg H_{\bfeta \bfeta} & \avg H_{\bfeta \ba_1} & \dots & \avg H_{\bfeta \ba_d} \\
        \avg H_{\ba_1 \btheta} & \avg H_{\ba_1 \bfeta} & \avg H_{\ba_1 \ba_1} & \dots & \avg H_{\ba_1 \ba_d} \\
        \vdots & \vdots & \vdots & \ddots & \vdots \\
        \avg H_{\ba_d \btheta} & \avg H_{\ba_d \bfeta} & \avg H_{\ba_d \ba_1} & \dots & \avg H_{\ba_d \ba_d}
    \end{pmatrix}.
\end{align*}
In what follows, we denote $\avg H^{-1}$ by $F$ and, for any $\bnu, \bgamma \in \{\btheta, \bfeta, \ba_1, \dots \ba_d\}$, $F_{\bnu \bgamma}$ stands for the corresponding block of $F$. Let us note that all the components of $\nabla \cL(\bups^\circ)$, except for the first $d$ ones, are equal to zero. Therefore, for any $j \in \{1, \dots, d\}$ we may write
\begin{align*}
    \ba_j^* - \Sigma \be_j
    &
    = - \big( \avg H^{-1} \, \bnabla \cL(\bups^\circ) \big)_{\ba_j}
    = - \big( F \, \bnabla \cL(\bups^\circ) \big)_{\ba_j}
    \\&
    = - F_{\ba_j \btheta} \bnabla_{\btheta} \cL(\bups^\circ)
    - F_{\ba_j \bfeta} \bnabla_{\bfeta} \cL(\bups^\circ)
    - \sum\limits_{k = 1}^d F_{\ba_j \ba_k} \bnabla_{\ba_k} \cL(\bups^\circ)
    \\&
    = - F_{\ba_j \btheta} \bnabla_{\btheta} \cL(\bups^\circ)
    = -\lambda F_{\ba_j \btheta} \btheta^\circ.
\end{align*}
Let us elaborate on $F_{\ba_j \btheta}$. Denote the nuisance parameter by $\chi = (\bfeta, A)$ and let
\[
    \avg H_{\chi \chi}
    =
    \begin{pmatrix}
        \avg H_{\bfeta \bfeta} & \avg H_{\bfeta A} \\
        \avg H_{A \bfeta} & \avg H_{AA}
    \end{pmatrix}
    =
    \begin{pmatrix}
        \avg H_{\bfeta \bfeta} & \avg H_{\bfeta \ba_1} & \dots & \avg H_{\bfeta \ba_d} \\
        \avg H_{\ba_1 \bfeta} & \avg H_{\ba_1 \ba_1} & \dots & \avg H_{\ba_1 \ba_d} \\
        \vdots & \vdots & \ddots & \vdots \\
        \avg H_{\ba_d \bfeta} & \avg H_{\ba_d \ba_1} & \dots & \avg H_{\ba_d \ba_d}
    \end{pmatrix}.
\]
stand for the corresponding block of $\avg H$. Introducing $J = \avg{H}_{\chi \chi}^{-1}$ and using the blockwise inversion formula \eqref{eq:h_block_inversion}, we obtain that
\[
    F_{\chi \btheta} = -\avg H_{\chi \chi}^{-1} H_{\chi \btheta} \left(\avg H / \avg H_{\chi \chi} \right)^{-1}
    = J H_{\chi \btheta} \left(\avg H / \avg H_{\chi \chi} \right)^{-1}.
\]
Thus, for any $j \in \{1, \dots, d\}$, we have
\[
    F_{\ba_j \btheta}
    = -J_{\ba_j \bfeta} H_{\bfeta \btheta} \left(\avg H / \avg H_{\chi \chi} \right)^{-1} - \sum\limits_{k = 1}^d J_{\ba_j \ba_k} H_{\ba_k \btheta} \left(\avg H / \avg H_{\chi \chi} \right)^{-1}.
\]
According to Proposition \ref{proposition: inverse of nuisance parameters}, it holds that $J_{\ba_j \ba_k} = O_d$ for all $j \neq k$. Then the expression for $F_{\ba_j \btheta}$ simplifies to
\[
    F_{\ba_j \btheta}
    = -J_{\ba_j \bfeta} H_{\bfeta \btheta} \left(\avg H / \avg H_{\chi \chi} \right)^{-1} - J_{\ba_j \ba_j} H_{\ba_j \btheta} \left(\avg H / \avg H_{\chi \chi} \right)^{-1}
\]
Let us recall that in the proof of Theorem \ref{theorem: bias}, we showed (see \eqref{eq:theta_star_representation})
\[
    \btheta^* - \btheta^\circ = -\lambda \left(\avg H / \avg H_{\chi \chi} \right)^{-1} \btheta^\circ.
\]
This yields that
\[
    \ba_j^* - \Sigma \be_j
    = -J_{\ba_j \bfeta} H_{\bfeta \btheta} (\btheta^* - \btheta^\circ)
    - J_{\ba_j \ba_j} H_{\ba_j \btheta} (\btheta^* - \btheta^\circ)
    \quad \text{for all $j \in \{1, \dots, d\}$}
\]
and, hence,
\begin{align}
    \label{eq: A - Sigma star decomposition using the average hessian}
    A^* - \Sigma
    &
    = \sum_{j = 1}^d \left( \be_j (\ba_j^*)^\top - \Sigma \be_j \be_j^\top \right)
    \nonumber \\
    &
    = - \sum_{j = 1}^d \be_j (\btheta^* - \btheta^\circ)^\top \left ( \left \{ \avg{(\ba_j^\top \btheta - \bfeta_j)} I_d + \avg{\ba_j \btheta^\top}   \right \}  R_3 + r_2 \avg{A}^\top \be_j \Tilde{\btheta}^\top \right ) \nonumber
    \\&
    = - \avg{(A \btheta - \bfeta)} (\btheta^* - \btheta^\circ)^\top R_3 - \avg{A (\btheta^* - \btheta^\circ) \btheta^\top} R_3 - r_2 \avg{A} (\btheta^* - \btheta^\circ) \Tilde{\btheta}^\top.
\end{align}
Let us consider the first term in the right-hand side of~\eqref{eq: A - Sigma star decomposition using the average hessian}. For the sake of brevity, we introduce
\begin{align*}
    \rho := \frac{7 \Vert \btheta^\circ \Vert}{\mu}.
\end{align*}
Due to Proposition~\ref{proposition: A theta - eta bias bound}, Lemma~\ref{lemma: localization lemma for bias}\ref{point: theta weak bound on bias, bias locating convex set} and the assumption $\Vert \btheta^\circ \Vert \le \rho \mu /7$, we have
\[
    \Vert \avg{A \btheta - \bfeta} \Vert \le \rho \mu \sqrt{\lambda} / 2 + 3 \Vert \Sigma \Vert \Vert \bb_\lambda \Vert
    \quad \text{and} \quad
    \Vert \btheta^* \Vert \le \Vert \btheta^\circ \Vert \le \rho \mu /7.
\]
This implies that
\begin{align}
    \Vert \avg{(A \btheta - \bfeta)} (\btheta^* - \btheta^\circ)^\top \, R_3 \Vert & \le \left ( \rho \mu \sqrt{\lambda} / 2 + 3 \Vert \Sigma \Vert \Vert \bb_\lambda \Vert \right ) \cdot (2 \rho \mu) \cdot \Vert R_3 \Vert / 7  \nonumber \\
    & \le  \rho^2 \sqrt{\lambda} / 3 + 2 \rho \Vert \Sigma \Vert \Vert \bb_\lambda \Vert / \mu, \label{eq: A - Sigma star decomposition using the average hessian - term 1 bound}
\end{align}
where we use $\Vert R_3 \Vert \le 2 \mu^{-2}$ due to Proposition~\ref{proposition: inverse of nuisance parameters}.
Then we bound last two terms of~\eqref{eq: A - Sigma star decomposition using the average hessian}. We have
\begin{align*}
    \left\Vert \avg{A (\btheta^* - \btheta^\circ) \btheta^\top} \right\Vert
    &
    \le \left\Vert \Sigma  (\btheta^* - \btheta^\circ) \avg{\btheta}^\top \right\Vert
    + \left\Vert \avg{(A - \Sigma) (\btheta^* - \btheta^\circ) \btheta^\top} \right\Vert, \\
    \Vert \avg{A} (\btheta^* - \btheta^\circ) \Tilde{\btheta}^\top \Vert & \le \left \Vert \Sigma (\btheta^* - \btheta^\circ) \Tilde{\btheta}^\top \right \Vert + \left \Vert \avg{A - \Sigma} (\btheta^*- \btheta^\circ ) \Tilde{\btheta}^\top \right \Vert.
\end{align*}
Using Lemma~\ref{lemma: localization lemma for bias}\ref{point: theta weak bound on bias, bias locating convex set} and the triangle inequality, we deduce that
\begin{align*}
    \left\Vert \avg{A (\btheta^* - \btheta^\circ) \btheta^\top} \right\Vert
    & \le \Vert \Sigma \Vert \left \{ \Vert \bb_\lambda \Vert + 3 \rho \Vert \btheta^\circ \Vert \right \} \Vert \avg{\btheta}^\top \Vert + \int_0^1 \Vert (A^{(t)} - \Sigma) (\btheta^* - \btheta^\circ) (\btheta^{(t)})^\top \Vert \rmd \m(t), \\
    \left\Vert \avg{A} (\btheta^* - \btheta^\circ) \Tilde{\btheta}^\top \right\Vert
    & \le \Vert \Sigma \Vert \left \{ \Vert \bb_\lambda \Vert + 3 \rho \Vert \btheta^\circ \Vert \right \} \Vert \avg{\btheta}^\top \Vert + \int_0^1 \Vert (A^{(t)} - \Sigma) (\btheta^* - \btheta^\circ) \Tilde{\btheta}^\top \Vert \rmd \m(t).
\end{align*}
Since $A^{(t)} \in \sfA^*$, we have $\Vert A^{(t)} - \Sigma \Vert \le \sqrt{\lambda}/3$. On the other hand, the vectors $\btheta^*$, $\btheta^\circ$, and $\btheta^{(t)}$ belong to $\Theta^*$. Thus, we can bound their norms by $\rho \mu$, as well as the norm of $\Tilde{\btheta}$ due to~\eqref{eq: tilde theta bound}, and obtain that
\begin{align*}
    \left\Vert \avg{A (\btheta^* - \btheta^\circ) \btheta^\top} \right\Vert
    &
    \le \rho \mu  \Vert \Sigma \Vert \left \{ \Vert \bb_\lambda \Vert + 3 \rho \Vert \btheta^\circ \Vert \right \} + 2 \rho^2 \mu^2 \sqrt{\lambda} / 3,
    \\
    \Vert \avg{A} (\btheta^* - \btheta^\circ) \Tilde{\btheta} \Vert
    &
    \le \rho \mu  \Vert \Sigma \Vert \left \{ \Vert \bb_\lambda \Vert + 3 \rho \Vert \btheta^\circ \Vert \right \} + 2 \rho^2 \mu^2 \sqrt{\lambda} / 3.
\end{align*}
It remains to bound the first terms of expressions above. By the assumptions of the lemma, we have
\begin{align*}
    \Vert \Sigma \Vert \Vert \btheta^\circ \Vert & \le \mu \sqrt{\lambda} / 24.
\end{align*}
This yields that
\begin{align}
\label{eq: A - Sigma star decomposition using the average hessian - term 2 bound}
    \left \Vert \avg{A (\btheta^* - \btheta^\circ) \btheta^\top} \right \Vert \le \rho^2 \mu^2 \sqrt{\lambda} + \rho \mu \Vert \Sigma \Vert \Vert \bb_\lambda \Vert, \\
\label{eq: A - Sigma star decomposition using the average hessian - term 3 bound}
    \left \Vert \avg{A} (\btheta^* - \btheta^\circ) \Tilde{\btheta}^\top \right \Vert \le \rho^2 \mu^2 \sqrt{\lambda} + \rho \mu \Vert \Sigma \Vert \Vert \bb_\lambda \Vert.
\end{align}
Combining the inequalities~\eqref{eq: A - Sigma star decomposition using the average hessian - term 1 bound},\eqref{eq: A - Sigma star decomposition using the average hessian - term 2 bound},\eqref{eq: A - Sigma star decomposition using the average hessian - term 3 bound} with \eqref{eq: A - Sigma star decomposition using the average hessian}, we finally get that
\begin{align*}
    \Vert A^* - \Sigma \Vert \le \rho^2 \sqrt{\lambda} / 3 +  \frac{2 \rho \Vert \Sigma \Vert \Vert \bb_\lambda \Vert}{\mu}+ \Vert R_3 \Vert \cdot (\rho^2 \mu^2 \sqrt{\lambda} + \Vert \Sigma \Vert \Vert \bb_\lambda \Vert) + r_2 (\rho^2 \mu^2 \sqrt{\lambda} + \Vert \Sigma \Vert \Vert \bb_\lambda \Vert),
\end{align*}
which is at most $4 \rho^2 \sqrt{\lambda} + 5 \rho \Vert \Sigma \Vert \Vert \bb_\lambda \Vert / \mu$ due to Proposition~\ref{proposition: inverse of nuisance parameters}.
\myendproof

\subsection{Proof of Lemma \ref{lem:a_theta_eta_expansion}}
\label{sec:lem_a_theta_eta_expansion}

Since $\bfeta^* = (A^* \btheta^* - \Sigma \btheta^\circ) / 2$, it holds that
\begin{align}
    \label{eq:a_theta_eta_triangle}
    \left\| A^* \btheta^* - \bfeta^* - \frac12 \Sigma \bb_\lambda \right\|
    &\notag
    = \frac12 \left\| A^* \btheta^* - \Sigma \btheta^\circ - \Sigma \bb_\lambda \right\|
    \\&
    \leq \frac12 \| A^* - \Sigma\| \|\btheta^*\| + \frac12 \left\| \Sigma (\btheta^* - \btheta^\circ - \bb_\lambda) \right\|.
\end{align}
Using Lemma \ref{lemma: Sigma bias optimal bound} and Lemma \ref{lemma: localization lemma for bias}\ref{point: theta weak bound on bias, bias locating convex set} we can easily bound the first term in the right-hand side:
\begin{equation}
    \label{eq:a_theta_eta_first_term}
    \| A^* - \Sigma\| \|\btheta^*\| \leq \left(\frac{14 \|\btheta^\circ\|}\mu \right)^2 \|\btheta^\circ\| \sqrt{\lambda} + \frac{35 \|\Sigma\| \|\btheta^\circ\|^2 \|\bb_\lambda\|}{\mu^2}.
\end{equation}
To bound the second term in the right-hand side of \eqref{eq:a_theta_eta_triangle}, we use the first-order optimality condition $\bnabla_{\btheta} \cL(\btheta^*, \bfeta^*, A^*) = \bzero$. Similarly to the proof of Lemma \ref{lemma: localization lemma for bias} (see \eqref{eq:first-order_optimality_corollary}), we obtain that
\begin{equation}
    \label{eq:first-order_optimality_corollary_2}
    \left( \frac12 (A^*)^\top A^* + \lambda I_d \right) (\btheta^* - \btheta^\circ)
    = -\lambda \btheta^\circ - \frac12 (A^*)^\top (A^* - \Sigma) \btheta^\circ.
\end{equation}
Let us introduce
\[
    B = \left( \frac12 \Sigma^2 + \lambda I_d \right)^{-1/2} \left( (A^*)^\top A^* - \Sigma^2 \right) \left( \frac12 \Sigma^2 + \lambda I_d \right)^{-1/2}.
\]
According to Lemma \ref{lem:diff_squares_op_norm_bound}, the operator norm of $B$ does not exceed
\[
    \|B\| \leq \|A^* - \Sigma\| \sqrt{\frac2\lambda} + \frac{\|A^* - \Sigma\|^2}{2 \lambda}.
\]
Due to Lemma \ref{lemma: Sigma bias optimal bound}, we have
\[
    \frac{\|A^* - \Sigma\|}{\sqrt{\lambda}} \leq \left(\frac{14 \|\btheta^\circ\|}\mu \right)^2 + \frac{35 \|\Sigma\| \|\btheta^\circ\| \|\bb_\lambda\|}{\mu^2 \sqrt{\lambda}}.
\]
The conditions of the lemma yield that $\|A^* - \Sigma\| / \sqrt{\lambda}$ is not greater than $1$. Indeed, since $\|\btheta^\circ\| / \mu = \rho / 7 \leq 1 / 49$ and $\mu \sqrt{\lambda} \geq 24 \|\Sigma\| \|\btheta^\circ\|$, it holds that
\[
    \frac{\|A^* - \Sigma\|}{\sqrt{\lambda}} \leq \left( \frac27 \right)^2 + \frac{35 \|\bb_\lambda\|}{24 \mu}
    \leq \left( \frac27 \right)^2 + \frac{35 \|\btheta^\circ\|}{24 \mu}
    \leq \left( \frac27 \right)^2 + \frac{5}{24 \cdot 7} < \frac18.
\]
This implies that
\[
    \|A^* - \Sigma\| \sqrt{\frac2\lambda} + \frac{\|A^* - \Sigma\|^2}{2 \lambda}
    \leq \frac{\|A^* - \Sigma\|}{\sqrt{\lambda}} \left( \sqrt{2} + \frac12 \right) 
    \leq \frac{2 \|A^* - \Sigma\|}{\sqrt{\lambda}}.
\]
and then
\begin{equation}
    \label{eq:b_bound}
    \|B\|
    \leq \frac{2 \|A^* - \Sigma\|}{\sqrt{\lambda}}
    \leq 2 \left(\frac{14 \|\btheta^\circ\|}\mu \right)^2 + \frac{70 \|\Sigma\| \|\btheta^\circ\| \|\bb_\lambda\|}{\mu^2 \sqrt{\lambda}} \leq \frac14.
\end{equation}
Using the identity \eqref{eq:first-order_optimality_corollary_2}, we obtain that
\begin{align*}
    \left\| \Sigma \left(\btheta^* - \btheta^\circ - \bb_\lambda \right) \right\|
    &
    = \Bigg\| \lambda \Sigma \left(\frac12 \Sigma^2 + \lambda I_d \right)^{-1} \btheta^\circ - \lambda \Sigma \left( \frac12 (A^*)^\top A^* + \lambda I_d \right)^{-1} \btheta^\circ
    \\&\quad
    - \frac12 \Sigma \left( \frac12 (A^*)^\top A^* + \lambda I_d \right)^{-1} (A^*)^\top (A^* - \Sigma) \btheta^\circ \Bigg\|
    \\&
    = \Bigg\| \lambda \Sigma \left(\frac12 \Sigma^2 + \lambda I_d \right)^{-1/2} \left(I_d - (I_d + B)^{-1} \right) \left(\frac12 \Sigma^2 + \lambda I_d \right)^{-1/2} \btheta^\circ
    \\&\quad
    - \frac12 \Sigma \left( \frac12 (A^*)^\top A^* + \lambda I_d \right)^{-1} (A^*)^\top (A^* - \Sigma) \btheta^\circ \Bigg\|.
\end{align*}
Due to the triangle inequality and
\[
    \left\| I_d - (I_d + B)^{-1} \right\|
    = \left\| B (I_d + B)^{-1} \right\|
    \leq \frac{\|B\|}{1 - \|B\|}
    \leq \frac{4 \|B\|}3,
\]
it holds that
\begin{align*}
    \left\| \Sigma \left(\btheta^* - \btheta^\circ - \bb_\lambda \right) \right\|
    &
    \leq \frac{4 \lambda \|B\|}3 \left\| \Sigma \left(\frac12 \Sigma^2 + \lambda I_d \right)^{-1/2} \right\| \left\| \left(\frac12 \Sigma^2 + \lambda I_d \right)^{-1/2} \btheta^\circ \right\|
    \\&\quad
    + \frac12 \left\| \Sigma \left( \frac12 (A^*)^\top A^* + \lambda I_d \right)^{-1} (A^*)^\top (A^* - \Sigma) \btheta^\circ \right\|.
\end{align*}
An upper bound on the first term in the right-hand side immediately follows from the inequalities
\[
    \left\| \Sigma \left(\frac12 \Sigma^2 + \lambda I_d \right)^{-1/2} \right\| \leq \sqrt{2}
    \quad \text{and} \quad
    \left\| \left(\frac12 \Sigma^2 + \lambda I_d \right)^{-1/2} \btheta^\circ \right\|
    \leq \frac{\|\btheta^\circ\|}{\sqrt{\lambda}},
\]
while the second one does not exceed
\begin{align*}
    &
    \frac12 \left\| \Sigma \left( \frac12 (A^*)^\top A^* + \lambda I_d \right)^{-1} (A^*)^\top (A^* - \Sigma) \btheta^\circ \right\|
    \\&
    \leq \frac12 \left\| A^* \left( \frac12 (A^*)^\top A^* + \lambda I_d \right)^{-1} (A^*)^\top (A^* - \Sigma) \btheta^\circ \right\|
    \\&\quad
    + \frac12 \left\| (A^* - \Sigma) \left( \frac12 (A^*)^\top A^* + \lambda I_d \right)^{-1} (A^*)^\top (A^* - \Sigma) \btheta^\circ \right\|
    \\&
    \leq \left\|(A^* - \Sigma) \btheta^\circ \right\| + \frac{\|A^* - \Sigma\|^2 \|\btheta^\circ\|}{2 \lambda} \left\| \left( \frac12 (A^*)^\top A^* + \lambda I_d \right)^{-1} A^* \right\|
    \\&
    \leq \left\|(A^* - \Sigma) \btheta^\circ \right\| + \frac{\|A^* - \Sigma\|^2 \|\btheta^\circ\|}{\sqrt{2 \lambda}}
\end{align*}
Thus, we obtain that
\begin{align*}
    &
    \left\| \Sigma \left(\btheta^* - \btheta^\circ - \bb_\lambda \right) \right\|
    \leq \frac{4 \|B\| \|\btheta^\circ\| \sqrt{2 \lambda}}3 + \left\|(A^* - \Sigma) \btheta^\circ \right\| + \frac{\|A^* - \Sigma\|^2 \|\btheta^\circ\|}{2 \lambda}
    \\&
    \leq \frac{8 \|\btheta^\circ\| \sqrt{2 \lambda}}3 \left( \left(\frac{14 \|\btheta^\circ\|}\mu \right)^2 + \frac{35 \|\Sigma\| \|\btheta^\circ\| \|\bb_\lambda\|}{\mu^2 \sqrt{\lambda}} \right)
    + \|A^* - \Sigma\| \|\btheta^\circ\| \left(1 + \frac{\|A^* - \Sigma\|}{\sqrt{2 \lambda}} \right)
    \\&
    \leq \frac{8 \|\btheta^\circ\| \sqrt{2 \lambda}}3 \left( \left(\frac{14 \|\btheta^\circ\|}\mu \right)^2 + \frac{35 \|\Sigma\| \|\btheta^\circ\| \|\bb_\lambda\|}{\mu^2 \sqrt{\lambda}} \right)
    + \|A^* - \Sigma\| \|\btheta^\circ\| \left(1 + \frac{1}{8 \sqrt{2}} \right)
    \\&
    \leq \left(\frac{8 \sqrt{2}}3 + 1 + \frac1{8 \sqrt{2}} \right) \left( \left(\frac{14 \|\btheta^\circ\|}\mu \right)^2 + \frac{35 \|\Sigma\| \|\btheta^\circ\| \|\bb_\lambda\|}{\mu^2 \sqrt{\lambda}} \right) \|\btheta^\circ\| \sqrt{\lambda}
    \\&
    \leq 5 \|\btheta^\circ\| \sqrt{\lambda} \left( \left(\frac{14 \|\btheta^\circ\|}\mu \right)^2 + \frac{35 \|\Sigma\| \|\btheta^\circ\| \|\bb_\lambda\|}{\mu^2 \sqrt{\lambda}} \right).
\end{align*}
This and the inequalities \eqref{eq:a_theta_eta_triangle} and \eqref{eq:a_theta_eta_first_term} yield that
\[
    \left\| A^* \btheta^* - \bfeta^* - \frac12 \Sigma \bb_\lambda \right\|
    \leq 3 \|\btheta^\circ\| \sqrt{\lambda} \left( \left(\frac{14 \|\btheta^\circ\|}\mu \right)^2 + \frac{35 \|\Sigma\| \|\btheta^\circ\| \|\bb_\lambda\|}{\mu^2 \sqrt{\lambda}} \right).
\]
\myendproof

\section{Results from the proof of Theorem~\ref{theorem: bias}}

\subsection{Proof of Lemma~\ref{lemma: Schur complement bound for bias}}

Since the measure $\m$ is fixed, we omit the superscript $\m$ and simply write $\avg H$ instead of  $\biasInner{H}^{\,\m}$. Similarly, for a matrix-valued function $f$, we define
\begin{align}
    \label{eq:ut_def}
    \biasInner{f} = \int_0^1 f(\bu^{(t)}) \, \rmd \m(t),
    \quad \text{where} \quad
    \bu^{(t)} = (\btheta^{(t)}, \bfeta^{(t)}, A^{(t)}) = (1 - t) \bups^{\circ} + t \bups^*.
\end{align}
Our analysis heavily relies on the block-form representation of $H(\bups)$: 
\[
    H = H(\bups) =
    \begin{pmatrix}
        H_{\btheta\btheta} & H_{\btheta\bfeta} & H_{\btheta A} \\
        H_{\bfeta\btheta} & H_{\bfeta\bfeta} & H_{\bfeta A} \\
        H_{A\btheta} & H_{A\bfeta} & H_{AA}
    \end{pmatrix}.
\]
Here, for any $\bnu$ and $\bgamma$ from $\{\btheta, \bfeta, A\}$, the block $H_{\bnu \bgamma} = H_{\bnu \bgamma}(\bups)$ corresponds to the second partial derivative of $\ttL$ with respect to $\bnu$ and $\bgamma$. In particular, the diagonal blocks are equal to
\begin{align*}
    H_{\btheta\btheta} = A^\top A + \lambda I_d,
    \quad
    H_{\bfeta\bfeta} = 2 I_d,
    \quad \text{and} \quad
    H_{AA} = \diag \big(H_{\ba_1 \ba_1}, \dots, H_{\ba_d \ba_d}\big),
\end{align*}
while the off-diagonal ones are given by
\[
    H_{\btheta \bfeta} = H_{\bfeta \btheta}^\top = - A^\top,
    \quad
    H_{\btheta A} = H_{A \btheta}^\top = \begin{pmatrix} H_{\btheta\ba_1} & \dots & H_{\btheta\ba_d} \end{pmatrix},
    \quad
    H_{\bfeta A} = H_{A \bfeta}^\top = \begin{pmatrix} H_{\bfeta\ba_1} & \dots & H_{\bfeta\ba_d} \end{pmatrix},
\]
where, for any $j \in \{1, \dots, d\}$,
\begin{align*}
    H_{\ba_j \ba_j} = \mu^2 I_d + \btheta \btheta^\top,
    \quad
    H_{\btheta \ba_j} = (\ba_j^\top \btheta - \bfeta_j) I_d + \ba_j \btheta^\top
    \quad \text{and} \quad
    H_{\bfeta \ba_j} = - \be_j \btheta^\top.
\end{align*}
Since the proof of Lemma \ref{lemma: Schur complement bound for bias} is quite cumbersome, we split it into several steps.

\medskip

\noindent\textbf{Step 1. Decomposing the Schur complement.}
\quad
Let us recall that we are interested in properties of the Schur complement of $\avg H_{\chi \chi}$ corresponding to the nuisance parameter $\chi = (\bfeta, A)$:
\[
    \avg H_{\chi \chi}
    =
    \begin{pmatrix}
        \avg H_{\bfeta \bfeta} & \avg H_{\bfeta A} \\
        \avg H_{A \bfeta} & \avg H_{AA}
    \end{pmatrix}
    =
    \begin{pmatrix}
        \avg H_{\bfeta \bfeta} & \avg H_{\bfeta \ba_1} & \dots & \avg H_{\bfeta \ba_d} \\
        \avg H_{\ba_1 \bfeta} & \avg H_{\ba_1 \ba_1} & \dots & \avg H_{\ba_1 \ba_d} \\
        \vdots & \vdots & \ddots & \vdots \\
        \avg H_{\ba_d \bfeta} & \avg H_{\ba_d \ba_1} & \dots & \avg H_{\ba_d \ba_d}
    \end{pmatrix}.
\]
In what follows, we denote the inverse of $\avg H_{\chi \chi}$ by $J$:
\[
    \avg H_{\chi \chi}^{-1} = J =
    \begin{pmatrix}
        J_{\bfeta \bfeta} & J_{\bfeta A} \\
        J_{A \bfeta} & J_{AA}
    \end{pmatrix}
    =
    \begin{pmatrix}
        J_{\bfeta \bfeta} & J_{\bfeta \ba_1} & \dots & J_{\bfeta \ba_d} \\
        J_{\ba_1 \bfeta} & J_{\ba_1 \ba_1} & \dots & J_{\ba_1 \ba_d} \\
        \vdots & \vdots & \ddots & \vdots \\
        J_{\ba_d \bfeta} & J_{\ba_d \ba_1} & \dots & J_{\ba_d \ba_d}
    \end{pmatrix}.
\]
Then it holds that
\begin{align}
    \label{eq:schur_complement_decomposition}
    \left ( \avg{H} /  \avg{H}_{\chi \chi} \right )
    &\notag
    = \avg H_{\btheta \btheta} - \avg H_{\btheta \chi} J \avg H_{\chi \btheta}
    \\&
    = \avg{A^\top A} + \lambda I_d - \avg{H}_{\btheta \bfeta} J_{\bfeta \bfeta} \avg{H}_{\bfeta \btheta} - \sum_{j = 1}^d \avg H_{\btheta \bfeta} J_{\bfeta \ba_j} \avg H_{\ba_j \btheta}
    \\&\quad\notag
    - \sum_{j = 1}^d \avg H_{\btheta \ba_j} J_{\ba_j \bfeta} \avg H_{\bfeta \btheta} - \sum_{j = 1}^d \avg H_{\btheta \ba_j} J_{\ba_j \ba_j} \avg H_{\ba_j \btheta}.
\end{align}
We simplify the terms in the right-hand side one by one using Proposition \ref{proposition: inverse of nuisance parameters}, which ensures that $J$ admits a nice decomposition. First, according to Proposition \ref{proposition: inverse of nuisance parameters}, we have
\begin{align*}
    \avg{H}_{\btheta \bfeta} J_{\bfeta \bfeta} \avg{H}_{\bfeta \btheta} = \left (\frac{1}{2} + r_1 \right ) \avg{A}^\top \avg{A},
    \quad \text{where} \quad
    0 \leq r_1 \leq \rho^2.
\end{align*}
Next, using the identity
\[
    A = \sum\limits_{j = 1}^d \be_j \ba_j^\top,
\]
we observe that the same proposition yields that the third term in the right-hand side of \eqref{eq:schur_complement_decomposition} is equal to
\begin{align}
     \sum\limits_{j = 1}^d \avg H_{\btheta \bfeta} J_{\bfeta \ba_j} \avg H_{\ba_j \btheta} & = -  r_2 \langle A \rangle^\top \sum_{j = 1}^d \be_j \Tilde{\btheta}^\top \avg{(\ba_j^\top \btheta - \bfeta_j)} - r_2  \avg{A}^\top \sum_{j = 1}^d \be_j \Tilde{\btheta}^\top \avg{\btheta \ba_j^\top} \nonumber \\
     & = - r_2 \avg{A}^\top \avg{(A \btheta - \bfeta)} \Tilde{\btheta}^\top - r_2 \avg{A}^\top \avg{ A \Tilde{\btheta}^\top \btheta},
     \label{eq: Schur complement of the Hessian -- term 1}
\end{align}
where
\[
    \Tilde{\btheta} = \mu^2 \left( \mu^2 I_d + (\avg{\btheta \btheta^\top}) \right)^{-1}\avg{\btheta}.
\]
Let us denote the rightmost side of~\eqref{eq: Schur complement of the Hessian -- term 1} by $\ttR_{\eqref{eq: Schur complement of the Hessian -- term 1}}$. The sum
\[
    \sum\limits_{j = 1}^d \avg H_{\btheta \ba_j} J_{\ba_j \bfeta} \avg H_{\bfeta \btheta}
\]
is nothing but the transpose of the left-hand side of \eqref{eq: Schur complement of the Hessian -- term 1}. Hence, it remains simplify the last term in \eqref{eq:schur_complement_decomposition}:
\begin{align}
    \label{eq: Schur complement of the Hessian -- term 2}
    \sum_{j = 1}^d \avg H_{\btheta \ba_j} J_{\ba_j \ba_j} \avg H_{\ba_j \btheta}
    &
    = \sum_{j = 1}^d \avg{(\ba_j^\top \btheta - \bfeta_j)}^2 \, R_3 + \sum_{j = 1}^d \avg{(\ba_j^\top \btheta -\bfeta_j)} \, R_3 \, \avg{\btheta \ba_j^\top}
    \nonumber \\
    & \quad
    + \sum_{j = 1}^d \avg{\ba_j \btheta ^\top} \, R_3 \, \avg{(\ba_j^\top \btheta - \bfeta_j)} + \sum_{j = 1}^d \avg{\ba_j \btheta ^\top} \, R_3 \,  \avg{\btheta \ba_j^\top}
    \\& \notag
    = \Vert \avg{A \btheta - \bfeta} \Vert^2 \, R_3 + R_3 \, \avg{\btheta \avg{(A \btheta - \bfeta)}^\top A } \notag\\
    & \quad + \avg{A^\top \avg{(A \btheta - \bfeta)} \btheta^\top} \, R_3
    + \avg{A^\top \btheta^\top} \, R_3 \, \avg{\btheta A}.
\end{align}
We denote the latter sum by $\ttR_{\eqref{eq: Schur complement of the Hessian -- term 2}}$. Summing up \eqref{eq:schur_complement_decomposition}, \eqref{eq: Schur complement of the Hessian -- term 1}, and \eqref{eq: Schur complement of the Hessian -- term 2}, we obtain the following identity:
\begin{align}
    &\notag
    (\Sigma^2 / 2 + \lambda I_d)^{-1} \left(\avg{H} / \avg{H}_{\chi \chi} \right)
    \\&
    = (\Sigma^2 / 2 + \lambda I_d)^{-1} \left[ \left (  \avg{A^\top A} - \frac{1}{2} \avg{A}^\top \avg{A} + \lambda I_d \right ) - r_1 \avg{A}^\top \avg{A} - \ttR_{\eqref{eq: Schur complement of the Hessian -- term 1}} - \ttR_{\eqref{eq: Schur complement of the Hessian -- term 1}}^\top - \ttR_{\eqref{eq: Schur complement of the Hessian -- term 2}} \right]. \label{eq: Schur complement decomposition -- normalized}
\end{align}

\medskip

\noindent\textbf{Step 2. The leading term expansion.}
\quad
On this step, we focus on the leading term in the right-hand side of~\eqref{eq: Schur complement decomposition -- normalized}
\begin{equation}
    \label{eq:schur_complement_decomposition_leading_term}
    (\Sigma^2 / 2 + \lambda I_d)^{-1} \left( \avg{A^\top A} - \frac{1}{2} \avg{A}^\top \avg{A} + \lambda I_d \right)
\end{equation}
leaving analysis of the remainders for future. We are going to show that \eqref{eq:schur_complement_decomposition_leading_term}
approximately equals $I_d$. To do so, we use the following proposition.

\begin{Prop}
\label{proposition: pointwise bound on AiAj}
    {Assume that
    \begin{align*}
        \rho := \frac{7 \Vert \btheta^\circ \Vert}{\mu} \le 1/7 \quad \text{and} \quad 7 \rho \Vert \Sigma \Vert \le \sqrt{\lambda}.
    \end{align*}
    }
    Then the following holds. For any $s, t \in [0, 1]$, we have
    \begin{align*}
        \Vert A^{(s)} - \Sigma \Vert \le 4 \rho^2 \sqrt{\lambda} + 5 \rho \Vert \Sigma \Vert \Vert \bb_\lambda \Vert/\mu \quad \text{and} \quad \Vert (A^{(s)})^\top A^{(t)} - \Sigma^2 \Vert \le 3 \rho \lambda + 3 \Vert \Sigma \Vert \Vert \bb_\lambda \Vert \sqrt{\lambda} / \mu.
    \end{align*}
\end{Prop}

Applying Proposition \ref{proposition: pointwise bound on AiAj} and using the bound $\Vert (\Sigma^2/2 + \lambda I_d)^{-1} \Vert \le \lambda^{-1}$, we infer
\begin{align*}
    \left\Vert (\Sigma^2 / 2 + \lambda I_d)^{-1} \left (\avg{A^\top A} - \Sigma^2 \right ) \right\Vert
    &
    \le \int_0^1 \left \Vert (\Sigma^2 / 2 + \lambda I_d)^{-1} ((A^{(t)})^\top A^{(t)} - \Sigma^2) \right \Vert \rmd \m(t)
    \\
    &
    \le \left ( 3 \rho \lambda + \frac{3 \Vert \Sigma \Vert \Vert \bb_\lambda \Vert \sqrt{\lambda}}{\mu} \right )\Vert (\Sigma^2/2 + \lambda I_d)^{-1} \Vert \le 3 \rho + \frac{3 \Vert \Sigma \Vert \Vert \bb_\lambda \Vert}{\mu \sqrt{\lambda}}.
\end{align*}
Similarly, it holds that
\begin{align*}
    \left\Vert (\Sigma^2 / 2 + \lambda I_d)^{-1} (\biasInner{A}^\top \biasInner{A} - \Sigma^2) \right\Vert
    &
    \le \int_0^1 \int_0^1 \Vert (\Sigma^2 / 2 + \lambda I_d)^{-1} \Vert \Vert (A^{(s)})^\top A^{(t)} - \Sigma^2 \Vert \rmd \m^{\otimes 2}(s, t)
    \\
    &
    \le 3 \rho + \frac{3 \Vert \Sigma \Vert \Vert \bb_\lambda \Vert}{\mu \sqrt{\lambda}}.
\end{align*}
Substituting these bounds into \eqref{eq: Schur complement decomposition -- normalized}, we obtain that
\begin{align}
    \left\Vert (\Sigma^2 / 2 + \lambda I_d)^{-1} (\avg{H} / \avg{H}_{\chi \chi}) - I_d \right\Vert
    &
    = \left\Vert (\Sigma^2 / 2 + \lambda I_d)^{-1} \left[ (\avg{H} / \avg{H}_{\chi \chi}) -  (\Sigma^2/2 + \lambda I_d) \right] \right\Vert
    \nonumber \\
    &
    \le \frac{9\rho}2 + \frac{3 \Vert \Sigma \Vert \Vert \bb_\lambda \Vert}{2 \mu \sqrt{\lambda}} + r_1 \Vert (\Sigma^2 / 2 + \lambda I_d)^{-1} \Sigma^2 \Vert \nonumber \\
    & \quad + \left ( 3 \rho  + \frac{3 \Vert \Sigma \Vert \Vert \bb_\lambda \Vert}{\mu \sqrt{\lambda}}\right ) \cdot r_1 \nonumber  + \left\Vert (\Sigma^2 / 2 + \lambda I_d)^{-1} \, \ttR_{\eqref{eq: Schur complement of the Hessian -- term 1}} \right\Vert \nonumber \\
    & \quad + \left\Vert (\Sigma^2 / 2 + \lambda I_d)^{-1} \, \ttR_{\eqref{eq: Schur complement of the Hessian -- term 1}}^\top \Vert + \Vert (\Sigma^2 / 2 + \lambda I_d)^{-1}  \, \ttR_{\eqref{eq: Schur complement of the Hessian -- term 2}} \right\Vert \nonumber.
\end{align}
Let us recall that, according to Proposition \ref{proposition: inverse of nuisance parameters}, $r_1 \le \rho^2$. This inequality, together with the bound
\[
    \left\Vert (\Sigma^2/2 + \lambda I_d)^{-1} \Sigma^2 \right\Vert \le 2,
\]
yields that
\begin{align}
    \left\Vert (\Sigma^2 / 2 + \lambda I_d)^{-1} (\avg{H} / \avg{H}_{\chi \chi}) - I_d \right\Vert
    &
    \le 6 \rho + \frac{2 \Vert \Sigma \Vert \Vert \bb_\lambda \Vert}{\mu \sqrt{\lambda}} +  \left\Vert (\Sigma^2 / 2 + \lambda I_d)^{-1} \, \ttR_{\eqref{eq: Schur complement of the Hessian -- term 1}} \right\Vert
    \nonumber \\ 
    & \quad
    + \left\Vert (\Sigma^2 / 2 + \lambda I_d)^{-1} \, \ttR_{\eqref{eq: Schur complement of the Hessian -- term 1}}^\top \right\Vert  + \left\Vert (\Sigma^2 / 2 + \lambda I_d)^{-1} \, \ttR_{\eqref{eq: Schur complement of the Hessian -- term 2}} \right\Vert. \label{eq: normalized remainder in Schur complement}
\end{align}

\medskip

\noindent \textbf{Step 3. Bounding the remainder I.}
\quad
{We start bounding the term $\Vert (\Sigma^2 / 2 + \lambda I_d)^{-1} \, \ttR_{\eqref{eq: Schur complement of the Hessian -- term 1}} \Vert$ , and then proceed with $\Vert (\Sigma^2 / 2 + \lambda I_d)^{-1} \, \ttR_{\eqref{eq: Schur complement of the Hessian -- term 1}}^\top \Vert$.} We have
\begin{align}
    \left\Vert (\Sigma^2 / 2 + \lambda I_d)^{-1} \, \ttR_{\eqref{eq: Schur complement of the Hessian -- term 1}} \right\Vert
    &
    \le r_2 \left\Vert (\Sigma^2 / 2 + \lambda I_d)^{-1} \right\Vert \, \left\Vert \avg{A}^\top \avg{(A \btheta - \bfeta)} \Tilde{\btheta}^\top \right\Vert
    \nonumber \\
    & \quad
    + r_2 \left\Vert (\Sigma^2 / 2 + \lambda I_d)^{-1} \, \avg{A}^\top \avg{A \btheta^\top \Tilde{\btheta}} \right\Vert.
    \label{eq: non-simplified upper bound on normalized term 1}
\end{align}
We bound $r_2$ by $\mu^{-2}$. Next, we bound the first term of the expression above. We have
\begin{align*}
    r_2 \left\Vert (\Sigma^2 / 2 + \lambda I_d)^{-1} \right\Vert
    \,
    \left\Vert \avg{A}^\top \avg{(A \btheta - \bfeta)} \Tilde{\btheta}^\top \right \Vert
    \le \frac{\Vert \biasInner{A} \Vert \Vert \biasInner{A \btheta - \bfeta} \Vert \Vert \Tilde{\btheta} \Vert}{\mu^2 \lambda} \le \frac{\Vert \biasInner{A} \Vert \Vert \biasInner{A \btheta - \bfeta} \Vert \, \Vert \Tilde{\btheta} \Vert}{\mu^2 \lambda}.
\end{align*}
Since $(\btheta^{(t)}, \bfeta^{(t)}, A^{(t)}) \in \Theta^* \times \sfH^* \times \sfA^*$, we have $\Vert \biasInner{A} \Vert \le \Vert \Sigma \Vert + \sqrt{\lambda} / 3$ by the definition of $\sfA^*$, and $\Vert \biasInner{\btheta} \Vert \le \Vert \btheta^\circ \Vert \le \rho \mu$ since $\Vert \btheta^* \Vert \le \Vert \btheta^\circ \Vert$ by Lemma~\ref{lemma: localization lemma for bias}\ref{point: theta weak bound on bias, bias locating convex set}. To bound $\Vert \biasInner{A \btheta - \bfeta} \Vert$, we use the following proposition.

\begin{Prop}
\label{proposition: A theta - eta bias bound}
    {Assume that
    \begin{align*}
        \rho := \frac{7 \Vert \btheta^\circ \Vert}{\mu} \le 1 \quad \text{ and } \quad \Vert \Sigma \Vert \Vert \btheta^\circ \Vert \le \mu \sqrt{\lambda} / 24.
    \end{align*}}
    For any $t \in [0, 1]$, we have
    \begin{align*}
        \Vert A^{(t)} \btheta^{(t)} - \bfeta^{(t)} \Vert \le \rho \mu \sqrt{\lambda} / 2 + 3 \Vert \Sigma \Vert \Vert \bb_\lambda \Vert.
    \end{align*}
\end{Prop}

It implies
\begin{align}
    \label{eq: upper bound on the part i of term 1}
    r_2 \left\Vert (\Sigma^2 / 2 + \lambda I_d)^{-1} \right \Vert \cdot \left\Vert \avg{A}^\top \avg{(A \btheta - \bfeta)} \Tilde{\btheta}^\top \right\Vert
    &\notag
    \le \frac{\left( \Vert \Sigma \Vert + \sqrt{\lambda} / 3\right) \cdot (\rho \mu \sqrt{\lambda}/ 2 + \Vert \Sigma \Vert \Vert \bb_\lambda \Vert) \cdot \rho \mu}{\mu^2 \lambda}
    \\&\notag
    \le \frac{\rho^2 \Vert \Sigma \Vert}{2 \sqrt{\lambda}} + \rho^2/6 + \frac{\rho \Vert \Sigma \Vert^2 \Vert \bb_\lambda \Vert }{\mu \lambda} + \frac{\rho \Vert \Sigma \Vert \Vert \bb_\lambda \Vert}{\mu \sqrt{\lambda}}
    \\& \notag
    \le \rho / 14 + \rho^2/6 + \frac{\Vert \Sigma \Vert \Vert \bb_\lambda \Vert}{\mu\sqrt{\lambda}}
    \\& 
    \le \rho / 6 + \frac{\Vert \Sigma \Vert \Vert \bb_\lambda \Vert}{\mu \sqrt{\lambda}},
\end{align}
where we used the condition $7 \rho \Vert \Sigma \Vert \le \sqrt{\lambda}$.

Next, we bound $r_2 \left\Vert (\Sigma^2 / 2 + \lambda I_d)^{-1} \avg{A}^\top  \avg{A \btheta^\top \Tilde{\btheta}} \right\Vert$. We have
\begin{align*}
    &
    r_2 \left\Vert (\Sigma^2 / 2 + \lambda I_d)^{-1} \avg{A}^\top \avg{A \btheta^\top \Tilde{\btheta}} \right\Vert
    \\&
    \le \mu^{-2} \int_0^1 \int_0^1 \left\Vert (\Sigma^2/2 + \lambda I_d)^{-1} (A^{(s)})^\top A^{(t)} \cdot \Tilde{\btheta}^\top \btheta^{(t)} \right\Vert \, \rmd \m^{\otimes 2}(s, t)
    \\&
    \le \left\Vert (\Sigma^2/2 + \lambda I_d)^{-1} \Sigma^2 \right\Vert \cdot \mu^{-2} \int_0^1 |\Tilde{\btheta}^\top \btheta^{(t)} | \, \rmd \m(t)
    \\&
    \quad + \mu^{-2} \, \left\Vert (\Sigma^2 /2 + \lambda I_d)^{-1} \right\Vert \int_0^1 \int_0^1 \Vert (A^{s})^\top A^{(t)} - \Sigma^2 \Vert \cdot |\Tilde{\btheta}^\top \btheta^{(t)} | \, \rmd \m^{\otimes 2}(s, t).
\end{align*}
By the Cauchy-Schwarz inequality and bound $\Vert \Tilde{\btheta} \Vert\le \rho \mu$ from Proposition~\ref{proposition: inverse of nuisance parameters}, we have $\mu^{-2} |\Tilde{\btheta}^\top \btheta^{(t)} | \le \rho \mu^{-1} \Vert \btheta^{(t)} \Vert$. The latter is at most $\rho^2 / 7$, since $\Vert \btheta^{(t)} \Vert \le \max \{\Vert \btheta^* \Vert, \Vert \btheta^\circ \Vert\} \le \Vert \btheta^\circ \Vert \le \rho \mu / 7$ from Lemma~\ref{lemma: localization lemma for bias}\ref{point: theta weak bound on bias, bias locating convex set}. Bounding $\Vert (\Sigma^2/2 + \lambda I_d)^{-1} \Sigma^2 \Vert$ by $1$ and applying Proposition~\ref{proposition: pointwise bound on AiAj}, we obtain
\begin{align}
    r_2 \left\Vert (\Sigma^2 / 2 + \lambda I_d)^{-1} \avg{A}^\top  \avg{A \btheta^\top \Tilde{\btheta}} \right\Vert
    &
    \le \rho^2 + \rho^2 \cdot \left ( 3 \rho \lambda + \frac{3 \Vert \Sigma \Vert \Vert \bb_\lambda \Vert \sqrt{\lambda}}{\mu}\right ) \left\Vert (\Sigma^2/ 2 + \lambda I_d)^{-1} \right\Vert \nonumber \\
    & \le \rho^2 + 3 \rho^3 + \rho^2/4 \le 3 \rho, \label{eq: upper bound on the part ii of term 1}
\end{align}
where we used $\Vert \Sigma \Vert \Vert \bb_\lambda \Vert \le \Vert \Sigma \Vert \Vert \btheta^\circ \Vert \le \mu \sqrt{\lambda} / 24$ from~\eqref{eq: sigma theta circ bound}.
Substituting bounds~\eqref{eq: upper bound on the part i of term 1},~\eqref{eq: upper bound on the part ii of term 1} in~\eqref{eq: non-simplified upper bound on normalized term 1}, we get
\begin{align}
\label{eq: normalized bound Schur complement of the Hessian -- term 1}
    \Vert (\Sigma^2 / 2 + \lambda I_d)^{-1} \, \ttR_{\eqref{eq: Schur complement of the Hessian -- term 1}} \Vert \le 7 \rho / 2 + \frac{\Vert \Sigma \Vert \Vert \bb_\lambda \Vert}{\mu \sqrt{\lambda}}.
\end{align}
{
Then, we proceed with the term $\Vert (\Sigma^2 / 2 + \lambda I_d)^{-1} \, \ttR_{\eqref{eq: Schur complement of the Hessian -- term 1}}^\top \Vert$. We have
\begin{align*}
    \Vert (\Sigma^2 / 2 + \lambda I_d)^{-1} \, \ttR_{\eqref{eq: Schur complement of the Hessian -- term 1}}^\top \Vert & \le r_2 \Vert (\Sigma^2 / 2 + \lambda I_d)^{-1}  \Vert \Vert \avg{A}^\top \avg{A \btheta - \bfeta} \Tilde{\btheta}^\top \Vert \\
    & \quad + r_2 \Vert (\Sigma^2 / 2 + \lambda I_d)^{-1}  \avg{ \btheta^\top \Tilde{\btheta}  A^\top} \avg{A}^\top \Vert.
\end{align*}
The first term can be bounded by~\eqref{eq: upper bound on the part i of term 1}. We bound the second term using the fact that for each $s \in [0, 1]$, the quantity $(\btheta^{(s)})^\top \Tilde{\btheta}$ is scalar. Hence, we have
\begin{align*}
    \Vert (\Sigma^2 / 2 + \lambda I_d)^{-1}  \avg{ \btheta^\top \Tilde{\btheta}  A^\top} \avg{A} \Vert & = \left \Vert  \int_{0}^1 \int_0^1  \Tilde{\btheta}^\top \btheta^{(s)} \, (\Sigma^2 / 2 + \lambda I_d)^{-1} (A^{(s)})^\top A^{(t)} \rmd \m^{\otimes 2}(s, t)\right \Vert \\
    & \le \rho^2 \mu^2 / 7 \int_0^1 \int^1 \left \Vert (\Sigma^2 / 2 + \lambda I_d)^{-1} (A^{(s)})^\top A^{(t)} \right \Vert \rmd \m^{\otimes 2}(s, t),
\end{align*}
where we used $\Vert \btheta^{(s)} \Vert \le \max \{ \Vert \btheta^* \Vert,  \Vert \btheta^\circ \Vert\} \le \rho \mu / 7$ from Lemma~\ref{lemma: localization lemma for bias}\ref{point: theta weak bound on bias, bias locating convex set} and $\Vert \Tilde{\btheta} \Vert\le \rho \mu$ from Proposition~\ref{proposition: inverse of nuisance parameters}. Then, we bound $\left \Vert (\Sigma^2 / 2 + \lambda I_d)^{-1} (A^{(s)})^\top A^{(t)} \right \Vert$ as before:
\begin{align*}
    \left \Vert (\Sigma^2 / 2 + \lambda I_d)^{-1} (A^{(s)})^\top A^{(t)} \right \Vert & \le \left \Vert (\Sigma^2 / 2 + \lambda I_d)^{-1} \Sigma^2 \right \Vert + \left \Vert (\Sigma^2 / 2 + \lambda I_d)^{-1} \left ( (A^{(s)})^\top A^{(t)}  -\Sigma ^2 \right ) \right \Vert \\
    & \le 1 + \frac{1}{\lambda} \left ( 3 \rho \lambda + \frac{3 \Vert \Sigma \Vert \Vert \bb_\lambda \Vert \sqrt{\lambda}}{\mu} \right ) = 1 + 3 \rho + \frac{3 \Vert \Sigma \Vert \Vert \bb_\lambda \Vert}{\mu \sqrt{\lambda}},
\end{align*}
where we used $\Vert (\Sigma^2 / 2 + \lambda I_d)^{-1} \Vert \le \lambda^{-1}$ and Proposition~\ref{proposition: pointwise bound on AiAj} for the second inequality. Hence, we have
\begin{align*}
    \Vert (\Sigma^2 / 2 + \lambda I_d)^{-1} \, \ttR_{\eqref{eq: Schur complement of the Hessian -- term 1}}^\top \Vert \le \rho/6 + \frac{\Vert \Sigma \Vert \Vert \bb_\lambda \Vert}{\mu \sqrt{\lambda}} + \rho^2 / 7 + 3 \rho^3 / 7 + \frac{3 \rho^2\Vert \Sigma \Vert \Vert \bb_\lambda \Vert}{7 \mu \sqrt{\lambda}}.
\end{align*}
Bounding $\Vert \Sigma \Vert \Vert \bb_\lambda \Vert \le \Vert \Sigma \Vert \Vert \btheta^\circ \Vert \le \mu \sqrt{\lambda} / 24$ due to~\eqref{eq: sigma theta circ bound}, we derive
\begin{align}
\label{eq: normalized bound Schur complement of the Hessian -- term 1 transposed}
    \Vert (\Sigma^2 / 2 + \lambda I_d)^{-1} \, \ttR_{\eqref{eq: Schur complement of the Hessian -- term 1}}^\top \Vert \le 7 \rho / 2 + \frac{\Vert \Sigma \Vert \Vert \bb_\lambda \Vert}{\mu \sqrt{\lambda}}.
\end{align}
}

\medskip

\noindent \textbf{Step 4. Bounding the remainder II.}
\quad
Next, we analyze $\Vert (\Sigma^2/2 + \lambda I_d)^{-1} \, \ttR_{\eqref{eq: Schur complement of the Hessian -- term 2}\Vert}$. The first three terms of~\eqref{eq: Schur complement of the Hessian -- term 2}, can be bounded in the same way:
\begin{align*}
    &
    \left \Vert (\Sigma^2/2 + \lambda I_d)^{-1} \left \{ \Vert \avg{A \btheta - \bfeta} \Vert^2 \, R_3 + R_3 \, \avg{\btheta \avg{(A \btheta - \bfeta)}^\top A } + \avg{A^\top \avg{(A \btheta - \bfeta)} \btheta^\top} \, R_3 \right \} \right \Vert
    \\&
    \le \frac{1}{\lambda} \left \{ \left\Vert \biasInner{A \btheta - \bfeta} \right\Vert^2 \cdot \Vert R_3 \Vert  + 2 \Vert R_3 \Vert \int_0^1 \Vert \btheta^{(t)} \Vert \, \Vert \biasInner{A \btheta - \bfeta} \Vert \, \Vert A^{(t)} \Vert \, \rmd \m(t) \right \}.
\end{align*}
Due to Proposition~\ref{proposition: A theta - eta bias bound}, the quantity $\Vert \biasInner{A \btheta - \bfeta} \Vert$ is at most $\rho \mu \sqrt{\lambda}/2 + \Vert \Sigma \Vert \Vert \bb_\lambda \Vert$, and
\begin{align*}
    \Vert \biasInner{A \btheta - \bfeta} \Vert^2 \le \rho^2 \mu^2 \lambda / 4 + \rho \mu \sqrt{\lambda}\Vert \Sigma \Vert \Vert \bb_\lambda \Vert + \Vert \Sigma\Vert^2 \Vert \bb_\lambda \Vert^2.
\end{align*}
Next,  the norm $\Vert R_3 \Vert$ is at most $2 \mu^{-2}$ due to Proposition~\ref{proposition: inverse of nuisance parameters} and $\Vert \Sigma \Vert \Vert \bb_\lambda \Vert \le \Vert \Sigma \Vert \Vert \btheta^\circ \Vert \le \mu \sqrt{\lambda} / 24$ due to~\eqref{eq: sigma theta circ bound},  therefore, the above is at most
\begin{align*}
    \frac{1}{\lambda} \left \{ \rho^2 \lambda/2 + \frac{\Vert \Sigma \Vert \Vert \bb_\lambda \Vert \sqrt{\lambda}}{\mu} + \left( \rho \sqrt{\lambda} \mu^{-1} + \frac{2 \Vert \Sigma \Vert \Vert \bb_\lambda \Vert}{\mu^2} \right) \int_0^1 \Vert \btheta^{(t)} \Vert  \Vert A^{(t)} \Vert \rmd \m(t)\right \}
\end{align*}
Again, we use $A^{(t)} \in \sfA^*$, and bound $\Vert A^{(t)} \Vert \le \Vert \Sigma \Vert + \sqrt{\lambda} / 3$. Since $\Vert \btheta^* \Vert \le \Vert \btheta^\circ \Vert$ due to Lemma~\ref{lemma: localization lemma for bias}\ref{point: theta weak bound on bias, bias locating convex set}, $\Vert \btheta^{(t)} \Vert \le \rho \mu / 7$. Hence, the above is at most
\begin{align*}
    \rho^2 / 2 + \left ( \frac{\rho^2}{\sqrt{\lambda}} + \frac{2 \rho \Vert \Sigma \Vert \Vert \bb_\lambda \Vert}{\mu \lambda}\right ) \left \{ \Vert \Sigma \Vert + \sqrt{\lambda} / 3 \right \} & \le \rho^2/ 2 + \frac{\rho}{7} + \rho^2/3 + \frac{\Vert \Sigma \Vert \Vert \bb_\lambda \Vert}{7 \mu \sqrt{\lambda}} + \frac{2 \rho \Vert \Sigma \Vert \Vert \bb_\lambda \Vert}{\mu \sqrt{\lambda}} \\
    & \le \rho + \frac{\Vert \Sigma \Vert \Vert \bb_\lambda \Vert}{\mu \sqrt{\lambda}},
\end{align*}
where we used the condition $7 \rho \Vert \Sigma \Vert \le \sqrt{\lambda}$. Thus, the bound
\begin{align}
    &\notag
    \left \Vert (\Sigma^2/2 + \lambda I_d)^{-1} \left \{ \Vert \avg{A \btheta - \bfeta} \Vert^2 \, R_3 + R_3 \, \avg{\btheta \avg{(A \btheta - \bfeta)}^\top A } + \avg{A^\top \avg{(A \btheta - \bfeta)} \btheta^\top} \, R_3 \right \} \right \Vert
    \\&
    \le \rho + \frac{\Vert \Sigma \Vert \Vert \bb_\lambda \Vert}{\mu \sqrt{\lambda}} \label{eq: normalized Shur complement bound on part i term 2}
\end{align}
holds. Finally, we bound the norm of $(\Sigma^2 /2 + \lambda I_d)^{-1} \avg{A^\top \btheta^\top} \, R_3 \, \avg{\btheta A}$. Applying the triangle inequality, we derive
\begin{align*}
    \left\Vert (\Sigma^2 /2 + \lambda I_d)^{-1} \avg{A^\top \btheta^\top} \, R_3 \, \avg{\btheta A} \right\Vert
    &
    \le \left\Vert (\Sigma^2 /2 + \lambda I_d)^{-1} \Sigma \right\Vert \, \left( \int_0^1 \Vert \btheta^{(t)} \Vert \rmd \m(t) \right) \Vert R_3 \Vert \, \Vert \avg{\btheta A} \Vert
    \\
    & \quad
    + \Vert (\Sigma^2/ 2 + \lambda I_d)^{-1} \Vert \left ( \int_0^1 \Vert A^{(t)} - \Sigma \Vert \Vert \btheta^{(t)} \Vert \rmd \m(t) \right ) \Vert R_3 \Vert \Vert \biasInner{\btheta A} \Vert.
\end{align*}
As before, we use the definition of $\sfA^*$, to get bound on the norm of $A^{(t)} \in \sfA^*$ and $A^{(t)} - \Sigma$, and, as a consequence, obtain
\begin{align*}
    \Vert \biasInner{\btheta A} \Vert
    \le \int_0^1\Vert \btheta^{(t)} \Vert \Vert A^{(t)} \Vert \rmd \m(t)
    \le \rho \mu \left \{ \Vert \Sigma \Vert + \sqrt{\lambda} / 3 \right \}.
\end{align*}
Next, we have $\Vert R_3 \Vert \le 2 \mu^{-2}$ and $\Vert A^{(t)} - \Sigma \Vert \le \sqrt{\lambda} /3$, hence, it yields
\begin{align}
    \label{eq:last_remainder_intermediate_bound}
    \left\Vert (\Sigma^2 /2 + \lambda I_d)^{-1} \avg{A^\top \btheta^\top} \, R_3 \, \avg{\btheta A} \right\Vert
    &\notag
    \le 2 \rho^2 \left\Vert (\Sigma^2/2 + \lambda I_d)^{-1} \Sigma \right\Vert
    \left \{\Vert \Sigma \Vert + \sqrt{\lambda} / 3 \right \} \\
    & \quad
    + \frac{2 \rho^2 \sqrt{\lambda}}{3} \, \left\Vert (\Sigma^2/2 + \lambda I_d)^{-1} \right\Vert \left \{ \Vert \Sigma \Vert + \sqrt{\lambda} / 3 \right \}.
\end{align}
Let $\lambda_1(\Sigma) \geq \dots \geq \lambda_1(\Sigma)$ stand for the eigenvalues of $\Sigma$. Then, due to the Cauchy-Schwarz inequality, it holds that
\[
    \left\| (\Sigma^2/2 + \lambda I_d)^{-1} \Sigma \right\|
    = \max\limits_{1 \leq j \leq d} \left\{ \frac{\lambda_j(\Sigma)}{\lambda_j(\Sigma)^2/2 + \lambda} \right\}
    \leq \sqrt{\frac2{\lambda}}.
\]
Applying this bound and the inequality $\Vert (\Sigma^2/ 2 + \lambda I_d)^{-1} \Vert \le 1 /\lambda$ to \eqref{eq:last_remainder_intermediate_bound}, we deduce
\begin{align}
    \left\Vert (\Sigma^2 /2 + \lambda I_d)^{-1} \avg{A^\top \btheta^\top} \, R_3 \, \avg{\btheta A} \right\Vert
    \le \frac{3 \rho^2}{\sqrt{\lambda}} \left \{ \Vert \Sigma \Vert + \sqrt{\lambda} / 3 \right \}
    \le \frac{3 \rho}{7} + \rho^2 \le \rho.
    \label{eq: normalized Shur complement bound part ii term 2}
\end{align}
As before, we took into account that $7 \rho \Vert \Sigma \Vert \le \sqrt{\lambda}$ and $\rho \le 1/2$ due to the proposition conditions. Thus, \eqref{eq: normalized Shur complement bound on part i term 2} and \eqref{eq: normalized Shur complement bound part ii term 2} imply that
\begin{align}
    \Vert (\Sigma^2/2 + \lambda I_d)^{-1} \, \ttR_{\eqref{eq: Schur complement of the Hessian -- term 2}} \Vert \le 2 \rho + \frac{\Vert \Sigma \Vert \Vert \bb_\lambda \Vert}{\mu \sqrt{\lambda}}. \label{eq: normalized bound Schur complement of the Hessian -- term 2}
\end{align}
Substituting the bounds~\eqref{eq: normalized bound Schur complement of the Hessian -- term 1},~\eqref{eq: normalized bound Schur complement of the Hessian -- term 1 transposed}, and~\eqref{eq: normalized bound Schur complement of the Hessian -- term 2} into~\eqref{eq: normalized remainder in Schur complement}, we obtain the desired bound:
\[
    \left \Vert (\Sigma^2 / 2 + \lambda I_d)^{-1} (\avg{H}^{\,\m} / \avg{H}^{\,\m}_{\chi \chi}) - I_d \right \Vert
    \le 15 \rho + \frac{5 \Vert \Sigma \Vert \Vert \bb_\lambda \Vert}{\mu \sqrt{\lambda}}.
\]
The proof is finished.

\myendproof

\subsection{Proof of Proposition~\ref{proposition: pointwise bound on AiAj}}

Due to the definition of $A^{(t)}$ (see \eqref{eq:ut_def}), we clearly have $\Vert A^{(t)} - \Sigma \Vert = t \Vert A^* - \Sigma \Vert$ for all $t \in [0, 1]$. Therefore, according to Lemma \ref{lemma: Sigma bias optimal bound}, it holds that
\[
    \Vert A^{(t)} - \Sigma \Vert
    \leq \Vert  A^* - \Sigma \Vert
    \le 4 \rho^2 \sqrt{\lambda} + \frac{5 \rho \Vert \Sigma \Vert \Vert \bb_\lambda \Vert }{\mu}.
\]
It only remains to consider $\Vert (A^{(s)})^\top A^{(t)} - \Sigma^2 \Vert$. The triangle yields that
\begin{align*}
    \Vert (A^{(s)})^\top A^{(t)} - \Sigma^2 \Vert
    &
    = \Vert (A^{(s)} - \Sigma) A^{(t)} + \Sigma (A^{(t)} - \Sigma) \Vert
    \\&
    \le \Vert A^{(s)} - \Sigma \Vert \Vert A^{(t)} \Vert + \Vert \Sigma \Vert \Vert A^{(t)} - \Sigma \Vert
    \\&
    \leq \|A^* - \Sigma\| \left( \|A^*\| + \|\Sigma\| \right).
\end{align*}
Since, according to Lemma \ref{lemma: localization bias corollary} apllied with $\rho_0 = 1/7$, $A^*$ and $\Sigma$ belong to the set $\sfA^*$, we have $\Vert A^* - \Sigma \Vert \leq \sqrt{\lambda} / 3$ and then
\[
    \Vert \Sigma \Vert + \Vert A^* \Vert
    \leq 2 \Vert \Sigma \Vert + \Vert A^* - \Sigma \Vert
    \leq 2 \Vert \Sigma \Vert + \frac{\sqrt{\lambda}}3.
\]
This yields that
\begin{align*}
    \Vert (A^{(s)})^\top A^{(t)} - \Sigma^2 \Vert
    & = \|A^* - \Sigma\| \left( \|A^*\| + \|\Sigma\| \right) \\
    & \le \left ( 4 \rho^2 \sqrt{\lambda} + \frac{5 \rho \Vert \Sigma \Vert \Vert \bb_\lambda \Vert }{\mu} \right ) (2 \Vert \Sigma \Vert + \sqrt{\lambda}/3) \\
    & \le 8 \rho^2 \Vert \Sigma \Vert \sqrt{\lambda} + 4 \rho^2 \lambda / 3 + \frac{10 \rho \Vert \Sigma \Vert^2 \Vert \bb_\lambda \Vert}{\mu} + \frac{5 \rho \Vert \Sigma \Vert \Vert \bb_\lambda \Vert \sqrt{\lambda}}{3 \mu}.
\end{align*}
Since $7 \rho \Vert \Sigma \Vert \le \sqrt{\lambda}$ and $\rho \le 1/7$, the latter is at most $3 \rho \lambda + 3 \Vert \Sigma \Vert \Vert \bb_\lambda \Vert \sqrt{\lambda} / \mu$.
\myendproof

\subsection{Proof of Proposition~\ref{proposition: A theta - eta bias bound}}

Due to the triangle inequality, $\Vert A^{(t)} \btheta^{(t)} - \bfeta^{(t)} \Vert$ satisfies
\begin{align*}
    \Vert A^{(t)} \btheta^{(t)} - \bfeta^{(t)} \Vert & \le \Vert A^{(t)} - \Sigma \Vert \Vert \btheta^{(t)} \Vert + \Vert \Sigma \btheta^{(t)} - \Sigma \btheta^\circ \Vert + \Vert \Sigma \btheta^\circ - \bfeta^{(t)} \Vert \\
    & = t \Vert A^* - \Sigma \Vert \Vert \btheta^{(t)} \Vert + t \Vert \Sigma \btheta^* - \Sigma \btheta^\circ \Vert + t \Vert \Sigma \btheta^\circ - \bfeta^* \Vert.
\end{align*}
Since $\bfeta^* = (A^* \btheta^* + \Sigma \btheta^\circ) / 2$, it holds that
\begin{align*}
    \Vert A^{(t)} \btheta^{(t)} - \bfeta^{(t)} \Vert & \le \Vert A^* - \Sigma \Vert \Vert \btheta^{(t)} \Vert + \Vert \Sigma \Vert \Vert \btheta^* - \btheta^\circ \Vert + \frac{1}{2} \Vert \Sigma \btheta^\circ - A^* \btheta^* \Vert \\
    & \le \Vert A^* - \Sigma \Vert \left( \Vert \btheta^{(t)} \Vert + \frac{\Vert \btheta^* \Vert}2 \right) + \frac{3}{4} \Vert \Sigma \Vert \Vert \btheta^\circ - \btheta^* \Vert
\end{align*}
Let us recall that $\btheta^{(t)}$ is a convex combination of $\btheta^*$ and $\btheta^\circ$, and $\Vert \btheta^* \Vert \le \Vert \btheta^\circ \Vert$ in view of Lemma~\ref{lemma: localization lemma for bias}\ref{point: theta weak bound on bias, bias locating convex set}. Moreover, we have $\Vert \btheta^{(t)} \Vert \le \Vert \btheta^\circ \Vert = \rho \mu / 7$ due to the conditions of the proposition. Applying  Lemma~\ref{lemma: localization lemma for bias}\ref{point: theta weak bound on bias, bias locating convex set} to $\|\btheta^* - \btheta^\circ\|$, we obtain that
\begin{align*}
    \Vert A^{(t)} \btheta^{(t)} - \bfeta^{(t)} \Vert
    \le \frac{3 \rho \mu \Vert A^* - \Sigma \Vert}{14} + 3 \Vert \Sigma \Vert \Vert \bb_\lambda \Vert + 3 \rho \Vert \Sigma \Vert \Vert \btheta^\circ \Vert.
\end{align*}
Using the fact that $\Vert \Sigma \Vert \Vert \btheta^\circ \Vert \le \mu \sqrt{\lambda} / 24$ by the assumptions of Lemma~\ref{lemma: localization bias corollary}, we obtain that 
\[
    3 \rho \Vert \Sigma \Vert \Vert \btheta^\circ \Vert \le \rho \mu \sqrt{\lambda} / 8.
\]
Finally, the chain of the inequalities
\[  
    \Vert A^* - \Sigma \Vert \le \mu^{-1} \sqrt{\lambda} \Vert \btheta^\circ \Vert \le \rho \sqrt{\lambda} / 7
\]
following from Lemma~\ref{lemma: localization lemma for bias}, implies that
\begin{align*}
    \Vert A^{(t)} \btheta^{(t)} - \bfeta^{(t)} \Vert \le \rho \mu \sqrt{\lambda} / 2 + 3 \Vert \Sigma \Vert \Vert \bb_\lambda \Vert.
\end{align*}
\myendproof

\section{Results from the proof of Theorem \ref{th:stoch_term}}

This section contained proof of auxiliary results related to Theorem \ref{th:stoch_term}.

\subsection{Preliminaries}

We start with the following result playing an important role in the proof of Lemma \ref{lem:stoch_term_expansion}.

\begin{Lem}
    \label{lem:stoch_term_rough_bound}
    Let us fix any $\rho \in [0, 1/16]$ and introduce a block diagonal matrix
    \[
        \ttD_*^2
        = \diag \left( (A^*)^\top A^* + \lambda I_d, 2 I_d, I_d \otimes \big(\mu^2 I_d + \btheta^* (\btheta^*)^\top \big) \right).
    \]
    Consider an event $\cE$ such that $\bups^*$ and $\widehat \bups$ belong to a convex set $ \cU \subset \Upsilon(\rho)$ and $171 \|\ttD_*^{-1} \z\| \leq \mu \sqrt{\lambda}$
    on $\cE$. Then, on this event, it holds that
    \[
        \|\ttD_* (\widehat \bups - \bups^*)\| \leq 9 \|\ttD_*^{-1} \z\|.
    \]
\end{Lem}

\begin{proof}[Proof of Lemma \ref{lem:stoch_term_rough_bound}]

Let us denote $r = 9 \|\ttD_*^{-1} \z\|$ for brevity and show that $\|\ttD (\widehat \bups - \bups^*)\| \leq r$ on the event $\cE$. For this purpose, we fix an arbitrary vector $\bu \in \R^{2d + d^2}$ of the norm $\|\bu\| = r$ and consider a function
\[
    f(t) = \ttL(\bups^* + t \, \ttD_*^{-1} \bu),
    \quad t > 0.
\]
We are going to prove that $f'(t) \neq 0$ for any $t > 1$ such that $\bups^* + t \, \ttD^{-1} \bu \in \cU$. This yields the desired bound, because $\widehat \bups$ is a stationary point of $\ttL$.
Our approach is based on the analysis of the first and the second order derivatives:
\[
    f'(t) = \bnabla \ttL(\bups^* + t\, \ttD_*^{-1} \bu)^\top \ttD_*^{-1} \bu
    \quad \text{and} \quad
    f''(t) = \bu^\top \ttD_*^{-1} \nabla^2 \ttL(\bups^* + t\, \ttD_*^{-1} \bu) \ttD_*^{-1} \bu.
\]
According to Taylor's formula with the integral remainder term, it holds that
\begin{equation}
    \label{eq:f_taylors_expansion}
    f'(1) - f'(0) - f''(0)
    = \int\limits_0^1 \big( f''(s) - f''(0) \big) \rmd s. 
\end{equation}
Let us fix any $s \in (0, 1)$ and consider $\big( f''(s) - f''(0) \big)$.
Applying Lemma \ref{lem:second_derivative_lipschitzness} with $\lambda_0 = \lambda$ and $\mu_0 = \mu$, we obtain that
\begin{align*}
    \left| f''(s) - f''(0) \right|
    &
    = \left| \bu^\top \ttD_*^{-1} \big( \nabla^2 \ttL(\bups^* + ts \, \ttD_*^{-1} \bu) - \nabla^2 \ttL(\bups^*) \big) \ttD_*^{-1} \bu \right|
    \\&
    \leq \frac{4 s \, \|\bu\|^3}{\mu \sqrt{\lambda}} + \frac{2 s^2 \, \|\bu\|^4}{\mu^2 \lambda}
    \\&
    = \frac{4 s \, r^3}{\mu \sqrt{\lambda}} + \frac{2 s^2 \, r^4}{\mu^2 \lambda}.
\end{align*}
Substituting this bound into \eqref{eq:f_taylors_expansion}, we observe that
\[
    \left| f'(1) - f'(0) - f''(0) \right|
    \leq \frac{4 r^3}{\mu \sqrt{\lambda}} \int\limits_0^1 s \rmd s + \frac{2 r^4}{\mu^2 \lambda} \int\limits_0^1 s^2 \rmd s
    = \frac{2 r^3}{\mu \sqrt{\lambda}} \left(1 + \frac{r}{3 \mu \sqrt{\lambda}} \right).
\]
Since $r = 9 \|\ttD_*^{-1} \z\|$ and $171 \|\ttD_*^{-1} \z\| \leq \mu \sqrt{\lambda}$ on the event $\cE$, it holds that
\[
    \left| f'(1) - f'(0) - f''(0) \right|
    \leq \frac{2 r^2}{19} \left(1 + \frac{1}{57} \right)
    = \frac{r^2}{19} \cdot \frac{58}{57}.
\]
Note that Lemma \ref{lem:block-diagonal_lower_bound} yields
\[
    f''(0)
    = \bu^\top \ttD_*^{-1} \nabla^2 \ttL(\bups^*) \ttD_*^{-1} \bu
    \geq \frac{(1 - 2\rho) \|\bu\|^2}4
    = \frac{(1 - 2\rho) r^2}4.
\]
Taking into account this inequality and the relation
\[
    |f'(0)|
    = \left| \bnabla \ttL(\bups^*)^\top \ttD_*^{-1} \bu \right|
    = \left| \bz^\top \ttD_*^{-1} \bu \right|
    \leq \|\ttD_*^{-1} \bz\| \|\bu\| 
    \leq \frac{r^2}{9},
\]
we obtain that
\[
    f'(1)
    \geq \frac{(1 - 2\rho) r^2}4 - \frac{r^2}{10} - \frac{2 r^2}{19} \cdot \frac{58}{57}.
\]
It is straightforward to check that
\[
    \frac{2}{19} \cdot \frac{58}{57}
    = \frac{6 \cdot 58}{57^2}
    \leq \frac{6 \cdot 58}{57^2 - 1}
    = \frac{6}{56}
    = \frac3{28}.
\]
Thus, it holds that
\[
    f'(1)
    = \left( \frac14 - \frac19 - \frac3{28} - \frac{\rho}2 \right) r^2
    = \left( \frac5{36} - \frac3{28} - \frac{\rho}2 \right) r^2
    = \left( \frac2{63} - \frac{\rho}2 \right) r^2
    \geq \frac{(1 - 16 \rho) r^2}{32}
    \geq 0.
\]
    
Since $f''(t) > 0$ for any $t > 0$ such that $\bups^* + t \, \ttD_*^{-1} \bu \in \cU$, the Lagrange mean value theorem yields that $f'(t) \geq f'(1) > 0$ for any $t \geq 1$ satisfying $\bups^* + t \, \ttD_*^{-1} \bu \in \cU$.
Hence, we deduce that $\| \ttD_* (\widehat \bups - \bups^*)\| \leq r$.

\end{proof}

\subsection{Proof of Lemma \ref{lem:stoch_term_expansion}}
\label{sec:lem_stoch_term_expansion_proof}

For convenience, we split the proof into several steps.

\medskip

\noindent
\textbf{Step 1: Taylor's expansion.}
\quad
Let us remind a reader that $H_* = \nabla^2 \cL(\bups^*) = \nabla^2 \ttL(\bups^*)$.
Due to the definition of $\widehat \bups$, it holds that $\bnabla \ttL (\widehat \bups) = \bzero$. Using the Newton-Leibniz rule, we obtain that
\begin{align*}
    \bnabla \ttL(\bups^* + H_*^{-1} \bz)
    &
    = \bnabla \ttL(\bups^* + H_*^{-1} \bz) - \bnabla \ttL(\widehat \bups)
    \\&
    = \int\limits_0^1 \nabla^2 \ttL \big(t \widehat\bups + (1 - t)(\bups^* + H_*^{-1} \bz \big) (\bups^* - \widehat\bups + H_*^{-1} \bz) \rmd t
    \\&
    = \int\limits_0^1 \nabla^2 \ttL \big(\bups^* + t (\widehat\bups - \bups^*) + (1 - t) H_*^{-1} \bz \big) (\bups^* - \widehat\bups + H_*^{-1} \bz) \rmd t.
\end{align*}
Thus, it holds that
\[
    \widehat\bups - \bups^* - H_*^{-1} \bz
    = -\left( H_* + R \right)^{-1} \bnabla \ttL(\bups^* + H_*^{-1} \bz),
\]
where we introduced
\[
    R = \int\limits_0^1 \left[ \nabla^2 \ttL \big(\bups^* + t (\widehat\bups - \bups^*) + (1 - t) H_*^{-1} \bz \big) - H_* \right] \rmd t.
\]
The representation
\[
    -(H_* + R)^{-1}
    = -H_*^{-1} H_* (H_* + R)^{-1}
    = -H_*^{-1} + H_*^{-1} R (H_* + R)^{-1}
\]
yields the equality
\[
    \widehat\bups - \bups^* - H_*^{-1} \bz
    = -H_*^{-1} \bnabla \ttL(\bups^* + H_*^{-1} \bz) + H_*^{-1} R (H_* + R)^{-1} \bnabla \ttL(\bups^* + H_*^{-1} \bz).
\]
Then it holds that
\begin{align*}
    \left\| S \big( \widehat\bups - \bups^* - H_*^{-1} \bz \big) \right\|
    &
    \leq \left\| S H_*^{-1} \bnabla \ttL(\bups^* + H_*^{-1} \bz) \right\| + \left\| S H_*^{-1} R (H_* + R)^{-1} \bnabla \ttL(\bups^* + H_*^{-1} \bz) \right\|
    \\&
    \leq  \left\| S H_*^{-1} \ttD_0 \right\| \left\| \ttD_0^{-1} \bnabla \ttL(\bups^* + H_*^{-1} \bz) \right\|
    \\&\quad
    + \left\| S H_*^{-1} \ttD_* \right\| \left\| \ttD_*^{-1} R (H_* + R)^{-1} \ttD_* \right\| \left\| \ttD_*^{-1} \bnabla \ttL(\bups^* + H_*^{-1} \bz) \right\|.
\end{align*}
Let us recall that
\[
    \ttD_*^2
    = \diag\left( (A^*)^\top A^* + \lambda I_d, 2 I_d, I_d \otimes \big(\mu^2 I_d + \btheta^* (\btheta^*)^\top \big) \right)
\]
and that $\mu_0$ does not exceed $\mu$ by the definition. Then it is straightforward to observe that $\ttD_*^2 \succeq \ttD_0^2$ and $\ttD_*^{-2} \preceq \ttD_0^{-2}$. 
Taking these inequalities into account, we obtain that
\begin{align}
    \label{eq:variance_remainder_triangle_inequality}
    \left\| S \big( \widehat\bups - \bups^* - H_*^{-1} \bz \big) \right\|
    &
    \leq  \left\| S H_*^{-1} \ttD_* \right\| \left\| \ttD_0^{-1} \bnabla \ttL(\bups^* + H_*^{-1} \bz) \right\| \left(1 + \left\| \ttD_*^{-1} R (H_* + R)^{-1} \ttD_* \right\| \right)
\end{align}
In the rest of the proof, we bound three terms in the right-hand side of \eqref{eq:variance_remainder_triangle_inequality} one by one.

\medskip

\noindent
\textbf{Step 2: upper bound on $\| \ttD_0^{-1} \bnabla \ttL(\bups^* + H_*^{-1} \bz) \|$.}
\quad
We start with an upper bound on the norm of $\ttD_0^{-1} \bnabla \ttL(\bups^* + H_*^{-1} \bz)$. Note that
\begin{align*}
    \ttD_0^{-1} \bnabla \ttL(\bups^* + H_*^{-1} \bz)
    &
    = \ttD_0^{-1} \bnabla \cL(\bups^* + H_*^{-1} \bz) - \ttD_0^{-1} \bz
    \\&
    = \ttD_0^{-1} \left( \bnabla \cL(\bups^* + H_*^{-1} \bz) - \bnabla \cL(\bups^*) \right) - \ttD_0^{-1} \bz.
\end{align*}
Using Taylor's expansion with the integral remainder term, we obtain that
\begin{align*}
    \left\| \ttD_0^{-1} \bnabla \ttL(\bups^* + H_*^{-1} \bz) \right\|
    &
    = \left\| - \ttD_0^{-1} \bz + \ttD_0^{-1} \int\limits_0^1 \bnabla^2 \cL(\bups^* + t \, H_*^{-1} \bz) \, H_*^{-1} \bz \, \rmd t \right\|
    \\&
    = \left\| \int\limits_0^1 \ttD_0^{-1} \left( \bnabla^2 \cL(\bups^* + t \, H_*^{-1} \bz) - H_* \right) \ttD_0^{-1} \, \ttD_0 H_*^{-1} \bz \, \rmd t \right\|
    \\&
    \leq \left\| \ttD_0 H_*^{-1} \bz \right\| \int\limits_0^1 \left\| \ttD_0^{-1} \left( \bnabla^2 \cL(\bups^* + t \, H_*^{-1} \bz) - H_* \right) \ttD_0^{-1} \right\| \, \rmd t.
\end{align*}
Lemma \ref{lem:second_derivative_lipschitzness} implies that
\[
    \left\| \ttD_0^{-1} \left( \bnabla^2 \cL(\bups^* + t \, H_*^{-1} \bz) - H_* \right) \ttD_0^{-1} \right\|
    \leq \frac{4 t}{\mu_0 \sqrt{\lambda}} \, \| \ttD_0 H_*^{-1} \bz \| + \frac{2 t^2}{\mu_0^2 \lambda} \| \ttD_0 H_*^{-1} \bz \|^2,
\]
Hence, it holds that
\begin{align}
    \label{eq:l_grad_upper_bound}
    \left\| \ttD_0^{-1} \bnabla \ttL(\bups^* + H_*^{-1} \bz) \right\|
    &\notag
    \leq \frac{4}{\mu_0 \sqrt{\lambda}} \, \| \ttD_0 H_*^{-1} \bz \|^2 \int\limits_0^1 t \rmd t + \frac{2}{\mu_0^2 \lambda} \| \ttD_0 H_*^{-1} \bz \|^3 \int\limits_0^1 t^2 \rmd t
    \\&
    = \frac{2}{\mu_0 \sqrt{\lambda}} \, \| \ttD_0 H_*^{-1} \bz \|^2 \left(1 + \frac{\| \ttD_0 H_*^{-1} \bz \|}{3 \mu_0 \sqrt{\lambda}} \right).
\end{align}

\medskip

\noindent
\textbf{Step 3: upper bound on $\|\ttD_*^{-1} R (H_* + R)^{-1} \ttD_* \|$.}
\quad
We proceed with an upper bound on the norm of $\ttD_*^{-1} R (H_* + R)^{-1} \ttD_*$. Representing the matrix of interest as the product of $\ttD_*^{-1} R \, \ttD_*^{-1}$ and $\ttD_* (H_* + R)^{-1} \ttD_*$, we obtain that
\[
    \left\|\ttD_*^{-1} R (H_* + R)^{-1} \ttD_* \right\|
    \leq \left\|\ttD_*^{-1} R \, \ttD_*^{-1} \right\|  \left\| \ttD_* (H_* + R)^{-1} \ttD_* \right\|.
\]
Due to the Jensen inequality, the first term in the right-hand side does not exceed
\begin{align*}
    \left\|\ttD_*^{-1} R \, \ttD_*^{-1} \right\|
    &
    = \left\| \int\limits_0^1 \ttD_*^{-1}\left[ \nabla^2 \ttL \big(\bups^* + t (\widehat\bups - \bups^*) + (1 - t) H_*^{-1} \bz \big) - H_* \right] \ttD_*^{-1} \, \rmd t \right\|
    \\&
    \leq \int\limits_0^1 \left\| \ttD_*^{-1}\left[ \nabla^2 \ttL \big(\bups^* + t (\widehat\bups - \bups^*) + (1 - t) H_*^{-1} \bz \big) - H_* \right] \ttD_*^{-1} \right\| \, \rmd t.
\end{align*}
Next, we can bound the expression under the integral using Lemmata \ref{lem:stoch_term_rough_bound}, \ref{lem:block-diagonal_lower_bound}, and \ref{lem:second_derivative_lipschitzness}. Indeed, in view of Lemma \ref{lem:second_derivative_lipschitzness}, for any $t \in [0, 1]$, we have
\begin{align*}
    &
    \left\| \ttD_*^{-1}\left( \nabla^2 \ttL \big(\bups^* + t (\widehat\bups - \bups^*) + (1 - t) H_*^{-1} \bz \big) - H_* \right) \ttD_*^{-1} \right\|
    \\&
    \leq \frac{4}{\mu \sqrt{\lambda}} \, \left\| t \, \ttD_* (\widehat\bups - \bups^*) + (1 - t) \ttD_* H_*^{-1} \bz \right\| + \frac{2}{\mu^2 \lambda} \left\| t \, \ttD_* (\widehat\bups - \bups^*) + (1 - t) \, \ttD_* H_*^{-1} \bz \right\|^2.
\end{align*}
Lemma \ref{lem:block-diagonal_lower_bound} implies that 
\[
    \left\| \ttD_* H_*^{-1} \ttD_* \right\| \leq \frac4{1 - 2\rho_0}
    \quad \text{and} \quad
    \|\ttD_* H_*^{-1} \z\| \leq \frac{4}{1 - 2\rho_0} \|\ttD_*^{-1} \z\|.
\]
Combining this bound with the result of
Lemma \ref{lem:stoch_term_rough_bound}, we get that
\begin{align*}
    \left\| t \, \ttD_* (\widehat\bups - \bups^*) + (1 - t) \ttD_* H_*^{-1} \bz \right\|
    &
    \leq t \left\| \ttD_* (\widehat\bups - \bups^*) \right\| + (1 - t) \left\| \ttD_* H_*^{-1} \bz \right\|
    \\&
    \leq 9t \|\ttD_*^{-1} \z\| + \frac{4 (1 - t)}{1 - 2\rho_0} \|\ttD_*^{-1} \z\|
    \\&
    \leq 9 \|\ttD_*^{-1} \z\|.
\end{align*}
The last inequality follows from the fact that $\rho \in [0, 1/16]$.
Similarly, it holds that
\begin{align*}
    \left\| t \, \ttD_* (\widehat\bups - \bups^*) + (1 - t) \ttD_* H_*^{-1} \bz \right\|^2
    \leq t \left\| \ttD_* (\widehat\bups - \bups^*) \right\|^2 + (1 - t) \left\| \ttD_* H_*^{-1} \bz \right\|^2
    \leq 81 \|\ttD_*^{-1} \z\|^2.
\end{align*}
Hence, we obtain that
\begin{equation}
    \label{eq:r_upper_bound}
    \left\|\ttD_*^{-1} R \, \ttD_*^{-1} \right\|
    \leq \frac{36}{\mu \sqrt{\lambda}} \, \|\ttD_*^{-1} \z\| \left(1 + \frac{9 \|\ttD_*^{-1} \z\|}{2 \mu \sqrt{\lambda}} \right)
    \leq \frac{36}{\mu \sqrt{\lambda}} \, \|\ttD_*^{-1} \z\| \left(1 + \frac{1}{36} \right)
    = \frac{37}{\mu \sqrt{\lambda}} \, \|\ttD_*^{-1} \z\|.
\end{equation}
The inequality \eqref{eq:r_upper_bound} and Lemma \ref{lem:block-diagonal_lower_bound} immediately yield an upper bound on $\|\ttD_* (H_* + R)^{-1} \ttD_*\|$. Indeed, the smallest eigenvalue of $\ttD_*^{-1} (H_* + R) \ttD_*^{-1}$ is at least
\[
    \frac{1 - 2\rho_0}4 - \left\|\ttD_*^{-1} R \, \ttD^{-1} \right\|
    \geq \frac14 - \frac{\rho_0}2 - \frac{37}{\mu \sqrt{\lambda}} \, \|\ttD_*^{-1} \z\|
    \geq \frac14 - \frac{1}{32} - \frac{3}{32}
    = \frac1{8}.
\]
Thus,
\[
    \|\ttD_* (H_* + R)^{-1} \ttD_*\| \leq 8
\]
and, as a consequence, we have
\begin{equation}
    \label{eq:r_h_plus_r_upper_bound}
    \left\|\ttD_*^{-1} R (H_* + R)^{-1} \ttD_* \right\|
    \leq 8 \cdot \frac{37}{\mu \sqrt{\lambda}} \, \|\ttD_*^{-1} \z\|
    = \frac{296}{\mu \sqrt{\lambda}} \, \|\ttD_*^{-1} \z\|.
\end{equation}

\medskip

\noindent
\textbf{Step 4: upper bound on $\|S H_*^{-1} \ttD_*\|$.}
\quad
It only remains to bound $\|S H_*^{-1} \ttD_*\|$ to finish the proof. According to Lemma \ref{lem:block-diagonal_lower_bound}, $4 H_* \succeq (1 - 2\rho_0) \ttD_*^2$. This yields that $(1 - 2\rho_0) H_*^{-1} \preceq 4 \ttD_*^{-2}$ and then
\[
    \|S H_*^{-1} \ttD_*\|
    = \|S \, \ttD_*^{-1} \, \ttD_* H_*^{-1} \ttD_*\|
    \leq \|S \, \ttD_*^{-1}\| \| \ttD_* H_*^{-1} \ttD_*\|
    \leq \frac4{1 - 2\rho_0} \|S \, \ttD_*^{-1}\|.
\]
Let us recall that
\[
    S = \diag\big( \Sigma^{1/2}, I_{d+1} \otimes O_d \big)
    \quad \text{and} \quad
    \ttD_*^2
    = \diag\left( (A^*)^\top A^* + \lambda I_d, 2 I_d, I_d \otimes \big(\mu^2 I_d + \btheta^* (\btheta^*)^\top \big) \right).
\]
Hence, it holds that
\begin{align*}
    \|S \, \ttD_*^{-1}\|
    &
    = \left\| \Sigma^{1/2} \left( (A^*)^\top A^* + \lambda I_d \right)^{-1/2} \right\|
    = \left\| \left( (A^*)^\top A^* + \lambda I_d \right)^{-1/2} \Sigma \left( (A^*)^\top A^* + \lambda I_d \right)^{-1/2} \right\|^{1/2}
    \\&
    \leq \left\| \left( (A^*)^\top A^* + \lambda I_d \right)^{-1/2} A^* \left( (A^*)^\top A^* + \lambda I_d \right)^{-1/2} \right\|^{1/2}
    \\&\quad
    + \left\| \left( (A^*)^\top A^* + \lambda I_d \right)^{-1/2} (\Sigma - A^*) \left( (A^*)^\top A^* + \lambda I_d \right)^{-1/2} \right\|^{1/2}
    \\&
    \leq \left\| \left( (A^*)^\top A^* + \lambda I_d \right)^{-1/2} \right\|^{1/2} + \sqrt{\frac{\left\| \Sigma - A^*\right\|}{\lambda}}
    \leq \lambda^{-1/4} + \sqrt{\frac{\left\| \Sigma - A^*\right\|}{\lambda}}
\end{align*}
and then
\begin{equation}
    \label{eq:s_hessian_inverse_upper_bound}
    \|S H_*^{-1} \ttD_*\|
    \leq \frac{4}{1 - 2\rho_0} \left( \lambda^{-1/4} + \sqrt{\frac{\left\| \Sigma - A^*\right\|}{\lambda}} \right).
\end{equation}

\medskip

\noindent
\textbf{Step 5: final bound.}
\quad
Summing up the inequalities \eqref{eq:variance_remainder_triangle_inequality}, \eqref{eq:l_grad_upper_bound}, \eqref{eq:r_h_plus_r_upper_bound}, and \eqref{eq:s_hessian_inverse_upper_bound}, we obtain that
\begin{align*}
    \left\| S \big( \widehat\bups - \bups^* - H^{-1} \bz \big) \right\|
    &
    \leq \frac{4}{1 - 2\rho_0} \left( \lambda^{-1/4} + \sqrt{\frac{\left\| \Sigma - A^*\right\|}{\lambda}} \right) \cdot \frac{2}{\mu_0 \sqrt{\lambda}} \, \| \ttD_0 H_*^{-1} \bz \|^2
    \\&\quad
    \cdot \left(1 + \frac{\| \ttD_0 H_*^{-1} \bz \|}{3 \mu_0 \sqrt{\lambda}} \right) \left(1 + \frac{296}{\mu \sqrt{\lambda}} \, \|\ttD_*^{-1} \z\| \right).
\end{align*}
Since, due to the conditions of Lemma \ref{lem:stoch_term_expansion}, $395 \|\ttD_*^{-1} \z\| \leq \mu \sqrt{\lambda}$ and $\rho \leq 1/16$, the expression in the right-hand side simplifies to
\begin{align*}
    \left\| S \big( \widehat\bups - \bups^* - H_*^{-1} \bz \big) \right\|
    &
    \leq \frac{8\| \ttD_0 H_*^{-1} \bz \|^2}{(1 - 1/8) \mu_0 \sqrt{\lambda}} \left( \lambda^{-1/4} + \sqrt{\frac{\left\| \Sigma - A^*\right\|}{\lambda}} \right) 
    \left(1 + \frac{\| \ttD_0 H_*^{-1} \bz \|}{3 \mu_0 \sqrt{\lambda}} \right) \left(1 + \frac{296}{395}\right)
    \\&
    \leq \frac{16}{\mu_0 \sqrt{\lambda}} \, \| \ttD_0 H_*^{-1} \bz \|^2 \left( \lambda^{-1/4} + \sqrt{\frac{\left\| \Sigma - A^*\right\|}{\lambda}} \right) \left(1 + \frac{\| \ttD_0 H_*^{-1} \bz \|}{3 \mu_0 \sqrt{\lambda}} \right).
\end{align*}
In the last line, we used the fact that $64 / 7 \cdot (1 + 296/395) < 16$. The proof is finished.

\myendproof

\subsection{Proof of Lemma \ref{lem:standardized_score_theta_bound}}
\label{sec:lem_standardized_score_theta_bound_proof}

The main idea of the proof of Lemma \ref{lem:standardized_score_theta_bound} is to represent $\breve H_{\btheta \btheta}^{-1} H_{\btheta \chi} H_{\chi \chi}^{-1} \bz_\chi$ as a linear transform of $(\bZ - \E \bZ)$ and $(\widehat \Sigma - \Sigma)$ and then apply the large deviation inequalities from Theorems \ref{th:covariance_concentration} and \ref{th:noise_concentration}. Despite the simple idea, the proof is quite technical, so we split it into several steps for convenience.

\medskip

\noindent\textbf{Step 1.}\quad We start with an observation that the matrix $\breve H_{\btheta \btheta}$ is close to $(A^\top A / 2 + \lambda I_d)$.

\begin{Lem}
    \label{lem:h_theta_schur}
    Assume that $\rho \leq 1/5$. Let us fix an arbitrary $\bups \in \Upsilon(\rho)$ and let $\breve H_{\btheta \btheta} = H_{\btheta \btheta} - H_{\btheta \chi} H_{\chi \chi}^{-1} H_{\chi \btheta}$ stand for the Schur complement of $H_{\chi\chi}$. Then it holds that
    \[
        \left\| \left( \frac12 A^\top A + \lambda I_d \right)^{-1/2} \breve H_{\btheta \btheta} \left( \frac12 A^\top A + \lambda I_d \right)^{-1/2} - I_d \right\|
        \leq 10 \rho^2.
    \]
\end{Lem}
The proof of Lemma \ref{lem:h_theta_schur} is postponed to Appendix \ref{sec:lem_h_theta_schur_proof}. Lemma \ref{lem:h_theta_schur} yields that the smallest eigenvalue of $( A^\top A/2 + \lambda I_d )^{-1/2} \breve H_{\btheta \btheta} ( A^\top A/2 + \lambda I_d )^{1/2}$
is at least $1 - 11\rho^2$. Then we deduce that the operator norm of
\[
    \left( \frac12 A^\top A + \lambda I_d \right)^{1/2} \breve H_{\btheta \btheta}^{-1} \left( \frac12 A^\top A + \lambda I_d \right)^{1/2}
\]
does not exceed $(1 - 10 \rho^2)^{-1}$ and then
\[
    \left\| \left( \frac12 A^\top A + \lambda I_d \right)^{1/2} \breve H_{\btheta \btheta}^{-1} H_{\btheta \chi} H_{\chi \chi}^{-1} \bz_\chi \right\|
    \leq \frac1{1 - 10\rho^2} \left\| \left( \frac12 A^\top A + \lambda I_d \right)^{-1/2} H_{\btheta \chi} H_{\chi \chi}^{-1} \bz_\chi \right\|.
\]
Thus, we can restrict our attention on a bit simpler term
\[
    \left( \frac12 A^\top A + \lambda I_d \right)^{-1/2} H_{\btheta \chi} H_{\chi \chi}^{-1} \bz_\chi.
\]

\medskip

\noindent\textbf{Step 2: explicit expression for $H_{\btheta \chi} H_{\chi \chi}^{-1} \bz_\chi$.}\quad
Let us elaborate on $H_{\btheta \chi} H_{\chi \chi}^{-1} \bz_\chi$. The next lemma shows that $H_{\btheta \chi} H_{\chi \chi}^{-1} \bz_\chi$ can be represented as a sum of linear transforms of  $(\bZ - \E \bZ)$ and $(\widehat \Sigma - \Sigma)$.

\begin{Lem}
    \label{lem:semiparametric_score_explicit_representation}
    Let $\rho \leq 1/2$ and let us fix any $\bups \in \Ups(\rho)$. With the introduced notations, it holds that
    \begin{align*}
        H_{\btheta \chi} H_{\chi \chi}^{-1} \z_\chi
        &
        = -\frac12 A^\top \bU + \frac12 A^\top (\widehat \Sigma - \Sigma)(\btheta - \btheta^\circ) + (\widehat \Sigma - \Sigma)(A \btheta - \bfeta)
        \\&\quad
        - \left(r_1 - \frac{r_2 \mu^2 \|\btheta\|^2}{\mu^2 + \|\btheta\|^2} \right) A^\top (\bZ - \E \bZ) + \frac{r_2 \mu^2}{\mu^2 + \|\btheta\|^2} \left( (A \btheta - \bfeta)^\top (\bZ - \E \bZ) \right) \btheta
        \\&\quad
        - \left(\frac{\|\btheta\|^2}{2\mu^2 + \|\btheta\|^2} - \frac{\|\btheta\|^2}{2 (\mu^2 + \|\btheta\|^2)} + \frac{r_2' \mu^4}{\mu^2 + \|\btheta\|^2} \right) A^\top (\widehat \Sigma - \Sigma) \btheta
        \\&\quad
        - \frac{1}{2\mu^2 + \|\btheta\|^2} \big( (A \btheta - \bfeta)^\top (\widehat \Sigma - \Sigma) \btheta \big) \btheta,
    \end{align*}
    where
    \[
        \bU
        = \bZ - \E \bZ - (\widehat \Sigma - \Sigma) \btheta^\circ
        = \frac1n \, \bbX \beps
    \]
    and the constants $r_1 \in [0, \rho^2]$, $r_2 = 0.5 \mu^{-2} + r_2'$, $r_2' \in [0, \rho^2 / \mu^{2}]$ are defined in Proposition \ref{proposition: inverse of nuisance parameters}.
\end{Lem}
The proof of Lemma \ref{lem:semiparametric_score_explicit_representation} is moved to Appendix \ref{sec:lem_semiparametric_score_explicit_representation}. 
The inequalities $0 \leq r_1 \leq \rho^2$ and $0 \leq r_2' \leq \rho^2 / \mu^2$ yield that
\[
    \left| r_1 - \frac{r_2 \mu^2 \|\btheta\|^2}{\mu^2 + \|\btheta\|^2} \right| \leq \rho^2
    \quad \text{and} \quad
    \left| \frac{\|\btheta\|^2}{2\mu^2 + \|\btheta\|^2} - \frac{\|\btheta\|^2}{2 (\mu^2 + \|\btheta\|^2)} + \frac{r_2' \mu^4}{\mu^2 + \|\btheta\|^2} \right| \leq 2 \rho^2.
\]
Then, according to Lemma \ref{lem:semiparametric_score_explicit_representation} and the triangle inequality, we have
\begin{align}
    \label{eq:semiparametric_score_norm_upper_bound}
    &\notag
    \left\| \left( \frac12 A^\top A + \lambda I_d \right)^{-1/2} H_{\btheta \chi} H_{\chi \chi}^{-1} \bz_\chi \right\|
    \\&\notag
    \leq \frac12 \left\| \left( \frac12 A^\top A + \lambda I_d \right)^{-1/2} A^\top \bU \right\|
    + \frac12 \left\| \left( \frac12 A^\top A + \lambda I_d \right)^{-1/2} A^\top (\widehat \Sigma - \Sigma) (\btheta - \btheta^\circ) \right\|
    \\&\quad
    + \rho^2 \left\| \left( \frac12 A^\top A + \lambda I_d \right)^{-1/2} A^\top (\bZ - \E \bZ) \right\|
    + 2 \rho^2 \left\| \left( \frac12 A^\top A + \lambda I_d \right)^{-1/2} A^\top (\widehat \Sigma - \Sigma) \btheta \right\|
    \\&\quad\notag
    + \frac{\lambda^{-1/2} \|\btheta\|}{\mu^2 + \|\btheta\|^2} \left| (A \btheta - \bfeta)^\top (\bZ - \E \bZ) \right|
    + \frac{\lambda^{-1/2} \|\btheta\|}{2\mu^2 + \|\btheta\|^2} \left| (A \btheta - \bfeta)^\top (\widehat \Sigma - \Sigma) \btheta \right|
    \\&\quad\notag
    + \left\| \left( \frac12 A^\top A + \lambda I_d \right)^{-1/2} (\widehat \Sigma - \Sigma) (A \btheta - \bfeta) \right\|.
\end{align}
Here we used the fact that
\[
    \left\| (A^\top A + \lambda I_d)^{-1/2} \btheta \right\|
    \leq \frac{\|\btheta\|}{\sqrt{\lambda}}.
\]
Since
\[
    \left\| \left( \frac12 A^\top A + \lambda I_d \right)^{-1/2} A^\top \right\| 
    \leq \sqrt{2}
    \quad \text{and} \quad
    \left\| \left( \frac12 A^\top A + \lambda I_d \right)^{-1/2}\right\| 
    \leq \frac1{\sqrt{\lambda}},
\]
the inequality \eqref{eq:semiparametric_score_norm_upper_bound} can be simplified even further:
\begin{align}
    \label{eq:semiparametric_score_norm_upper_bound_simplified}
    \left\| \left( \frac12 A^\top A + \lambda I_d \right)^{-1/2} H_{\btheta \chi} H_{\chi \chi}^{-1} \bz_\chi \right\|
    &\notag
    \leq \frac{\|\bU\|}{\sqrt{2}} + \frac{\|(\widehat \Sigma - \Sigma)(\btheta - \btheta^\circ)\|}{\sqrt{2}} + \frac{\|(\widehat \Sigma - \Sigma) (A \btheta - \bfeta)\|}{\sqrt{\lambda}}
    \\&\quad \notag
    + \sqrt{2} \rho^2 \, \|\bZ - \E \bZ\|
    + 2 \sqrt{2} \rho^2 \left\| (\widehat \Sigma - \Sigma) \btheta \right\|
    \\&\quad
    + \frac{\lambda^{-1/2} \|\btheta\|}{\mu^2 + \|\btheta\|^2} \left| (A \btheta - \bfeta)^\top (\bZ - \E \bZ) \right|
    \\&\quad\notag
    + \frac{\lambda^{-1/2} \|\btheta\|}{2\mu^2 + \|\btheta\|^2} \left| (A \btheta - \bfeta)^\top (\widehat \Sigma - \Sigma) \btheta \right|.
\end{align}
In other words, Lemma \ref{lem:semiparametric_score_explicit_representation} allows us to reduce the study of a complicated expression for $\breve H_{\btheta \btheta}^{-1} H_{\btheta \chi} H_{\chi \chi}^{-1} \bz_\chi$ to analysis of much more tractable terms. In the rest of the proof, we bound the summands in the right-hand side of \eqref{eq:semiparametric_score_norm_upper_bound} one by one applying Theorems \ref{th:covariance_concentration} and \ref{th:noise_concentration}.

\medskip

\noindent\textbf{Step 3: large deviation bounds.}
\quad
Let us recall that on the event $\cE_1$ we have
\[
    \|(\widehat \Sigma - \Sigma) \btheta^\circ\| \leq 4 C_X \|\Sigma\| \|\btheta^\circ\| \sqrt{\frac{\ttr(\Sigma) + \log(4 / \delta)}n}
    \quad \text{and} \quad
    \|\bU\| \leq 8 \sigma \|\Sigma\|^{1/2} \sqrt{\frac{\ttr(\Sigma) + \log(4 / \delta)}n}.
\]
Moreover, according to Theorem \ref{th:covariance_concentration}, each of the  inequalities \eqref{eq:hat_sigma_linear_functional_1_deviation_bound}--\eqref{eq:hat_sigma_linear_functional_4_deviation_bound} holds with probability at least $1 - \delta/2$. Note that the conditions of Theorem \ref{th:covariance_concentration} are fulfilled,
because $\ttr(\Sigma) + \log(4 / \delta) \leq n \leq 4n$ by our assumptions. Similarly, due to Theorem \ref{th:noise_concentration}, with probability at least $1 - \delta/2$, it holds that
\[
    \left| (A \btheta - \bfeta)^\top \bU \right|
    \leq 8 \sigma \|\Sigma^{1/2} (A \btheta - \bfeta)\| \sqrt{\frac{1 + \log(4 / \delta)}n}.
\]
The union bound yields that there exists an event $\cE_2$ of probability at least $1 - 5\delta / 2$ where the inequalities \eqref{eq:hat_sigma_linear_functional_1_deviation_bound}--\eqref{eq:z_linear_functional_deviation_bound} hold simultaneously. On the intersection of the events $\cE_1$ and $\cE_2$, the right-hand side of \eqref{eq:semiparametric_score_norm_upper_bound_simplified} does not exceed
\begin{align*}
    &
    \left\| \left( \frac12 A^\top A + \lambda I_d \right)^{-1/2} H_{\btheta \chi} H_{\chi \chi}^{-1} \bz_\chi \right\|
    \\&
    \leq 2 \sqrt{2} \|\Sigma\|^{1/2} \left(
    2 \sigma + C_X \|\Sigma^{1/2} (\btheta - \btheta^\circ)\| + \frac{\sqrt{2} C_X \|\Sigma^{1/2} (A \btheta - \bfeta)\|}{\sqrt{\lambda}} \right) \sqrt{\frac{\ttr(\Sigma) + \log(4 / \delta)}n}
    \\&\quad
    + 4 \sqrt{2} \rho^2 \|\Sigma\|^{1/2} \left(
    2 \sigma + 3 C_X \|\Sigma^{1/2} \btheta^\circ\| \right) \sqrt{\frac{\ttr(\Sigma) + \log(4 / \delta)}n}
    \\&\quad
    + \frac{8 \|\btheta\|}{\mu^2 + \|\btheta\|^2} \cdot \frac{\|\Sigma^{1/2} (A \btheta - \bfeta)\|}{\sqrt{\lambda}} \left( C_X \|\Sigma^{1/2} \btheta\| + \sigma \right) \sqrt{\frac{1 + \log(4 / \delta)}n}.
\end{align*}
Finally, since $\bups \in \Ups(\rho)$, we have
\[
    \|\btheta\|^2 \leq \rho^2 (\mu^2 + \|\btheta\|^2)
    \quad \text{and} \quad
    \frac{\|\btheta\|}{\mu^2 + \|\btheta\|^2}
    \leq \frac{\|\btheta\|}{\mu \sqrt{\mu^2 + \|\btheta\|^2}} \leq \frac{\rho}{\mu}.
\]
This yields the desired bound:
\begin{align*}
    &
    \left\| \left( \frac12 A^\top A + \lambda I_d \right)^{1/2} \breve H_{\btheta \btheta}^{-1} H_{\btheta \chi} H_{\chi \chi}^{-1} \bz_\chi \right\|
    \\&
    \leq \frac1{1 - 10\rho^2} \left\| \left( \frac12 A^\top A + \lambda I_d \right)^{-1/2} H_{\btheta \chi} H_{\chi \chi}^{-1} \bz_\chi \right\|
    \\&
    \leq \frac{2 \sqrt{2} \|\Sigma\|^{1/2}}{1 - 10\rho^2} \left(
    2 \sigma + C_X \|\Sigma^{1/2} (\btheta - \btheta^\circ)\| + \frac{\sqrt{2} C_X \|\Sigma^{1/2} (A \btheta - \bfeta)\|}{\sqrt{\lambda}} \right) \sqrt{\frac{\ttr(\Sigma) + \log(4 / \delta)}n}
    \\&\quad
    + \frac{4 \sqrt{2} \rho^2 \|\Sigma\|^{1/2}}{1 - 10\rho^2} \left(
    2 \sigma + 3 C_X \|\Sigma^{1/2} \btheta^\circ\| \right) \sqrt{\frac{\ttr(\Sigma) + \log(4 / \delta)}n}
    \\&\quad
    + \frac{8 \rho}{(1 - 10\rho^2) \mu} \cdot \frac{\|\Sigma^{1/2} (A \btheta - \bfeta)\|}{\sqrt{\lambda}} \left( C_X \|\Sigma^{1/2} \btheta\| + \sigma \right) \sqrt{\frac{1 + \log(4 / \delta)}n}.
\end{align*}
\myendproof

\subsection{Proof of Lemma \ref{lem:standardized_score_chi_bound}}
\label{sec:lem_standardized_score_chi_bound_proof}

As before, we will use the blockwise representation of the Hessian of $\cL(\bups)$.
We would like to recall that $\breve H_{\chi \chi} = H_{\chi \chi} - H_{\chi \btheta} H_{\btheta \btheta}^{-1} H_{\chi \btheta}$, where
\[
    H_{\btheta \btheta} = A^\top A + \lambda I_d,
    \quad
    H_{\chi \chi} =
    \begin{pmatrix}
        2 I_d & H_{\bfeta A} \\
        H_{A \bfeta} & H_{AA}
    \end{pmatrix}
    \quad \text{and} \quad
    H_{\btheta \chi} =
    \begin{pmatrix}
        -A^\top & H_{\btheta A}
    \end{pmatrix}.
\]
The key observation in the proof of Lemma \ref{lem:standardized_score_chi_bound} is that the off-diagonal blocks of $\breve H_{\chi \chi}$ are small. To be more precise, let us introduce a block-diagonal matrix
\[
    \ttD_\chi^2 = \diag\left( 2 I_d - A (A^\top A + \lambda I_d)^{-1} A^\top, H_{AA} \right),
\]
and consider the difference
\[
    Q = \breve H_{\chi \chi} - \ttD_\chi^2 =
    \begin{pmatrix}
        O_d & -A H_{\btheta \btheta}^{-1} H_{\btheta A} + H_{\bfeta A} \\
        -H_{A \btheta} H_{\btheta \btheta}^{-1} A^\top + H_{A \bfeta} & -H_{A \btheta} H_{\btheta \btheta}^{-1} H_{\btheta A}
    \end{pmatrix}.
\]
We are going to show that $\|\ttD_{\chi}^{-1} Q \, \ttD_{\chi}^{-1}\| \leq (1 + \sqrt{2}) \rho + 2\rho^2$. As usual, we split the rest of the proof into several steps for convenience.

\medskip

\noindent\textbf{Step 1: upper bound on $\|\ttD_{\chi}^{-1} Q \, \ttD_{\chi}^{-1}\|$.}
\quad
Since $A (A^\top A + \lambda I_d)^{-1} A^\top \preceq I_d$ and $\ttD_{\chi}^2 \succeq \diag( I_d, H_{AA})$,
it is enough to check that
\[
    \left\| \diag( I_d, H_{AA}^{-1/2}) \; Q \; \diag( I_d, H_{AA}^{-1/2}) \right\| \leq (1 + \sqrt{2}) \rho + 2\rho^2.
\]
For this purpose, we fix arbitrary $\bv \in \R^d$ and $\bw \in \R^{d^2}$ and study
\[
    \begin{pmatrix}
        \bv^\top & \bw^\top
    \end{pmatrix}
    Q
    \begin{pmatrix}
        \bv \\ \bw
    \end{pmatrix}.
\]
Due to the Cauchy-Schwarz inequality, it holds that
\begin{align*}
    \begin{pmatrix}
        \bv^\top & \bw^\top
    \end{pmatrix}
    Q
    \begin{pmatrix}
        \bv \\ \bw
    \end{pmatrix}
    &
    = 2 \bv^\top A H_{\btheta \btheta}^{-1} H_{\btheta A} \bw + 2 \bv^\top H_{\bfeta A} \bw + \bw^\top H_{A \btheta} H_{\btheta \btheta}^{-1} H_{\btheta A} \bw
    \\&
    \leq 2 \|\bv\| \|A H_{\btheta \btheta}^{-1/2}\| \|H_{\btheta \btheta}^{-1/2} H_{\btheta A} H_{AA}^{-1/2} \| \|H_{AA}^{1/2} \bw\|
    \\& \quad
    + 2 \|\bv\| \| H_{\bfeta A} H_{AA}^{-1/2} \| \|H_{AA}^{1/2} \bw\|
    + \|H_{\btheta \btheta}^{-1/2} H_{\btheta A} H_{AA}^{-1/2} \|^2 \|H_{AA}^{1/2} \bw\|^2.
\end{align*}
According to Lemma \ref{lem:non-diagonal_blocks_bound}, we have
\[
    \| H_{\bfeta A} H_{AA}^{-1/2} \| \leq \rho
    \quad \text{and} \quad
    \|H_{\btheta \btheta}^{-1/2} H_{\btheta A} H_{AA}^{-1/2} \|
    \leq \rho \sqrt{2}.
\]
This, together with the inequality
\[
    \|A H_{\btheta \btheta}^{-1/2}\|
    = \left\| A (A^\top A + \lambda I_d)^{-1/2} \right\|
    \leq 1,
\]
yields that
\begin{align*}
    \begin{pmatrix}
        \bv^\top & \bw^\top
    \end{pmatrix}
    Q
    \begin{pmatrix}
        \bv \\ \bw
    \end{pmatrix}
    &
    \leq 2 \rho (1 + \sqrt{2}) \|\bv\| \, \|H_{AA}^{1/2} \bw\| + 2 \rho^2 \|H_{AA}^{1/2} \bw\|^2
    \\&
    \leq \rho (1 + \sqrt{2}) \left( \|\bv\|^2 + \|H_{AA}^{1/2} \bw\|^2 \right) + 2 \rho^2 \|H_{AA}^{1/2} \bw\|^2
    \\&
    \leq ((1 + \sqrt{2}) \rho + 2\rho^2) \left( \|\bv\|^2 + \|H_{AA}^{1/2} \bw\|^2 \right).
\end{align*}
Hence, we proved that
\begin{equation}
    \label{eq:remainder_q_bound}
    \|\ttD_{\chi}^{-1} Q \, \ttD_{\chi}^{-1}\|
    \leq \left\| \diag( I_d, H_{AA}^{-1/2}) \; Q \; \diag( I_d, H_{AA}^{-1/2}) \right\|
    \leq (1 + \sqrt{2}) \rho + 2\rho^2.
\end{equation}

\medskip

\noindent\textbf{Step 2: upper bound on $\|\widetilde \ttD \breve H_{\chi \chi}^{-1} \z_\chi\|$.}
\quad
With the inequality \eqref{eq:remainder_q_bound} the upper bound on the norm of $\widetilde \ttD \breve H_{\chi \chi}^{-1} \z_\chi$ is straightforward. Indeed, it holds that
\begin{align*}
    \breve H_{\chi \chi}^{-1} \z_\chi
    = \left( \ttD_\chi^2 + Q \right)^{-1} \z_\chi
    &
    = \ttD_\chi^{-1} \left( I_d + \ttD_\chi^{-1} Q \, \ttD_\chi^{-1} \right)^{-1} \ttD_\chi^{-1} \z_\chi
    \\&
    = \ttD_\chi^{-1} \left( I_d - \left( I_d + \ttD_\chi^{-1} Q \, \ttD_\chi^{-1} \right)^{-1} \right) \ttD_\chi^{-1} \z_\chi
    \\&
    = \ttD_\chi^{-2} \z_\chi - \ttD_\chi^{-1} \left( I_d + \ttD_\chi^{-1} Q \, \ttD_\chi^{-1} \right)^{-1} \ttD_\chi^{-1} Q \, \ttD_\chi^{-2} \z_\chi.
\end{align*}
Then \eqref{eq:remainder_q_bound} and the triangle inequality imply that
\begin{align*}
    \left\| \widetilde \ttD \breve H_{\chi \chi}^{-1} \z_\chi \right\|
    &
    \leq \left\| \widetilde \ttD \, \ttD_\chi^{-2} \z_\chi \right\| + \left\| \widetilde \ttD \, \ttD_\chi^{-1} \left( I_d + \ttD_\chi^{-1} Q \, \ttD_\chi^{-1} \right)^{-1} \ttD_\chi^{-1} Q \, \ttD_\chi^{-2} \z_\chi \right\|
    \\&
    \leq \left\| \widetilde \ttD \, \ttD_\chi^{-2} \z_\chi \right\|
    + \left\| \widetilde \ttD \, \ttD_\chi^{-1} \right\| \cdot \frac{ \|\ttD_\chi^{-1} Q \, \ttD_\chi^{-1}\|}{1 - \|\ttD_\chi^{-1} Q \, \ttD_\chi^{-1}\|} \cdot \left\| \ttD_\chi^{-1} \z_\chi \right\|
    \\&
    \leq \left\| \widetilde \ttD \, \ttD_\chi^{-2} \z_\chi \right\| + \frac{((1 + \sqrt{2}) \rho + 2\rho^2) \sqrt{2}}{1 - (1 + \sqrt{2}) \rho - 2\rho^2} \cdot \left\| \ttD_\chi^{-1} \z_\chi \right\|.
\end{align*}
Here we used the fact that
\[
    \ttD_\chi^2
    \succeq \diag(I_d, H_{AA})
    = \diag\big(I_d, I_d \otimes (\mu^2 I_d + \btheta\btheta^\top) \big)
    \succeq \frac12 \diag\big(2 I_d, I_d \otimes (\widetilde\mu^2 I_d + \btheta\btheta^\top) \big)
    \succeq \frac12 \widetilde \ttD^2
\]
which yields $\|\widetilde \ttD \, \ttD_\chi^{-1}\| \leq \sqrt{2}$. Since $\rho \leq 1/16$, we obtain that
\[
    \left\| \widetilde \ttD \breve H_{\chi \chi}^{-1} \z_\chi \right\|
    \leq \left\| \widetilde \ttD \, \ttD_\chi^{-2} \z_\chi \right\| + \frac{((1 + \sqrt{2}) \rho + \rho/8) \sqrt{2}}{1 - (1 + \sqrt{2}) / 16 - 1 / 128} \, \left\| \ttD_\chi^{-1} \z_\chi \right\|
    \leq \left\| \ttD_\chi^{-2} \z_\chi \right\| + \frac{9 \rho}2 \, \left\| \ttD_\chi^{-1} \z_\chi \right\|.
\]
It only remains to note that
\begin{align*}
    \left\| \widetilde \ttD \, \ttD_\chi^{-2} \z_\chi \right\|^2
    &
    \leq \left\| \diag\left(I_d, \frac{\sqrt{\widetilde \mu^2 + \|\btheta\|^2}}{\mu^2} I_{d^2} \right) \z_\chi \right\|^2
    \\&
    = \left\| \diag\left(I_d, \frac{\sqrt{\widetilde \mu^2 + \|\btheta\|^2}}{\mu^2} I_{d^2} \right)
    \begin{pmatrix}
        \bZ - \E \bZ \\ \mu^2 \, \rmvec(\widehat \Sigma - \Sigma)
    \end{pmatrix}
    \right\|^2
    \\&
    = \left\| \bZ - \E \bZ \right\|^2 + \left(\widetilde \mu^2 + \|\btheta\|^2\right) \left\| \widehat \Sigma - \Sigma \right\|_{\F}^2
\end{align*}
and, similarly,
\[
    \left\| \ttD_\chi^{-1} \z_\chi \right\|^2
    \leq \left\| \diag\left(I_d, \frac{1}{\mu} I_{d^2} \right)
    \begin{pmatrix}
        \bZ - \E \bZ \\ \mu^2 \, \rmvec(\widehat \Sigma - \Sigma)
    \end{pmatrix}
    \right\|^2
    = \left\| \bZ - \E \bZ \right\|^2 + \mu^2 \left\| \widehat \Sigma - \Sigma \right\|_{\F}^2. 
\]
Then
\[
    \left\| \widetilde \ttD \breve H_{\chi \chi}^{-1} \z_\chi \right\|
    \leq \left(1 + \frac{9 \rho}2 \right) \left( \|\bU\| + \|(\widehat \Sigma - \Sigma) \btheta^\circ\| \right) + \left(\sqrt{\widetilde\mu^2 + \|\btheta\|^2} + \frac{9 \rho \mu}2 \right) \left\| \widehat \Sigma - \Sigma \right\|_{\F}
\]
and, taking into account that
\[
    \|\widehat \Sigma - \Sigma\|_{\F} \leq 4 C_X \|\Sigma\| \sqrt{\frac{\ttr(\Sigma)^2 + \log(4 / \delta)}n},
    \quad 
    \|(\widehat \Sigma - \Sigma) \btheta^\circ\| \leq 4 C_X \|\Sigma\| \|\btheta^\circ\| \sqrt{\frac{\ttr(\Sigma) + \log(4 / \delta)}n},
\]
and
\[
    \|\bU\| \leq 8 \sigma \|\Sigma\|^{1/2} \sqrt{\frac{\ttr(\Sigma) + \log(4 / \delta)}n}.
\]
on the event $\cE_1$ (see \eqref{eq:e1_sigma} and \eqref{eq:e1_u}), we deduce the desired bound:
\begin{align*}
    \left\| \widetilde \ttD \breve H_{\chi \chi}^{-1} \z_\chi \right\|
    &
    \leq 8 \left(1 + \frac{9 \rho}2 \right) \sigma \|\Sigma\|^{1/2} \sqrt{\frac{\ttr(\Sigma) + \log(4 / \delta)}n}
    \\&\quad
    + 4 \left(1 + \frac{9 \rho}2 \right) C_X \|\Sigma\| \|\btheta^\circ\| \sqrt{\frac{\ttr(\Sigma) + \log(4 / \delta)}n}
    \\&\quad
    + 4 \left(\sqrt{\widetilde \mu^2 + \|\btheta\|^2} + \frac{9 \rho \mu}2\right) C_X \|\Sigma\| \sqrt{\frac{\ttr(\Sigma)^2 + \log(4 / \delta)}n}.
\end{align*}
\myendproof

\subsection{Proof of Lemma \ref{lem:standardized_score_simplified}}
\label{sec:lem_standardized_score_simplified_proof}

We would like to remind a reader that (see \eqref{eq:standardized_score})
\[
    \bs_{\btheta}^* = -\breve H_{\btheta \btheta}^{-1}(\bups^*) H_{\btheta \chi}(\bups^*) H_{\chi \chi}^{-1}(\bups^*) \z_\chi,
\]
where, as before, $\breve H_{\btheta \btheta}(\bups^*) = H_* / H_{\chi \chi}(\bups^*)$
stands for the Schur complement of $H_{\chi \chi}(\bups^*)$. The proof of Lemma \ref{lem:standardized_score_simplified} consists of two steps. First, using Lemma \ref{lem:h_theta_schur}, we show that $\bs_{\btheta}^*$ is close to
\[
    \left( \frac12 (A^*)^\top A^* + \lambda I_d \right)^{-1} H_{\btheta \chi}(\bups^*) H_{\chi \chi}^{-1}(\bups^*) \z_\chi.
\]
After that, we quantify the difference between
\[
    \left( \frac12 (A^*)^\top A^* + \lambda I_d \right)^{-1} H_{\btheta \chi}(\bups^*) H_{\chi \chi}^{-1}(\bups^*) \z_\chi
    \quad \text{and} \quad
    \left( \frac12 \Sigma^2 + \lambda I_d \right)^{-1} H_{\btheta \chi}(\bups^*) H_{\chi \chi}^{-1}(\bups^*) \z_\chi.
\]

\medskip

\noindent\textbf{Step 1.}\quad
Let us introduce
\[
    R = \left( \frac12 (A^*)^\top A^* + \lambda I_d \right)^{-1/2} \breve H_{\btheta \btheta}(\bups^*) \left( \frac12 (A^*)^\top A^* + \lambda I_d \right)^{-1/2} - I_d
\]
and note that
\begin{align*}
    &
    \left\| \Sigma^{1/2} \bs_{\btheta}^*
    + \Sigma^{1/2} \left( \frac12 (A^*)^\top A^* + \lambda I_d \right)^{-1} H_{\btheta \chi}(\bups^*) H_{\chi \chi}^{-1}(\bups^*) \z_\chi \right\|
    \\&
    = \left\| \Sigma^{1/2} \left( \frac12 (A^*)^\top A^* + \lambda I_d \right)^{-1/2} R (I_d + R)^{-1} \left( \frac12 (A^*)^\top A^* + \lambda I_d \right)^{-1/2} H_{\btheta \chi}(\bups^*) H_{\chi \chi}^{-1}(\bups^*) \z_\chi \right\|
    \\&
    \leq \left\| R (I_d + R)^{-1} \right\| \left\| \Sigma^{1/2} \left( \frac12 (A^*)^\top A^* + \lambda I_d \right)^{-1/2}  \right\| \left\| \left( \frac12 (A^*)^\top A^* + \lambda I_d \right)^{-1/2} H_{\btheta \chi}(\bups^*) H_{\chi \chi}^{-1}(\bups^*) \z_\chi \right\|.
\end{align*}
According to Lemma \ref{lemma: localization lemma for bias} (see \ref{point: eta weak bound on bias, bias locating convex set} and \ref{point: theta weak bound on bias, bias locating convex set}), it holds that
\[
    \|A^* \btheta^* - \bfeta^*\| \leq \sqrt{\lambda} \|\btheta^\circ\|
    \quad \text{and} \quad
    \|\btheta^*\| \leq \|\btheta^\circ\|.
\]
This means that $\btheta^*$ belongs to $\Ups(\rho^*)$ with $\rho^* = \|\btheta^\circ\| / \mu \leq 1/112$. Then, due to Lemma \ref{lem:h_theta_schur}, the operator norm of $R$ does not exceed $10 (\rho^*)^2$. This yields that
\[
    \left\|R (I_d + R)^{-1} \right\|
    \leq \|R\| \left\|(I_d + R)^{-1} \right\|
    \leq \frac{\|R\|}{1 - \|R\|}
    \leq \frac{10 (\rho^*)^2}{1 - 10 (\rho^*)^2}
    \leq 1.001 (\rho^*)^2
    = \frac{1.001 \|\btheta^\circ\|^2}{\mu^2},
\]
and, as a consequence, we obtain that
\begin{align}
    \label{eq:leading_stoch_term_upper_bound}
    &
    \left\| \Sigma^{1/2} \bs_{\btheta}^* 
    + \Sigma^{1/2} \left( \frac12 (A^*)^\top A^* + \lambda I_d \right)^{-1} H_{\btheta \chi}(\bups^*) H_{\chi \chi}^{-1}(\bups^*) \z_\chi \right\|
    \\&\notag
    \leq \frac{1.001 \|\btheta^\circ\|^2}{\mu^2} \left\| \Sigma^{1/2} \left( \frac12 (A^*)^\top A^* + \lambda I_d \right)^{-1/2} \right\| \left\| \left( \frac12 (A^*)^\top A^* + \lambda I_d \right)^{-1/2} H_{\btheta \chi}(\bups^*) H_{\chi \chi}^{-1}(\bups^*) \z_\chi \right\|.
\end{align}
We are going to show that, under the conditions of the lemma, the operator norm of
\[
    \Sigma^{1/2} \left( (A^*)^\top A^* / 2 + \lambda I_d \right)^{-1/2}
\]
is not greater than $\lambda^{-1/4}$. The proof of this fact is based on the following lemma.
\begin{Lem}
    \label{lem:diff_squares_op_norm_bound}
    Let us fix arbitrary $A \in \R^{d \times d}$ and $\lambda > 0$. Then it holds that
    \[
        \frac12 \left\| \left(\frac12 \Sigma^2 + \lambda I_d \right)^{-1/2} \left( A^\top A - \Sigma^2 \right) \left(\frac12 \Sigma^2 + \lambda I_d \right)^{-1/2} \right\|
        \leq \|A - \Sigma\| \sqrt{\frac{2}{\lambda}} + \frac{\|A - \Sigma\|^2}{2 \lambda}.
    \]
\end{Lem}
We postpone the proof of Lemma \ref{lem:diff_squares_op_norm_bound} to Appendix \ref{sec:lem_diff_squares_op_norm_bound_proof} and proceed with the proof of Lemma \ref{lem:standardized_score_simplified}. Taking into account \eqref{eq:stoch_term_mu_lambda_conditions} and applying Lemma \ref{lemma: Sigma bias optimal bound}, we note that
\[
    \frac{\|A^* - \Sigma\|}{\sqrt{\lambda}}
    \leq \left( \frac{14 \|\btheta^\circ\|}{\mu} \right)^2 + \frac{5 \|\bb_\lambda\|}{64 \cdot 18 \mu}
    \leq \left( \frac{14 \|\btheta^\circ\|}{\mu}\right)^2 + \frac{\|\bb_\lambda\|}{230 \mu}.
\]
Due to the definition of $\bb_\lambda$ and the condition $112 \|\btheta^\circ\| \leq \mu$, the right-hand side does not exceed 
\[
    \left( \frac{14 \|\btheta^\circ\|}{\mu}\right)^2 + \frac{\|\bb_\lambda\|}{230 \mu}
    \leq \frac1{64} + \frac1{112 \cdot 230} < \frac1{63}.
\]
Then, due to Lemma \ref{lem:diff_squares_op_norm_bound}, it holds that
\begin{align}
    \label{eq:diff_squares_op_norm_bound_number}
    &\notag
    \left\| \left(\frac12 \Sigma^2 + \lambda I_d \right)^{-1/2} \left(\frac12 (A^*)^\top A^* - \frac12 \Sigma^2 \right) \left(\frac12 \Sigma^2 + \lambda I_d \right)^{-1/2} \right\|
    \\&
    \leq \left[ \left( \frac{14 \|\btheta^\circ\|}{\mu}\right)^2 + \frac{\|\bb_\lambda\|}{230 \mu} \right] \left( \sqrt{2} + \frac1{2 \cdot 63} \right)
    \leq \left( \frac{17 \|\btheta^\circ\|}{\mu}\right)^2 + \frac{\|\bb_\lambda\|}{160 \mu} \leq \frac1{43}.
\end{align}
In the last line, we used the inequalities
\[
    \sqrt{2} + 1/126 < (17/14)^2,
    \quad
    (\sqrt{2} + 1/126) / 230 < 1 / 160,
    \quad\text{and}\quad
    \frac{\|\bb_{\lambda}\|}{\mu} \leq \frac{\|\btheta^\circ\|}{\mu} \leq \frac{1}{112}.
\]
Let us denote the matrix in the left-hand side by $B$:
\begin{equation}
    \label{eq:b}
    B = \left(\frac12 \Sigma^2 + \lambda I_d \right)^{-1/2} \left(\frac12 (A^*)^\top A^* - \frac12 \Sigma^2 \right) \left(\frac12 \Sigma^2 + \lambda I_d \right)^{-1/2}.
\end{equation}
Then it is straightforward to observe that
\begin{align}
    \label{eq:sigma_matrix_product_bound}
    &\notag
    \left\| \left( \frac12 \Sigma^2 + \lambda I_d \right)^{1/2} \left( \frac12 (A^*)^\top A^* + \lambda I_d \right)^{-1/2} \right\|^2
    \\&
    = \left\| \left( \frac12 \Sigma^2 + \lambda I_d \right)^{1/2} \left( \frac12 \Sigma^2 + \lambda I_d + \frac12 (A^*)^\top A^* - \frac12 \Sigma^2 \right) \left( \frac12 \Sigma^2 + \lambda I_d \right)^{1/2} \right\|
    \\&\notag
    = \left\| (I_d + B)^{-1} \right\|
    \leq \frac1{1 - \|B\|}
    \leq \frac{43}{42}.
\end{align}
This inequality and \eqref{eq:leading_stoch_term_upper_bound} immediately imply that
\begin{align*}
    &
    \left\| \Sigma^{1/2} \bs_{\btheta}^* 
    + \Sigma^{1/2} \left( \frac12 (A^*)^\top A^* + \lambda I_d \right)^{-1} H_{\btheta \chi}(\bups^*) H_{\chi \chi}^{-1}(\bups^*) \z_\chi \right\|
    \\&
    \leq \frac{1.001 \|\btheta^\circ\|^2}{\mu^2} \sqrt{\frac{43}{42}} \, \left\| \Sigma^{1/2} \left( \frac12 \Sigma^2 + \lambda I_d \right)^{-1/2} \right\| \left\| \left( \frac12 (A^*)^\top A^* + \lambda I_d \right)^{-1/2} H_{\btheta \chi}(\bups^*) H_{\chi \chi}^{-1}(\bups^*) \z_\chi \right\|.
\end{align*}
Due to the Cauchy-Schwarz inequality,
\begin{equation}
    \label{eq:sigma_cauchy-schwarz}
    \left\| \Sigma^{1/2} \left( \frac12 \Sigma^2 + \lambda I_d \right)^{-1/2} \right\|^2
    = \left\| \left( \frac12 \Sigma^2 + \lambda I_d \right)^{-1/2} \Sigma \left( \frac12 \Sigma^2 + \lambda I_d \right)^{-1/2} \right\| \leq \frac{1}{\sqrt{2 \lambda}}.
\end{equation}
This yields that
\begin{align}
    \label{eq:standardized_score_simplified_first_term}
    &\notag
    \left\| \Sigma^{1/2} \bs_{\btheta}^* 
    + \Sigma^{1/2} \left( \frac12 (A^*)^\top A^* + \lambda I_d \right)^{-1} H_{\btheta \chi}(\bups^*) H_{\chi \chi}^{-1}(\bups^*) \z_\chi \right\|
    \\&
    \leq \frac{1.001 \|\btheta^\circ\|^2}{\mu^2} \sqrt{\frac{43}{42}} \cdot (2\lambda)^{-1/4} \left\| \left( \frac12 (A^*)^\top A^* + \lambda I_d \right)^{-1/2} H_{\btheta \chi}(\bups^*) H_{\chi \chi}^{-1}(\bups^*) \z_\chi \right\|.
\end{align}
Hence, we proved that $\Sigma^{1/2} \bs_{\btheta}^*$ is close to $\Sigma^{1/2} \left( (A^*)^\top A^* / 2 + \lambda I_d \right)^{-1} H_{\btheta \chi}(\bups^*) H_{\chi \chi}^{-1}(\bups^*) \z_\chi$. Our next goal is to show that
\[
    \Sigma^{1/2} \left( (A^*)^\top A^* / 2 + \lambda I_d \right)^{-1} H_{\btheta \chi}(\bups^*) H_{\chi \chi}^{-1}(\bups^*) \z_\chi
    \quad \text{and} \quad
    \Sigma^{1/2} \left( \Sigma^2 / 2 + \lambda I_d \right)^{-1} H_{\btheta \chi}(\bups^*) H_{\chi \chi}^{-1}(\bups^*) \z_\chi
\]
are close as well.

\medskip

\noindent\textbf{Step 2.}
\quad Using \eqref{eq:b} we rewrite the difference between
\[
    -\Sigma^{1/2} \left( (A^*)^\top A^* / 2 + \lambda I_d \right)^{-1} H_{\btheta \chi}(\bups^*) H_{\chi \chi}^{-1}(\bups^*) \z_\chi
    \quad \text{and} \quad
    -\Sigma^{1/2} \left( \Sigma^2 / 2 + \lambda I_d \right)^{-1} H_{\btheta \chi}(\bups^*) H_{\chi \chi}^{-1}(\bups^*) \z_\chi
\]
in the following form:
\begin{align*}
    &
    \left\| \Sigma^{1/2} \left[ \left( \frac12 (A^*)^\top A^* + \lambda I_d \right)^{-1} - \left( \frac12 \Sigma^2 + \lambda I_d \right)^{-1} \right] H_{\btheta \chi}(\bups^*) H_{\chi \chi}^{-1}(\bups^*) \z_\chi \right\|
    \\&
    = \left\| \Sigma^{1/2} \left( \frac12 \Sigma^2 + \lambda I_d \right)^{-1/2} \left( (I_d + B)^{-1} - I_d \right) \left( \frac12 \Sigma^2 + \lambda I_d \right)^{-1/2} H_{\btheta \chi}(\bups^*) H_{\chi \chi}^{-1}(\bups^*) \z_\chi \right\|
    \\&
    = \left\| \Sigma^{1/2} \left( \frac12 \Sigma^2 + \lambda I_d \right)^{-1/2} B (I_d + B)^{-1} \left( \frac12 \Sigma^2 + \lambda I_d \right)^{-1/2} H_{\btheta \chi}(\bups^*) H_{\chi \chi}^{-1}(\bups^*) \z_\chi \right\|.
\end{align*}
Since, due to \eqref{eq:diff_squares_op_norm_bound_number}, the operator norm of $B (I_d + B)^{-1}$ does not exceed
\[
    \left\| B (I_d + B)^{-1} \right\|
    \leq \frac{\|B\|}{1 - \|B\|}
    \leq \frac{43}{42} \left( \left( \frac{17 \|\btheta^\circ\|}{\mu}\right)^2 + \frac{\|\bb_\lambda\|}{160 \mu} \right),
\]
we obtain that
\begin{align*}
    &
    \left\| \Sigma^{1/2} \left[ \left( \frac12 (A^*)^\top A^* + \lambda I_d \right)^{-1} - \left( \frac12 \Sigma^2 + \lambda I_d \right)^{-1} \right] H_{\btheta \chi}(\bups^*) H_{\chi \chi}^{-1}(\bups^*) \z_\chi \right\|
    \\&
    \leq \frac{43}{42} \left( \left( \frac{17 \|\btheta^\circ\|}{\mu}\right)^2 + \frac{\|\bb_\lambda\|}{160 \mu} \right) \left\| \Sigma^{1/2} \left( \frac12 \Sigma^2 + \lambda I_d \right)^{-1/2} \right\|
    \\&\quad
    \cdot \left\| \left( \frac12 \Sigma^2 + \lambda I_d \right)^{-1/2} H_{\btheta \chi}(\bups^*) H_{\chi \chi}^{-1}(\bups^*) \z_\chi \right\|.
\end{align*}
Combining this inequality with \eqref{eq:sigma_matrix_product_bound} and \eqref{eq:sigma_cauchy-schwarz}, we obtain that
\begin{align}
    \label{eq:standardized_score_simplified_second_term}
    &
    \left\| \Sigma^{1/2} \left[ \left( \frac12 (A^*)^\top A^* + \lambda I_d \right)^{-1} - \left( \frac12 \Sigma^2 + \lambda I_d \right)^{-1} \right] H_{\btheta \chi}(\bups^*) H_{\chi \chi}^{-1}(\bups^*) \z_\chi \right\|
    \\&\notag
    \leq \frac{43}{42} \left( \left( \frac{17 \|\btheta^\circ\|}{\mu}\right)^2 + \frac{\|\bb_\lambda\|}{160 \mu} \right)  \cdot (2\lambda)^{-1/4} \cdot \sqrt{\frac{43}{42}} \, \left\| \left( \frac12 (A^*)^\top A^* + \lambda I_d \right)^{-1/2} H_{\btheta \chi}(\bups^*) H_{\chi \chi}^{-1}(\bups^*) \z_\chi \right\|.
\end{align}
Finally, summing up the inequalities \eqref{eq:standardized_score_simplified_first_term} and \eqref{eq:standardized_score_simplified_second_term}, we obtain the desired bound:
\begin{align*}
    &
    \left\| \Sigma^{1/2} \bs_{\btheta}^* 
    + \Sigma^{1/2} \left( \Sigma^2 + \lambda I_d \right)^{-1} H_{\btheta \chi}(\bups^*) H_{\chi \chi}^{-1}(\bups^*) \z_\chi \right\|
    \\&
    \leq \frac{1.001 \|\btheta^\circ\|^2}{\mu^2} \sqrt{\frac{43}{42}} \cdot (2\lambda)^{-1/4} \left\| \left( \frac12 (A^*)^\top A^* + \lambda I_d \right)^{-1/2} H_{\btheta \chi}(\bups^*) H_{\chi \chi}^{-1}(\bups^*) \z_\chi \right\|
    \\&\quad
    + \frac{43}{42} \left( \left( \frac{17 \|\btheta^\circ\|}{\mu}\right)^2 + \frac{\|\bb_\lambda\|}{160 \mu} \right)  \cdot (2\lambda)^{-1/4} \cdot \sqrt{\frac{43}{42}} \, \left\| \left( \frac12 (A^*)^\top A^* + \lambda I_d \right)^{-1/2} H_{\btheta \chi}(\bups^*) H_{\chi \chi}^{-1}(\bups^*) \z_\chi \right\|
    \\&
    \leq \lambda^{-1/4} \left( \left( \frac{17 \|\btheta^\circ\|}{\mu}\right)^2 + \frac{\|\bb_\lambda\|}{160 \mu} \right) \left\| \left( \frac12 (A^*)^\top A^* + \lambda I_d \right)^{-1/2} H_{\btheta \chi}(\bups^*) H_{\chi \chi}^{-1}(\bups^*) \z_\chi \right\|.
\end{align*}
It only remains to note that, due to \eqref{eq:standardized_score_theta_bound_simplified_2}, the expression in the right-hand side does not exceed
\begin{align*}
    &
    \left(4 \sqrt{2} + \frac{8 \sqrt{2}}{112^2} + \frac{8}{112^2} \right) \frac{\|\btheta^\circ\|}{14 \lambda^{1/4}} 
    \left( \left( \frac{17 \|\btheta^\circ\|}{\mu}\right)^2 + \frac{\|\bb_\lambda\|}{160 \mu} \right) \sqrt{\Psi(n, \delta)}
    \\&
    \leq \frac{\|\btheta^\circ\|}{2 \lambda^{1/4}} 
    \left( \left( \frac{17 \|\btheta^\circ\|}{\mu}\right)^2 + \frac{\|\bb_\lambda\|}{160 \mu} \right) \sqrt{\Psi(n, \delta)},
\end{align*}
and the claim of the lemma follows.

\myendproof

\subsection{Proof of Lemma \ref{lem:semiparametric_remainder_1}}
\label{sec:lem_semiparametric_remainder_1_proof}

The proof of the lemma is based on the expansion of $H_{\btheta \chi}(\bups^*) H_{\chi \chi}^{-1}(\bups^*) \z_\chi$ from Lemma \ref{lem:semiparametric_score_explicit_representation}.
According to Lemma \ref{lemma: localization lemma for bias} (see \ref{point: eta weak bound on bias, bias locating convex set} and \ref{point: theta weak bound on bias, bias locating convex set}), it holds that
\[
    \|A^* \btheta^* - \bfeta^*\| \leq \sqrt{\lambda} \|\btheta^\circ\|
    \quad \text{and} \quad
    \|\btheta^*\| \leq \|\btheta^\circ\|.
\]
This means that $\btheta^*$ belongs to $\Ups(\rho^*)$ with $\rho^* = \|\btheta^\circ\| / \mu \leq 1/112$. Applying Lemma \ref{lem:semiparametric_score_explicit_representation} with $\bups = \bups^*$ and $\rho = \rho^*$, we obtain that
\begin{align*}
    &
    H_{\btheta \chi}(\bups^*) H_{\chi \chi}^{-1}(\bups^*) \z_\chi
    + \frac12 (A^*)^\top \bU
    - \frac12 (A^*)^\top (\widehat \Sigma - \Sigma) (\btheta^* - \btheta^\circ)
    - (\widehat \Sigma - \Sigma) (A^* \btheta^* - \bfeta^*)
    \\&
    = -\left(r_1 - \frac{r_2 \mu^2 \|\btheta^*\|^2}{\mu^2 + \|\btheta^*\|^2} \right) (A^*)^\top (\bZ - \E \bZ)
    + \frac{r_2 \mu^2}{\mu^2 + \|\btheta^*\|^2} \left( (A^* \btheta^* - \bfeta^*)^\top (\bZ - \E \bZ) \right) \btheta^*
    \\&\quad
    - \left(\frac{\|\btheta\|^2}{2\mu^2 + \|\btheta\|^2} - \frac{\|\btheta\|^2}{2 (\mu^2 + \|\btheta\|^2)} + \frac{r_2' \mu^4}{\mu^2 + \|\btheta\|^2} \right) (A^*)^\top (\widehat \Sigma - \Sigma) \btheta^*
    \\&\quad
    - \frac{1}{2\mu^2 + \|\btheta^*\|^2} \big( (A^* \btheta^* - \bfeta^*)^\top (\widehat \Sigma - \Sigma) \btheta^* \big) \btheta^*,
\end{align*}
where $r_1 \in [0, (\rho^*)^2]$, $r_2 = 0.5\mu^{-2} + r_2'$, and $r_2' \in [0, (\rho^*)^2 / \mu^{2}]$ are defined in Proposition \ref{proposition: inverse of nuisance parameters}. Let us denote
\begin{align*}
    \ttR
    &
    = -\left(r_1 - \frac{r_2 \mu^2 \|\btheta^*\|^2}{\mu^2 + \|\btheta^*\|^2} \right) (A^*)^\top (\bZ - \E \bZ)
    \\&\quad
    - \left(\frac{\|\btheta\|^2}{2\mu^2 + \|\btheta\|^2} - \frac{\|\btheta\|^2}{2 (\mu^2 + \|\btheta\|^2)} + \frac{r_2' \mu^4}{\mu^2 + \|\btheta\|^2} \right) (A^*)^\top (\widehat \Sigma - \Sigma) \btheta^*
\end{align*}
and
\[
    \ttQ = \frac{r_2 \mu^2}{\mu^2 + \|\btheta^*\|^2} \left( (A^* \btheta^* - \bfeta^*)^\top (\bZ - \E \bZ) \right) \btheta^*
    - \frac{1}{2\mu^2 + \|\btheta^*\|^2} \big( (A^* \btheta^* - \bfeta^*)^\top (\widehat \Sigma - \Sigma) \btheta^* \big) \btheta^*.
\]
With the introduced notations, it holds that
\begin{align}
    \label{eq:standardized_score_representation}
    &
    H_{\btheta \chi}(\bups^*) H_{\chi \chi}^{-1}(\bups^*) \z_\chi
    + 0.5 \btau
    = \ttR + \ttQ.
\end{align}
Let us recall that on the event $\cE_2^*$ we have (see \eqref{eq:hat_sigma_linear_functional_star_3_deviation_bound}--\eqref{eq:z_linear_functional_star_deviation_bound})
\[
    | (A^* \btheta^* - \bfeta^*)^\top (\widehat \Sigma - \Sigma) \btheta^* |
    \leq 4 C_X \|\Sigma^{1/2} \btheta^*\| \|\Sigma^{1/2} (A^* \btheta^* - \bfeta^*)\| \sqrt{\frac{1 + \log(4 / \delta)}n},
\]
\[
    | (A^* \btheta^* - \bfeta^*)^\top (\widehat \Sigma - \Sigma) \btheta^\circ |
    \leq 4 C_X \|\Sigma^{1/2} \btheta^\circ\| \|\Sigma^{1/2} (A^* \btheta^* - \bfeta^*)\| \sqrt{\frac{1 + \log(4 / \delta)}n},
\]
and
\[
    \left| (A^* \btheta^* - \bfeta^*)^\top \bU \right|
    \leq 8 \sigma \|\Sigma^{1/2} (A^* \btheta^* - \bfeta^*)\| \sqrt{\frac{1 + \log(4 / \delta)}n}.
\]
Since $\|A^*\btheta^* - \bfeta^*\| \leq \sqrt{\lambda} \|\btheta^\circ\|$ and $\|\btheta^*\| \leq \|\btheta^\circ\|$ due to Lemma \ref{lemma: localization lemma for bias}, it holds that
\begin{align*}
    &
    \left| \frac{r_2 \mu^2}{\mu^2 + \|\btheta^*\|^2} \left( (A^* \btheta^* - \bfeta^*)^\top (\bZ - \E \bZ) \right)
    - \frac{1}{2\mu^2 + \|\btheta^*\|^2} \big( (A^* \btheta^* - \bfeta^*)^\top (\widehat \Sigma - \Sigma) \btheta^* \big) \right|
    \\&
    \leq \frac{4 \|\Sigma\|^{1/2} \|A^* \btheta^* - \bfeta^*\|}{\mu^2} \left((1 + (\rho^*)^2 \right) \left( \sigma + C_X \|\Sigma\|^{1/2} \|\btheta^\circ\| \right) \sqrt{\frac{1 + \log(4 / \delta)}n}
    \\&
    \leq 1.001 \cdot \frac{4 \|\btheta^\circ\|^2 \sqrt{\lambda}}{\mu^2} \left( \frac{\sigma \|\Sigma\|^{1/2}}{\|\btheta^\circ\|} + C_X \|\Sigma\| \right) \sqrt{\frac{1 + \log(4 / \delta)}n}
    \\&
    \leq 1.001 \cdot \frac{2 \|\btheta^\circ\|^2}{7 \mu^2} \sqrt{\lambda \Psi(n, \delta)}.
\end{align*}
Let $\lambda_1(\Sigma), \dots, \lambda_d(\Sigma)$ be the eigenvalues of $\Sigma$. Let us fix an arbitrary $j \in \{1, \dots, d\}$. Applying the Young inequality
\[
    ab \leq \frac{a^p}{p} + \frac{b^q}{q}
    \quad \text{with $a = \sqrt{\lambda_j(\Sigma)}$, $b = (2 \lambda / 3)^{3/4}$, $p = 4$, and $q = 4/3$,}
\]
we obtain that
\[
    \frac{\sqrt{\lambda_j(\Sigma)}}{\lambda_j(\Sigma)^2 / 2 + \lambda}
    \leq \frac{(2\lambda/3)^{-3/4}}2
    \quad \text{for any $j \in \{1, \dots, d\}$,}
\]
and then
\begin{equation}
    \label{eq:young_inequality_op_norm_bound}
    \left\| \Sigma^{1/2} \left( \frac12 \Sigma^2 + \lambda I_d \right)^{-1} \right\|
    = \max\limits_{1 \leq j \leq d} \left\{ \frac{\sqrt{\lambda_j(\Sigma)}}{\lambda_j(\Sigma)^2 / 2 + \lambda} \right\}
    \leq \frac{(2\lambda/3)^{-3/4}}2.
\end{equation}
Then, using the inequality
\[
    \left\| \Sigma^{1/2} \left( \frac12 \Sigma^2 + \lambda I_d \right)^{-1} \btheta^* \right\|
    \leq \left\| \Sigma^{1/2} \left( \frac12 \Sigma^2 + \lambda I_d \right)^{-1} \right\| \|\btheta^*\|
    \leq \frac12 \left( \frac{2\lambda}3 \right)^{-3/4} \|\btheta^\circ\|,
\]
we conclude that
\begin{align}
    \label{eq:q_bound}
    \left\| \Sigma^{1/2} \left( \frac12 \Sigma^2 + \lambda I_d \right)^{-1} \ttQ \right\|
    &\notag
    \leq \frac12 \left( \frac{2\lambda}3 \right)^{-3/4} \|\btheta^\circ\| \cdot 1.001 \cdot \frac{2 \|\btheta^\circ\|^2}{7 \mu^2} \sqrt{\lambda \Psi(n, \delta)}
    \\&
    \leq \frac{\|\btheta^\circ\|^3}{5 \mu^2 \lambda^{1/4}} \sqrt{\Psi(n, \delta)}.
\end{align}
On the other hand, since $0 \leq r_1 \leq (\rho^*)^2$ and $0 \leq r_2' \leq (\rho^*)^2 / \mu^2$, it holds that
\[
    \left|r_1 - \frac{r_2 \mu^2 \|\btheta^*\|^2}{\mu^2 + \|\btheta^*\|^2} \right| \leq (\rho^*)^2 = \frac{\|\btheta^\circ\|^2}{\mu^2}
\]
and
\[
    \left|\frac{\|\btheta\|^2}{2\mu^2 + \|\btheta\|^2} - \frac{\|\btheta\|^2}{2 (\mu^2 + \|\btheta\|^2)} + \frac{r_2' \mu^4}{\mu^2 + \|\btheta\|^2} \right| \leq 2 (\rho^*)^2 = \frac{2\|\btheta^\circ\|^2}{\mu^2}.
\]
Then
\begin{align*}
    \left\| \Sigma^{1/2} \left( \frac12 \Sigma^2 + \lambda I_d \right)^{-1} \ttR \right\|
    &
    \leq \frac{\|\btheta^\circ\|^2}{\mu^2} \left\| \Sigma^{1/2} \left( \frac12 \Sigma^2 + \lambda I_d \right)^{-1} (A^*)^\top (\bZ - \E \bZ) \right\|
    \\&\quad
    + \frac{2\|\btheta^\circ\|^2}{\mu^2} \left\| \Sigma^{1/2} \left( \frac12 \Sigma^2 + \lambda I_d \right)^{-1} (A^*)^\top (\widehat \Sigma - \Sigma) \btheta^* \right\|.
\end{align*}
Let us consider the operator norm of 
\[
    \Sigma^{1/2} \left( \frac12 \Sigma^2 + \lambda I_d \right)^{-1} (A^*)^\top.
\]
Due to the triangle inequality, we have
\[
    \left\| \Sigma^{1/2} \left( \frac12 \Sigma^2 + \lambda I_d \right)^{-1} (A^*)^\top \right\|
    \leq \left\| \Sigma^{1/2} \left( \frac12 \Sigma^2 + \lambda I_d \right)^{-1} \Sigma \right\| + \left\| \Sigma^{1/2} \left( \frac12 \Sigma^2 + \lambda I_d \right)^{-1}\right\| \|\Sigma - A^*\|.
\]
Let us note that we have already bounded the operator norm of $\Sigma^{1/2} (\Sigma^2 / 2 + \lambda I_d )^{-1}$ (see \eqref{eq:young_inequality_op_norm_bound}). Similarly, using Young's inequality
\[
    ab \leq \frac{a^p}{p} + \frac{b^q}{q}
    \quad \text{with $a = \lambda_j(\Sigma)^{3/2}$, $b = (6 \lambda)^{1/4}$, $p = 4/3$, and $q = 4$,}
\]
we observe that
\begin{align*}
    \left\| \Sigma^{1/2} \left( \frac12 \Sigma^2 + \lambda I_d \right)^{-1} \Sigma \right\|
    = \max\limits_{1 \leq j \leq d} \left\{ \frac{\lambda_j(\Sigma)^{3/2}}{\lambda_j(\Sigma)^2 / 2 + \lambda} \right\}
    \leq \frac{3}{2 (6\lambda)^{1/4}}.
\end{align*}
Thus, it holds that 
\[
    \left\| \Sigma^{1/2} \left( \frac12 \Sigma^2 + \lambda I_d \right)^{-1} (A^*)^\top \right\|
    \leq \frac{3}{2 (6\lambda)^{1/4}} + \frac{(2\lambda/3)^{-3/4}}2 \|A^* - \Sigma\|.
\]
Moreover, Lemma \ref{lemma: Sigma bias optimal bound} and the condition \eqref{eq:stoch_term_mu_lambda_conditions} imply that
\[
    \frac{\|A^* - \Sigma\|}{\sqrt{\lambda}}
    \leq \left( \frac{14 \|\btheta^\circ\|}{\mu}\right)^2 + \frac{\|\bb_\lambda\|}{230 \mu}
    \leq \frac1{64} + \frac1{112 \cdot 230} < \frac1{63},
\]
and we obtain that
\begin{equation}
    \label{eq:young_inequality_op_norm_bound_2}
    \left\| \Sigma^{1/2} \left( \frac12 \Sigma^2 + \lambda I_d \right)^{-1} (A^*)^\top \right\|
    \leq \frac{3}{2 (6\lambda)^{1/4}} + \frac{(2\lambda/3)^{-3/4}}{126} \leq \lambda^{1/4}.
\end{equation}
This yields that
\begin{align*}
    \left\| \Sigma^{1/2} \left( \frac12 \Sigma^2 + \lambda I_d \right)^{-1} \ttR \right\|
    &
    \leq \frac{\|\btheta^\circ\|^2}{\mu^2} \left\| \Sigma^{1/2} \left( \frac12 \Sigma^2 + \lambda I_d \right)^{-1} (A^*)^\top (\bZ - \E \bZ) \right\|
    \\&\quad
    + \frac{2\|\btheta^\circ\|^2}{\mu^2} \left\| \Sigma^{1/2} \left( \frac12 \Sigma^2 + \lambda I_d \right)^{-1} (A^*)^\top (\widehat \Sigma - \Sigma) \btheta^* \right\|
    \\&
    \leq \frac{\|\btheta^\circ\|^2}{\mu^2 \lambda^{1/4}} \left( \|\bU\| + \|(\widehat\Sigma - \Sigma) \btheta^\circ\| + 2 \|(\widehat \Sigma - \Sigma) \btheta^*\| \right).
\end{align*}
Recall that, on the intersection of $\cE_1$ and $\cE_2^*$, we have (see \eqref{eq:e1_sigma}, \eqref{eq:e1_u}, and \eqref{eq:hat_sigma_linear_functional_1_deviation_bound})
\[
    \|(\widehat \Sigma - \Sigma) \btheta^\circ\| \leq 4 C_X \|\Sigma\| \|\btheta^\circ\| \sqrt{\frac{\ttr(\Sigma) + \log(4 / \delta)}n},
    \quad
    \|(\widehat \Sigma - \Sigma) \btheta^*\| \leq 4 C_X \|\Sigma\| \|\btheta^*\| \sqrt{\frac{\ttr(\Sigma) + \log(4 / \delta)}n},
\]
and
\[
    \|\bU\| \leq 8 \sigma \|\Sigma\|^{1/2} \sqrt{\frac{\ttr(\Sigma) + \log(4 / \delta)}n}.
\]
Taking into account that, due to Lemma \ref{lemma: localization lemma for bias}, $\|\btheta^*\| \leq \|\btheta^\circ\|$, we finally get that
\begin{align}
    \label{eq:remainder_1_r_bound}
    \left\| \Sigma^{1/2} \left( \frac12 \Sigma^2 + \lambda I_d \right)^{-1} \ttR \right\|
    &\notag
    \leq \frac{8 \|\btheta^\circ\|^3}{\mu^2 \lambda^{1/4}} \left( \frac{\sigma \|\Sigma\|^{1/2}}{\|\btheta^\circ\|} + \frac{3 C_X \|\Sigma\|}2 \right) \sqrt{\frac{\ttr(\Sigma) + \log(4 / \delta)}n}
    \\&
    \leq \frac{4 \|\btheta^\circ\|^3}{7 \mu^2 \lambda^{1/4}} \sqrt{\Psi(n, \delta)}.
\end{align}
The identity \eqref{eq:standardized_score_representation} and the inequalities \eqref{eq:q_bound} and \eqref{eq:remainder_1_r_bound} yield that
\begin{align*}
    \left\| H_{\btheta \chi}(\bups^*) H_{\chi \chi}^{-1}(\bups^*) \z_\chi
    + 0.5 \btau \right\|
    \leq \frac{\|\btheta^\circ\|^3}{5 \mu^2 \lambda^{1/4}} \sqrt{\Psi(n, \delta)} + \frac{4 \|\btheta^\circ\|^3}{7 \mu^2 \lambda^{1/4}} \sqrt{\Psi(n, \delta)}
    \leq \frac{\|\btheta^\circ\|^3}{\mu^2 \lambda^{1/4}} \sqrt{\Psi(n, \delta)}.
\end{align*}
\myendproof

\subsection{Proof of Lemma \ref{lem:semiparametric_remainder_2}}
\label{sec:lem_semiparametric_remainder_2_proof}

Note that, due to the definitions of $\btau$ and $\bzeta$, we have
\begin{align*}
    \btau - \bzeta
    &
    = (A^* - \Sigma) \bU
    + (A^*)^\top (\widehat \Sigma - \Sigma) (\btheta^* - \btheta^\circ - \bb_\lambda)
    \\&\quad
    - (A^* - \Sigma)^\top (\widehat \Sigma - \Sigma) \bb_\lambda
    - (\widehat \Sigma - \Sigma) (2 A^* \btheta^* - 2 \bfeta^* - \bb_\lambda).
\end{align*}
Then the triangle inequality yields that
\begin{align*}
    \left\|\Sigma^{1/2} \left( \Sigma^2 + 2 \lambda I_d \right)^{-1} (\btau - \bzeta) \right\|
    &
    \leq \left\|\Sigma^{1/2} \left( \Sigma^2 + 2 \lambda I_d \right)^{-1} \right\| \|A^* - \Sigma\| \left( \|\bU\| + \left\| (\widehat \Sigma - \Sigma) \bb_\lambda \right\| \right)
    \\&\quad
    + \left\|\Sigma^{1/2} \left( \Sigma^2 + 2 \lambda I_d \right)^{-1} (A^*)^\top \right\| \left\| (\widehat \Sigma - \Sigma) (\btheta^* - \btheta^\circ - \bb_\lambda) \right\|
    \\&\quad
    + \left\|\Sigma^{1/2} \left( \Sigma^2 + 2 \lambda I_d \right)^{-1} \right\| \left\| (\widehat \Sigma - \Sigma) (2 A^* \btheta^* - 2 \bfeta^* - \Sigma \bb_\lambda) \right\|.
\end{align*}
In view of \eqref{eq:young_inequality_op_norm_bound} and \eqref{eq:young_inequality_op_norm_bound_2}, we obtain that
\begin{align*}
    \left\|\Sigma^{1/2} \left( \Sigma^2 + 2 \lambda I_d \right)^{-1} (\btau - \bzeta) \right\|
    &
    \leq \frac{(2\lambda/3)^{-3/4}}4 \, \|A^* - \Sigma\| \left( \|\bU\| + \left\| (\widehat \Sigma - \Sigma) \bb_\lambda \right\| \right)
    \\&\quad
    + \frac1{2 \lambda^{1/4}} \left\| (\widehat \Sigma - \Sigma) (\btheta^* - \btheta^\circ - \bb_\lambda) \right\|
    \\&\quad
    + \frac{(2\lambda/3)^{-3/4}}2 \, \left\| (\widehat \Sigma - \Sigma) \left( A^* \btheta^* - \bfeta^* - \frac12 \Sigma \bb_\lambda \right) \right\|.
\end{align*}
According to Theorem \ref{th:covariance_concentration}, there exists an event $\cE_3$, $\p(\cE_3) \geq 1 - 3\delta / 2$ such that the following inequalities holds simultaneously on $\cE_3$:
\begin{align*}
    \left\| (\widehat \Sigma - \Sigma) \bb_\lambda \right\|
    &
    \leq 4 C_X \|\Sigma\|^{1/2} \|\Sigma^{1/2} \bb_\lambda\| \sqrt{\frac{\ttr(\Sigma) + \log(4 / \delta)}n},
    \\
    \left\| (\widehat \Sigma - \Sigma) (\btheta^* - \btheta^\circ - \bb_\lambda) \right\|
    &
    \leq 4 C_X \|\Sigma\|^{1/2} \|\Sigma^{1/2} (\btheta^* - \btheta^\circ - \bb_\lambda)\| \sqrt{\frac{\ttr(\Sigma) + \log(4 / \delta)}n},
    \\
    \left\| (\widehat \Sigma - \Sigma) \left( A^* \btheta^* - \bfeta^* - \frac12 \Sigma \bb_\lambda \right) \right\|
    &
    \leq 4 C_X \|\Sigma\| \left\| A^* \btheta^* - \bfeta^* - \frac12 \Sigma \bb_\lambda \right\| \sqrt{\frac{\ttr(\Sigma) + \log(4 / \delta)}n}.
\end{align*}
From now on, we restrict our attention on the intersection of $\cE_1$ and $\cE_3$. Since on $\cE_1$ (see \eqref{eq:e1_u})
\[
    \|\bU\| \leq 8 \sigma \|\Sigma\|^{1/2} \sqrt{\frac{\ttr(\Sigma) + \log(4 / \delta)}n},
\]
it holds that
\begin{align}
    \label{eq:tau_zeta_diff_triangle_inequality}
    &\notag
    \left\|\Sigma^{1/2} \left( \Sigma^2 + 2 \lambda I_d \right)^{-1} (\btau - \bzeta) \right\|
    \\&\notag
    \leq (2\lambda/3)^{-3/4} \|A^* - \Sigma\| \left( 2 \sigma \|\Sigma\|^{1/2} + C_X \|\Sigma\| \|\bb_\lambda\| \right) \sqrt{\frac{\ttr(\Sigma) + \log(4 / \delta)}n}
    \\&\quad
    + \frac{2 C_X}{\lambda^{1/4}} \|\Sigma\| \| \btheta^* - \btheta^\circ - \bb_\lambda \| \sqrt{\frac{\ttr(\Sigma) + \log(4 / \delta)}n}
    \\&\quad\notag
    + 2 (2\lambda/3)^{-3/4} C_X \|\Sigma\| \left\| A^* \btheta^* - \bfeta^* - \frac12 \Sigma \bb_\lambda \right\| \sqrt{\frac{\ttr(\Sigma) + \log(4 / \delta)}n}.
\end{align}
It remains to bound the term in the right-hand side one by one. Applying Lemma \ref{lemma: Sigma bias optimal bound} and taking into account the inequalities \eqref{eq:stoch_term_mu_lambda_conditions} and $\|\bb_\lambda\| \leq \|\btheta^\circ\|$, we obtain that
\begin{align*}
    &
    (2\lambda/3)^{-3/4} \, \|A^* - \Sigma\| \left( 8 \sigma \|\Sigma\|^{1/2} + C_X \|\Sigma\| \|\bb_\lambda\| \right) \sqrt{\frac{\ttr(\Sigma) + \log(4 / \delta)}n}
    \\&
    \leq (2\lambda/3)^{-3/4} \cdot \sqrt{\lambda} \left( \left( \frac{14 \Vert \btheta^\circ \Vert}{\mu} \right )^2 + \frac{\Vert \bb_\lambda \Vert }{230 \mu} \right) \left( 2 \sigma \|\Sigma\|^{1/2} + C_X \|\Sigma\| \|\btheta^\circ\| \right) \sqrt{\frac{\ttr(\Sigma) + \log(4 / \delta)}n}.
\end{align*}
Due to Lemma \ref{lem:a_theta_eta_expansion}, it holds that
\begin{align*}
    &
    2 (2\lambda/3)^{-3/4} C_X \|\Sigma\| \left\| A^* \btheta^* - \bfeta^* - \frac12 \Sigma \bb_\lambda \right\| \sqrt{\frac{\ttr(\Sigma) + \log(4 / \delta)}n}
    \\&
    \leq \frac{6 \cdot (3/2)^{3/4}}{\lambda^{1/4}} C_X \|\Sigma\| \|\btheta^\circ\| \left( \left( \frac{14 \Vert \btheta^\circ \Vert}{\mu} \right )^2 + \frac{\Vert \bb_\lambda \Vert }{230 \mu} \right) \sqrt{\frac{\ttr(\Sigma) + \log(4 / \delta)}n}.
\end{align*}
Finally, we use the following result to bound the second term in the right-hand side of \eqref{eq:tau_zeta_diff_triangle_inequality}.
\begin{Lem}
    \label{lem:bias_expansion_rough_bound}
    Assume that $\|\btheta^\circ\| \leq \mu / 49$ and $\|\Sigma\| \|\btheta^\circ\| \leq \mu \sqrt{\lambda} / 24$. Then it holds that
    \[
        \left\|\btheta^* - \btheta^\circ - \bb_\lambda \right\| \leq \frac{7 \|\btheta^\circ\|}2 \left( \left(\frac{14 \|\btheta^\circ\|}\mu \right)^2 + \frac{70 \|\Sigma\| \|\btheta^\circ\| \|\bb_\lambda\|}{\mu^2 \sqrt{\lambda}} \right).
    \]
\end{Lem}
We provide the proof of Lemma \ref{lem:bias_expansion_rough_bound} in Appendix \ref{sec:lem_bias_expansion_rough_bound_proof}. It immediately implies that
\begin{align*}
    &
    \frac{2 C_X}{\lambda^{1/4}} \|\Sigma\| \| \btheta^* - \btheta^\circ - \bb_\lambda \| \sqrt{\frac{\ttr(\Sigma) + \log(4 / \delta)}n}
    \\&
    \leq \frac{7 C_X \|\Sigma\| \|\btheta^\circ\|}{\lambda^{1/4}} \left( \left(\frac{14 \|\btheta^\circ\|}\mu \right)^2 + \frac{70 \|\Sigma\| \|\btheta^\circ\| \|\bb_\lambda\|}{\mu^2 \sqrt{\lambda}} \right) \sqrt{\frac{\ttr(\Sigma) + \log(4 / \delta)}n}
    \\&
    \leq \frac{7 C_X \|\Sigma\| \|\btheta^\circ\|}{\lambda^{1/4}} \left( \left(\frac{14 \|\btheta^\circ\|}\mu \right)^2 + \frac{\|\bb_\lambda\|}{230 \mu} \right) \sqrt{\frac{\ttr(\Sigma) + \log(4 / \delta)}n}.
\end{align*}
Hence, we obtain that
\begin{align*}
    &
    \left\|\Sigma^{1/2} \left( \Sigma^2 + 2 \lambda I_d \right)^{-1} (\btau - \bzeta) \right\|
    \\&
    \leq \frac{2 \|\btheta^\circ\|}{(2/3)^{3/4} \lambda^{1/4}} \left( \frac{\sigma \|\Sigma\|^{1/2}}{\|\btheta^\circ\|} + \left( 1 + 3 + \frac{7}{2 \cdot (3/2)^{3/4}} \right) C_X \|\Sigma\| \right)
    \\&\quad
    \cdot \left( \left(\frac{14 \|\btheta^\circ\|}\mu \right)^2 + \frac{\|\bb_\lambda\|}{230 \mu} \right) \sqrt{\frac{\ttr(\Sigma) + \log(4 / \delta)}n}
    \\&
    \leq \frac{2 \|\btheta^\circ\|}{(2/3)^{3/4} \lambda^{1/4}} \left( \frac{\sigma \|\Sigma\|^{1/2}}{\|\btheta^\circ\|} + 7 C_X \|\Sigma\| \right) \left( \left(\frac{14 \|\btheta^\circ\|}\mu \right)^2 + \frac{\|\bb_\lambda\|}{230 \mu} \right) \sqrt{\frac{\ttr(\Sigma) + \log(4 / \delta)}n}
    \\&
    \leq \frac{7 \|\btheta^\circ\|}{(2/3)^{3/4} \lambda^{1/4}} \left( \frac{\sigma \|\Sigma\|^{1/2}}{\|\btheta^\circ\|} + 2 C_X \|\Sigma\| \right) \left( \left(\frac{14 \|\btheta^\circ\|}\mu \right)^2 + \frac{\|\bb_\lambda\|}{230 \mu} \right) \sqrt{\frac{\ttr(\Sigma) + \log(4 / \delta)}n}.
\end{align*}
It remains to note that
\[
    \left(\frac{14 \|\btheta^\circ\|}\mu \right)^2 + \frac{\|\bb_\lambda\|}{230 \mu}
    \leq \frac1{\sqrt{2}} \left( \left(\frac{17 \|\btheta^\circ\|}\mu \right)^2 + \frac{\|\bb_\lambda\|}{160 \mu} \right)
\]
and
\[
    \left( \frac{\sigma \|\Sigma\|^{1/2}}{\|\btheta^\circ\|} + 2 C_X \|\Sigma\| \right) \sqrt{\frac{\ttr(\Sigma) + \log(4 / \delta)}n}
    \leq \frac{\sqrt{\Psi(n, \delta)}}{14}.
\]
Then it holds that
\begin{align*}
    \left\|\Sigma^{1/2} \left( \Sigma^2 + 2 \lambda I_d \right)^{-1} (\btau - \bzeta) \right\|
    &
    \leq \frac{(3/2)^{3/4} \|\btheta^\circ\|}{2 \sqrt{2} \lambda^{1/4}} \left( \left(\frac{17 \|\btheta^\circ\|}\mu \right)^2 + \frac{\|\bb_\lambda\|}{160 \mu} \right) \sqrt{\Psi(n, \delta)}
    \\&
    \leq \frac{12 \|\btheta^\circ\|}{25 \lambda^{1/4}} \left( \left(\frac{17 \|\btheta^\circ\|}\mu \right)^2 + \frac{\|\bb_\lambda\|}{160 \mu} \right) \sqrt{\Psi(n, \delta)}.
\end{align*}
\myendproof

\subsection{Auxiliary results}

\subsubsection{Proof of Lemma \ref{lem:h_theta_schur}}
\label{sec:lem_h_theta_schur_proof}

Applying Proposition \ref{proposition: inverse of nuisance parameters}, we obtain that
\begin{align}
    \label{eq:h_theta_schur_complement_decomposition}
    \breve H_{\btheta \btheta} - \left( \frac12 A^\top A + \lambda I_d \right)
    &
    = \left[ -r_1 + \frac{2 r_2 \mu^2 \|\btheta\|^2}{\mu^2 + \|\btheta\|^2} - \frac{\|\btheta\|^2}{\mu^2}\left(1 + r_3' \mu^2 \|\btheta\|^2\right) \right] A^\top A
    \\&\quad\notag
    - \frac1{\mu^2} \|A \btheta - \bfeta\|^2 I_d
    + \frac{r_2 \mu^2}{\mu^2 + \|\btheta\|^2} \left( \btheta (A \btheta - \bfeta)^\top A + A^\top (A \btheta - \bfeta) \btheta^\top \right),
\end{align}
where
\[
    0 \leq r_1 \leq \rho^2,
    \quad
    0 \leq r_2 \leq \frac1{\mu^2},
    \quad \text{and} \quad
    0 \leq r_3' \leq \frac{2}{\mu^4}.
\]
In the rest of the proof, we study the terms in the right-hand side of \eqref{eq:h_theta_schur_complement_decomposition} one by one. First, let us note that the inequalities $\|\btheta\|^2 \leq \rho^2 (\mu^2 + \|\btheta\|^2)$ and $\rho^2 \leq 1/2$ yield that
\begin{align*}
    \left| -r_1 + \frac{2 r_2 \mu^2 \|\btheta\|^2}{\mu^2 + \|\btheta\|^2} - \frac{\|\btheta\|^2}{\mu^2}\left(1 + r_3' \mu^2 \|\btheta\|^2\right) \right|
    &
    \leq \rho^2 + \frac{2 \|\btheta\|^2}{\mu^2 + \|\btheta\|^2} + \frac{\|\btheta\|^2 (\mu^2 + 2 \|\btheta\|^2)}{\mu^4}
    \\&
    \leq \rho^2 + 2 \rho^2 + 2 \rho^4
    \leq 3 \rho^2 + \frac{2\rho^2}{25}.
\end{align*}
Then the first term in the right-hand side of \eqref{eq:h_theta_schur_complement_decomposition} does not exceed
\[
    \left( 3 \rho^2 + \frac{2\rho^2}{25} \right) \left\| \left( \frac12 A^\top A + \lambda I_d \right)^{-1/2} A^\top A \left( \frac12 A^\top A + \lambda I_d \right)^{-1/2} \right\|
    \leq 6 \rho^2 + 2 \rho^4
    \leq 6 \rho^2 + \frac{4 \rho^2}{25}.
\]
Second, it holds that
\[
    \frac{\|A \btheta - \bfeta\|^2}{\mu^2} \left\| \left( \frac12 A^\top A + \lambda I_d \right)^{-1} \right\|
    \leq \rho^2 \lambda \cdot \frac1{\lambda}
    = \rho^2.
\]
Finally, let us consider the latter term in the right-hand side of \eqref{eq:h_theta_schur_complement_decomposition}. The triangle inequality and the submultiplicativity of the operator norm imply that
\begin{align*}
    &
    \frac{r_2 \mu^2}{\mu^2 + \|\btheta\|^2} \left\| \left(\frac12 A^\top A + \lambda I_d \right)^{-1/2} \left( \btheta (A \btheta - \bfeta)^\top A + A^\top (A \btheta - \bfeta) \btheta^\top \right) \left(\frac12 A^\top A + \lambda I_d \right)^{-1/2} \right\|
    \\&
    \leq \frac{2}{\mu^2 + \|\btheta\|^2} \left\| \left(\frac12 A^\top A + \lambda I_d \right)^{-1/2} \btheta \right\| \left\| A^\top \left(\frac12 A^\top A + \lambda I_d \right)^{-1/2} \right\| \left\|A \btheta - \bfeta\right\|
    \\&
    \leq \frac{2}{\mu^2 + \|\btheta\|^2} \cdot \frac{\|\btheta\|}{\sqrt{\lambda}} \cdot \sqrt{2} \cdot \rho \mu \sqrt{\lambda} 
    \leq \frac{2 \sqrt{2} \rho \mu}{\sqrt{\mu^2 + \|\btheta\|^2}} \cdot \frac{\|\btheta\|}{\sqrt{\mu^2 + \|\btheta\|^2}}
    \leq \frac{2 \sqrt{2} \rho \mu}{\mu} \cdot \rho = 2\sqrt{2} \rho^2.
\end{align*}
Hence, it holds that
\[
    \left\| \left( \frac12 A^\top A + \lambda I_d \right)^{-1/2} \breve H_{\btheta \btheta} \left( \frac12 A^\top A + \lambda I_d \right)^{-1/2} - I_d \right\|
    \leq 6\rho^2 + \frac{4 \rho^2}{25} + \rho^2 + 2\sqrt{2} \rho^2
    \leq 10 \rho^2.
\]
\myendproof

\subsubsection{Proof of Lemma \ref{lem:semiparametric_score_explicit_representation}}
\label{sec:lem_semiparametric_score_explicit_representation}

The proof of Lemma \ref{lem:semiparametric_score_explicit_representation} relies on the expression \eqref{eq:h_block_form_kronecker} for the Hessian of $\nabla^2 \cL(\bups)$ and Proposition \ref{proposition: inverse of nuisance parameters}.
Let us denote the inverse of $H_{\chi \chi}$ by $J$.
It holds that
\begin{align*}
    H_{\btheta \chi} J \bz_\chi
    &
    =
    \begin{pmatrix}
        H_{\btheta \bfeta} & H_{\btheta A}
    \end{pmatrix}
    \begin{pmatrix}
        J_{\bfeta \bfeta} & J_{\bfeta A} \\ J_{A \bfeta} & J_{AA}
    \end{pmatrix}
    \begin{pmatrix} \bz_{\bfeta} \\ \bz_A \end{pmatrix}
    \\&
    =
    H_{\btheta \bfeta} J_{\bfeta \bfeta} \bz_{\bfeta} + H_{\btheta \bfeta} J_{\bfeta A} \bz_A
    + H_{\btheta A} J_{A \bfeta} \bz_{\bfeta} + H_{\btheta A} J_{AA} \bz_A.
\end{align*}
In the rest of the proof, we study the terms in the right-hand-side one by one. Using Proposition \ref{proposition: inverse of nuisance parameters}, we obtain that
\begin{align*}
    H_{\btheta A} J_{A \bfeta}
    &
    = r_2 \left( (A \btheta - \bfeta)^\top \otimes I_d + A^\top \otimes \btheta^\top \right) \left( I_d \otimes \Tilde{\btheta} \right)
    \\&
    = r_2 (A \btheta - \bfeta)^\top \otimes \Tilde{\btheta}
    + r_2 A^\top \otimes \left( \btheta^\top \Tilde{\btheta} \right)
    \\&
    = r_2 \Tilde{\btheta} (A \btheta - \bfeta)^\top
    + r_2 \left( \btheta^\top \Tilde{\btheta} \right) A^\top,
\end{align*}
where $\Tilde{\btheta} = \mu^2 (\mu^2 I_d + \btheta \btheta^\top)^{-1} \btheta$. The Woodbury matrix identity yields that
\[
    \Tilde{\btheta} = \frac{\mu^2}{\mu^2 + \|\btheta\|^2} \btheta
    \quad \text{and} \quad
    \btheta^\top \Tilde{\btheta} = \frac{\mu^2 \|\btheta\|^2}{\mu^2 + \|\btheta\|^2}.
\]
Then we have
\begin{equation}
    \label{eq:h_thetaA_J_Aeta_z_eta}
    H_{\btheta A} J_{A \bfeta} \z_{\bfeta}
    = \frac{r_2 \mu^2}{\mu^2 + \|\btheta\|^2} \left( (A \btheta - \bfeta)^\top (\bZ - \E \bZ) \right) \btheta
    + \frac{r_2 \mu^2 \|\btheta\|^2}{\mu^2 + \|\btheta\|^2} A^\top (\bZ - \E \bZ).
\end{equation}
Similarly, we note that
\begin{equation}
    \label{eq:h_thetaeta_J_etaeta_z_eta}
    H_{\btheta \bfeta} J_{\bfeta \bfeta} \z_{\bfeta}
    = -A^\top J_{\bfeta \bfeta} \z_{\bfeta}
    = -\left(\frac12 + r_1 \right) A^\top (\bZ - \E \bZ).
\end{equation}
Due to the properties of the Kronecker product and the vectorization operator (see, for instance, \citep[Appendix A, eq. (25)]{puchkin24}), it holds that
\[
    (I_d \otimes \Tilde{\btheta}\,^\top) \rmvec(\widehat \Sigma - \Sigma)
    = \rmvec\left( I_d (\widehat \Sigma - \Sigma) \Tilde{\btheta} \right)
    = (\widehat \Sigma - \Sigma) \Tilde{\btheta}.
\]
Thus, we obtain that
\begin{equation}
    \label{eq:h_thetaeta_J_etaA_z_A}
    H_{\btheta \bfeta} J_{\bfeta A} \z_A
    = -r_2 \mu^2 A^\top (\widehat \Sigma - \Sigma) \Tilde{\btheta}
    = -\frac{r_2 \mu^4}{\mu^2 + \|\btheta\|^2} A^\top (\widehat \Sigma - \Sigma) \btheta.
\end{equation}
It remains to consider $H_{\btheta A} J_{AA} \z_A$.
Let us focus on $J_{AA}$. Note that, due to \eqref{eq:h_block_form_kronecker}, the matrix $H_{\chi \chi}$ can be represented in the following compact form using the Kronecker product:
\[
    H_{\chi \chi} =
    \begin{pmatrix}
        2 I_d & -I_d \otimes \btheta^\top \\
        -I_d \otimes \btheta & I_d \otimes (\mu^2 I_d + \btheta \btheta^\top)
    \end{pmatrix}.
\]
Then, according to the block-inversion formula (see \eqref{eq:h_block_inversion} in Appendix \ref{sec:hessian}), it holds that
\begin{align*}
    J_{AA}
    &
    = \left( I_d \otimes (\mu^2 I_d + \btheta \btheta^\top) - \frac12 \left( I_d \otimes \btheta^\top \right) \left( I_d \otimes \btheta \right) \right)^{-1}
    \\&
    = \left( I_d \otimes \left(\mu^2 I_d + \frac12 \btheta \btheta^\top \right) \right)^{-1}
    = I_d \otimes \left(\mu^2 I_d + \frac12 \btheta \btheta^\top \right)^{-1}.
\end{align*}
Using the Woodbury matrix identity, we obtain that
\[
    J_{AA}
    = I_d \otimes \left( \frac1{\mu^2} I_d - \frac{1}{\mu^2 (2\mu^2 + \|\btheta\|^2)} \btheta \btheta^\top \right).
\]
Since $\z_A = \mu^2 \rmvec(\widehat \Sigma - \Sigma)$, we have
\begin{align}
    \label{eq:h_thetaA_J_AA_z_A}
    H_{\btheta A} J_{AA} \z_A
    &\notag
    = \mu^2 \left( (A \btheta - \bfeta)^\top \otimes I_d \right) \left( I_d \otimes \left( \frac1{\mu^2} I_d - \frac{1}{\mu^2 (2\mu^2 + \|\btheta\|^2)} \btheta \btheta^\top \right) \right) \rmvec\left( \widehat \Sigma - \Sigma \right)
    \\&\quad\notag
    + \mu^2 \left( A^\top \otimes \btheta^\top \right) \left( I_d \otimes \left( \frac1{\mu^2} I_d - \frac{1}{\mu^2 (2\mu^2 + \|\btheta\|^2)} \btheta \btheta^\top \right) \right) \rmvec\left( \widehat \Sigma - \Sigma \right)
    \\&
    = \left( (A \btheta - \bfeta)^\top \otimes \left(I_d - \frac{1}{2\mu^2 + \|\btheta\|^2} \btheta \btheta^\top \right) \right) \rmvec\left( \widehat \Sigma - \Sigma \right)
    \\&\quad\notag
    + \left( A^\top \otimes \btheta^\top \left( I_d - \frac{1}{2\mu^2 + \|\btheta\|^2} \btheta \btheta^\top \right)\right) \rmvec\left( \widehat \Sigma - \Sigma \right).
\end{align}
Using the properties of the Kronecker product and the vectorization operator $\rmvec$ one more time, we obtain that
\begin{align}
    \label{eq:h_thetaA_J_AA_z_A_first_term}
    &\notag
    \left( (A \btheta - \bfeta)^\top \otimes \left(I_d - \frac{1}{2\mu^2 + \|\btheta\|^2} \btheta \btheta^\top \right) \right) \rmvec\left( \widehat \Sigma - \Sigma \right)
    \\&
    = \rmvec\left( (A \btheta - \bfeta)^\top (\widehat \Sigma - \Sigma) \left(I_d - \frac{1}{2\mu^2 + \|\btheta\|^2} \btheta \btheta^\top \right) \right)
    \\&\notag
    = \left(I_d - \frac{1}{2\mu^2 + \|\btheta\|^2} \btheta \btheta^\top \right) (\widehat \Sigma - \Sigma) (A \btheta - \bfeta)
\end{align}
and
\begin{align}
    \label{eq:h_thetaA_J_AA_z_A_second_term}
    \left( A^\top \otimes \btheta^\top \left( I_d - \frac{1}{2\mu^2 + \|\btheta\|^2} \btheta \btheta^\top \right)\right) \rmvec\left( \widehat \Sigma - \Sigma \right)
    &\notag
    = \frac{2\mu^2}{2\mu^2 + \|\btheta\|^2}\left( A^\top \otimes \btheta^\top \right) \rmvec\left( \widehat \Sigma - \Sigma \right)
    \\&
    = \frac{2\mu^2}{2\mu^2 + \|\btheta\|^2} \rmvec\left( A^\top (\widehat \Sigma - \Sigma) \btheta \right)
    \\&\notag
    = \frac{2\mu^2}{2\mu^2 + \|\btheta\|^2} A^\top (\widehat \Sigma - \Sigma) \btheta.
\end{align}
Substituting the equalities \eqref{eq:h_thetaA_J_AA_z_A_first_term} and \eqref{eq:h_thetaA_J_AA_z_A_second_term} into \eqref{eq:h_thetaA_J_AA_z_A}, we obtain that
\begin{align}
    \label{eq:h_thetaA_J_AA_z_A_simplified}
    H_{\btheta A} J_{AA} \z_A
    &\notag
    = (\widehat \Sigma - \Sigma) (A \btheta - \bfeta) - \frac{1}{2\mu^2 + \|\btheta\|^2} \big( (A \btheta - \bfeta)^\top (\widehat \Sigma - \Sigma) \btheta \big) \btheta
    \\&\quad
    + \frac{2\mu^2}{2\mu^2 + \|\btheta\|^2} A^\top (\widehat \Sigma - \Sigma) \btheta.
\end{align}
Summing up \eqref{eq:h_thetaA_J_Aeta_z_eta}, \eqref{eq:h_thetaeta_J_etaeta_z_eta}, \eqref{eq:h_thetaeta_J_etaA_z_A}, and \eqref{eq:h_thetaA_J_AA_z_A_simplified}, we conclude that
\begin{align*}
    H_{\btheta \chi} J \z_\chi
    &
    = -\left(\frac12 + r_1 - \frac{r_2 \mu^2 \|\btheta\|^2}{\mu^2 + \|\btheta\|^2} \right) A^\top (\bZ - \E \bZ)
    + \frac{r_2 \mu^2}{\mu^2 + \|\btheta\|^2} \left( (A \btheta - \bfeta)^\top (\bZ - \E \bZ) \right) \btheta
    \\&\quad
    + \left( \frac{2\mu^2}{2\mu^2 + \|\btheta\|^2} - \frac{r_2 \mu^4}{\mu^2 + \|\btheta\|^2} \right) A^\top (\widehat \Sigma - \Sigma) \btheta
    - \frac{1}{2\mu^2 + \|\btheta\|^2} \big( (A \btheta - \bfeta)^\top (\widehat \Sigma - \Sigma) \btheta \big) \btheta
    \\&\quad
    + (\widehat \Sigma - \Sigma) (A \btheta - \bfeta).
\end{align*}
Due to the definition of $r_2$ (see Proposition \ref{proposition: inverse of nuisance parameters}), it holds that $r_2 = 0.5 \mu^{-2} + r_2'$, where $0 \leq r_2' \leq \rho^2 / \mu^2$. Using this identity, we can simplify the coefficient before $A^\top (\widehat \Sigma - \Sigma) \btheta$:
\begin{align*}
    \frac{2\mu^2}{2\mu^2 + \|\btheta\|^2} - \frac{r_2 \mu^4}{\mu^2 + \|\btheta\|^2}
    &
    = 1 - \frac{\|\btheta\|^2}{2\mu^2 + \|\btheta\|^2} - \frac{\mu^2}{2 (\mu^2 + \|\btheta\|^2)} - \frac{r_2' \mu^4}{\mu^2 + \|\btheta\|^2}
    \\&
    = 1 - \frac{\|\btheta\|^2}{2\mu^2 + \|\btheta\|^2} - \frac12 + \frac{\|\btheta\|^2}{2 (\mu^2 + \|\btheta\|^2)} - \frac{r_2' \mu^4}{\mu^2 + \|\btheta\|^2}
    \\&
    = \frac12 - \frac{\|\btheta\|^2}{2\mu^2 + \|\btheta\|^2} + \frac{\|\btheta\|^2}{2 (\mu^2 + \|\btheta\|^2)} - \frac{r_2' \mu^4}{\mu^2 + \|\btheta\|^2}.
\end{align*}
Finally, taking into account the representation $\bZ - \E \bZ = \bU + (\widehat \Sigma - \Sigma) \btheta^\circ$, we obtain that
\begin{align*}
    H_{\btheta \chi} J \z_\chi
    &
    = -\frac12 A^\top \bU + \frac12 A^\top (\widehat \Sigma - \Sigma)(\btheta - \btheta^\circ) + (\widehat \Sigma - \Sigma)(A \btheta - \bfeta)
    \\&\quad
    - \left(r_1 - \frac{r_2 \mu^2 \|\btheta\|^2}{\mu^2 + \|\btheta\|^2} \right) A^\top (\bZ - \E \bZ) + \frac{r_2 \mu^2}{\mu^2 + \|\btheta\|^2} \left( (A \btheta - \bfeta)^\top (\bZ - \E \bZ) \right) \btheta
    \\&\quad
    - \left(\frac{\|\btheta\|^2}{2\mu^2 + \|\btheta\|^2} - \frac{\|\btheta\|^2}{2 (\mu^2 + \|\btheta\|^2)} + \frac{r_2' \mu^4}{\mu^2 + \|\btheta\|^2} \right) A^\top (\widehat \Sigma - \Sigma) \btheta
    \\&\quad
    - \frac{1}{2\mu^2 + \|\btheta\|^2} \big( (A \btheta - \bfeta)^\top (\widehat \Sigma - \Sigma) \btheta \big) \btheta.
\end{align*}
\myendproof

\subsubsection{Proof of Lemma \ref{lem:diff_squares_op_norm_bound}}
\label{sec:lem_diff_squares_op_norm_bound_proof}

\noindent
\textbf{Step 1: a bound on the operator norm.}
\quad
Let us start with an upper bound on the operator norm of
\[
    \frac12 \left(\frac12 \Sigma^2 + \lambda I_d \right)^{-1/2} \left( A^\top A - \Sigma^2 \right) \left(\frac12 \Sigma^2 + \lambda I_d \right)^{-1/2}.
\]
Due to the definition of $S$ and the triangle inequality, it holds that
\begin{align*}
    &
    \frac12 \left\| \left(\frac12 \Sigma^2 + \lambda I_d \right)^{-1/2} \left( A^\top A - \Sigma^2 \right) \left(\frac12 \Sigma^2 + \lambda I_d \right)^{-1/2} \right\|
    \\&
    \leq \left\| \left(\frac12 \Sigma^2 + \lambda I_d \right)^{-1/2} \Sigma (A - \Sigma) \left(\frac12 \Sigma^2 + \lambda I_d \right)^{-1/2} \right\|
    \\&\quad
    + \frac12 \left\| \left(\frac12 \Sigma^2 + \lambda I_d \right)^{-1/2} (A^\top - \Sigma)(A - \Sigma) \left(\frac12 \Sigma^2 + \lambda I_d \right)^{-1/2} \right\|.
\end{align*}
Let us note that
\begin{equation}
    \label{eq:sigma_inequalities}
    \left\| \left(\frac12 \Sigma^2 + \lambda I_d \right)^{-1/2} \Sigma \right\| \leq \sqrt{2}
    \quad \text{and} \quad
    \left\| \left(\frac12 \Sigma^2 + \lambda I_d \right)^{-1/2} \right\| \leq \frac1{\sqrt{\lambda}}.
\end{equation}
Taking into account submultiplicativity of the operator norm, we obtain that
\[
    \left\| \left(\frac12 \Sigma^2 + \lambda I_d \right)^{-1/2} \Sigma (A - \Sigma) \left(\frac12 \Sigma^2 + \lambda I_d \right)^{-1/2} \right\|
    \leq \|A - \Sigma\| \sqrt{\frac{2}{\lambda}}
\]
and
\[
   \frac12 \left\| \left(\frac12 \Sigma^2 + \lambda I_d \right)^{-1/2} (A^\top - \Sigma)(A - \Sigma) \left(\frac12 \Sigma^2 + \lambda I_d \right)^{-1/2} \right\|
   \leq \frac{\|A - \Sigma\|^2}{2 \lambda}.
\]
Hence, it holds that
\begin{align*}
    \frac12 \left\| \left(\frac12 \Sigma^2 + \lambda I_d \right)^{-1/2} \left( A^\top A - \Sigma^2 \right) \left(\frac12 \Sigma^2 + \lambda I_d \right)^{-1/2} \right\|
    \leq \|A - \Sigma\| \sqrt{\frac{2}{\lambda}} + \frac{\|A - \Sigma\|^2}{2 \lambda}.
\end{align*}
\myendproof

\subsubsection{Proof of Lemma \ref{lem:bias_expansion_rough_bound}}
\label{sec:lem_bias_expansion_rough_bound_proof}

The proof relies on the identity
\[
    \left( \frac12 (A^*)^\top A^* + \lambda I_d \right) (\btheta^* - \btheta^\circ)
    = -\lambda \btheta^\circ - \frac12 (A^*)^\top (A^* - \Sigma) \btheta^\circ
\]
following from the first-order optimality condition $\bnabla_{\btheta} \cL(\btheta^*, \bfeta^*, A^*) = \bzero$ (see \eqref{eq:first-order_optimality_corollary} in the proof of Lemma \ref{lemma: localization lemma for bias} for the details).
Let us introduce
\[
    B = \left( \frac12 \Sigma^2 + \lambda I_d \right)^{-1/2} \left( (A^*)^\top A^* - \Sigma^2 \right) \left( \frac12 \Sigma^2 + \lambda I_d \right)^{-1/2}.
\]
Then we can rewrite the difference $(\btheta^* - \btheta^\circ - \bb_\lambda)$ in the following form:
\begin{align*}
    \btheta^* - \btheta^\circ - \bb_\lambda
    &
    = \lambda \left( \frac12 \Sigma^2 + \lambda I_d \right) \btheta^\circ - \lambda \left( \frac12 (A^*)^\top A^* + \lambda I_d \right)^{-1} \btheta^\circ
    \\&\quad
    - \frac12 \left( \frac12 (A^*)^\top A^* + \lambda I_d \right)^{-1} (A^*)^\top (A^* - \Sigma) \btheta^\circ
    \\&
    = \lambda \left( \frac12 \Sigma^2 + \lambda I_d \right)^{-1/2} \left( I_d - (I_d + B)^{-1} \right) \left( \frac12 \Sigma^2 + \lambda I_d \right)^{-1/2} \btheta^\circ 
    \\&\quad
    - \frac12 \left( \frac12 (A^*)^\top A^* + \lambda I_d \right)^{-1} (A^*)^\top (A^* - \Sigma) \btheta^\circ.
\end{align*}
The inequalities
\[
    \left\| \left(\frac12 \Sigma^2 + \lambda I_d \right)^{-1/2} \right\| \leq \frac{1}{\sqrt{\lambda}},
    \quad
    \left\| \left( \frac12 (A^*)^\top A^* + \lambda I_d \right)^{-1} (A^*)^\top \right\| \leq \frac{1}{\sqrt{2\lambda}},
\]
and
\[
    \left\| I_d - (I_d + B)^{-1} \right\|
    = \left\| B (I_d + B)^{-1} \right\|
    \leq \frac{\|B\|}{1 - \|B\|}
\]
yield that
\[
    \left\| \btheta^* - \btheta^\circ - \bb_\lambda \right\|
    \leq \left( \frac{\|B\|}{1 - \|B\|} + \frac{\|A^* - \Sigma\|}{2 \sqrt{2 \lambda}} \right) \|\btheta^\circ\|.
\]
Similarly to the proof of Lemma \ref{lem:a_theta_eta_expansion}, we deduce that (see \eqref{eq:b_bound})
\[
    \|B\|
    \leq \frac{2 \|A^* - \Sigma\|}{\sqrt{\lambda}}
    \leq 2 \left(\frac{14 \|\btheta^\circ\|}\mu \right)^2 + \frac{70 \|\Sigma\| \|\btheta^\circ\| \|\bb_\lambda\|}{\mu^2 \sqrt{\lambda}} \leq \frac14.
\]
Then it holds that
\begin{align*}
    \left\| \btheta^* - \btheta^\circ - \bb_\lambda \right\|
    &
    \leq \left( \frac{4\|B\|}3 + \frac{\|A^* - \Sigma\|}{2 \sqrt{2 \lambda}} \right) \|\btheta^\circ\|
    \\&
    \leq \left( \frac{8 \|A^* - \Sigma\|}{3 \sqrt{\lambda}}  + \frac{\|A^* - \Sigma\|}{2 \sqrt{2 \lambda}} \right) \|\btheta^\circ\|
    \\&
    \leq \frac{7 \|\btheta^\circ\|}2 \cdot \frac{\|A^* - \Sigma\|}{\sqrt{\lambda}}
    \\&
    \leq \frac{7 \|\btheta^\circ\|}2 \left( \left(\frac{14 \|\btheta^\circ\|}\mu \right)^2 + \frac{70 \|\Sigma\| \|\btheta^\circ\| \|\bb_\lambda\|}{\mu^2 \sqrt{\lambda}} \right).
\end{align*}

\section{Concentration inequalities}
\label{sec:concentration_inequalities}

\subsection{Proof of Theorem \ref{th:covariance_concentration}}

Following the ideas of \cite{zhivotovskiy21}, we use the tools of PAC-Bayesian analysis to prove a large deviation inequality. The core of this approach is the following variational inequality (see, e. g., \citep[Proposition 2.1]{catoni17}). 

\begin{Lem}
    \label{lem:pac-bayes}
    Let $ \bX, \bX_1, \dots, \bX_n$ be i.i.d. random elements on a measurable space $\cX$. Let $\Theta$ be a parameter space equipped with a measure $\mu$ (which is also referred to as prior). Let $f : \cX \times \Theta \rightarrow \R$. Then, with probability at least $1 - \delta$, it holds that
    \[
        \E_{\btheta \sim \rho} \frac1n \sum\limits_{i = 1}^n f(\bX_i, \btheta)
        \leq \E_{\btheta \sim \rho} \log \E_\bX e^{f(X, \btheta)} + \frac{\KL(\rho, \mu) + \log(1 / \delta)}n
    \]
    simultaneously for all $\rho \ll \mu$.
\end{Lem}

In what follows, we show that a high probability upper bound on $\|B(\widehat \Sigma - \Sigma)\|_{\F}$ can be deduced from the PAC-Bayes variational inequality. 
We split the proof into several steps for convenience.

\bigskip

\noindent{\bf Step 0: auxiliary constants}.
\quad
Before we move to the proof, let us introduce auxiliary constants. We define
\begin{equation}
    \label{eq:beta}
    \beta = 2 \, \ttr(\Sigma^{1/2} A^\top A \Sigma^{1/2}) \, \ttr(\Sigma^{1/2} B^\top B \Sigma^{1/2}),
\end{equation}
\begin{equation}
    \label{eq:R}
    R = \sqrt{\frac2\beta \, \Tr(\Sigma^{1/2} A^\top A \Sigma^{1/2}) \, \Tr(\Sigma^{1/2} B^\top B \Sigma^{1/2})},
\end{equation}
and
\begin{align}
    \label{eq:lambda}
    \lambda
    &\notag
    = \frac1{C_X \sqrt{\|\Sigma^{1/2} A^\top A \Sigma^{1/2}\| \, \|\Sigma^{1/2} B^\top B \Sigma^{1/2}}\|}
    \\&\quad
    \cdot \sqrt{\frac{\ttr(\Sigma^{1/2} A^\top A \Sigma^{1/2}) \, \ttr(\Sigma^{1/2} B^\top B \Sigma^{1/2}) + \log(2/\delta)}{4n}}, 
\end{align}
where $C_X$ is defined in Assumption \ref{as:quadratic}.
We move to the proof of Theorem \ref{th:covariance_concentration}.

\bigskip

\noindent{\bf Step 1: a variational inequality.}
\quad
On this step, we make some preparations. We specify a prior $\mu$ and a family of measures $\rho$ from Lemma \ref{lem:pac-bayes}. A reader will immediately see the result of these efforts on the second step where we relate the variational approach with a high probability upper bound on the Frobenius norm of $B(\widehat \Sigma - \Sigma)$.
Let us introduce $\bxi = \Sigma^{-1/2} \bX$, $\bxi_i = \Sigma^{-1/2} \bX_i$, $1 \leq i \leq n$, and a random matrix
\[
    \Psi = \frac1{\sqrt \beta} \Sigma^{1/2} B^\top Z A \Sigma^{1/2},
\]
where $Z$ is a matrix of size $(q \times d)$ with i.i.d. $\cN(0, 1)$ entries and $\beta$ is defined in \eqref{eq:beta}. Note that
\[
    \rmvec(\Psi) \sim \cN\big(\bzero, \beta^{-1} \, (\Sigma^{1/2} A^\top A \Sigma^{1/2}) \otimes (\Sigma^{1/2} B^\top B \Sigma^{1/2} ) \big)
\]
in this case (see, for instance, \citep{leng18} for the details). Denote its density (with respect to the volume measure on the image of $(\Sigma^{1/2} A^\top A \Sigma^{1/2}) \otimes (\Sigma^{1/2} B^\top B \Sigma^{1/2})$) by $\sfp$ and, for any $U \in \R^{d \times d}$, define
\[
    \sfp_U(\bx) = \frac1{\pi_R} \, \sfp \big(\bx - \rmvec(U) \big) \1( \|\bx - \rmvec(U)\| \leq R),
\]
where $\rmvec(U) \in \R^{d^2}$ is a vector obtained by reshaping of $U$ and $\pi_R \in (0, 1)$ is a normalizing constant.
Then $(\bxi_1^\top U \bxi_1 + \dots \bxi_n^\top U \bxi_n) / n - \E \bxi^\top U \bxi$ can be represented in the following form:
\[
    \frac1n \sum\limits_{i = 1}^n \bxi_i^\top U \bxi_i - \E \bxi^\top U \bxi
    = \E_{\rmvec(\Phi) \sim \sfp_U} \left( \frac1n \sum\limits_{i = 1}^n \bxi_i^\top \Phi \bxi_i - \E_X \bxi^\top \Phi \bxi \right).
\]
According to Lemma \ref{lem:pac-bayes}, for any $\delta \in (0, 1)$ there exists an event $\cE$, such that $\p(\cE) \geq 1 - \delta$ and 
\begin{align*}
    \lambda \left( \frac1n \sum\limits_{i = 1}^n \bxi_i^\top U \bxi_i - \E \bxi^\top U \bxi \right)
    &
    = \lambda \E_{\rmvec(\Phi) \sim \sfp_U} \left( \frac1n \sum\limits_{i = 1}^n \bxi_i^\top \Phi \bxi_i - \E_X \bxi^\top \Phi \bxi \right)
    \\&
    \leq -\lambda \E_{\rmvec(\Phi) \sim \sfp_U} \E_\bX \bxi^\top \Phi \bxi + \E_{\rmvec(\Phi) \sim \sfp_U} \log \E_\bX e^{\lambda \bxi^\top \Phi \bxi}
    \\&\quad
    + \frac{\KL(\sfp_U, \sfp) + \log(1/\delta)}n
\end{align*}
simultaneously for all $U \in \R^{d \times d}$ on $\cE$. From now on, we restrict our attention on this event.

\bigskip

\noindent{\bf Step 2: relating the variational inequality with the Frobenius norm.} Rewrite the Frobenius norm of $B(\widehat \Sigma - \Sigma) A^\top$ as follows:
\begin{align*}
    \|B(\widehat \Sigma - \Sigma) A^\top\|_{\F}
    &
    = \sup\limits_{\substack{V \in \R^{d \times q}, \\ \|V\|_{\F} = 1}} \Tr\big(V B (\widehat \Sigma - \Sigma) A^\top \big)
    = \sup\limits_{\substack{V \in \R^{d \times q}, \\ \|V\|_{\F} = 1}} \left( \frac1n \sum\limits_{i = 1}^n \bX_i^\top A^\top V B \bX_i - \E \bX^\top A^\top V B \bX \right)
    \\&
    = \sup\limits_{\substack{U = \Sigma^{1/2} A^\top V B \Sigma^{1/2}, \\ \|V\|_{\F} = 1}} \left( \frac1n \sum\limits_{i = 1}^n \bxi_i^\top U \bxi_i - \E \bxi^\top U \bxi \right).
\end{align*}
On the event $\cE$, defined on Step 1, it holds that
\begin{align}
    \label{eq:fr_norm_bound}
    \lambda \|B(\widehat \Sigma - \Sigma) A^\top\|_{\F}
    &\notag
    \leq \sup\limits_{\substack{U = \Sigma^{1/2} V B \Sigma^{1/2}, \\ \|V\|_{\F} = 1}} \Bigg(  -\lambda \E_{\rmvec(\Phi) \sim \sfp_U} \E_\bX \bxi^\top \Phi \bxi \\&\quad
    + \E_{\rmvec(\Phi) \sim \sfp_U} \log \E_\bX e^{\lambda \bxi^\top \Phi \bxi}
    + \frac{\KL(\sfp_U, \sfp) + \log(1/\delta)}n \Bigg).
\end{align}
It only remains to bound the terms in the right hand side.

\bigskip

\noindent{\bf Step 3: a bound on $\KL(\sfp_U, \sfp)$.} 
\quad Applying Lemma \ref{lem:kl} with $P = \Sigma^{1/2} A^\top$, $Q = \Sigma^{1/2} B^\top$, and $U = \Sigma^{1/2} A^\top V B \Sigma^{1/2}$, $\|V\|_{\F} = 1$, we obtain that
\begin{equation}
    \label{eq:kl_bound}    
    \KL(\sfp_U, \sfp)
    \leq \log2 + \frac{\beta}2 \left\| (\Sigma^{1/2} A^\top)^\dag U (B \Sigma^{1/2})^\dag \right\|_{\F}^2
    \leq \log 2 + \frac{\beta}2 \|V\|_{\F}^2
    = \log 2 + \frac{\beta}2.
\end{equation}

\medskip

\noindent{\bf Step 4: a bound on $\E_{\rmvec(\Phi) \sim \sfp_U} \log \E_\bX e^{\lambda \bxi^\top \Phi \bxi}$.} 
\quad Note that, since $U = \Sigma^{1/2} A^\top V B \Sigma^{1/2}$, $\|V\|_{\F} = 1$, the following inequality holds almost surely:
\begin{align*}
    \|\Phi\|_{\F}^2
    \leq 2 \|U\|_{\F}^2 + 2 R^2
    &
    \leq 2 \|\Sigma^{1/2} A^\top\|^2 \|V\|_{\F}^2 \|B \Sigma^{1/2}\|^2 + 2 R^2
    \\&
    = 2 \|\Sigma^{1/2} A^\top A \Sigma^{1/2}\| \|\Sigma^{1/2} B^\top B \Sigma^{1/2}\| + 2 R^2.
\end{align*}
By the definition of $R$ (see \eqref{eq:R} and \eqref{eq:beta}), we have
\[
    \|\Phi\|_{\F}^2 \leq 4 \|\Sigma^{1/2} A^\top A \Sigma^{1/2}\| \|\Sigma^{1/2} B^\top B \Sigma^{1/2}\|
    \quad \text{almost surely.}
\]
Note that this inequality yields that $\lambda$, defined in \eqref{eq:lambda}, does not exceed $1 / (C_X \|\Phi\|_{\F})$.
Hence, due to Assumption \ref{as:quadratic}, it holds that
\begin{equation}
    \label{eq:exp_moment_bound}
    - \lambda \E_{\bxi} (\bxi^\top \Phi \bxi) + \log \E_{\bxi} e^{\lambda \bxi^\top \Phi \bxi}
    \leq C_X^2 \lambda^2 \|\Phi\|_{\F}^2
    \leq 4 C_X^2 \lambda^2 \|\Sigma^{1/2} A^\top A \Sigma^{1/2}\| \|\Sigma^{1/2} B^\top B \Sigma^{1/2}\|.
\end{equation}
Summing up the inequalitites \eqref{eq:fr_norm_bound}, \eqref{eq:kl_bound}, and \eqref{eq:exp_moment_bound}, we obtain that
\[
    \lambda \|B(\widehat \Sigma - \Sigma) A^\top\|_{\F}
    \leq 4 C_X^2 \lambda^2 \|\Sigma^{1/2} A^\top A \Sigma^{1/2}\| \|\Sigma^{1/2} B^\top B \Sigma^{1/2}\| + \frac{\beta/2 + \log(2 / \delta)}n
\]
on the event $\cE$.
Dividing both parts by $\lambda$ and taking into account \eqref{eq:beta} and \eqref{eq:lambda}, we conclude that
\begin{align*}
    \|B (\widehat \Sigma - \Sigma) A^\top\|_{\F}
    &
    \leq 4 C_X \|\Sigma^{1/2} A^\top A \Sigma^{1/2}\| \, \|\Sigma^{1/2} B^\top B \Sigma^{1/2}\|
    \\&\quad
    \cdot \sqrt{ \frac{\ttr(\Sigma^{1/2} A^\top A \Sigma^{1/2}) \ttr(\Sigma^{1/2} B^\top B \Sigma^{1/2}) + \log(2 / \delta)}n}
\end{align*}
with probability at least $1 - \delta$.

\subsection{Proof of Theorem \ref{th:noise_concentration}}

The proof of this statement is also based on the PAC-Bayesian approach. Note that
\[
    \left\| \frac1n \sum\limits_{i = 1}^n B \bX_i \eps_i \right\|
    = \sup\limits_{\|\bv\| = 1} \frac1n \sum\limits_{i = 1}^n \bv^\top B \bX_i \eps_i
    = \sup\limits_{\substack{\bu = \Sigma^{1/2} B^\top \bv, \\ \|\bv\| = 1}} \frac1n \sum\limits_{i = 1}^n \bu^\top \Sigma^{-1/2} \bX_i \eps_i.
\]
Let $\mu$ be the probability measure on $\R^d$ corresponding to the Gaussian distribution
\[
    \mathcal N\left(\bzero, \frac1{2 \ttr(\Sigma^{1/2} B^\top B \Sigma^{1/2})} \Sigma^{1/2} B^\top B \Sigma^{1/2} \right).
\]
For any $\bu$, such that $\bu = \Sigma^{1/2} B^\top \bv$ for some unit vector $\bv$, define a probability measure $\rho_\bu$ by the formula
\[
    \rmd \rho_\bu(\bx) = \frac1\pi \1\left(\|\bx - \bu\| \leq \|B \Sigma^{1/2}\| \right) \rmd \mu(\bx - \bu),
\]
where
\[
    \pi
    = \mu\left( \left\{\bx : \|\bx - \bu\| \leq \|B \Sigma^{1/2}\| \right\} \right)
\]
is a normalizing constant. Then it holds that
\[
    \sup\limits_{\substack{\bu = \Sigma^{1/2} B^\top \bv, \\ \|\bv\| = 1}} \frac1n \sum\limits_{i = 1}^n \bu^\top \Sigma^{-1/2} \bX_i \eps_i
    = \sup\limits_{\substack{\bu = \Sigma^{1/2} B^\top \bv, \\ \|\bv\| = 1}} \frac1n \E_{\bxi \sim \rho_\bu} \sum\limits_{i = 1}^n \bxi^\top \Sigma^{-1/2} \bX_i \eps_i.
\]
Let us fix
\[
    \lambda = \frac1{4 \sigma} \sqrt{ \frac{\ttr(\Sigma^{1/2} B^\top B \Sigma^{1/2}) + \log(2 / \delta)}{\|\Sigma\| n}}.
\]
According to Lemma \ref{lem:pac-bayes}, for any $\delta \in (0, 1)$, with probability at least $(1 - \delta)$, we have
\[
    \lambda \sup\limits_{\substack{\bu = \Sigma^{1/2} B^\top \bv, \\ \|\bv\| = 1}} \frac1n \E_{\bxi \sim \rho_\bu} \sum\limits_{i = 1}^n \bxi^\top \Sigma^{-1/2} \bX_i \eps_i
    \leq \E_{\bxi \sim \rho_\bu} \log \E_{\bX, \eps} e^{\lambda \bxi^\top \Sigma^{-1/2} \bX \eps} 
    + \frac{\KL(\rho_\bu, \mu) + \log(1 / \delta)}n
\]
simultaneously for all $\rho_\bu$, such that $\bu = \Sigma^{1/2} B^\top \bv$ for some unit vector $\bv$. Applying Lemma \ref{lem:kl} with $P = \Sigma^{1/2} B^\top$, $Q = 1$, and $\beta = 2 \ttr(\Sigma^{1/2} B^\top B \Sigma^{1/2})$, we obtain that
\[
    \KL(\rho_\bu, \mu)
    \leq \log 2 + \ttr(\Sigma^{1/2} B^\top B \Sigma^{1/2}) \left\| (\Sigma^{1/2} B^\top)^\dag \bu \right\|^2
    \leq \log 2 + \ttr(\Sigma^{1/2} B^\top B \Sigma^{1/2}).
\]
Besides, since $\|\Sigma^{-1/2} \bX_1 \eps_1\|_{\psi_1} = \sigma$ and
\[
    \lambda = \frac1{4 \sigma} \sqrt{ \frac{\ttr(\Sigma^{1/2} B^\top B \Sigma^{1/2}) + \log(2 / \delta)}{\|\Sigma\| n} }
    \leq \frac1{4 \sigma \sqrt{\|\Sigma^{1/2} B^\top B \Sigma^{1/2}\|}}
    \leq \frac1{2 \sigma \|\xi\|},
\]
Lemma 2 from \cite{zhivotovskiy21} yields that
\[
    \log \E_{\bX, \eps} e^{\lambda \bxi^\top \Sigma^{-1/2} \bX \eps}
    \leq 4 \lambda^2 \sigma^2 \|\bxi\|^2
    \leq 8 \lambda^2 \sigma^2 \left( \|\bu\|^2 + \|B \Sigma^{1/2}\|^2 \right)
    \leq 16 \lambda^2 \sigma^2 \|B \Sigma^{1/2}\|^2.
\]
Here we used the fact that $\|\bxi\|^2 \leq 2 \|\bu\|^2 + 2 \|B \Sigma^{1/2}\|^2 \leq 4 \|B \Sigma^{1/2}\|^2$ almost surely.
Hence, with probability at least $(1 - \delta)$, it holds that
\begin{align*}
    \left\| B (\bZ - \E \bZ) \right\|
    = \frac{1}n \left\| \sum\limits_{i = 1}^n B \bX_i \eps_i \right\|
    &
    \leq 16 \lambda \sigma^2 \|\Sigma^{1/2} B^\top B \Sigma^{1/2}\| + \frac{\ttr(\Sigma^{1/2} B^\top B \Sigma^{1/2}) + \log(2 / \delta)}{\lambda n}
    \\&
    = 8 \sigma \left\|B \Sigma^{1/2} \right\| \sqrt{\frac{\ttr(\Sigma^{1/2} B^\top B \Sigma^{1/2}) + \log(2 / \delta)}n}.
\end{align*}
\myendproof

\subsection{Auxiliary results}

\begin{Lem}
    \label{lem:kl}
    Let $\beta$ be a positive number and let $P \in \R^{p \times r}$ and $Q \in \R^{q \times s}$ be arbitrary matrices. 
    Let $\sfp$ be the probability density on ${\rm Im}((QQ^\top) \otimes (P P^\top))$ (with respect to the volume measure $\cV$) of the Gaussian distribution $\cN\big(0, \beta^{-1} \, (QQ^\top) \otimes (P P^\top) \big)$, where $\otimes$ stands for the Kronecker product. That is,
    \[
        \sfp(\bx) \propto \exp\left\{ - \frac{\beta}2 \bx^\top \big((QQ^\top) \otimes (P P^\top) \big)^\dag \bx \right\},
        \quad \bx \in {\rm Im}((QQ^\top) \otimes (P P^\top)).
    \]
    Fix a matrix $U \in \R^{p \times q}$, such that $\rmvec(U) \in {\rm Im}((QQ^\top) \otimes (P P^\top))$, and set
    \[
        R = \sqrt{\frac2\beta \, \Tr(PP^\top) \, \Tr(QQ^\top)}.
    \]
    Define the density $\sfp_U$ as follows:
    \[
        \sfp_U(\bx) = \frac{1}{\pi_R} \sfp(\bx - \rmvec(U)) \1 \big(\|\bx - \rmvec(U)\| \leq R \big), \quad \bx \in {\rm Im}((QQ^\top) \otimes (P P^\top)),
    \]
    where
    \[
        \pi_R = \int\limits_{\|\bx\| \leq R} \sfp(\bx) \rmd \cV(\bx) = \p_{\bxi \sim \sfp} \big( \|\bxi\| \leq R \big)
    \]
    is a normalizing constant. Then it holds that
    \[
        \KL(\sfp_U, \sfp) = \log \frac1{\pi_R} + \frac{\beta}2 \|P^{\dag} U (Q^\dag)^{\top} \|_{\F}^2
        \leq \log 2 + \frac{\beta}2 \|P^{\dag} U (Q^\dag)^{\top} \|_{\F}^2.
    \]
\end{Lem}

\begin{proof}
    Denote $S = (QQ^\top) \otimes (PP^\top)$ and $\bu = \rmvec(U)$ for brevity.
    By the definition of $\KL(\sfp_U, \sfp)$, it holds that
    \begin{align*}
        \KL(\sfp_U, \sfp)
        &
        = \int \log\frac{\sfp_U(\bx)}{\sfp(\bx)} \, \sfp_U(\bx) \rmd \cV(\bx)
        \\&
        = \int \log\frac1{\pi_R} \, \sfp_U(\bx) \rmd \cV(\bx) + \frac{\beta}2 \int\limits_{\|\bx - \bu\| \leq R} \left[ \bx^\top S^{\dag} \bx - (\bx - \bu)^\top S^{\dag} (\bx - \bu) \right] \rmd \cV(\bx)
        \\&
        = \log\frac1{\pi_R} + \frac{\beta}2 \int\limits_{\|\bx - \bu\| \leq R} \left[ 2 \bx^\top S^{\dag} \bu - \bu^\top S^{\dag} \bu \right] \rmd \cV(\bx)
        \\&
        = \log\frac1{\pi_R} + \frac{\beta}2 \left( 2 \bu^\top S^{\dag} \bu - \bu^\top S^{\dag} \bu\right)
        = \log\frac1{\pi_R} + \frac{\beta}2 \bu^\top S^{\dag} \bu.
    \end{align*}
    Let us recall the following properties of the Kronecker product:
    \begin{align*}
        &
        \big( (QQ^\top) \otimes (PP^\top) \big)^{\dag} = (QQ^\top)^{\dag} \otimes (PP^\top)^{\dag},
        \\&
        \big( (QQ^\top)^{\dag} \otimes (PP^\top)^{\dag} \big) \rmvec(U) = \rmvec \big( (PP^\top)^{\dag} U (QQ^\top)^{\dag} \big).
    \end{align*}
    Using these two inequalities, we obtain that
    \begin{align*}
        \bu^\top S^{\dag} \bu
        &
        = \rmvec(U)^\top \rmvec \big( (PP^\top)^{\dag} U (QQ^\top)^{\dag} \big)
        = \Tr\big(U^\top (PP^\top)^{\dag} U (QQ^\top)^{\dag} \big)
        \\&
        = \Tr\big( U^\top (P^\dag)^\top P^\dag U (Q^\dag)^\top Q^\dag \big)
        = \Tr\big( Q^\dag U^\top (P^\dag)^\top P^\dag U (Q^\dag)^\top \big)
        \\&
        = \|P^\dag U (Q^\dag)^\top\|_{\F}^2.
    \end{align*}
    Hence, it holds that
    \[
        \KL(\rho_U, \mu) = \log \frac1{\pi_R} + \frac{\beta}2 \|P^\dag U (Q^\dag)^\top\|_{\F}^2.
    \]
    To prove the second part of the lemma, note that
    \[
        \pi_R = \p( \|\bxi\| \leq R),
        \quad \text{where $\bxi \sim \cN(0, \beta^{-1} \, (QQ^\top) \otimes (P P^\top))$}.
    \]
    Using Chebyshev's inequality, we obtain that
    \begin{align*}
        \pi_R
        &
        \geq 1 - \frac{\E \|\bxi\|^2}{R^2}
        = 1 - \frac{\E \Tr(\bxi\bxi^\top)}{R^2}
        = 1 - \frac{\Tr\big( (QQ^\top) \otimes (P P^\top) \big)}{\beta R^2}
        \\&
        = 1 - \frac{\Tr\big( (QQ^\top) \big) \Tr \big( (P P^\top) \big)}{\beta R^2}
        \geq \frac12,
    \end{align*}
    and thus,
    \[
        \KL(\sfp_U, \sfp)
        = \log \frac1{\pi_R} + \frac{\beta}2 \|P^\dag U (Q^\dag)^\top\|_{\F}^2
        \leq \log 2 + \frac{\beta}2 \|P^\dag U (Q^\dag)^\top\|_{\F}^2.
    \]
\end{proof}

\section{Proof of Corollary \ref{co:risk_bound}}
\label{sec:co_risk_bound_proof}

According to Theorem \ref{theorem: bias} and \ref{th:stoch_term}, with probability at least $(1 - 13 \delta / 2)$, it holds that
\begin{equation}
    \label{eq:risk_expansion}
    \left\| \Sigma^{1/2} (\widehat \btheta - \btheta^\circ) \right\|
    \leq \left\|\Sigma^{1/2} \bb_\lambda \right\| + \left\| \Sigma^{1/2} (\Sigma^2 + 2\lambda I_d)^{-1} \bzeta \right\| + \diamondsuit,
\end{equation}
where $\bzeta = \Sigma \bU - \Sigma (\widehat \Sigma - \Sigma) \bb_\lambda - (\widehat \Sigma - \Sigma) \Sigma \bb_{\lambda}$ and $\diamondsuit$ defined in \eqref{eq:risk_expansion_remainder}. In the rest of the proof, we show that the norm of $\Sigma^{1/2} (\Sigma^2 + 2\lambda I_d)^{-1} \bzeta$ does not exceed
\[
    4 \left(2\sigma + C_X \|\Sigma^{1/2} \bb_\lambda\| \right) \sqrt{\frac{k^* + r_4(k^*) + \log(4 / \delta)}n}
    + \frac{2 C_X \|\Sigma^{3/2} \bb_{\lambda}\|}{\sqrt{2\lambda}} \sqrt{\frac{4 k^* + 4 r_2(k^*) + \log(4 / \delta)}n}.
\]
In order to do so, we use the triangle inequality
\begin{align}
    \label{eq:zeta_norm_triangle_inequality}
    \left\| \Sigma^{1/2} (\Sigma^2 + 2\lambda I_d)^{-1} \bzeta \right\|
    &\notag
    \leq \left\| \Sigma^{1/2} (\Sigma^2 + 2\lambda I_d)^{-1} \Sigma \bU \right\| + \left\| \Sigma^{1/2} (\Sigma^2 + 2\lambda I_d)^{-1} \Sigma (\widehat\Sigma - \Sigma) \bb_\lambda \right\|
    \\&\quad
    + \left\| \Sigma^{1/2} (\Sigma^2 + 2\lambda I_d)^{-1} (\widehat\Sigma - \Sigma) \Sigma \bb_\lambda \right\|
\end{align}
and bound the three terms in the right-hand side one by one using Theorems \ref{th:covariance_concentration} and \ref{th:noise_concentration}. Before we proceed, we formulate an auxiliary result.
\begin{Lem}
    \label{lem:effective_rank_bounds}
    Let $r_q(k)$ and $k^* = k^*(\lambda)$ be as defined in \eqref{eq:rqk} and \eqref{eq:k_star}, respectively. Then it holds that
    \[
        \ttr\left(\Sigma^4 (\Sigma^2 + 2\lambda I_d)^{-2} \right)
        \leq \left( 1 + \frac{2\lambda}{\|\Sigma\|^2} \right)^2 \big( k^* + r_4(k^*) \big)
    \]
    and
    \[
        \ttr\left(\Sigma^2 (\Sigma^2 + 2\lambda I_d)^{-2} \right)
        \leq 4 k^* + 4 r_2(k^*).
    \]
\end{Lem}
The proof of Lemma \ref{lem:effective_rank_bounds} is postponed to Section \ref{sec:lem_effective_rank_bounds_proof}.
It helps us to verify the conditions of Theorems \ref{th:covariance_concentration} and \ref{th:noise_concentration}. Indeed, taking into account that $r_4(k) \leq r_2(k)$ for all $k \in \mathbb N$, we observe that the condition \eqref{eq:k_star_condition} and Lemma \ref{lem:effective_rank_bounds} yield
\[
    \ttr\left(\Sigma^4 (\Sigma^2 + 2\lambda I_d)^{-2} \right) + \log(4/\delta)
    \leq \left( 1 + \frac{2\lambda}{\|\Sigma\|^2} \right)^2 \big( k^* + r_4(k^*) \big) + \log(4/\delta)
    \leq n.
\]
Then, according to Theorems \ref{th:covariance_concentration} and \ref{th:noise_concentration}, with probability at least $(1 - \delta)$, we simultaneously have
\begin{align*}
    &
    \left\| \Sigma^{1/2} (\Sigma^2 + 2\lambda I_d)^{-1} \Sigma (\widehat\Sigma - \Sigma) \bb_\lambda \right\|
    \\&
    \leq 4 C_X \left\|\Sigma^2 (\Sigma^2 + 2\lambda I_d)^{-1} \right\| \, \|\Sigma^{1/2} \bb_\lambda\| \sqrt{\frac{\ttr(\Sigma^4 (\Sigma^2 + 2\lambda I_d)^{-2}) + \log(4 / \delta)}n}
    \\&
    = \frac{4 C_X \|\Sigma\|^2}{\|\Sigma\|^2 + 2\lambda} \, \|\Sigma^{1/2} \bb_\lambda\| \sqrt{\frac{\ttr(\Sigma^4 (\Sigma^2 + 2\lambda I_d)^{-2}) + \log(4 / \delta)}n}
\end{align*}
and
\begin{align*}
    \left\| \Sigma^{1/2} (\Sigma^2 + 2\lambda I_d)^{-1} \Sigma \bU \right\|
    &
    \leq 8 \sigma \left\| \Sigma^2 (\Sigma^2 + 2\lambda I_d)^{-1} \right\| \sqrt{\frac{\ttr(\Sigma^4 (\Sigma^2 + 2\lambda I_d)^{-2}) + \log(4 / \delta)}n}
    \\&
    = \frac{8 \sigma \|\Sigma\|^2}{\|\Sigma\|^2 + 2\lambda} \sqrt{\frac{\ttr(\Sigma^4 (\Sigma^2 + 2\lambda I_d)^{-2}) + \log(4 / \delta)}n}.
\end{align*}
In view of Lemma \ref{lem:effective_rank_bounds}, we obtain that, on the same event, it holds that
\begin{equation}
    \label{eq:leading_stoch_term_1st}
    \left\| \Sigma^{1/2} (\Sigma^2 + 2\lambda I_d)^{-1} \Sigma (\widehat\Sigma - \Sigma) \bb_\lambda \right\|
    \leq 4 C_X \|\Sigma^{1/2} \bb_\lambda\| \sqrt{\frac{k^* + r_4(k^*) + \log(4 / \delta)}n}
\end{equation}
and
\begin{equation}
    \label{eq:leading_stoch_term_2nd}
    \left\| \Sigma^{1/2} (\Sigma^2 + 2\lambda I_d)^{-1} \Sigma \bU \right\|
    \leq 8 \sigma \sqrt{\frac{k^* + r_4(k^*) + \log(4 / \delta)}n}.
\end{equation}
It remains to bound the norm of $\Sigma^{1/2} (\Sigma^2 + 2\lambda I_d)^{-1} (\widehat\Sigma - \Sigma) \Sigma \bb_\lambda$ to finish the proof. Note that Lemma \ref{lem:effective_rank_bounds} and \eqref{eq:k_star_condition} imply that
\[
    \ttr\left(\Sigma^2 (\Sigma^2 + 2\lambda I_d)^{-2} \right)
    + \log(4/\delta)
    \leq 4 k^* + 4 r_2(k^*) + \log(4/\delta) \leq 4n.
\]
Then, due to Theorem \ref{th:covariance_concentration}, with probability at least $(1 - \delta/2)$, we have
\begin{align}
    \label{eq:leading_stoch_term_3rd}
    &\notag
    \left\| \Sigma^{1/2} (\Sigma^2 + 2\lambda I_d)^{-1} (\widehat\Sigma - \Sigma) \Sigma \bb_\lambda \right\|
    \\&
    \leq 4 C_X \left\|\Sigma (\Sigma^2 + 2\lambda I_d)^{-1} \right\| \, \| \Sigma^{3/2} \bb_{\lambda} \| \sqrt{\frac{\ttr(\Sigma^2 (\Sigma^2 + 2\lambda I_d)^{-2}) + \log(4 / \delta)}n}
    \\&\notag
    \leq \frac{2 C_X \|\Sigma^{3/2} \bb_{\lambda}\|}{\sqrt{2\lambda}} \sqrt{\frac{4 k^* + 4 r_2(k^*) + \log(4 / \delta)}n}.
\end{align}
Here we used the fact that
\[
    \left\|\Sigma (\Sigma^2 + 2\lambda I_d)^{-1} \right\|
    = \max\limits_{1 \leq j \leq d} \frac{\sigma_j}{\sigma_j^2 + 2\lambda}
    \leq \frac{1}{2 \sqrt{2\lambda}}.
\]
Summing up the inequalities \eqref{eq:zeta_norm_triangle_inequality}--\eqref{eq:leading_stoch_term_3rd} and using the union bound, we obtain that, with probability at least $(1 - 3\delta / 2)$,
\begin{align}
    \label{eq:leading_stoch_term_bound}
    \left\| \Sigma^{1/2} (\Sigma^2 + 2\lambda I_d)^{-1} \bzeta \right\|
    &\notag
    \leq 4 \left(2\sigma + C_X \|\Sigma^{1/2} \bb_\lambda\| \right) \sqrt{\frac{k^* + r_4(k^*) + \log(4 / \delta)}n}
    \\&\quad
    + \frac{2 C_X \|\Sigma^{3/2} \bb_{\lambda}\|}{\sqrt{2\lambda}} \sqrt{\frac{4 k^* + 4 r_2(k^*) + \log(4 / \delta)}n}
\end{align}
Finally, \eqref{eq:risk_expansion} and \eqref{eq:leading_stoch_term_bound} yield that
\begin{align*}
    \left\| \Sigma^{1/2} (\widehat \btheta - \btheta^\circ) \right\|
    &
    \leq \left\|\Sigma^{1/2} \bb_\lambda\right\| + 4 \left(2\sigma + C_X \|\Sigma^{1/2} \bb_\lambda\| \right) \sqrt{\frac{k^* + r_4(k^*) + \log(4 / \delta)}n}
    \\&\quad
    + \frac{2 C_X \|\Sigma^{3/2} \bb_{\lambda}\|}{\sqrt{2\lambda}} \sqrt{\frac{4 k^* + 4 r_2(k^*) + \log(4 / \delta)}n}
    + \diamondsuit
\end{align*}
with probability at least $(1 - 8\delta)$.

\myendproof

\subsection{Proof of Lemma \ref{lem:effective_rank_bounds}}
\label{sec:lem_effective_rank_bounds_proof}

Due to the definition of the effective rank, we have
\[
    \ttr\left(\Sigma^4 (\Sigma^2 + 2\lambda I_d)^{-2} \right)
    = \frac{\Tr(\Sigma^4 (\Sigma^2 + 2\lambda I_d)^{-2})}{\|\Sigma^4 (\Sigma^2 + 2\lambda I_d)^{-2}\|}
    = \left( \frac{\|\Sigma\|^2 + 2\lambda}{\|\Sigma\|^2} \right)^2 \sum\limits_{j = 1}^d \frac{\sigma_j^4}{(\sigma_j^2 + 2\lambda)^2}.
\]
Taking into account that the operator norm of $\Sigma^4 (\Sigma^2 + 2\lambda I_d)^{-2}$ is equal to $\|\Sigma\|^4 / (\|\Sigma\|^2 + 2\lambda)^2$, we obtain that
\begin{align*}
    \ttr\left(\Sigma^4 (\Sigma^2 + 2\lambda I_d)^{-2} \right)
    &
    = \left( \frac{\|\Sigma\|^2 + 2\lambda}{\|\Sigma\|^2} \right)^2 \left( \sum\limits_{j = 1}^{k^*} \frac{\sigma_j^4}{(\sigma_j^2 + 2\lambda)^2} + \sum\limits_{j > k^*} \frac{\sigma_j^4}{(\sigma_j^2 + 2\lambda)^2} \right)
    \\&
    \leq \left( \frac{\|\Sigma\|^2 + 2\lambda}{\|\Sigma\|^2} \right)^2 \left( \sum\limits_{j = 1}^{k^*} 1 + \sum\limits_{j > k^*} \frac{\sigma_j^4}{\sigma_{k^* + 1}^4} \right)
    \\&
    = \left( 1 + \frac{2\lambda}{\|\Sigma\|^2} \right)^2 \left( k^* + r_4(k^*) \right).
\end{align*}
Similarly, it holds that
\begin{align*}
    \ttr\left(\Sigma^2 (\Sigma^2 + 2\lambda I_d)^{-2} \right)
    &
    = \left( \max\limits_{1 \leq k \leq d} \frac{\sigma_k^2}{(\sigma_k^2 + 2\lambda)^2} \right)^{-1} \sum\limits_{j = 1}^d \frac{\sigma_j^2}{(\sigma_j^2 + 2\lambda)^2}
    \\&
    = \min\limits_{1 \leq k \leq d} \frac{(\sigma_k^2 + 2\lambda)^2}{\sigma_k^2} \left( \sum\limits_{j = 1}^{k^*} \frac{\sigma_j^2}{(\sigma_j^2 + 2\lambda)^2} + \sum\limits_{j > k^*} \frac{\sigma_j^2}{(\sigma_j^2 + 2\lambda)^2} \right)
    \\&
    \leq \frac{(\sigma_{k^*}^2 + 2\lambda)^2}{\sigma_{k^*}^2} \sum\limits_{j = 1}^{k^*} 1 + \frac{(\sigma_{k^* + 1}^2 + 2\lambda)^2}{\sigma_{k^* + 1}^2} \sum\limits_{j > k^*} \frac{\sigma_j^2}{(\sigma_j^2 + 2\lambda)^2}.
\end{align*}
Since the definition of $k^*$ (see \eqref{eq:k_star}) yields that
\[
    \frac{(\sigma_{k^*}^2 + 2\lambda)^2}{\sigma_{k^*}^2}
    \leq \frac{(\sigma_{k^*}^2 + \sigma_{k^*}^2)^2}{\sigma_{k^*}^2} = 4 \sigma_{k^*}^2
\]
and
\[
    \frac{(\sigma_{k^* + 1}^2 + 2\lambda)^2}{\sigma_{k^* + 1}^2} \sum\limits_{j > k^*} \frac{\sigma_j^2}{(\sigma_j^2 + 2\lambda)^2}
    \leq \frac{(\sigma_{k^* + 1}^2 + 2\lambda)^2}{4\lambda^2} \sum\limits_{j > k^*} \frac{\sigma_j^2}{\sigma_{k^* + 1}^2}
    \leq \frac{(2\lambda + 2\lambda)^2}{4\lambda^2} \sum\limits_{j > k^*} \frac{\sigma_j^2}{\sigma_{k^* + 1}^2}
    = 4 r_2(k^*),
\]
we get the desired bound:
\[
    \ttr\left(\Sigma^2 (\Sigma^2 + 2\lambda I_d)^{-2} \right)
    \leq 4 k^* + 4 r_2(k^*).
\]
\myendproof

\section{Properties of the bias $\bb_\lambda$}
\label{sec:bias_properties}

In this section, we provide rigorous proofs of some properties of $\bb_\lambda = -2\lambda (\Sigma^2 + 2\lambda I_d)^{-1} \btheta^\circ$ mentioned in Section \ref{sec:statistical_properties}. Let us introduce auxiliary notation to simplify further derivations. Throughout this section, $\sigma_1 \geq \sigma_2 \geq \dots \geq \sigma_d$ stand for the eigenvalues of $\Sigma$ and
\[
    \Sigma = \sum\limits_{j = 1}^d \sigma_j \bv_j \bv_j^\top
\]
is its eigenvalue decomposition. For any $k \in \{1, \dots, d\}$, we define $\beta_k = \bv_k^\top \btheta^\circ$ and the projections of $\btheta$ onto the linear span of $\bv_1, \dots, \bv_k$ and its orthogonal complement, respectively:
\[
    \btheta_{\leq k}^\circ = \sum\limits_{j = 1}^k \beta_j \bv_j
    \quad \text{and} \quad
    \btheta_{>k}^\circ = \btheta^\circ - \btheta_{\leq k}^\circ.
\]
Similarly to the statement of Corollary \ref{co:risk_bound}, we denote
\[
    k^* = k^*(\lambda) = \max\left\{ k \in \mathbb N : \sigma_k^2 \geq 2 \lambda \right\}.
\]
The first result relates the squared norms of $\Sigma^{1/2} \bb_\lambda$ and $\Sigma^{3/2} \bb_\lambda$ with the ones of $\btheta_{\leq k^*}^\circ$ and $\btheta_{> k^*}^\circ$.

\begin{Lem}
    \label{lem:bias_upper_bound}
    With the notation introduced above, it holds that
    \[
        \left\|\Sigma^{1/2} \bb_\lambda \right\|^2
        \leq \frac{\sigma_{k^*}^2}4 \left\|\Sigma^{-1/2} \btheta^\circ_{\leq k^*}\right\|^2 + \left\|\Sigma^{1/2} \btheta^\circ_{> k^*}\right\|^2
    \]
    and
    \[
        \frac{\left\|\Sigma^{3/2} \bb_\lambda \right\|^2}{2\lambda}
        \leq \sigma_{k^*}^2 \left\|\Sigma^{-1/2} \btheta^\circ_{\leq k^*}\right\|^2 + \frac14 \left\|\Sigma^{1/2} \btheta^\circ_{> k^*}\right\|^2.
    \]
\end{Lem}

\begin{proof}
    Let us first elaborate on the upper bound on $\|\Sigma^{1/2} \bb_\lambda \|^2$. It holds that
    \begin{align*}
        \left\|\Sigma^{1/2} \bb_\lambda \right\|^2
        &
        = (2 \lambda)^2 \left\| \Sigma^{1/2} (\Sigma^2 + 2\lambda I_d)^{-1} \btheta^\circ \right\|^2
        = 4 \lambda^2 \sum\limits_{j = 1}^d \frac{\sigma_j \beta_j^2}{(\sigma_j^2 + 2\lambda)^2}
        \\&
        = 4 \lambda^2 \sum\limits_{j = 1}^{k^*} \frac{\sigma_j \beta_j^2}{(\sigma_j^2 + 2\lambda)^2} + 4 \lambda^2 \sum\limits_{j > k^*} \frac{\sigma_j \beta_j^2}{(\sigma_j^2 + 2\lambda)^2}.
    \end{align*}
    Due to the Cauchy inequality, we have $8\lambda \sigma_j^2 \leq (\sigma_j^2 + 2\lambda)^2$ for any $j \in \{1, \dots, d\}$. This yields that
    \begin{align*}
        \left\|\Sigma^{1/2} \bb_\lambda \right\|^2
        &
        = 4 \lambda^2 \sum\limits_{j = 1}^{k^*} \frac{\sigma_j \beta_j^2}{(\sigma_j^2 + 2\lambda)^2} + 4 \lambda^2 \sum\limits_{j > k^*} \frac{\sigma_j \beta_j^2}{(\sigma_j^2 + 2\lambda)^2}
        \\&
        \leq \frac{\lambda}2 \sum\limits_{j = 1}^{k^*} \sigma_j^{-1} \beta_j^2 + \sum\limits_{j > k^*} \sigma_j \beta_j^2
        = \frac{\lambda}2 \left\|\Sigma^{-1/2} \btheta^\circ_{\leq k^*}\right\|^2 + \left\|\Sigma^{1/2} \btheta^\circ_{> k^*}\right\|^2.
    \end{align*}
    Due to the definition of $k^*$, the right-hand side does not exceed $\sigma_{k^*}^2 \|\Sigma^{-1/2} \btheta^\circ_{\leq k^*}\|^2 / 4 + \|\Sigma^{1/2} \btheta^\circ_{> k^*}\|^2$. Similarly, we obtain that
    \begin{align*}
        \frac{\left\|\Sigma^{3/2} \bb_\lambda \right\|^2}{2\lambda}
        &
        = 2 \lambda \sum\limits_{j = 1}^{k^*} \frac{\sigma_j^3 \beta_j^2}{(\sigma_j^2 + 2\lambda)^2} + 2 \lambda \sum\limits_{j > k^*} \frac{\sigma_j^3 \beta_j^2}{(\sigma_j^2 + 2\lambda)^2}
        \leq 2 \lambda \sum\limits_{j = 1}^{k^*} \sigma_j^{-1} \beta_j^2 + \frac14 \sum\limits_{j > k^*} \sigma_j \beta_j^2
        \\&
        = 2 \lambda \left\|\Sigma^{-1/2} \btheta^\circ_{\leq k^*}\right\|^2 + \frac14 \left\|\Sigma^{1/2} \btheta^\circ_{> k^*}\right\|^2
        \leq \sigma_{k^*}^2 \left\|\Sigma^{-1/2} \btheta^\circ_{\leq k^*}\right\|^2 + \frac14 \left\|\Sigma^{1/2} \btheta^\circ_{> k^*}\right\|^2.
    \end{align*}
\end{proof}

\begin{Lem}
    \label{lem:ridge_comparison}
    Let $\tau^2 = 2 \lambda$. Then it holds that
    \[
        \max\left\{ \left\|\Sigma^{1/2} \bb_\lambda \right\|^2, \frac{\left\|\Sigma^{3/2} \bb_\lambda \right\|^2}{2\lambda} \right\}
        \leq 2 \tau^2 \left\|\Sigma^{1/2} (\Sigma + \tau I_d)^{-1} \btheta^\circ \right\|^2.
    \]
\end{Lem}

\begin{proof}
    The proof is trivial. First, note that
    \begin{align*}
        \left\|\Sigma^{1/2} \bb_\lambda \right\|^2
        &
        = (2 \lambda)^2 \left\| \Sigma^{1/2} (\Sigma^2 + 2\lambda I_d)^{-1} \btheta^\circ \right\|^2
        = 4 \lambda^2 \sum\limits_{j = 1}^d \frac{\sigma_j \beta_j^2}{(\sigma_j^2 + 2\lambda)^2}
        \\&
        = \tau^4 \sum\limits_{j = 1}^d \frac{\sigma_j \beta_j^2}{(\sigma_j^2 + \tau^2)^2}
        \leq \tau^2 \sum\limits_{j = 1}^d \frac{\sigma_j \beta_j^2}{\sigma_j^2 + \tau^2}
        \leq 2 \tau^2 \sum\limits_{j = 1}^d \frac{\sigma_j \beta_j^2}{(\sigma_j + \tau)^2}.
    \end{align*}
    The expression in the right-hand side is nothing but $2 \tau^2 \|\Sigma^{1/2} (\Sigma + \tau I_d)^{-1} \btheta^\circ\|^2$. This concludes the first part of the proof. The second one follows from the inequalities
    \begin{align*}
        \frac{\left\|\Sigma^{3/2} \bb_\lambda \right\|^2}{2\lambda}
        &
        = 2\lambda \sum\limits_{j = 1}^d \frac{\sigma_j^3 \beta_j^2}{(\sigma_j^2 + 2\lambda)^2}
        = \tau^2 \sum\limits_{j = 1}^d \frac{\sigma_j^3 \beta_j^2}{(\sigma_j^2 + \tau^2)^2}
        \\&
        \leq \tau^2 \sum\limits_{j = 1}^d \frac{\sigma_j \beta_j^2}{\sigma_j^2 + \tau^2}
        \leq 2\tau^2 \sum\limits_{j = 1}^d \frac{\sigma_j \beta_j^2}{(\sigma_j + \tau)^2}
        = 2 \tau^2 \left\|\Sigma^{1/2} (\Sigma + \tau I_d)^{-1} \btheta^\circ \right\|^2.
    \end{align*}
\end{proof}

\section{Proof of Lemma~\ref{lemma: localization_probabilistic_conditions} and auxiliary results}

\subsection{Proof of Lemma~\ref{lemma: localization_probabilistic_conditions}}

We start the proof with the following non-probabilistic lemma.
\begin{Lem}
    \label{lem:estimate_rough_bounds}
    For some $\rho_0 \le 1/8$, define sets
    \begin{align*}
        \Theta & = \left \{
            \btheta \in \R^d 
            \mid 
            \Vert \btheta \Vert \le \rho_0 \mu \; \text{and} \; \Vert \Sigma \Vert \Vert \btheta - \btheta^* \Vert \le 5 \rho_0 \mu \sqrt{\lambda} / 96
        \right \}, \\
        \mathsf{H} & = \left \{ \bfeta \in \R^d \mid \Vert \bfeta - \bZ \Vert \le \rho_0 \mu \sqrt{\lambda} / 3
        \right \}, \\
        \mathsf{A} & = \left \{ A \in \R^{d \times d} \mid \Vert A \Vert \le 3 \Vert \Sigma \Vert + \sqrt{\lambda} / 3   \right \}.
    \end{align*}
    Suppose that 
    \begin{enumerate}[label=(\roman*)]
        \item \label{condition: Sigma theta constant bounds} $\Vert \btheta^\circ \Vert \le \rho_0 \mu / 7$ and $\Vert \Sigma \Vert \Vert \btheta^\circ \Vert \le  \rho_0 \mu \sqrt{\lambda} / (18 \cdot 16)$;
        \item \label{condition: Sigma bounded via lambda} $\Vert \Sigma \Vert \Vert \btheta^\circ \Vert \Vert \widehat \Sigma - \Sigma \Vert \le \rho_0 \mu \lambda / (2 \cdot 96)$ and $\Vert \widehat \Sigma - \Sigma \Vert \le \Vert \Sigma \Vert$;
        \item \label{condition: Xe bounding with lambda}
        $\Vert \Sigma \Vert \Vert \bU \Vert \le \rho_0 \mu \lambda / (2 \cdot 96)$ and
        $\Vert \bU \Vert \le \rho_0 \mu \sqrt{\lambda} / (2 \cdot 96)$.
    \end{enumerate}
    Then, we have the following:
    \begin{align*}
        (\bzero, \bZ, \widehat \Sigma), (\btheta^*, A^*, \bfeta^*), (\widehat \btheta, \widehat A, \widehat \bfeta ) \in \Theta \times \sfA \times \sfH.
    \end{align*}
    and $\Theta \times \sfH \times \sfA \subset \Ups(\rho_0)$. Moreover, $\Vert \widehat A - \widehat \Sigma \Vert \le \sqrt{\lambda} / 3$.
\end{Lem}

Next, we will choose $\lambda$, to satisfy conditions~\ref{condition: Sigma bounded via lambda}-\ref{condition: Xe bounding with lambda} using the following proposition proved in Section~\ref{section: proof of loc condition satisfied}.

\begin{Prop}
\label{proposition: localization conditions satisfied}
Fix $\delta \in (0; 1)$ and $\rho_0 \le 1/8$. Grant Assumptions~\ref{as:quadratic} and \ref{as:noise}. Assume that 
\begin{enumerate}[label= (\roman*)]
    \item $\Vert \btheta^\circ \Vert \le \rho_0 \mu / 7$;
    \item $\Vert \Sigma \Vert \Vert \btheta^\circ \Vert \le \rho_0 \mu \sqrt{\lambda}$;
    \item 
    \begin{align*}
        \frac{\lambda}{\Vert \Sigma \Vert} \wedge \sqrt{\lambda} &  \ge \frac{2^{11} \sigma \Vert \Sigma \Vert^{1/2}}{\rho_0 \mu} \sqrt{\frac{\ttr(\Sigma) + \log(4/\delta)}{n}}, \\
        \lambda & \ge \frac{2^{13}\Vert \Sigma \Vert^2 \Vert \btheta^\circ \Vert}{\rho_0 \mu} (1 + C_X) \sqrt{\frac{4 \ttr(\Sigma) + \log(2/\delta)}{n}}, \\
    n & \ge 2^{12} (1 + C_X)^2 \left ( \ttr(\Sigma) + \log(2/\delta) \right ).
    \end{align*}
\end{enumerate}
Then, conditions~\ref{condition: Sigma theta constant bounds}-\ref{condition: Xe bounding with lambda} of Lemma~\ref{lem:estimate_rough_bounds} are satisfied and the upper bound
\begin{align*}
    \Vert \bU \Vert \le 8 \sigma \Vert \Sigma \Vert^{1/2} \sqrt{\frac{\ttr(\Sigma) + \log(4/\delta)}{n}}
\end{align*}
holds with probability at least $1 - \delta$.
\end{Prop}

Clearly, Lemma~\ref{lem:estimate_rough_bounds} combined with Proposition~\ref{proposition: localization conditions satisfied} implies the desired result.

\subsection{Proof of Lemma~\ref{lem:estimate_rough_bounds}}

\textbf{Step 1. Bounding the norm of $\widehat \btheta$.} We start from the following chain of inequalities:
\begin{align*}
    \lambda /2 \Vert \widehat  \btheta \Vert^2 \le \cL(\widehat \btheta, \widehat \bfeta, \widehat \Sigma ) \le \cL(\bzero, \bZ, \widehat \Sigma) = \frac{1}{2} \Vert \bZ \Vert ^2.
\end{align*}
It implies $\Vert \widehat \btheta \Vert^2 \le \Vert \bZ \Vert^2/ \lambda$. Next, we have
\begin{align*}
    \Vert \bZ \Vert \le \Vert \widehat \Sigma \btheta^\circ \Vert + \Vert \bU \Vert \le \Vert \widehat \Sigma - \Sigma \Vert \Vert \btheta \Vert + \Vert \Sigma \Vert \Vert \btheta^\circ \Vert + \Vert \bU \Vert
\end{align*}
Using $\Vert \widehat \Sigma - \Sigma \Vert \le \Vert \Sigma \Vert$ from condition~\ref{condition: Sigma bounded via lambda}, $\Vert \Sigma \Vert \Vert \btheta^\circ \Vert \le \rho_0 \mu \sqrt{\lambda}/ (18 \cdot 8)$ from condition~\ref{condition: Sigma theta constant bounds} and $\Vert \bU \Vert \le \rho_0 \mu \sqrt{\lambda} / 18$ from condition~\ref{condition: Xe bounding with lambda}, we infer
\begin{align}
\label{eq: bound on bZ}
    \Vert \bZ \Vert \le 2 \Vert \Sigma \Vert \Vert \btheta^\circ \Vert + \Vert \bU \Vert \le \rho_0 \mu \sqrt{\lambda} / 6.
\end{align}
It yields 
\begin{align}
    \Vert \widehat \btheta \Vert \le \Vert \bZ \Vert / \sqrt{\lambda} \le \rho_0 \mu / 6.
\end{align}
Note that bound~\eqref{eq: bound on bZ} also implies that $(\bzero, \bZ, \widehat \Sigma) \in \Theta \times \sfH \times \sfA$.

\noindent \textbf{Step 2. Weak perturbation bounds.} A bound on $\Vert \widehat \btheta \Vert$ allows us applying the following lemma:
\begin{Lem}
\label{lemma: localization lemma}
    Assume that $\Vert \Sigma \Vert \Vert \btheta^\circ \Vert \le 4 \mu \sqrt{\lambda}$. Then, if $\max \{\Vert \widehat \btheta \Vert, \Vert \btheta^* \Vert, \Vert \btheta^\circ \Vert \} \le \rho_0 \mu$ for $\rho_0 \le 1/2$, then 
    \begin{align*}
            \Vert \widehat \btheta - \btheta^* \Vert \le \frac{8 \Vert \Sigma - \widehat \Sigma \Vert \Vert \btheta^\circ \Vert}{\sqrt{\lambda}}  + \frac{2 \Vert \bU \Vert}{\sqrt{\lambda}}.
        \end{align*}
\end{Lem}
Applying Lemma~\ref{lemma: localization lemma for bias} with $\rho_0$ in place of $\rho$, we get $\Vert \btheta^* \Vert \le \Vert \btheta^\circ \Vert$, and the latter is at most $\rho_0 \mu / 7$ by the conditions of the lemma to be proved. Therefore, assumptions of Lemma~\ref{lemma: localization lemma} are satisfied, and its conclusion holds.

\noindent \textbf{Step 3. Proving that $\widehat \btheta, \btheta^* \in \Theta$.} Since $\Vert \btheta^* \Vert \le \Vert \btheta^\circ \Vert \le \rho_0 \mu$ due to Lemma~\ref{lemma: localization lemma for bias}\ref{point: theta weak bound on bias, bias locating convex set}, the vector $\btheta^*$ belongs to $\Theta$ obviously. Next, we check that $\widehat \btheta \in \Theta$. It requires to bound
    \begin{align}
        \Vert \Sigma \Vert \Vert \widehat \btheta - \btheta^* \Vert \le \frac{8 \Vert \Sigma \Vert \Vert \widehat \Sigma - \Sigma \Vert \Vert \btheta^\circ \Vert}{\sqrt{\lambda}} + \frac{2 \Vert \Sigma \Vert \Vert \bU \Vert}{\sqrt{\lambda}}. \label{eq: Sigma hat theta hat deviation from Z}
    \end{align}
    Using conditions~\ref{condition: Sigma bounded via lambda},\ref{condition: Xe bounding with lambda}, we bound the latter by $4 \rho_0 \mu \sqrt{\lambda}/96 + \rho_0 \mu \sqrt{96} = 5 \rho_0 \mu \sqrt{\lambda}/96$, so $\widehat \btheta \in \Theta$.

    \noindent \textbf{Step 4. Proving that $\widehat A, A^* \in \sfA$.} Using
    \begin{align*}
        \frac{\mu^2}{2} \Vert \widehat A - \widehat \Sigma \Vert^2 \le \cL(\widehat \btheta, \widehat \bfeta, \widehat A) \le \cL(\bzero, \bZ, \widehat \Sigma ) = \Vert \bZ \Vert^2/2,
    \end{align*}
    we get $\Vert \widehat A - \widehat \Sigma \Vert\le \Vert \bZ \Vert / \mu$. Using~\eqref{eq: bound on bZ}, we establish $\Vert \widehat A - \widehat \Sigma \Vert \le \rho_0 \sqrt{\lambda} / 6 \le \sqrt{\lambda}/3$. This implies
    \begin{align*}
        \Vert \widehat A \Vert \le 3 \Vert \Sigma \Vert + \sqrt{\lambda}/3,
    \end{align*}
    since $\Vert \widehat \Sigma \Vert \le 2 \Vert \Sigma \Vert$ due to condition~\ref{condition: Sigma bounded via lambda}. Thus, we have $\widehat A \in \sfA$.
    
    Next, applying Lemma~\ref{lemma: localization lemma for bias}\ref{point: A* weak bound on bias, bias locating convex set}, we get with $\rho_0$ in place of $\rho$, we obtain
    \begin{align*}
        \Vert A^* - \widehat \Sigma \Vert \le \mu^{-1} \sqrt{\lambda} \Vert \btheta^\circ \Vert + \Vert \Sigma - \widehat \Sigma \Vert \le \rho_0 \sqrt{\lambda} / 7 + \Vert \Sigma \Vert \le \Vert \Sigma \Vert + \sqrt{\lambda} /3,
    \end{align*}
    where we used conditions~\ref{condition: Sigma theta constant bounds},\ref{condition: Sigma bounded via lambda} in the second inequality. Hence, $\Vert A^* \Vert \le 3 \Vert \Sigma \Vert + \sqrt{\lambda}/3$ and $A^* \in \sfA$.

    \noindent\textbf{Step 5. Proving that $\bfeta^*, \widehat \bfeta \in \sfH$.} To prove that $\widehat \bfeta \in \sfH$, note that $\widehat \bfeta = \frac{1}{2} (\widehat A \widehat \btheta - \bZ)$. It yields
    \begin{align*}
        \Vert \widehat \bfeta - \bZ \Vert & = \frac{1}{2} \Vert \widehat A \widehat \btheta - \bZ \Vert \le \Vert \widehat A \Vert \Vert \widehat \btheta \Vert/2 + \Vert \bZ \Vert /2 \\
        & \le (3 \Vert \Sigma \Vert + \sqrt{\lambda} / 3) \Vert \widehat \btheta \Vert/2 + \rho_0 \mu \sqrt{\lambda} / 12 \\
        & \le 3 \Vert \Sigma \Vert \Vert \btheta^* \Vert/2 + 3 \Vert \Sigma \Vert \Vert \widehat \btheta - \btheta^* \Vert / 2 + \sqrt{\lambda} \Vert \widehat \btheta \Vert / 6 + \rho_0 \mu \sqrt{\lambda} / 12 \\
        & \le 3 \Vert \Sigma \Vert \Vert \btheta^\circ \Vert/2 + 3 \Vert \Sigma \Vert \Vert \widehat \btheta - \btheta^* \Vert / 2 + \sqrt{\lambda} \Vert \widehat \btheta \Vert / 6 + \rho_0 \mu \sqrt{\lambda} / 12 \\
        & \le \frac{\rho_0 \mu \sqrt{\lambda}}{6 \cdot 16} + \frac{15 \rho_0 \mu \sqrt{\lambda}}{96} + \frac{\rho_0 \mu \sqrt{\lambda}}{6} + \frac{\rho_0 \mu \sqrt{\lambda}}{12} \le \rho_0 \mu \sqrt{\lambda}/3.
    \end{align*}
    where we used $\widehat A \in \sfA$ and $\Vert \bZ \Vert \le \rho_0 \mu \sqrt{\lambda} / 6$ from~\eqref{eq: bound on bZ} for the second inequality,  $\Vert \btheta^* \Vert \le \Vert \btheta^\circ \Vert$ from Lemma~\ref{lemma: localization lemma for bias}\ref{point: theta weak bound on bias, bias locating convex set} for the fourth inequality, $\widehat \btheta \in \Theta$ and $\Vert \Sigma \Vert \Vert \btheta^\circ \Vert \le \mu \sqrt{\lambda} / (18 \cdot 16)$ from condition~\ref{condition: Sigma theta constant bounds} for the fifth inequality. 
    
    To prove $\bfeta^* \in \sfH$, we bound it as follows:
    \begin{align*}
        \Vert \bfeta^* - \bZ \Vert & = \left \Vert \frac{1}{2} (A^* \btheta^* + \Sigma \btheta^\circ ) - \bZ  \right \Vert \\
        & \le \left ( \frac{\Vert A^* \Vert \Vert \btheta^* \Vert + \Vert \Sigma \Vert \Vert \btheta^\circ \Vert}{2} \right )  + \Vert \bZ \Vert \\
        & \le \frac{\Vert A^* \Vert + \Vert \Sigma \Vert}{2} \cdot \Vert \btheta^\circ \Vert + \Vert \bZ \Vert \\
        & \le \Vert \Sigma \Vert \Vert \btheta^\circ \Vert + \Vert \btheta^\circ \Vert \sqrt{\lambda}/3 + \Vert \bZ \Vert \\
        & \le \rho_0 \mu \sqrt{\lambda}/(18 \cdot 16) + \rho_0 \mu \sqrt{\lambda} / 21 + \rho_0 \mu \sqrt{\lambda} / 6 \le \rho_0 \mu \sqrt{\lambda} / 3
    \end{align*}
    where we used $A^* \in \sfA$ in the second inequality. To establish the fourth inequality, we applied bound~\eqref{eq: bound on bZ} and condition~\ref{condition: Sigma theta constant bounds}.

    \noindent \textbf{Step 6. Endgame.} We are left to prove that $\Theta \times \sfH \times \sfA \subset \Ups(\rho_0)$. We state it as the following proposition.
    \begin{Prop}
    \label{proposition: product convex subset of Ups}
        We have
        \begin{align*}
            \Theta \times \sfH \times \sfA \subset \Ups(\rho_0).
        \end{align*}
    \end{Prop}
    \begin{proof}
        Consider arbitrary $(\btheta, \bfeta, A) \in \Theta \times \sfH \times \sfA$. It is enough to prove that $\Vert \btheta \Vert \le \rho_0 \mu$ and $\Vert A \btheta - \bfeta \Vert \le \rho_0 \mu \sqrt{\lambda}$. The former inequality is satisfied by the definition of $\Theta$. To satisfy the second inequality, we bound $\Vert A \btheta - \bfeta\Vert$ as follows:
    \begin{align*}
        \Vert A \btheta - \bfeta \Vert & \le \Vert A \Vert \Vert \btheta \Vert + \Vert \bZ \Vert + \Vert \bfeta - \bZ \Vert \\
        & \le (\Vert \Sigma \Vert + \sqrt{\lambda} / 3) \Vert \btheta \Vert + \Vert \bZ \Vert + \rho_0 \mu \sqrt{\lambda} \\
        & \le \Vert \Sigma \Vert \Vert \btheta^* \Vert + \Vert \Sigma \Vert \Vert \btheta - \btheta^* \Vert + \Vert \btheta \Vert \sqrt{\lambda} / 3 + \Vert \bZ \Vert + \rho_0 \mu \sqrt{\lambda} / 3 \\
        & \le \Vert \Sigma \Vert \Vert \btheta^* \Vert + 5 \rho_0 \mu \sqrt{\lambda} / 96 + \rho_0 \mu \sqrt{\lambda} / 3 + \Vert \bZ \Vert + \rho_0 \mu \sqrt{\lambda}.
    \end{align*}
    where the second inequality holds by the definitions of $\sfA, \sfH$, and the fourth inequality holds by the definition of $\Theta$. Due to~\eqref{eq: bound on bZ}, we have $\Vert \bZ \Vert \le \rho_0 \mu \sqrt{\lambda}/6$, hence, we have
    \begin{align*}
        \Vert A \btheta - \bfeta \Vert \le \left (1 - 11/96 \right ) \rho_0 \mu \sqrt{\lambda} + \Vert \Sigma \Vert \Vert \btheta^* \Vert.
    \end{align*}
    Due to Lemma~\ref{lemma: localization lemma for bias}\ref{point: theta weak bound on bias, bias locating convex set}, we have $\Vert \btheta^* \Vert \le \Vert \btheta^\circ \Vert$. Finally, the quantity $\Vert \Sigma \Vert \Vert \btheta^\circ \Vert$ is at most $11 \rho_0 \mu \sqrt{\lambda} / 96$ due to condition~\ref{condition: Sigma theta constant bounds}, so $\Vert A \btheta - \bfeta \Vert \le \rho_0 \mu \sqrt{\lambda}$.
    \end{proof}

\subsection{Proof of Lemma~\ref{lemma: localization lemma}}

\textbf{Step 1. Constructing a finite difference equation.} The idea is to employ the equation
\begin{align*}
    \bnabla \ttL(\widehat \btheta, \widehat \bfeta, \widehat A) - \bnabla \cL(\btheta^*, \bfeta^*, A^*) = \bzero,
\end{align*}
which holds by definitions of $\widehat \bups = (\widehat \btheta, \widehat \bfeta, \widehat A)$ and $\bups^* = (\btheta^*, \bfeta^*, A^*)$. We have
\begin{align*}
    \bnabla \ttL(\btheta, \bfeta, A) = 
    \begin{pmatrix}
        -  A^\top \bfeta + A^\top A \btheta + \lambda \btheta \\
        - \bZ + 2 \bfeta - A \btheta \\
        A \btheta \btheta^\top - \bfeta \btheta^\top + \mu^2 (A - \widehat{\Sigma})
    \end{pmatrix} \text{ and } \nabla \cL(\btheta, \bfeta, A) = 
    \begin{pmatrix}
        -  A^\top \bfeta + A^\top A \btheta + \lambda \btheta \\
        - \E\bZ + 2 \bfeta - A \btheta \\
        A \btheta \btheta^\top - \bfeta \btheta^\top + \mu^2 (A - \Sigma)
    \end{pmatrix}.
\end{align*}

Define an operator $\stochasticDelta$, acting on a matrix-valued functions $f(\btheta, \bfeta, A)$ as $\stochasticDelta(f) = f(\widehat A, \widehat \eta, \widehat \theta) - f(A^*, \eta^*, \theta^*)$. Note that it possesses the following identities for any two matrix-valued functions $f, g$ such that their product is defined:
\begin{align}
\label{eq: stochastic delta product property}
    \stochasticDelta(fg) = \stochasticDelta(f) \, g(\widehat \btheta, \widehat \bfeta, \widehat A) + f(\btheta^*, \bfeta^*, A^*) \, \stochasticDelta(g); \notag \\
    \stochasticDelta(f g) = f(\widehat \btheta, \widehat \bfeta, \widehat A) \, \stochasticDelta(g) + \stochasticDelta(f) \, g(\btheta^*, \bfeta^*, A^*).
\end{align}
Using $\nabla \ttL(\widehat \btheta, \widehat \bfeta, \widehat A) - \nabla \cL(\btheta^*, \bfeta^*, A^*) = 0$, we obtain
\begin{align*}
    \begin{cases}
        - \stochasticDelta(A^\top \bfeta) + \stochasticDelta(A^\top A \btheta) + \lambda \stochasticDelta(\btheta) & = \bzero \\
        - (\bZ - \E \bZ) + 2 \stochasticDelta(\bfeta) - \stochasticDelta(A \btheta) & = \bzero \\
        \stochasticDelta(A \btheta \btheta^\top) - \stochasticDelta(\bfeta \btheta^\top) + \mu^2 \stochasticDelta(A) & = \mu^2(\widehat{\Sigma} - \Sigma)
    \end{cases}
\end{align*}
The second equation implies $\stochasticDelta(\bfeta) = \frac{1}{2} (\stochasticDelta(A \btheta) + (\bZ - \E \bZ))$. Substituting it in the above system, we get
\begin{align*}
    \stochasticDelta(A^\top \bfeta) & = \widehat A^\top \stochasticDelta(\bfeta) + \stochasticDelta^\top(A) \bfeta^* \\
    & = \frac{1}{2} \widehat A^\top \stochasticDelta(A \btheta) + \frac{1}{2} \widehat A^\top \stochasticDelta_{\bZ} + \stochasticDelta^\top(A) \bfeta^*, \\
    \stochasticDelta(\bfeta \btheta^\top) & = \stochasticDelta(\bfeta) \widehat \btheta^\top + \bfeta^* \stochasticDelta^\top(\btheta) \\
    & = \frac{1}{2} \stochasticDelta(A \btheta) \widehat \btheta^\top + \frac{1}{2} \stochasticDelta_{\bZ} \widehat{\btheta}^\top + \bfeta^* \stochasticDelta^\top (\btheta),
\end{align*}
where we denote $\bZ - \E \bZ$ by $\stochasticDelta_{\bZ}$ for brevity.
Since $\stochasticDelta(A \btheta \btheta^\top) = \stochasticDelta(A \btheta) \widehat \btheta^\top + A^* \btheta^* \stochasticDelta^\top(\btheta)$ and $\stochasticDelta(A^\top A \btheta) = \widehat A^\top \stochasticDelta(A \btheta) + \stochasticDelta^\top(A) A^* \btheta^*$ due to~\eqref{eq: stochastic delta product property}, we have
\begin{align}
\label{eq: system on Deltas without eta}
    \begin{cases}
        \frac{1}{2} \widehat A^\top \stochasticDelta(A \btheta) - \frac{1}{2} \widehat A^\top \stochasticDelta_{\bZ} + \stochasticDelta^\top(A) (A^* \btheta^* -  \bfeta^*) + \lambda \stochasticDelta(\btheta) & = \bzero, \\
        \frac{1}{2} \stochasticDelta(A \btheta) \widehat \btheta^\top - \frac{1}{2} \stochasticDelta_{\bZ} \widehat \btheta^\top + (A^* \btheta^* - \bfeta^*) \stochasticDelta^\top(\btheta) + \mu^2 \stochasticDelta(A) & = \mu^2 ( \widehat \Sigma - \Sigma).
    \end{cases}
\end{align}

\noindent \textbf{Step 2. Bounding the norm of $\Vert \stochasticDelta(A) \Vert$.} Deriving $\stochasticDelta(A \btheta) = \stochasticDelta(A) \widehat \btheta + A^* \stochasticDelta(\btheta)$ from~\eqref{eq: stochastic delta product property} and substituting it into the last equation of~\eqref{eq: system on Deltas without eta}, we obtain
\begin{align*}
    \stochasticDelta(A) \left ( \mu^2 I_d + \frac{1}{2} \widehat \theta \widehat \theta^\top \right ) = \mu^2 (\widehat \Sigma - \Sigma) - (A^* \theta^* - \eta^*) \stochasticDelta^\top(\theta) + \frac{1}{2} \stochasticDelta_{\bZ} \widehat \theta^\top.
\end{align*}
For brevity, denote $A^* \btheta^* - \bfeta^*$ by $\bb_{\bfeta}$.
Rearranging terms, we get
\begin{align*}
    \stochasticDelta(A) & = \mu^2  (\widehat \Sigma - \Sigma) \left (\mu^2 + \frac{1}{2} \widehat \btheta \widehat \btheta^\top \right )^{-1} -  \bb_{\bfeta} \stochasticDelta^\top(\btheta) \left (\mu^2 + \frac{1}{2} \widehat \btheta \btheta^\top \right )^{-1} \nonumber \\
    & \quad +  \frac{1}{2} \stochasticDelta_{\bZ} \widehat \btheta^\top \left (\mu^2 + \frac{1}{2} \widehat \btheta \btheta^\top \right )^{-1},
\end{align*}
and, therefore, we bound
\begin{align}
\label{eq: Delta(A) norm variance weak bound with Delta(theta)}
    \Vert \stochasticDelta(A) \Vert \le \Vert \Sigma - \widehat \Sigma \Vert + \mu^{-2}\Vert \bb_{\bfeta} \Vert \Vert \stochasticDelta(\btheta) \Vert + \frac{1}{2} \mu^{-2} \Vert \stochasticDelta_{\bZ} \Vert \Vert \widehat\btheta \Vert.
\end{align}

\noindent \textbf{Step 3. Establishing a bound on $\Vert \stochasticDelta(\btheta) \Vert$.} We bound $\Vert \stochasticDelta(\btheta) \Vert$ using $\Vert \stochasticDelta(A) \btheta^* \Vert$ and $\Vert \stochasticDelta^\top(A) \bb_{\bfeta} \Vert$. From the first equation of~\eqref{eq: system on Deltas without eta}, we obtain
\begin{align*}
    \frac{1}{2} \widehat A^\top \stochasticDelta(A \btheta) - \frac{1}{2} \widehat A^\top \stochasticDelta_{\bZ} + \stochasticDelta^\top(A) \bb_{\bfeta} + \lambda \stochasticDelta(\btheta) = \bzero.
\end{align*}
Since $\stochasticDelta(A \btheta) = \widehat A \stochasticDelta(\btheta) + \stochasticDelta(A) \btheta^*$, we have
\begin{align}
\label{eq: equation on Delta(btheta)}
    \left ( \frac{1}{2} \widehat{A}^\top \widehat{A} + \lambda I_d \right )\stochasticDelta(\btheta) = \frac{1}{2} \widehat{A}^\top \stochasticDelta_{\bZ} - \stochasticDelta^\top(A) \bb_{\bfeta} - \frac{1}{2} \widehat{A}^\top \stochasticDelta(A) \btheta^*.
\end{align}
Next, we bound the norms $\left \Vert \left ( \frac{1}{2} \widehat{A}^\top \widehat{A} + \lambda I_d \right )^{-1} \right \Vert$ and $\left \Vert \left ( \frac{1}{2} \widehat{A}^\top \widehat{A} + \lambda I_d \right )^{-1} \widehat A \right \Vert $. While the first norm is clearly at most $1/\lambda$, for the second norm we need to maximize
\begin{align*}
    \left \Vert \left ( \frac{1}{2} \widehat{A}^\top \widehat{A} + \lambda I_d \right )^{-1} \widehat A \right \Vert  = \max_i \frac{\sigma_i}{\sigma_i^2/2 + \lambda}
\end{align*}
over all possible singular values $\sigma_1, \ldots, \sigma_n$ of $\widehat A$. However, we have
\begin{align*}
    \max_i \frac{\sigma_i}{\sigma_i^2/2 + \lambda} \le \max_{x \in \R_+} \frac{x}{x^2/2 + \lambda} = \frac{\sqrt{2 \lambda}}{2 \lambda} \le \frac{1}{\sqrt{2 \lambda}}
\end{align*}
by elementary calculus. 
Thus, multiplying both sides of~\eqref{eq: equation on Delta(btheta)} by $\left ( \frac{1}{2} \widehat{A}^\top \widehat{A} + \lambda I_d \right )^{-1}$ and applying bounds on $\left \Vert \left ( \frac{1}{2} \widehat{A}^\top \widehat{A} + \lambda I_d \right )^{-1} \right \Vert$ and $\left \Vert \left ( \frac{1}{2} \widehat{A}^\top \widehat{A} + \lambda I_d \right )^{-1} \widehat A \right \Vert $, we derive
\begin{align*}
    \Vert \stochasticDelta(\btheta) \Vert & \le \frac{\Vert \stochasticDelta_{\bZ} \Vert}{\sqrt{8 \lambda}} + \frac{1}{\lambda} \Vert \stochasticDelta^\top(A) \bb_{\bfeta} \Vert + \frac{1}{2 \sqrt{\lambda}} \Vert \stochasticDelta(A) \btheta^* \Vert \\
    & \le \frac{\Vert \stochasticDelta_{\bZ} \Vert}{\sqrt{8\lambda}} + \frac{\Vert \stochasticDelta(A) \Vert}{\sqrt{\lambda}} \left (\frac{1}{2} \Vert \btheta^* \Vert + \frac{\Vert \bb_{\bfeta} \Vert}{\sqrt{\lambda}} \right ) \\
    & \le\frac{\Vert \stochasticDelta_{\bZ} \Vert}{\sqrt{8 \lambda}} + \frac{1}{\sqrt{\lambda}} \left (\Vert \Sigma - \widehat \Sigma \Vert + \frac{\mu^{-2}}{2} \Vert \stochasticDelta_{\bZ} \Vert \Vert \widehat \btheta \Vert \right ) \left (\frac{1}{2} \Vert \btheta^* \Vert + \frac{\Vert \bb_{\bfeta} \Vert}{\sqrt{\lambda}} \right ) \\
    & \quad + \frac{\Vert \bb_{\bfeta} \Vert \Vert \stochasticDelta(\btheta) \Vert}{\mu^2 \sqrt{\lambda}} \left (\frac{1}{2} \Vert \btheta^* \Vert + \frac{\Vert \bb_{\bfeta} \Vert}{\sqrt{\lambda}} \right ),
\end{align*}
where the third inequality holds due to~\eqref{eq: Delta(A) norm variance weak bound with Delta(theta)}. 
According to Lemma~\ref{lemma: localization lemma for bias}\ref{point: bb bfeat weak bound, bias locating convex set}, we have $\Vert \bb_{\bfeta} \Vert \le \sqrt{\lambda} \Vert \btheta^\circ \Vert$. It implies
\begin{align*}
    \Vert \stochasticDelta(\btheta) \Vert & \le \frac{\stochasticDelta_{\bZ}}{\sqrt{8 \lambda}} + \frac{1}{\sqrt{\lambda}} \left (\Vert \Sigma - \widehat \Sigma \Vert + \frac{\mu^{-2}}{2} \Vert \stochasticDelta_{\bZ} \Vert \Vert \widehat \btheta \Vert \right ) \left ( \frac{1}{2} \Vert \btheta^* \Vert + \Vert \btheta^\circ \Vert \right ) \\
    & \quad + \mu^{-2} \Vert \stochasticDelta(\btheta) \Vert \Vert \btheta^\circ \Vert  \left ( \frac{1}{2} \Vert \btheta^* \Vert + \Vert \btheta^\circ \Vert \right ).
\end{align*}
By assumptions of the lemma, we have $\Vert \btheta^\circ \Vert, \Vert \btheta^* \Vert$ and $\Vert \widehat \btheta \Vert$ are at most $\rho_0 \mu$, which implies
\begin{align*}
    \Vert \stochasticDelta(\btheta) \Vert  & \le \frac{\Vert \stochasticDelta_{\bZ} \Vert}{\sqrt{8 \lambda}} + \frac{1}{\sqrt{\lambda}} \left ( \Vert \Sigma - \widehat \Sigma \Vert + \rho_0 \mu^{-1} / 2 \Vert \stochasticDelta_\bZ \Vert \right ) \left (\frac{1}{2} \Vert \btheta^* \Vert + \Vert \btheta^\circ \Vert \right) \\
    & \quad + 2 \rho^2_0 \Vert \stochasticDelta(\btheta) \Vert.
\end{align*}
Since $\stochasticDelta_{\bZ} = (\Sigma - \widehat \Sigma ) \btheta^\circ + \bU$, we may simplify the inequality above and obtain
\begin{align*}
    (1 - 2 \rho^2_0) \Vert \stochasticDelta(\btheta) \Vert & \le \frac{\Vert \stochasticDelta_{\bZ} \Vert}{2 \sqrt{\lambda}} \cdot (1 + 2 \rho^2_0) + \frac{\Vert \Sigma - \widehat \Sigma \Vert}{\sqrt{\lambda}} \left ( \Vert \btheta^* \Vert + \Vert \btheta^\circ \Vert \right ) \\
    & \le \frac{2 \Vert \Sigma - \widehat \Sigma \Vert}{\sqrt{\lambda}} (\Vert \btheta^* \Vert + \Vert \btheta^\circ \Vert) + \frac{\Vert \bU \Vert}{\sqrt{\lambda}},
\end{align*}
where we used $\rho_0 \le 1/2$.
It implies the desired bound on $\stochasticDelta(\btheta) = \widehat  \btheta - \btheta^*$:
\begin{align*}
    \Vert \stochasticDelta(\btheta) \Vert \le \frac{4 \Vert \Sigma - \widehat \Sigma \Vert}{\sqrt{\lambda}} (\Vert \btheta^* \Vert + \Vert \btheta^\circ \Vert) + \frac{2 \Vert \bU \Vert}{\sqrt{\lambda}}.
\end{align*}
Applying bound $\Vert \btheta^* \Vert \le \Vert \btheta^\circ \Vert$ from Lemma~\ref{lemma: localization lemma for bias}\ref{point: theta weak bound on bias, bias locating convex set}, we obtain  the lemma. \myendproof

\subsection{Proof of Proposition~\ref{proposition: localization conditions satisfied}}
\label{section: proof of loc condition satisfied}

\textbf{Step 1. Bound on $\Vert \bU \Vert$.} To ensure conditions~\ref{condition: Sigma bounded via lambda},\ref{condition: Xe bounding with lambda}, we shall obtain bounds on $\Vert \widehat \Sigma - \Sigma \Vert$ and $\Vert \bU \Vert$ with high probability. For the latter quantity, we apply Theorem~\ref{th:noise_concentration} with $B = I_d$, and obtain
\begin{align}
\label{eq: U hp bound}
    \Vert \bU \Vert \le 8 \sigma \Vert \Sigma \Vert^{1/2} \sqrt{\frac{\ttr(\Sigma) + \log(4/\delta)}{n}},
\end{align}
with probability at least $1 - \delta /2 $, provided $n \ge \ttr(\Sigma) + \log(4/\delta)$.

\textbf{Step 2. Bound on $\Vert \widehat \Sigma - \Sigma \Vert$.} First, we will show that under Assumption~\ref{as:quadratic}, a random vector $\bX_1$ satisfies $\psi_2-L_2$-equivalence. In what follows, $\Vert \cdot \Vert_{\psi_1}, \Vert \cdot \Vert_{\psi_2}$ stand for $\psi_1$ and $\psi_2$ Orlicz norms.

\begin{Lem}[Lemma B.3 of~\cite{puchkin24}]
Suppose that a random vector $\bX$ satisfies Assumption~\ref{as:quadratic}. Then, for any vector $\bu \in \R^d$, it holds that
\begin{align*}
    \Vert \bu^\top \bX \Vert_{\psi_2}^2 \le \frac{1 + C_X}{\log 2} \cdot \bu^\top \Sigma \bu.
\end{align*}
\end{Lem}

Then, we deduce the concentration bound on $\widehat \Sigma$ from the following theorem.

\begin{Lem}[Theorem 1 of~\cite{zhivotovskiy21}]
    Assume that $M_1, \ldots, M_n$ are independent copies of a $d$ by $d$ positive semi-definite symmetric matrix $M$ with mean $\E M = \Sigma$. Let $M$ satisfy for some $\kappa \ge 1$
    \begin{align*}
        \Vert \bu^\top M \bu \Vert_{\psi_1} \le \kappa^2 \bu^\top \Sigma \bu
    \end{align*}
    for all $\bu \in \R^d$. Then, for any $\delta > 0$, with probability $1 - \delta$, it holds that
    \begin{align*}
        \left \Vert \frac{1}{n} \sum_{i = 1}^n M_i - \Sigma \right \Vert \le 20 \kappa^2 \Vert \Sigma \Vert \sqrt{\frac{4 \ttr(\Sigma) + \log(1/\delta)}{n}},
    \end{align*}
    whenever $n \ge 4 \ttr(\Sigma) + \log(1/\delta)$
\end{Lem}

Combining the above two lemmata and using an inequality $\Vert \xi^2 \Vert_{\psi_1} \le \Vert \xi \Vert_{\psi_2}^2$ suitable for any random variable $\xi$, we derive
\begin{align}
\label{eq: Sigma hp deviation bound in the spectral norm}
    \Vert \widehat \Sigma - \Sigma \Vert \le \frac{20 (1 + C_X)}{\log 2} \cdot \Vert \Sigma \Vert \sqrt{\frac{4 \ttr(\Sigma) + \log(2/\delta)}{n}}
\end{align}
with probability at least $1 - \delta/2$.

\textbf{Step 3. Satisfiability of conditions~\ref{condition: Sigma bounded via lambda},\ref{condition: Xe bounding with lambda} of Lemma~\ref{lem:estimate_rough_bounds}.} Note that if
\begin{align}
    \lambda & \ge \frac{2^{13}\Vert \Sigma \Vert^2 \Vert \btheta^\circ \Vert}{\rho_0 \mu} (1 + C_X) \sqrt{\frac{4 \ttr(\Sigma) + \log(2/\delta)}{n}}, \label{eq: lambda_1_lower_bound_proof_of_cond_satisfiability} \\
    n & \ge 2^{12} (1 + C_X)^2 \left ( \ttr(\Sigma) + \log(2/\delta) \right ) \label{eq: n_lower_bound_proof_of_cond_satisfiability},
\end{align}
then condition~\ref{condition: Sigma bounded via lambda} of Lemma~\ref{lem:estimate_rough_bounds} is satisfied with probability at least $1 - \delta/2$ due to~\eqref{eq: Sigma hp deviation bound in the spectral norm}. Similarly, due to~\eqref{eq: U hp bound}, condition~\ref{condition: Xe bounding with lambda} of Lemma~\ref{lem:estimate_rough_bounds} is satisfied with probability at least $1 - \delta /2$ if
\begin{align}
    \frac{\lambda}{\Vert \Sigma \Vert} \wedge \sqrt{\lambda} \ge \frac{2^{11} \sigma \Vert \Sigma \Vert^{1/2}}{\rho_0 \mu} \sqrt{\frac{\ttr(\Sigma) + \log(4/\delta)}{n}} \label{eq: lambda_2_lower_bound_proof_of_cond_satisfiability}
\end{align}
holds. The union bound implies that conditions~\ref{condition: Sigma bounded via lambda}-\ref{condition: Xe bounding with lambda} simultaniously hold with probability at least $1 - \delta$, provided inequalities~\eqref{eq: lambda_1_lower_bound_proof_of_cond_satisfiability}-\eqref{eq: lambda_2_lower_bound_proof_of_cond_satisfiability} are satisfied.

\myendproof

\section{Proof of Theorem~\ref{theorem: rates of convergence}}

\noindent \textbf{Step 1. Establishing a finite difference equation.} We consider the following three stationary point conditions:
\begin{align*}
    \nabla_{\btheta} \cL(\btheta_t, (A_{t - 1} \btheta_t + \bZ)/2, A_{t - 1}) & = \bzero, \\
    \nabla_{A} \cL(\btheta_t, (A_t \btheta_t + \bZ) / 2, A_t) & = \bzero, \\
    \nabla_{\bfeta} \cL(\widehat \btheta, (\widehat A \widehat \btheta + \bZ) / 2, \widehat A) & = \bzero,
\end{align*}
The first two equations can be expanded as follows:
\begin{align}
\label{eq: system on theta_t withour Delta}
    \begin{cases}
        \left ( \frac{1}{2} A_{t - 1}^\top A_{t - 1} + \lambda I_d \right ) \btheta_t - \frac{1}{2} A^\top_{t - 1} \bZ = \bzero \\
        A_t \btheta_t \btheta_t^\top / 2  - \frac{1}{2} \bZ \btheta^\top + \mu^2 (A_t - \widehat \Sigma) = \bzero.
    \end{cases}
\end{align}
For a matrix-valued function $f(\btheta, \bfeta, A)$, we define an operator
\begin{align*}
    \Delta_t(f) = f(\btheta_t, \bfeta_t, A_t) - f(\widehat \btheta, \widehat \bfeta, \widehat A).
\end{align*}
Note that for product of any matrix-valued functions $f, g$ of appropriate shapes, we can write
\begin{align}
\label{eq: Delta-t product expansion}
    \Delta_t(fg) & = f(\btheta_t, \bfeta_t, A_t) \Delta_t(g) + \Delta_t(f) g(\widehat \btheta, \widehat \bfeta, \widehat A) \notag \\
    \Delta_t(fg) & = \Delta_t(f) g(\btheta_t, \bfeta_t, A_t) + f(\widehat \btheta, \widehat \bfeta, \widehat A) \Delta_t(g).
\end{align}
Using this notation, we subtract $\nabla_{\btheta} \cL(\widehat \btheta, (\widehat A \widehat \btheta + \bZ)/2, \widehat A) = \bzero$ from the first equation of~\eqref{eq: system on theta_t withour Delta} and $\nabla_A \cL(\widehat \btheta, (\widehat A \widehat \btheta + \bZ)/2, \widehat A) = \bzero$ from the second equation of~\eqref{eq: system on theta_t withour Delta}, and obtain
\begin{align*}
    \begin{cases}
        \left (\frac{1}{2} A_{t - 1}^\top A_{t - 1} + \lambda I_d \right ) \Delta_t(\btheta) + \frac{1}{2} \Delta_{t - 1}(A^\top A) \widehat \btheta - \frac{1}{2} \Delta_{t - 1}^\top(A) \bZ & = \bzero, \\
        \Delta_t(A) (\btheta_t \btheta_t^\top/ 2 + \mu^2 I_d) + \frac{1}{2} \widehat A \Delta_t(\btheta \btheta^\top) - \frac{1}{2} \bZ \Delta^\top_t(\btheta) & = \bzero.
    \end{cases}
\end{align*}
Expanding $\Delta_{t - 1}(A^\top A) = A_{t - 1}^\top \Delta_{t - 1}(A) + \Delta_{t - 1}^\top (A)\widehat A$, {$\Delta_{t}(\btheta \btheta^\top) = \Delta_{t}(\btheta) \btheta_t^\top + \widehat \btheta \Delta_t^\top(\btheta)$ due to~\eqref{eq: Delta-t product expansion}}, and rearranging terms, we obtain
\begin{align}
\label{eq: system on Delta_t in negative feedback form}
    \begin{cases}
        \Delta_t(\btheta) & = - \frac{1}{2} (A_{t - 1}^\top A_{t - 1}/2 + \lambda I_d)^{-1} \left [ A_{t - 1}^\top  \Delta_{t - 1}(A) \widehat \btheta + \frac{1}{2 } \Delta_{t - 1}(A) (\widehat A \widehat \btheta - \bZ) \right ], \\
        \Delta_{t}(A) & = - \frac{1}{2} \left [\widehat A \Delta_t(\btheta) \btheta_t^\top  + (\widehat A \widehat \btheta - \bZ ) \Delta_t(\btheta) \right ](\mu^2 I_d + \btheta_t \btheta_t^\top / 2)^{-1}.
    \end{cases}
\end{align}

\noindent \textbf{Step 2. Recursively bounding the norm of $\Delta_t(\btheta)$ and $\Delta_t(A)$.} Let us start with the first equation of the system above. We have
\begin{align*}
    \Vert \Delta_t(\btheta) \Vert \le \Vert \Delta_{t - 1}(A) \Vert \left [ \Vert (A_{t - 1}^\top A_{t - 1} + \lambda I_d)^{-1} A_{t - 1}^\top \Vert \Vert \widehat \btheta \Vert + \Vert (A_{t - 1}^\top A_{t - 1} + 2 \lambda I_d)^{-1} \Vert \Vert \widehat A \widehat \btheta - \bZ \Vert\right ]. 
\end{align*}
Let $\sigma_1, \ldots, \sigma_d$ be the singular values of $A_{t - 1}$. Using elementary calculus, we obtain
\begin{align*}
    \Vert (A_{t - 1}^\top A_{t - 1} + \lambda I_d)^{-1} A_{t - 1}^\top \Vert = \max_i \frac{\sigma_i}{\sigma_i^2 + \lambda} \le \max_{x \in \R_+} \frac{x}{x^2 + \lambda} \le \frac{1}{\sqrt{\lambda}}.
\end{align*}
Then, using $\Vert  (A_{t - 1}^\top A_{t - 1} + 2 \lambda I_d)^{-1} \Vert \le 1/(2\lambda)$, we elaborate
\begin{align}
\label{eq: Delta_t theta bound via Delta_t A}
    \Vert \Delta_t(\btheta) \Vert \le \Vert \Delta_{t - 1}(A) \Vert \left [ \frac{\Vert \widehat \btheta \Vert}{\sqrt{\lambda}} + \frac{\Vert \widehat A  \widehat \btheta - \bZ \Vert}{2 \lambda} \right ].
\end{align}
Next, we analyze the norm of $\Delta_t(A)$. From the second equation~\eqref{eq: system on Delta_t in negative feedback form}, we infer
\begin{align*}
    \Vert \Delta_t(A) \Vert \le \mu^{-2} \left [ \Vert \widehat A \Vert \Vert \btheta_t \Vert + \Vert \widehat A \widehat \btheta - \bZ \Vert\right ] \cdot \Vert \Delta_{t}(\btheta) \Vert / 2.
\end{align*}
Bounding $\Vert \btheta_t \Vert \le \Vert \widehat \btheta \Vert + \Vert \Delta_t(\btheta) \Vert$ and applying bound~\eqref{eq: Delta_t theta bound via Delta_t A}, we obtain
\begin{align*}
    \Vert \Delta_t(A) \Vert & \le \frac{\mu^{-2}}{2}  \left [ \Vert \widehat A \Vert \Vert \widehat \btheta \Vert + \Vert \widehat A \widehat \btheta - \bZ \Vert \right ] \left [ \frac{\Vert \widehat \btheta \Vert}{\sqrt{\lambda}} + \frac{\Vert \widehat A \widehat \btheta - \bZ \Vert }{2 \lambda} \right ] \Vert \Delta_{t - 1}(A) \Vert  \\
     & \quad + \frac{\mu^{-2}}{2}   \Vert \widehat A \Vert  \left [ \frac{\Vert \widehat \btheta \Vert}{\sqrt{\lambda}} + \frac{\Vert \widehat A \widehat \btheta - \bZ \Vert }{2 \lambda} \right ]^2 \Vert \Delta_{t - 1}(A) \Vert^2.
\end{align*}

\noindent \textbf{Step 3. Exploiting the recursion.} We will prove by induction, that
\begin{align}
\label{eq: contraction of Delta A}
    \Vert\Delta_t(A) \Vert \le \rho_0^{2t} \Vert \widehat A - \widehat \Sigma \Vert.
\end{align}
Clearly, the base $t = 0$ holds, by the definition of $A_0 = \widehat \Sigma$. Assume that~\eqref{eq: contraction of Delta A} holds for $t$, and let us prove it for $t + 1$. We have
\begin{align*}
    \Vert \Delta_{t + 1}(A) \Vert & \le \frac{\mu^{-2}}{2}  \left [ \Vert \widehat A \Vert \Vert \widehat \btheta \Vert + \Vert \widehat A \widehat \btheta - \bZ \Vert +  \Vert \widehat A \Vert  \left [ \frac{\Vert \widehat \btheta \Vert}{\sqrt{\lambda}} + \frac{\Vert \widehat A \widehat \btheta - \bZ \Vert }{2 \lambda} \right ] \cdot \Vert \Delta_{t}(A) \Vert \right ] \\
    & \quad \times \left [ \frac{\Vert \widehat \btheta \Vert}{\sqrt{\lambda}} + \frac{\Vert \widehat A \widehat \btheta - \bZ \Vert }{2 \lambda} \right ]  \Vert \Delta_{t}(A) \Vert \\
    & \le \frac{\mu^{-2}}{2} \left [ \Vert \widehat A \Vert \Vert \widehat \btheta \Vert + \Vert \widehat A \widehat \btheta - \bZ \Vert +  \Vert \widehat A \Vert  \left [ \frac{\Vert \widehat \btheta \Vert}{\sqrt{\lambda}} + \frac{\Vert \widehat A \widehat \btheta - \bZ \Vert }{2 \lambda} \right ] \cdot \Vert \widehat A - \widehat \Sigma \Vert \right ] \\
    & \quad \times \left [ \frac{\Vert \widehat \btheta \Vert}{\sqrt{\lambda}} + \frac{\Vert \widehat A \widehat \btheta - \bZ \Vert }{2 \lambda} \right ]  \Vert \Delta_{t}(A) \Vert.
\end{align*}
Next, we use the fact that $\widehat A \in \sfA$ and $\widehat \btheta \in \Theta$. The definitions of sets $\Theta, \sfA$ imply $\Vert \widehat A - \widehat \Sigma \Vert \le \sqrt{\lambda}/3$, $\Vert \widehat \btheta \Vert \le \rho_0 \mu$ and 
\begin{align}
\label{eq: rough bound on stochastic eta difference}
    \Vert \widehat A \widehat \btheta - \bZ \Vert \le \Vert \widehat A - \widehat \Sigma \Vert \Vert \widehat \btheta \Vert + \Vert \widehat \Sigma \widehat \btheta - \bZ \Vert \le \sqrt{\lambda} \Vert \widehat \btheta \Vert / 3 + \rho_0 \mu \sqrt{\lambda} / 3 \le 2 \rho_0 \mu \sqrt{\lambda} / 3.
\end{align}
It yields
\begin{align*}
    \Vert \Delta_{t + 1}(A) \Vert & \le \frac{\mu^{-2}}{2} \left [\Vert \widehat A \Vert \Vert \widehat \btheta \Vert + \rho_0 \mu \sqrt{\lambda}/3 + \Vert \widehat A \Vert \Vert \widehat \btheta \Vert /3 + \frac{\Vert \widehat A \Vert \Vert \widehat A \widehat \btheta - \bZ \Vert}{3 \sqrt{\lambda}}\right ] \cdot \frac{2 \rho_0 \mu}{3 \sqrt{\lambda}} \cdot \Vert \Delta_{t}(A) \Vert \\
    & = \rho_0 \cdot \left [\rho_0/3 + \frac{4 \Vert \widehat A \Vert \Vert \widehat \btheta \Vert}{9\mu \sqrt{\lambda}} + \frac{\Vert \widehat A \Vert \Vert \widehat A \widehat\btheta - \bZ \Vert}{9 \mu \lambda}\right ] \Vert \Delta_t(A) \Vert.
\end{align*}
Bound~\eqref{eq: rough bound on stochastic eta difference} does not allow establishing $\Vert \Delta_t(A)\Vert \le \rho_0^2 \Vert \Delta_t(A) \Vert$ from the above inequality, which would imply~\eqref{eq: contraction of Delta A}. Instead, we use the definition of $(\widehat \btheta, \widehat \bfeta, \widehat A)$:
\begin{align*}
    \cL(\widehat \btheta, \widehat \bfeta, \widehat A) \le \cL(\btheta^\circ, (\widehat \Sigma \btheta^\circ + \bZ)/2, \widehat \Sigma) = \frac{1}{4} \Vert \widehat \Sigma \btheta^\circ - \bZ \Vert^2 + \lambda/2 \Vert \btheta^\circ \Vert^2.
\end{align*}
Using $\widehat \bfeta = (\widehat A \widehat \btheta + \bZ)/2$, we get
\begin{align*}
    \cL(\widehat \btheta, \widehat \bfeta, \widehat A) = \frac{1}{4} \Vert \widehat A \widehat \btheta - \bZ \Vert^2 + \frac{\mu^2}{2} \Vert \widehat A - \widehat \Sigma \Vert + \frac{\lambda}{2} \Vert \widehat \btheta \Vert,
\end{align*}
so we have
\begin{align}
\label{eq: definition of argminimum bound of A theta - Z}
    \Vert \widehat A \widehat \btheta - \bZ \Vert \le \Vert \widehat \Sigma \btheta^\circ - \bZ \Vert + \sqrt{2 \lambda} \Vert \btheta^\circ \Vert =  \Vert \bU \Vert + \sqrt{2 \lambda} \Vert \btheta^\circ \Vert,
\end{align}
where we used decomposition $\bZ = \widehat \Sigma \btheta^\circ + \bU$. Therefore, it holds that
\begin{align*}
    \Vert \Delta_{t + 1}(A) \Vert \le \rho_0 \cdot \left [\rho_0 /3 + \frac{4 \Vert \widehat A \Vert \Vert \widehat \btheta\Vert}{9 \mu \sqrt{\lambda}} + \frac{\Vert \widehat A \Vert \Vert \bU \Vert}{9 \mu \lambda} + \frac{2 \Vert \widehat A \Vert \Vert \btheta^\circ \Vert}{9 \mu \sqrt{\lambda}}\right ] \cdot \Vert \Delta_t(A) \Vert.
\end{align*}
 Under assumptions of the theorem, Proposition~\ref{proposition: localization conditions satisfied} implies the satisfaction of conditions~\ref{condition: Sigma theta constant bounds}-\ref{condition: Xe bounding with lambda} of Lemma~\ref{lem:estimate_rough_bounds} and the upper bound
 \begin{align}
 \label{eq: U bound in the proof of convergence theorem}
     \Vert \bU \Vert \le 8 \sigma \Vert \Sigma \Vert^{1/2} \sqrt{\frac{\ttr(\Sigma) + \log(4/\delta)}{n}}
 \end{align}
 with probability at least $1 - \delta$. So, we assume that these conditions hold.
Then, we have
\begin{align*}
    \Vert \widehat A \Vert \le \Vert \widehat \Sigma \Vert + \Vert \widehat A - \widehat \Sigma \Vert \le 2 \Vert \Sigma \Vert + \sqrt{\lambda}/3,
\end{align*}
where we used condition~\ref{condition: Sigma bounded via lambda} of Lemma~\ref{lem:estimate_rough_bounds} and the fact that $\widehat A \in \sfA$. It yields
\begin{align*}
    \Vert \Delta_{t + 1}(A) \Vert \le \rho_0 \cdot \left [\rho_0 /3 + \frac{8 \Vert \Sigma\Vert \Vert \widehat \btheta\Vert}{9 \mu \sqrt{\lambda}} + \frac{4 \Vert \widehat \btheta \Vert}{27 \mu} + \frac{2 \Vert \Sigma \Vert \Vert \bU \Vert}{9 \mu \lambda} + \frac{\Vert \bU \Vert}{27 \mu \sqrt{\lambda}} + \frac{2 \Vert \Sigma \Vert \Vert \btheta^\circ \Vert}{9 \mu \sqrt{\lambda}}\right ] \cdot \Vert \Delta_t(A) \Vert.
\end{align*}
Using the fact that $\widehat \btheta \in \Theta$ and conditions~\ref{condition: Sigma bounded via lambda},\ref{condition: Xe bounding with lambda} of Lemma~\ref{lem:estimate_rough_bounds}, we obtain
\begin{align*}
    \Vert \Delta_{t + 1}(A) \Vert \le \rho_0 \cdot \left [\rho_0 /2 + \frac{8 \Vert \Sigma \Vert \Vert \widehat \btheta \Vert}{9 \mu \sqrt{\lambda}}\right ] \cdot \Vert \Delta_t(A) \Vert
\end{align*}
From Lemma~\ref{lemma: localization lemma}, we have
\begin{align*}
    \Vert \widehat \btheta - \btheta^* \Vert \le \frac{4 \Vert \widehat \Sigma - \Sigma \Vert (\Vert \btheta^* \Vert + \Vert \btheta^\circ \Vert)}{\sqrt{\lambda}} + \frac{2 \Vert \bU \Vert}{\sqrt{\lambda}} \le \frac{8 \Vert \btheta^\circ \Vert}{7} + \frac{2 \Vert \bU \Vert}{\sqrt{\lambda}},
\end{align*}
where we used $\Vert \widehat \Sigma - \Sigma \Vert \le \sqrt{\lambda} / 7$ from condition~\ref{condition: Sigma bounded via lambda} and $\Vert \btheta^* \Vert \le \Vert \btheta^\circ \Vert$ from Lemma~\ref{lemma: localization lemma for bias}\ref{point: theta weak bound on bias, bias locating convex set}. Thus, we have 
\begin{align}
\label{eq: theta hat bound via theta circ}
    \Vert \widehat \btheta \Vert \le \Vert \btheta^* \Vert + \frac{8 \Vert \btheta^\circ \Vert}{7} + \frac{2 \Vert \bU \Vert}{\sqrt{\lambda}} \le \frac{15 \Vert \btheta^\circ \Vert}{7} + \frac{2 \Vert \bU \Vert}{\sqrt{\lambda}},
\end{align}
and
\begin{align*}
    \Vert \Delta_{t + 1}(A) \Vert & \le \rho_0 \cdot \left [ \rho_0/ 2 + \frac{120 \Vert \Sigma \Vert \Vert \btheta^\circ \Vert}{63 \mu \sqrt{\lambda}} + \frac{16 \Vert \Sigma \Vert \Vert \bU \Vert}{9 \mu \lambda}\right ] \Vert \Delta_{t}(A) \Vert \\
    & \le \rho_0^2 \, \Vert \Delta_t(A) \Vert,
\end{align*}
where we used conditions~\ref{condition: Sigma theta constant bounds},\ref{condition: Xe bounding with lambda} of Lemma~\ref{lem:estimate_rough_bounds}. Hence, the inequality~\eqref{eq: contraction of Delta A} indeed holds.

\noindent \textbf{Step 4. Final bound.} Substituting~\eqref{eq: contraction of Delta A} into~\eqref{eq: Delta_t theta bound via Delta_t A}, we obtain
\begin{align*}
    \Vert \Delta_t (\btheta) \Vert \le \rho_0^{2(t - 1)} \Vert \widehat A - \widehat \Sigma \Vert \cdot \left [ \frac{\Vert \widehat \btheta \Vert}{\sqrt{\lambda}} + \frac{\Vert \widehat A  \widehat \btheta - \bZ \Vert}{2 \lambda} \right ]
\end{align*}
Using bounds~\eqref{eq: definition of argminimum bound of A theta - Z},\eqref{eq: theta hat bound via theta circ} and $\Vert \widehat A - \widehat \Sigma \Vert \le \sqrt{\lambda} / 3$ from the fact that $\widehat A \in \sfA$, we obtain
\begin{align*}
    \Vert \Delta_t(\btheta) \Vert \le \rho_0^{2(t - 1)} \left [ 2 \Vert \btheta^\circ \Vert + \frac{\Vert \bU \Vert}{3 \sqrt{\lambda}} \right ].
\end{align*}
Replacing $t$ in the inequality above with $T$ and bounding $\Vert \bU \Vert$ by~\eqref{eq: U bound in the proof of convergence theorem} complete the proof.

\myendproof

\end{document}